\documentclass[11pt,a4paper]{article}
\usepackage[T1]{fontenc}
\usepackage{lmodern}

\usepackage{iftex}
\RequireXeTeX
\newcommand{\siondedication}{%
  \begingroup
  \TeXXeTstate=1
  \font\hebrewdedication="[DejaVuSans.ttf]" at 10.5pt
  \hebrewdedication
  \beginR לזכרו של מוריס סיון ז״ל\endR
  \endgroup
}
\usepackage{microtype}
\usepackage{needspace}
\usepackage[margin=26mm,headheight=14pt,footskip=12mm]{geometry}
\usepackage{amsmath,amssymb,amsthm,mathtools,mathrsfs}
\numberwithin{equation}{section}
\usepackage{aliascnt}
\usepackage{hyperref}
\usepackage[nameinlink,noabbrev]{cleveref}
\hypersetup{hidelinks,pdftitle={Sion Extension in Net Convergence Groups and Pseudo-D-Lattices},pdfauthor={Alexandre Reggiolli Teixeira},pdfsubject={Net convergence, hereditary exhaustivity, bounded summation and measure extension}}
\allowdisplaybreaks[2]
\newtheorem{theorem}{Theorem}[section]
\newaliascnt{proposition}{theorem}
\newtheorem{proposition}[proposition]{Proposition}
\aliascntresetthe{proposition}
\newaliascnt{lemma}{theorem}
\newtheorem{lemma}[lemma]{Lemma}
\aliascntresetthe{lemma}
\newaliascnt{corollary}{theorem}
\newtheorem{corollary}[corollary]{Corollary}
\aliascntresetthe{corollary}
\theoremstyle{definition}
\newaliascnt{definition}{theorem}
\newtheorem{definition}[definition]{Definition}
\aliascntresetthe{definition}
\newtheorem*{definition*}{Definition}
\newaliascnt{example}{theorem}
\newtheorem{example}[example]{Example}
\aliascntresetthe{example}
\theoremstyle{remark}
\newaliascnt{remark}{theorem}
\newtheorem{remark}[remark]{Remark}
\aliascntresetthe{remark}
\newtheorem*{remark*}{Remark}
\crefname{theorem}{Theorem}{Theorems}
\crefname{proposition}{Proposition}{Propositions}
\crefname{lemma}{Lemma}{Lemmas}
\crefname{corollary}{Corollary}{Corollaries}
\crefname{definition}{Definition}{Definitions}
\crefname{example}{Example}{Examples}
\crefname{remark}{Remark}{Remarks}
\crefname{section}{Section}{Sections}
\newcommand{\N}{\mathbb N}

\newcommand{\HH}{\mathcal H}
\newcommand{\Hs}{\mathcal H_\sigma}
\newcommand{\KK}{\mathcal K}

\newcommand{\clo}{\operatorname{cl}}
\newcommand{\ev}{\mathsf E}

\newcommand{\alg}{\operatorname{alg}}
\newcommand{\qsub}{\preccurlyeq}
\newcommand{\tto}[1]{\xrightarrow{#1}}
\newcommand{\onull}{\xrightarrow{o}}

\newcommand{\symdiff}{\mathbin{\triangle}}

\title{Measure Theory in Convergence Spaces \\ \Large Part I: Sion-Carathéodory Extension in Net Convergence Groups}
\author{Alexandre Reggiolli Teixeira\\[3pt]
\small Departamento de Estat\'istica, Universidade Estadual de Campinas\\
\small Rua S\'ergio Buarque de Holanda, 651, Campinas--SP, Brazil\\
\small \texttt{riccimath123@gmail.com}, \texttt{a163407@dac.unicamp.br}}
\date{\siondedication}
\begin{document}
\maketitle
\begin{abstract}
We study extension of additive measures when the range carries a prescribed separated Abelian net convergence. On a $\sigma$-complete, lattice-$\sigma$-continuous pseudo-$D$-lattice, we prove an uncontrolled modular extension theorem using finite-response reflection, closed countable summation and uniform closure. Completeness is imposed only on the limit-stable hull of the source values. No topology, lattice order or positive control is required in its range. Modular interval identities, an order-null separation argument and two uniform transfer arguments construct the extension before establishing its normality. A separate relative lifting theorem combines this modular construction with a nonmodular Sasaki argument. Its cone conditions are verified for positive lattice-group values, lattice-normed groups, products and limit-stable subgroups. These verifications give explicit positive, dominated lattice-normed and topological corollaries. Uniqueness holds among all $\sigma$-additive extensions to the generated domain. Sion's inner and outer nets exist in the original convergence and recover the extension there. For Boolean domains we retain the stronger identification with the full regulated splitting class and the finite-variation and finite-semivariation results.
\end{abstract}
\clearpage
\tableofcontents
\clearpage

\section{Introduction}\label{sec:intro}
Our aim with this paper, the first part of a trilogy\footnote{The other two works concern convergence theorems for set functions and abstract integration.} on measure theory in convergence spaces, is to understand how general we can be in extending a given set function. The classical Carath\'eodory theorem~\cite[Chapters~II--III]{Halmos} gives this passage a familiar form: countable coverings produce an outer measure, and a splitting identity reveals the domain on which that outer measure is additive. At the heart of the argument lies a question that continues to accompany measure theory to this day, and in various levels of generality: how much of the measure on the enlarged domain is already carried by its original values, and what must convergence preserve for those values to reach their new domain?

Sion's classical and decisive construction~\cite{Sion} gives this question an answer in the form of two successive limits. Inner approximation gathers the values available from the source, and outer approximation carries them towards the events to be measured. We follow this process with convergence prescribed directly on an Abelian group, allowing the events themselves to form a pseudo-$D$-lattice. The passage brings the algebra of approximation into the foreground. Residuals record what remains between comparable events, and the possibility of extending a measure depends on how its values behave through these operations. Our contribution is to make this passage within the original net convergence, and to place the modular construction at the foundation of a common lifting theorem that also reaches positive measures without a modularity assumption. Thus, giving what can be considered, as far as we know, the most general extension of premeasures to date. We now comment on some details.

The uncontrolled extension theorem, Theorem~\ref{mod:main}, constructs a modular measure on the ambiently generated domain $\sigma_L(A)$ from its values on a $D$-sublattice $A$ of $L$. The group operations and the prescribed convergence provide the entire range structure. Completeness is required only on the least limit-stable hull of the source values. The central work is to carry the small errors supplied by hereditary exhaustivity through the residual calculus and into countable approximation. Modular interval identities control these errors through the algebraic operations. The convergence principles preserve the estimates as the approximations accumulate, until each event in the generated domain acquires a value independent of its representation. Countable additivity and ambient normality emerge from this construction. The extension is unique among all additive $\sigma$-additive extensions, and Sion's original inner and outer nets recover it on $\sigma_L(A)$.

The modular construction also opens a passage to measures whose values carry positivity. In Theorem~\ref{rel:lifting}, a continuous homomorphism connects such a measure to a modular image, and a limit-stable additive cone retains the information needed to lift the extension. A Sasaki recursion supplies hereditary estimates whose persistence at countable limit stages is a decisive issue. There, a contextual approximation argument carries normality across the preceding domains. The relative range conditions make the accumulated error estimates separate points in the kernel, so that the lifted values remain determined by the original measure. This is the technical connection through which the modular and nonmodular extension arguments become instances of one theorem.

The older extension theories give this connection its natural tests. The work of Avallone, De Simone and Vitolo~\cite{Avallone06,Avallone16} and Weber's general theorem~\cite[Corollary~6.11]{Weber18} provide the topological pseudo-$D$-lattice setting. Their extension conclusions return here through a direct verification of the intrinsic convergence conditions. In the ordered setting, the role of weak $\sigma$-distributivity has a long history in Wright's characterization~\cite{Wright} and Fremlin's integral argument~\cite{Fremlin75}. Together with super Dedekind completeness, it verifies the relative range conditions and yields our positive pseudo-$D$-lattice extension theorem. This theorem addresses the extension problem studied by Barbieri~\cite{Barbieri09,Barbieri14} on commutative $D$-lattices and carries it to nonmodular positive measures on pseudo-$D$-lattices. For a lattice-normed measure, a given positive source majorant accompanies the values as an additional coordinate. Under these lattice hypotheses and sequential bo-completeness, the graph construction verifies the lifting conditions, and the extended domination is a consequence of the extended measure.

The Boolean part of the paper retains a further freedom in the range. Its regulated extension uses a weaker reflection principle, admitting the complete non-topological example constructed below. Here the source approximations determine a full regulated $\sigma$-algebra, and the upper and lower approximation nets converge throughout that domain. Under the stronger hypotheses for the all-subset process, this domain is exactly Sion's splitting class. On a general pseudo-$D$-lattice, the two Sion processes exist on the ambient lattice and identify the extension on $\sigma_L(A)$. The subtleties of the measurable domain already appear in the study of Kalmbach measurability by d'Andrea, de Lucia and Maitland Wright~\cite{DAndreaDeLuciaWright94}.

Boolean approximation also allows the source bounds to accompany the extension in a more exact form. Volauf's positive extension results~\cite{Volauf77,Volauf80} return through a representation constructed from the original charge. In the lattice-normed theory, the finite-variation theorem of Kusraev and Malyugin~\cite{KM} leads naturally to the question of how much of the argument can be sustained by semivariation. We answer this question by representing the source values together with the positive variations of their dominated coordinates. The resulting additive lift preserves the exact semivariation on its full regulator domain, including examples with infinite vector variation. Weak $\sigma$-distributivity is needed only in the band generated by the source bound. The same construction determines the least isotone submajorant of the simple integral and proves its sequential order continuity. Extension thereby carries its own bounds into the first stage of integration.

The regularity of the extension records how the source fragments continue to approximate its values in the original convergence. This continuity between the source and the enlarged domain is the thread we follow through the paper. It also leads towards the convergence and integration questions of the subsequent parts of this work.

\section{Net convergence and the range principles}
\label{sec:nets}\label{sec:carriers}\label{gen:axioms}
We work in ZFC and use the net conventions of O'Brien, Troitsky
and van der Walt~\cite[Sections~2--3]{OTW}. The distinction between
adherence and closure is taken in the sense of Bilokopytov,
Conradie, Troitsky and van der Walt~\cite[Section~2]{BCTW}, whose
locally solid framework also encompasses topological and
non-topological order convergence.

\begin{definition}[Nets and quasi-subnets]
A \emph{directed domain} is a nonempty preordered set in which
every two elements have an upper bound. A \emph{net} in a set $X$
is a map $u:\Lambda\to X$ from a directed domain. Its value tail
at $\alpha\in\Lambda$ is $T_u(\alpha)=\{u_\beta:\beta\geq\alpha\}$.
For a net $v:\Gamma\to X$, the relation
\begin{equation}\label{eq:qsub}
 v\qsub u\quad\Longleftrightarrow\quad
 \forall\alpha\in\Lambda\ \exists\gamma\in\Gamma:\
 T_v(\gamma)\subseteq T_u(\alpha)
\end{equation}
means that $v$ is a \emph{quasi-subnet} of $u$. The nets are
\emph{tail equivalent}, written $u\sim v$, if $u\qsub v$ and
$v\qsub u$.
\end{definition}

\begin{definition}[Net convergence, separation and continuity]
A \emph{net convergence} $\theta$ on $X$ assigns limits to nets
in $X$, written $u\tto{\theta}x$, subject to the following axioms.
Every constant net with value $x$ converges to $x$. If
$u\tto{\theta}x$ and $v\qsub u$, then $v\tto{\theta}x$.
Finally, if $u:\Lambda\to X$ and $v:\Gamma\to X$ both converge
to $x$, their \emph{mixed net}
\[
 w_{(\alpha,\beta,i)}=
 \begin{cases}u_\alpha,&i=0,\\ v_\beta,&i=1,\end{cases}
 \qquad (\alpha,\beta,i)\in\Lambda\times\Gamma\times\{0,1\},
\]
converges to $x$. Here the domain is ordered by
$(\alpha,\beta,i)\leq(\alpha',\beta',i')$ if
$\alpha\leq\alpha'$ and $\beta\leq\beta'$, with no condition on
$i,i'$. This last axiom is called \emph{finite mixing}.
The pair $(X,\theta)$ is a \emph{net convergence space}.
It is \emph{separated}, or \emph{Hausdorff}, if every net has at
most one limit. It is $T_1$ if a constant net with value $x$ can
converge only to $x$. A map $f:(X,\theta)\to(Y,\eta)$ is
\emph{continuous} if $u\tto{\theta}x$ implies
$f\circ u\tto{\eta}f(x)$ for every net $u$ and every $x\in X$.
\end{definition}

The tails of the mixed net are $T_u(\alpha)\cup T_v(\beta)$.
Finite mixing is equivalent to the pointwise-selection axiom
in~\cite[Section~3]{OTW}. Indeed, for nets $u,v$ on a common
domain $\Lambda$, every net $z$ with
$z_\alpha\in\{u_\alpha,v_\alpha\}$ is a quasi-subnet of their
mixed net. For a mixed-net threshold $(\alpha_0,\beta_0,i_0)$,
an upper bound $\gamma_0\ge\alpha_0,\beta_0$ satisfies
\[
 T_z(\gamma_0)\subseteq T_u(\alpha_0)\cup T_v(\beta_0).
\]
Conversely, on
$\Lambda\times\Gamma\times\{0,1\}$ the nets with values
$u_\alpha$ and $v_\beta$ are tail equivalent to $u$ and $v$,
respectively. Pointwise selection between these two nets gives
$w$, and therefore gives finite mixing.

These conventions require only a set of representatives. Indeed,
for $\mathscr T=\{T_u(\alpha):\alpha\in\Lambda\}$, the domain
$I=\{(A,a):A\in\mathscr T,\ a\in A\}$, ordered by reverse
inclusion of $A$, carries the net $e(A,a)=a$ with
$T_e(A,a)=\bigcup_{B\in\mathscr T,\ B\subseteq A}B=A$.
The possible $\mathscr T$ form a subset of $\mathcal P(\mathcal P(X))$.
For finitely many simultaneous nets the same construction applied
to their tuple preserves every pointwise relation and all
coordinate tail families. An increasing, respectively decreasing,
net in a partially ordered set also has a monotone representative: on $\mathscr T$ use
$\min A$, respectively $\max A$, in reverse inclusion order.
For $A=T_u(\alpha)$ and $\beta\geq\alpha$,
$T_u(\beta)\subseteq A$ and
$u_\beta=\min T_u(\beta)$, respectively
$u_\beta=\max T_u(\beta)$. Thus every element of $A$ occurs
in the representative tail at $A$, and no element outside $A$ occurs.

\begin{definition}[Net convergence groups and completeness]
An \emph{abelian net convergence group} is an abelian group $X$
with a net convergence $\theta$ such that, for nets
$u:\Lambda\to X$ and $v:\Gamma\to X$,
\[
 u\tto{\theta}x,\quad v\tto{\theta}y
 \quad\Longrightarrow\quad
 u\boxplus v\tto{\theta}x+y,\qquad
 -u\tto{\theta}-x,
\]
where $u\boxplus v=(u_\alpha+v_\beta)_{(\alpha,\beta)}$ is
indexed by the coordinatewise directed product
$\Lambda\times\Gamma$, and $(-u)_\alpha=-u_\alpha$.
For $r\geq1$, write
$\Sigma_ru=(u_{\alpha_1}+\cdots+u_{\alpha_r})_{(\alpha_1,\ldots,\alpha_r)}$
on $\Lambda^r$. For subsets of $X$, sums are Minkowski sums,
$A+B=\{a+b:a\in A,\ b\in B\}$, with empty sum $\{0\}$.
A net $u$ is \emph{Cauchy} if
$u\boxplus(-u)\tto{\theta}0$. A subset $C\subseteq X$ is
\emph{complete} if every Cauchy net with values in $C$ converges
to a point of $C$.
\end{definition}

The Cauchy condition is the group form of the independent-index
definition in~\cite[Section~3]{OTW}. For convergence groups,
separation is equivalent to the $T_1$ condition. Separation
implies $T_1$ because constant nets converge to their values.
Conversely, suppose $T_1$ holds and $u\to x,y$. Then
$u\boxplus(-u)\to x-y$. Every two tails of $u$ intersect,
so every tail of $u\boxplus(-u)$ contains $0$. The constant
zero net is consequently its quasi-subnet and converges to
$x-y$. The $T_1$ condition gives $x-y=0$.
In the convergence statements below, $X$ denotes a separated Abelian
net convergence group with convergence $\theta$.

\begin{definition}[Evaluation nets and uniform convergence]
For a family of nonempty sets $B_\alpha\subseteq X$ indexed by
a directed domain $\Lambda$, its \emph{evaluation net} is
$\ev_\alpha(B_\alpha)(\alpha,b)=b$ on
$\{(\alpha,b):\alpha\in\Lambda,\ b\in B_\alpha\}$, ordered by
$(\alpha,b)\leq(\beta,c)$ if $\alpha\leq\beta$. Write
$(B_\alpha)_\alpha\Rightarrow0$ if
$\ev_\alpha(B_\alpha)\tto{\theta}0$, and call the family
\emph{null}. For a nonempty set $J$, a net of functions
$f_\alpha:J\to X$ converges \emph{uniformly} to $f:J\to X$ if
\[
 (\{f_\alpha(j)-f(j):j\in J\})_\alpha\Rightarrow0.
\]
It is \emph{uniformly Cauchy} if
\[
 (\{f_\alpha(j)-f_\beta(j):j\in J\})_{(\alpha,\beta)}
 \Rightarrow0,
\]
where the pair $(\alpha,\beta)$ ranges over
$\Lambda\times\Lambda$ with its coordinatewise order.
\end{definition}

The evaluation tails satisfy
\begin{equation}\label{eq:evaltail}
 T_{\ev(B)}(\alpha,b)=\bigcup_{\beta\geq\alpha}B_\beta.
\end{equation}
For decreasing families this is $B_\alpha$. Consequently eventual
membership in every $B_\alpha$ implies comparison with
$\ev_\alpha(B_\alpha)$ under \eqref{eq:qsub}. More explicitly, if
$v:\Gamma\to X$ satisfies
\[
 \forall\alpha\ \exists\gamma_\alpha\ \forall\gamma\ge\gamma_\alpha:
 v_\gamma\in B_\alpha,
\]
then $T_v(\gamma_\alpha)\subseteq
T_{\ev(B)}(\alpha,b)$ for every $b\in B_\alpha$.
Independent sums
satisfy
\begin{equation}\label{eq:algtails}
 T_{u\boxplus v}(\alpha,\beta)=T_u(\alpha)+T_v(\beta),\qquad
 T_{\Sigma_ru}(\alpha_1,\ldots,\alpha_r)=\sum_{j=1}^rT_u(\alpha_j).
\end{equation}
For decreasing families, evaluation of sums is tail equivalent
to the independent sum of the evaluations: at $(\alpha,\beta)$
both tails are exactly $B_\alpha+C_\beta$.
On a common directed domain, the tails of the diagonal evaluation
are $B_\gamma+C_\gamma$. Given $\alpha,\beta$, an upper bound
$\gamma\ge\alpha,\beta$ gives
$B_\gamma+C_\gamma\subseteq B_\alpha+C_\beta$.
The reverse tail comparison uses the product threshold $(\gamma,\gamma)$.
For null families on different domains, the family
$(B_\alpha\cup C_\beta)_{(\alpha,\beta)}$ has the tails of the
mixed evaluation net and is therefore null. On a common domain,
the preceding upper-bound argument gives nullity of
$(B_\alpha\cup C_\alpha)_\alpha$.
If a net $u_\alpha$ is reindexed on $\Lambda\times J$ without changing
its values, its tail at $(\alpha,j)$ is $T_u(\alpha)$.
Restriction to $\{\alpha:\alpha\ge\alpha_0\}$ has the following
tail comparisons. A restricted tail is an original tail, and for
an original threshold $\beta$, an upper bound
$\gamma\ge\beta,\alpha_0$ supplies a restricted tail contained in
$T_u(\beta)$. Applying these comparisons in both coordinates of the
independent difference net proves preservation of the Cauchy property.

\begin{definition}[Limit-stable sets and stable hulls]
A subset $C\subseteq X$ is \emph{limit-stable} if
$u:\Lambda\to C$ and $u\tto{\theta}x$ imply $x\in C$.
The \emph{stable hull} of $A\subseteq X$ is
\[
 [A]_\theta=\clo_\theta A
 =\bigcap\{C\subseteq X:A\subseteq C,\ C\text{ is limit-stable}\}.
\]
This is the closure convention of~\cite[Section~2]{BCTW}.
It is distinct from \emph{adherence}, the set of limits of nets
with values in $A$, which need not itself be limit-stable.
\end{definition}

\begin{lemma}[Stable hulls]\label{def:stablehull}
The stable hull is monotone, extensive and idempotent. For
$A,B\subseteq X$ and $x\in X$,
\begin{equation}\label{eq:hull-addition}
 [x+A]_\theta=x+[A]_\theta,\qquad [-A]_\theta=-[A]_\theta,
 \qquad [A]_\theta+[B]_\theta\subseteq[A+B]_\theta.
\end{equation}
\end{lemma}
\begin{proof}
For limit-stable sets $C_\xi$, a net with values in
$\bigcap_\xi C_\xi$ has every limit in every $C_\xi$. Thus
$[A]_\theta$ is limit-stable and
\[
 A\subseteq[A]_\theta,\qquad
 A\subseteq B\Longrightarrow[A]_\theta\subseteq[B]_\theta,
 \qquad [[A]_\theta]_\theta=[A]_\theta.
\]
Translation by $x$ and inversion preserve limits, and so do their
inverses. Consequently $C$ is stable if and only if $x+C$ is
stable, and if and only if $-C$ is stable. Applying these
bijections to the defining intersections proves the first two
identities in \eqref{eq:hull-addition}.

Put $K=[A+B]_\theta$. For every $b\in B$, the set $K-b$ is
stable and contains $A$, giving $[A]_\theta\subseteq K-b$.
For every $a\in[A]_\theta$, therefore, $B\subseteq K-a$.
Stability of $K-a$ gives $[B]_\theta\subseteq K-a$, and hence
\[
 a\in[A]_\theta,\ b\in[B]_\theta\Longrightarrow a+b\in K.
\]
\end{proof}

Stable hulls can also be described entirely through the prescribed convergence.
For $A\subseteq X$, write $A\vdash_\theta x$ if $x$
is the root value of a well-founded, set-sized tree whose leaves
have values in $A$ and whose every internal value is the
$\theta$-limit of a net of its child values. Then
\[
    A\vdash_\theta x
    \quad\Longleftrightarrow\quad
    x\in[A]_\theta.
\]
Indeed, well-founded induction places every derivable value in
every limit-stable set containing $A$. Conversely, the derivable
values contain $A$, by the one-node trees. If a net
$(x_\alpha)_{\alpha\in\Lambda}$ of derivable values converges to $x$,
fix one derivation tree $T_\alpha$ for each $x_\alpha$.
The disjoint union of the set-indexed family $(T_\alpha)_\alpha$,
with a new root whose children are their roots indexed by $\Lambda$,
is a set-sized tree. A strictly descending path below its new root
lies entirely in one $T_\alpha$, so the new tree is well founded.
It derives $x$. Thus the derivable values are limit-stable.
Thus every stable-hull assertion can be expressed through these
limit certificates. Requirements concerning families of hulls
retain the evaluation nets of all corresponding root values.
This reformulation preserves the distinction between adherence,
which records one passage to a limit, and closure, which allows
well-founded iterations. It changes neither the prescribed
convergence nor the hypotheses and conclusions of the theory.

\begin{definition}[Source and range data]\label{gen:premeasure}\label{gen:data}\label{gen:he}
Let $\mathbb B$ be a sigma-complete Boolean algebra, and let
$\mathcal H\subseteq\mathbb B$ be a Boolean ring. Write
$a\setminus b=a\wedge b^c$ and
\[
 \mathcal H_\sigma=\{\bigvee_na_n:a_n\in\mathcal H\},\qquad
 \mathcal H_\delta=\{u^c:u\in\mathcal H_\sigma\},\qquad
 1\in\mathcal H_\sigma.
\]
A premeasure $m:\mathcal H\to X$ is finitely additive, $m(0)=0$,
and satisfies $\sum_{n\leq N}m(a_n)\to m(a)$ whenever the $a_n$
are disjoint and $a=\bigvee_na_n\in\mathcal H$. Put
\[
\begin{gathered}
 M=[m(\mathcal H)]_\theta,\qquad M-M\subseteq D=-D,\\
 V_m(s)=\{m(a):a\in\mathcal H,\ a\leq s\},\qquad
 \Gamma_m(s)=[V_m(s)]_\theta.
\end{gathered}
\]
The range hypothesis is completeness of $M$. Uniform hereditary
exhaustivity, denoted (HE), is the condition
\begin{equation}\label{gen:HE}
 h_n\in\mathcal H\text{ disjoint}\quad\Longrightarrow\quad
 (V_m(h_n))_n\Rightarrow0. 
\end{equation}
\end{definition}
Thus $0\in\Gamma_m(s)\subseteq M\subseteq D$. No additive closure
or stability of $D$ is imposed. For $a_n\uparrow a$ in $\mathcal H$,
put $d_1=a_1$ and $d_n=a_n\setminus a_{n-1}$ for $n\ge2$.
These elements belong to $\mathcal H$, are pairwise disjoint, and satisfy
$\bigvee_{j\le n}d_j=a_n$ and $\bigvee_jd_j=a$.
Finite additivity gives $m(a_n)=\sum_{j\le n}m(d_j)$, and the
premeasure axiom gives $m(a_n)\to m(a)$.

\begin{definition}[The three range principles]\label{gen:block}\label{gen:summation}\label{gen:uniform}
Let $\mathfrak I$ be a collection of set-sized directed domains
containing $\mathbb N$. The notation $\mathsf S_{\rm seq}$ permits
only sequential rows. The notation $\mathsf S_{\rm all}$ permits every domain.
The following properties refer to the original convergence.

\emph{Block reflection}, $\mathsf B_\omega(D)$: if $0\in A_n$ and
$\sum_{n\in F}A_n\subseteq D$ for every finite $F\subseteq\mathbb N$,
then
\[
 \left[\forall\,1=k_0<k_1<\cdots:\
 \left(\sum_{k_{j-1}\leq n<k_j}A_n\right)_j\Rightarrow0\right]
 \ \Longrightarrow\
 \left(\bigcup_{F\subseteq\{N,N+1,\ldots\}\ \text{finite}}
                   \sum_{n\in F}A_n\right)_N\Rightarrow0.
\]
\emph{Closed countable summation}, $\mathsf S_{\mathfrak I}(D)$:
for decreasing rows $0\in U_{i,\alpha}\subseteq D$ with
$(U_{i,\alpha})_{\alpha\in I_i}\Rightarrow0$, $I_i\in\mathfrak I$,
put, on the coordinatewise directed product $P=\prod_iI_i$,
\begin{equation}\label{gen:Q-preliminary}
 A_\varphi=\bigcup_{F\subseteq\mathbb N\ \text{finite}}
                 \sum_{i\in F}U_{i,\varphi(i)},\qquad
 Q_\varphi=D\cap[A_\varphi]_\theta.
\end{equation}
The required conclusion is $(Q_\varphi)_{\varphi\in P}\Rightarrow0$.

\emph{Uniform closure}, $\mathsf U(M)$: for a nonempty set $J$, a
directed domain $B$, and points
$x_n(j),y_n(j),x_n^\beta(j),y_n^\beta(j)\in M$, the conditions
\[
 \forall\beta:\ (\{x_n^\beta(j)-y_n^\beta(j):j\in J\})_n\Rightarrow0,
\]
\[
 (\{x_n^\beta(j)-x_n(j):n\geq1,j\in J\})_\beta\Rightarrow0,
 \quad
 (\{y_n^\beta(j)-y_n(j):n\geq1,j\in J\})_\beta\Rightarrow0
\]
imply $(\{x_n(j)-y_n(j):j\in J\})_n\Rightarrow0$.
\end{definition}

\begin{remark*}[Motivation for the range principles]
These hypotheses abstract the control of approximation errors
afforded by zero-neighbourhoods and order bounds. Disjoint tails
must remain small under finite assembly, while countably many
vanishing errors must admit common control that survives passage
to limits. Uniform approximation must preserve these estimates.
Theorem~\ref{gen:extension} shows that the resulting
conditions suffice for extension. Their mathematical content is
established through the topological verification in
Corollary~\ref{class:topological} and the equivalence between
bounded sequential summation and weak sigma-distributivity in
Theorem~\ref{inv:wsd-local}. The classical specializations then
arise by verifying these conditions for the source data or its
additive representation. This gives a methodological parallel
with Carath\'eodory's measurability criterion: compatibility
conditions are formulated from the original data, and their force
becomes apparent in the measurable domain and additive extension
they produce.
\end{remark*}

We first use block reflection to strengthen hereditary exhaustivity
to a uniform estimate over finite disjoint unions. The next lemma
obtains this estimate using only the original source values.

\begin{lemma}[Uniform source tails]\label{gen:ut}\label{gen:ut-cauchy}
Let $m:\mathcal H\to X$ be finitely additive, with $m(0)=0$
and $m(\mathcal H)\subseteq D$. Under \textup{(HE)} and
$\mathsf B_\omega(D)$, for every disjoint
$(h_n)\subseteq\mathcal H$,
\[
 T_N=\{m(c):c\in\mathcal H,\ c\leq\bigvee_{n\in F}h_n,
            \ F\subseteq\{N,N+1,\ldots\}\text{ finite}\}
 \quad\text{satisfies}\quad (T_N)_N\Rightarrow0.
\]
Equivalently, $f_N(c)=m(c\wedge\bigvee_{n\leq N}h_n)$ is uniformly
Cauchy in $c\in\mathcal H$.
\end{lemma}
\begin{proof}
Let $F\subseteq\mathbb N$ be finite. For $c_n\in\mathcal H$ with
$c_n\leq h_n$, $n\in F$, the elements $c_n$ are disjoint and
$c=\bigvee_{n\in F}c_n\in\mathcal H$. Conversely, if
$c\in\mathcal H$ and $c\leq\bigvee_{n\in F}h_n$, then
$c=\bigvee_{n\in F}(c\wedge h_n)$. Finite additivity in these two
decompositions proves
\begin{equation}\label{gen:fragment-sum}
 \sum_{n\in F}V_m(h_n)=V_m\left(\bigvee_{n\in F}h_n\right)
 \subseteq m(\mathcal H)\subseteq D.
\end{equation}
The formula includes $F=\varnothing$, when both sides are
$\{0\}$. Put $A_n=V_m(h_n)$, so $0\in A_n$. For every
$1=k_0<k_1<\cdots$, the elements
\[
 g_j=\bigvee_{k_{j-1}\leq n<k_j}h_n\in\mathcal H
 \quad(j\geq1)
\]
are pairwise disjoint. By (HE) and \eqref{gen:fragment-sum},
\[
 \left(\sum_{k_{j-1}\leq n<k_j}A_n\right)_j
 =(V_m(g_j))_j\Rightarrow0.
\]
All premises of $\mathsf B_\omega(D)$ now hold. Its conclusion,
together with \eqref{gen:fragment-sum}, gives $(T_N)_N\Rightarrow0$.

Define $E_N=\{f_p(c)-f_q(c):p,q\geq N,\ c\in\mathcal H\}$.
For $p\geq q\geq N$, finite additivity gives
\[
 f_p(c)-f_q(c)=m\left(c\wedge\bigvee_{q<n\leq p}h_n\right)
 \in T_{N+1}.
\]
For $q\geq p\geq N$ the difference belongs to $-T_{N+1}$.
If $\varnothing\ne F\subseteq\{N+1,N+2,\ldots\}$ is finite and
$c\leq\bigvee_{n\in F}h_n$, then
$f_{\max F}(c)=m(c)$ and $f_N(c)=0$. The value associated with
$F=\varnothing$ is $0\in E_N$. Therefore
\[
 T_{N+1}\subseteq E_N\subseteq T_{N+1}\cup(-T_{N+1}).
\]
Inversion and finite mixing prove that nullity of $(T_N)_N$
implies nullity of $(E_N)_N$. The first inclusion proves the
converse. The decreasing family $(T_{N+1})_N$ has a tail contained
in $T_K$ whenever $N+1\ge K$, and every $T_{N+1}$ is a tail of
the unshifted family. For the final equivalence, the evaluation
of $\{f_p(c)-f_q(c):c\in\mathcal H\}$ has at $(P,Q)$ the tail
\[
 H_{P,Q}=\{f_p(c)-f_q(c):p\ge P,\ q\ge Q,\ c\in\mathcal H\}.
\]
Here $H_{N,N}=E_N$, and
$E_{\max(P,Q)}\subseteq H_{P,Q}$. Thus this evaluation and
$(E_N)_N$ are tail equivalent, proving uniform Cauchyness.
\end{proof}

\begin{lemma}[Hull control and countable assembly]\label{gen:relative}\label{gen:flatten}
Under $\mathsf S_{\mathfrak I}(D)$, every admissible decreasing null
row $0\in U_\alpha\subseteq D$ satisfies
$(D\cap[U_\alpha]_\theta)_\alpha\Rightarrow0$.
For independently labelled null rows $U_{n,i,\alpha}$, put
\[
 A_{n,\Phi}=\bigcup_{F\ \text{finite}}\sum_{i\in F}U_{n,i,\Phi(n,i)},
 \quad Q_{n,\Phi}=D\cap[A_{n,\Phi}]_\theta,
\]
\[
 A^*_\Phi=\bigcup_{E\subseteq\mathbb N^2\ \text{finite}}
            \sum_{(n,i)\in E}U_{n,i,\Phi(n,i)},
 \qquad Q^*_\Phi=D\cap[A^*_\Phi]_\theta.
\]
Then $(Q^*_\Phi)_\Phi\Rightarrow0$, and
\begin{equation}\label{gen:flatten-inclusion}
 D\cap\left[\bigcup_{F\ \text{finite}}\sum_{n\in F}Q_{n,\Phi}\right]_\theta
 \subseteq Q^*_\Phi.
\end{equation}
\end{lemma}
\begin{proof}
Suppose $U_\alpha$ is indexed by $I\in\mathfrak I$. In the
summation axiom set $I_1=I$, $U_{1,\alpha}=U_\alpha$ and, for
$i\geq2$, set $I_i=\mathbb N$, $U_{i,k}=\{0\}$. Then
\[
 A_\varphi=U_{\varphi(1)},\qquad
 Q_\varphi=D\cap[U_{\varphi(1)}]_\theta.
\]
For every product index $\varphi$, the union of the evaluation
values at indices $\psi\geq\varphi$ equals
$D\cap[U_{\varphi(1)}]_\theta$, by decrease of the row.
Every $\alpha\in I$ occurs as $\varphi(1)$, so these tails are
exactly the tails of $(D\cap[U_\alpha]_\theta)_\alpha$.
The one-row conclusion therefore holds.

Write $I_{n,i}$ for the domain of $U_{n,i,\alpha}$. A bijection
$e:\mathbb N\to\mathbb N^2$ induces the order isomorphism
\[
 \prod_{(n,i)\in\mathbb N^2}I_{n,i}
 \longrightarrow\prod_{r\in\mathbb N}I_{e(r)},
 \qquad \Phi\longmapsto(\Phi(e(r)))_r.
\]
Finite subsets of $\mathbb N$ correspond under $e$ to finite
subsets of $\mathbb N^2$. Hence the summation axiom for the rows
$U_{e(r),\alpha}$ gives $(Q^*_\Phi)_\Phi\Rightarrow0$.

Fix $\Phi$ and a finite $F\subseteq\mathbb N$. For each
$n\in F$, a term of $A_{n,\Phi}$ is a finite sum indexed by
some $F_n\subseteq\mathbb N$. Their sum is indexed by
$\bigcup_{n\in F}(\{n\}\times F_n)$, whose pairs are distinct.
Thus $\sum_{n\in F}A_{n,\Phi}\subseteq A^*_\Phi$ and
\[
 \sum_{n\in F}Q_{n,\Phi}
 \subseteq\sum_{n\in F}[A_{n,\Phi}]_\theta
 \subseteq\left[\sum_{n\in F}A_{n,\Phi}\right]_\theta
 \subseteq[A^*_\Phi]_\theta.
\]
The middle inclusion is the finite iteration of
\eqref{eq:hull-addition}. Taking the union over $F$ and then its
stable hull preserves containment in the stable set
$[A^*_\Phi]_\theta$. Intersection with $D$ proves
\eqref{gen:flatten-inclusion}.
\end{proof}

\begin{definition}[Finite-response and matrix reflection]\label{def:finite-response}\label{gen:matrixreflection}
Finite-response reflection on $Y\subseteq X$, denoted $\mathsf F(Y)$,
asserts the following.
For every net $u:\Lambda\to Y$, let a strategy assign an index
$\sigma(h)\in\Lambda$ to each finite legal history
$h=((\alpha_j,S_j))_{j\leq r}$, where
$\varnothing\ne S_j\subseteq T_u(\alpha_j)$ is finite. If a
strategy satisfies
\[
 \forall((\alpha_n,S_n))_{n\geq1}\text{ with }
 \alpha_n=\sigma((\alpha_j,S_j)_{j<n}):\quad(S_n)_n\Rightarrow0,
\]
then $u\tto{\theta}0$.
Matrix reflection $\mathsf R_\omega(D)$ asserts that
\[
 0\in A_{n,k}\subseteq D,\quad A_{n+1,k}\subseteq A_{n,k}
 \subseteq A_{n,k+1},\quad
 \forall\varphi\in\mathbb N^{\mathbb N}:\ (A_{n,\varphi(n)})_n\Rightarrow0
\]
implies $(\bigcup_kA_{n,k})_n\Rightarrow0$.
\end{definition}

\begin{proposition}[Reflection comparison]\label{gen:reflection-block}
Finite-response reflection on $D$ implies $\mathsf R_\omega(D)$,
which implies $\mathsf B_\omega(D)$.
\end{proposition}
\begin{proof}
Assume finite-response reflection on $D$. Given the matrix in
$\mathsf R_\omega(D)$, put $V_n=\bigcup_kA_{n,k}$ and
$u=\ev_n(V_n)$. Since $V_{n+1}\subseteq V_n$ and $0\in V_n$,
\[
 (n,0)\in\operatorname{dom}(u),\qquad T_u(n,0)=V_n.
\]
For every legal history of length $r$, define
$\sigma(h)=(r+1,0)$. In every play of this strategy,
$\varnothing\ne S_n\subseteq V_n$ is finite. The increasing
family $(A_{n,k})_k$ covers $V_n$, so the integer
\[
 \varphi(n)=\min\{k\geq1:S_n\subseteq A_{n,k}\}
\]
exists. The matrix premise gives
$(A_{n,\varphi(n)})_n\Rightarrow0$. Since
$\bigcup_{n\geq N}S_n\subseteq\bigcup_{n\geq N}A_{n,\varphi(n)}$
for every $N$, its evaluation has the response evaluation as a
quasi-subnet. Thus $(S_n)_n\Rightarrow0$ for every play.
Finite-response reflection gives $u\to0$.

Assume next $\mathsf R_\omega(D)$, and let $(A_n)$ satisfy the
premises of $\mathsf B_\omega(D)$. Put
\[
 E_{N,L}=\begin{cases}\displaystyle\sum_{N\leq n\leq L}A_n,&L\geq N,\\
 \{0\},&L<N.
 \end{cases}
\]
The finite-sum bound and $0\in A_n$ give
$0\in E_{N,L}\subseteq D$ and
$E_{N+1,L}\subseteq E_{N,L}\subseteq E_{N,L+1}$.
Fix $\varphi\in\mathbb N^{\mathbb N}$. Define $k_0=1$, $k_1=2$
and, for $j\geq1$,
\[
 k_{j+1}=1+\max\bigl(\{k_j\}\cup\{\varphi(n):1\leq n<k_j\}\bigr),
 \qquad B_j=\sum_{k_{j-1}\leq n<k_j}A_n.
\]
Then $k_{j+1}>k_j$, so the block premise gives
$(B_j)_j\Rightarrow0$. Moreover,
\[
 \bigcup_{j\geq J}(B_j+B_{j+1})
 \subseteq\left(\bigcup_{j\geq J}B_j\right)
             +\left(\bigcup_{j\geq J}B_j\right),
\]
which proves $(B_j+B_{j+1})_j\Rightarrow0$ by independent
addition and quasi-subnet stability. If $k_{j-1}\leq N<k_j$,
then $\varphi(N)<k_{j+1}$, and insertion of zeros gives
$E_{N,\varphi(N)}\subseteq B_j+B_{j+1}$. Consequently
\[
 \bigcup_{N\geq k_{J-1}}E_{N,\varphi(N)}
 \subseteq\bigcup_{j\geq J}(B_j+B_{j+1})\qquad(J\geq1).
\]
Thus $(E_{N,\varphi(N)})_N\Rightarrow0$ for every $\varphi$.
Matrix reflection gives $(\bigcup_LE_{N,L})_N\Rightarrow0$.
Finally, insertion of zeros and the finite maximum of every
nonempty finite $F$ give
\[
 \bigcup_LE_{N,L}
 =\bigcup_{F\subseteq\{N,N+1,\ldots\}\ \mathrm{finite}}
      \sum_{n\in F}A_n,
\]
including the empty sum $\{0\}$. This is block reflection.
\end{proof}

\begin{proposition}[Restriction]\label{gen:inheritance}
For $M_0\subseteq M$ and $0\in D_0=-D_0\subseteq D$, block,
matrix and finite-response reflection restrict to $D_0$.
Moreover, $\mathsf S_{\mathfrak I}(D)$ implies $\mathsf S_{\mathfrak I}(D_0)$,
and $\mathsf U(M)$ implies $\mathsf U(M_0)$. A stable subset of a
complete set is complete.
\end{proposition}
\begin{proof}
For block reflection, every finite-sum inclusion
$\sum_{n\in F}A_n\subseteq D_0$ also holds with $D$ in place of
$D_0$. The block premise and conclusion involve the same sets
$A_n$. For matrix reflection, $A_{n,k}\subseteq D_0$ implies
$A_{n,k}\subseteq D$, with the same monotonicities, selected
families and row unions. The respective principles on $D$
therefore give both principles on $D_0$.
For finite-response reflection, a net $u:\Lambda\to D_0$ is also
a net into $D$, with identical tails, legal histories, strategies
and response evaluations. Every strategy satisfying the premise
on $D_0$ satisfies it on $D$, yielding $u\to0$.

Given admissible rows in $D_0$, form their finite-sum sets
$A_\varphi$. They are admissible rows in $D$, and
\[
 \bigcup_{\psi\geq\varphi}(D_0\cap[A_\psi]_\theta)
 \subseteq\bigcup_{\psi\geq\varphi}(D\cap[A_\psi]_\theta).
\]
Both families contain $0$. Hence their evaluations satisfy the
quasi-subnet relation, and $\mathsf S_{\mathfrak I}(D)$ gives
$\mathsf S_{\mathfrak I}(D_0)$.
For $\mathsf U(M_0)$, its four point families belong to $M$.
All three premises and the conclusion of $\mathsf U(M)$ use
those same point differences. Applying $\mathsf U(M)$ proves
$\mathsf U(M_0)$.

Let $C_0\subseteq C$, with $C$ complete and $C_0$ limit-stable.
For every Cauchy net $u:\Lambda\to C_0$, completeness gives
$x\in C$ with $u\to x$. Limit stability of $C_0$ gives
$x\in C_0$. Thus $C_0$ is complete.
\end{proof}

\section{Pseudo-\texorpdfstring{$D$}{D}-lattices and sequential generation}
\label{pd:section}

We first establish the source calculus used in both extension
arguments. All orders, joins, meets, and order limits in this section
belong to the source. No order on a group of values is involved.

\begin{definition}\label{pd:definition}
A pseudo-effect algebra is a partial algebra $(L,+,0,1)$ with the
following properties. The expressions $(a+b)+c$ and $a+(b+c)$ exist
under exactly the same conditions and then agree. Every $a$ has unique
complements $a^-$ and $a^\sim$ satisfying
\[
 a^-+a=1,\qquad a+a^\sim=1.
\]
If $a+b$ exists, there are $d,e\in L$ such that
$a+b=d+a=b+e$. Existence of either $1+a$ or $a+1$ implies $a=0$.
The induced order is
\[
 a\le b\quad\Longleftrightarrow\quad
 a+c=b\text{ for some }c\in L.
\]
A pseudo-$D$-lattice is a pseudo-effect algebra whose induced order
is a lattice. A $D$-sublattice contains $0,1$ and is closed under
finite meets and joins, both residuals of comparable elements, and
every ambient sum of two of its elements.
\end{definition}

We use the cancellation and residual calculus of
pseudo-effect algebras from \cite{DvurecenskijVetterlein01}, in the
interval formulation of \cite[Propositions~1.2--1.3]{AvalloneVitolo12Uniformities}.
Our conventions are
\begin{equation}\label{pd:eq-residual-conventions}
 (b\backslash a)+a=b,\qquad a+(a/b)=b
 \qquad(a\le b).
\end{equation}
In particular, the slash in $a/b$ records the smaller argument first.
The basic interval facts are recorded explicitly because the order
of the summands will matter.

\begin{lemma}[Residual and translation calculus]
\label{pd:residual-calculus}
For a fixed $b$, the maps
\[
 a\longmapsto b\backslash a,\qquad
 a\longmapsto a/b
\]
are mutually inverse order anti-isomorphisms of $[0,b]$. Left
translation by $a$ is an order isomorphism
\[
 [0,a^\sim]\longrightarrow[a,1],\qquad t\longmapsto a+t,
\]
with inverse $b\mapsto a/b$. Right translation by $a$ is an order
isomorphism
\[
 [0,a^-]\longrightarrow[a,1],\qquad t\longmapsto t+a,
\]
with inverse $b\mapsto b\backslash a$. These maps preserve every
existing supremum and infimum with the indicated order polarity.
For $a\le b\le c$,
\begin{equation}\label{pd:eq-residual-chains}
 c\backslash a=(c\backslash b)+(b\backslash a),
 \qquad
 a/c=(a/b)+(b/c).
\end{equation}
Shrinking either coordinate of a defined sum preserves its existence
and decreases its value.
\end{lemma}

\begin{proof}
The cancellation law and the existence and uniqueness of the
residuals are the algebraic facts imported from the references
preceding the lemma. For $a\le b\le c$,
\[
 \bigl((c\backslash b)+(b\backslash a)\bigr)+a
 =(c\backslash b)+\bigl((b\backslash a)+a\bigr)
 =(c\backslash b)+b=c,
\]
and
\[
 a+\bigl((a/b)+(b/c)\bigr)
 =\bigl(a+(a/b)\bigr)+(b/c)=b+(b/c)=c.
\]
Residual uniqueness proves \eqref{pd:eq-residual-chains}. The
defining sums also give
\[
 (b\backslash a)/b=a,\qquad b\backslash(a/b)=a.
\]
For $a\le a'\le b$,
\[
 b\backslash a=(b\backslash a')+(a'\backslash a)
 \quad\Longrightarrow\quad b\backslash a'\le b\backslash a,
\]
\[
 a/b=(a/a')+(a'/b)
 \quad\Longrightarrow\quad a'/b\le a/b.
\]
Thus the two sectional maps are inverse order anti-isomorphisms.

If $c=t+a$ exists, then
\[
 c^-+(t+a)=1,\qquad (c^-+t)+a=1,\qquad
 c^-+t=a^-,\qquad t\le a^-.
\]
Conversely, $t\le a^-$ gives $a^-=r+t$ for
$r=a^-\backslash t$, and
\[
 (r+t)+a=a^-+a=1,\qquad r+(t+a)=1.
\]
Thus $t+a$ exists. If $c=a+t$ exists, then
\[
 (a+t)+c^\sim=1,\qquad a+(t+c^\sim)=1,\qquad
 t+c^\sim=a^\sim,\qquad t\le a^\sim.
\]
Conversely, $t\le a^\sim$ gives $a^\sim=t+r$ for
$r=t/a^\sim$, and
\[
 a+(t+r)=a+a^\sim=1,\qquad (a+t)+r=1.
\]
Thus $a+t$ exists. The residual inverses satisfy
\[
 (t+a)\backslash a=t,\qquad (b\backslash a)+a=b,
 \qquad
 a/(a+t)=t,\qquad a+(a/b)=b.
\]
If $u\le v$ and $v+a$ exists, a residual representation $v=r+u$
gives
\[
 v+a=(r+u)+a=r+(u+a),\qquad u+a\le v+a.
\]
If $u+a\le v+a$, let $r=(v+a)\backslash(u+a)$. Then
\[
 r+(u+a)=v+a,\qquad (r+u)+a=v+a,\qquad r+u=v.
\]
Hence $u\le v$. For left translation, the representation
$v=u+(u/v)$ gives
\[
 a+v=(a+u)+(u/v),\qquad a+u\le a+v.
\]
If $a+u\le a+v$ and $r=(a+u)/(a+v)$, then
\[
 (a+u)+r=a+v,\qquad a+(u+r)=a+v,\qquad u+r=v.
\]
This proves both order-isomorphism assertions. If $u\le U$,
$v\le V$, and $U+V$ exists, the two translation inequalities give
\[
 u+v\le u+V\le U+V.
\]

For an order isomorphism $\tau:I\to J$ and a nonempty family
$B\subseteq I$ with $s=\sup_I B$,
\[
 \tau(b)\le\tau(s)\quad(b\in B),
\]
\[
 \bigl[\tau(b)\le z\text{ for every }b\in B\bigr]
 \ \Longrightarrow\
 \bigl[b\le\tau^{-1}(z)\text{ for every }b\in B\bigr]
 \ \Longrightarrow\ s\le\tau^{-1}(z)
 \ \Longrightarrow\ \tau(s)\le z.
\]
Thus $\tau(s)=\sup_J\tau(B)$. Replacing upper bounds by lower
bounds proves $\tau(\inf_I B)=\inf_J\tau(B)$. For an order
anti-isomorphism $\rho:I\to J$, the corresponding implications are
\[
 \rho(s)\le\rho(b)\quad(b\in B),\qquad
 \bigl[z\le\rho(b)\text{ for all }b\in B\bigr]
 \Longrightarrow
 \bigl[b\le\rho^{-1}(z)\text{ for all }b\in B\bigr]
 \Longrightarrow z\le\rho(s).
\]
Therefore $\rho(\sup_I B)=\inf_J\rho(B)$, and
$\rho(\inf_I B)=\sup_J\rho(B)$. These formulas apply to each
translation and sectional complement above.
\end{proof}

Consequently, for $p,q\le c$,
\begin{equation}\label{pd:eq-sectional-demorgan}
 c\backslash(p\vee q)
  =(c\backslash p)\wedge(c\backslash q),
 \qquad
 (p\wedge q)/c=(p/c)\vee(q/c).
\end{equation}
The other sectional De Morgan identities follow by applying the
inverse sectional map. The same statements hold for every existing
family of extrema. We will also use the opposite pseudo-effect
algebra, whose addition is $a\mathbin{\boxplus}b=b+a$. It has the
same order and residuals
\begin{equation}\label{pd:eq-opposite}
 b\backslash_{\rm op}a=a/b,\qquad
 a/_{\rm op}b=b\backslash a.
\end{equation}

From now on $L$ is $\sigma$-complete and lattice-$\sigma$-continuous.
The latter condition means that $x_n\uparrow x$ implies
$x_n\wedge y\uparrow x\wedge y$ for every $y\in L$.
We write $x_n\onull x$ for sequential order convergence, meaning
that there are $p_n\uparrow x$ and $q_n\downarrow x$ with
$p_n\le x_n\le q_n$.

\begin{lemma}[Lattice continuity]\label{pd:sigma-continuity}
Both finite lattice operations preserve sequential order
convergence jointly. In particular,
\begin{align}
 x_n\uparrow x,\ y_n\uparrow y
   &\ \Longrightarrow\ x_n\wedge y_n\uparrow x\wedge y,
   \label{pd:eq-diagonal-meet}\\
 x_n\downarrow x,\ y_n\downarrow y
   &\ \Longrightarrow\ x_n\vee y_n\downarrow x\vee y.
   \label{pd:eq-diagonal-join}
\end{align}
Every order-convergent sequence has the canonical brackets
\begin{equation}\label{pd:eq-canonical-brackets}
 p_N=\bigwedge_{n\ge N}x_n\uparrow x,\qquad
 q_N=\bigvee_{n\ge N}x_n\downarrow x.
\end{equation}
\end{lemma}

\begin{proof}
For $z_n\downarrow z$, the complement identities and the assumed
increasing-meet continuity give
\[
 z_n^-\uparrow z^-,
 \qquad
 (z_n\vee c)^-=z_n^-\wedge c^-\uparrow z^-\wedge c^-=(z\vee c)^-.
\]
Applying the inverse complement $t\mapsto t^\sim$ yields
$z_n\vee c\downarrow z\vee c$. The remaining fixed-coordinate
formulas are
\[
 \bigvee_n(z_n\vee c)=\left(\bigvee_nz_n\right)\vee c
       \quad(z_n\uparrow),
 \qquad
 \bigwedge_n(z_n\wedge c)=\left(\bigwedge_nz_n\right)\wedge c
       \quad(z_n\downarrow).
\]
For the first identity, a common upper bound of all $z_n\vee c$
is an upper bound of every $z_n$ and of $c$. For the second,
a common lower bound is below every $z_n$ and below $c$.

Suppose $x_n\uparrow x$ and $y_n\uparrow y$. For
$d=\bigvee_n(x_n\wedge y_n)$ and each $k$,
\[
 x_k\wedge y_n\le x_n\wedge y_n\le d\quad(n\ge k),
 \qquad
 x_k\wedge y=\bigvee_{n\ge k}(x_k\wedge y_n)\le d,
\]
\[
 x\wedge y=\bigvee_k(x_k\wedge y)\le d\le x\wedge y.
\]
This proves \eqref{pd:eq-diagonal-meet}. Also, if
$d=\bigvee_n(x_n\vee y_n)$, then
\[
 x=\bigvee_nx_n\le d,\qquad y=\bigvee_ny_n\le d,
 \qquad x\vee y\le d\le x\vee y.
\]
Thus $x_n\vee y_n\uparrow x\vee y$.

Suppose $x_n\downarrow x$ and $y_n\downarrow y$. The preceding
increasing-meet formula gives
\[
 (x_n\vee y_n)^-=x_n^-\wedge y_n^-
       \uparrow x^-\wedge y^-=(x\vee y)^-,
\]
hence \eqref{pd:eq-diagonal-join}. If
$d=\bigwedge_n(x_n\wedge y_n)$, then
\[
 d\le\bigwedge_nx_n=x,\qquad d\le\bigwedge_ny_n=y,
 \qquad x\wedge y\le d\le x\wedge y.
\]
Thus $x_n\wedge y_n\downarrow x\wedge y$.

For $x_n\onull x$ and $y_n\onull y$, let
\[
 a_n\uparrow x,\quad b_n\downarrow x,\quad
 c_n\uparrow y,\quad d_n\downarrow y,\qquad
 a_n\le x_n\le b_n,\quad c_n\le y_n\le d_n.
\]
For each $\circ\in\{\wedge,\vee\}$,
\[
 a_n\circ c_n\le x_n\circ y_n\le b_n\circ d_n,\qquad
 a_n\circ c_n\uparrow x\circ y,\quad
 b_n\circ d_n\downarrow x\circ y.
\]
These are the required order brackets.

Finally, suppose $a_n\uparrow x$, $b_n\downarrow x$, and
$a_n\le x_n\le b_n$. For the canonical brackets and every $N$,
\[
 a_N\le\bigwedge_{n\ge N}x_n=p_N,\qquad
 q_N=\bigvee_{n\ge N}x_n\le b_N.
\]
For every $k\ge N$, $p_N\le x_k\le b_k$, whence
$p_N\le\bigwedge_{k\ge N}b_k=x$. Also
$q_N\ge x_k\ge a_k$, whence
$q_N\ge\bigvee_{k\ge N}a_k=x$. Therefore
\[
 a_N\le p_N\le x\le q_N\le b_N,\qquad
 p_N\le p_{N+1},\quad q_{N+1}\le q_N,
\]
\[
 x=\bigvee_Na_N\le\bigvee_Np_N\le x,\qquad
 x\le\bigwedge_Nq_N\le\bigwedge_Nb_N=x.
\]
\end{proof}

\begin{definition}\label{pd:normality-definition}
Let $E$ be a $D$-sublattice of $L$, and let $X$ be a separated
Abelian convergence group with its prescribed net convergence.
A map $f:E\to X$ is additive if $f(0)=0$ and
$f(a+b)=f(a)+f(b)$ for every defined sum in $E$.
It is modular if
\[
 f(a\vee b)+f(a\wedge b)=f(a)+f(b).
\]
An ordered orthogonal sequence $(a_n)$ has defined partial sums
$s_N=a_1+\cdots+a_N$. Intrinsic $\sigma$-additivity requires
\[
 f(s_N)\longrightarrow f(s)
\]
whenever $s=\sup_E s_N$ exists. Ambient $\sigma$-additivity requires
this identity when $s=\sup_Ls_N$ belongs to $E$.
Ambient normality means
\begin{equation}\label{pd:eq-ambient-normality}
 c_n\in E,\quad c_n\downarrow_L0
 \quad\Longrightarrow\quad f(c_n)\longrightarrow0.
\end{equation}
Intrinsic normality uses the infimum in $E$ instead.
\end{definition}

\begin{lemma}[Normality and additivity]\label{pd:ambient-normality}
For an additive map into a separated Abelian convergence group,
intrinsic $\sigma$-additivity is equivalent to intrinsic normality.
The ambient versions are also equivalent. Either form of normality
implies continuity at every monotone sequential limit of the
corresponding kind that belongs to the domain. Intrinsic normality
implies ambient normality.
\end{lemma}

\begin{proof}
Let $Q=E$ for the intrinsic assertions and $Q=L$ for the ambient
assertions. The restricted sectional maps are anti-isomorphisms
of the intervals in $E$, because both residuals belong to $E$.
The restricted translations are isomorphisms between the
corresponding intervals in $E$. Thus the extremum formulas of
Lemma~\ref{pd:residual-calculus} hold with $\sup_Q$ and $\inf_Q$.

Suppose $c_n\in E$, $c_n\downarrow_Q0$, and sigma-additivity
holds for orthosums computed in $Q$ that belong to $E$. Put
$d_n=c_n\backslash c_{n+1}$. The ordered sums satisfy
\[
 D_1=d_1=c_1\backslash c_2,\qquad D_1+c_2=c_1.
\]
If $D_N+c_{N+1}=c_1$, then
\[
 D_N+(d_{N+1}+c_{N+2})=D_N+c_{N+1}=c_1,
\]
\[
 (D_N+d_{N+1})+c_{N+2}=c_1,\qquad
 D_{N+1}=D_N+d_{N+1}=c_1\backslash c_{N+2}.
\]
Induction proves the existence of every $D_N$ and
$D_N=c_1\backslash c_{N+1}$. Moreover,
\[
 \sup_QD_N
 =c_1\backslash\left(\inf_Qc_{N+1}\right)
 =c_1\backslash0=c_1.
\]
Sigma-additivity and group subtraction now give
\[
 f(D_N)\longrightarrow f(c_1),\qquad
 f(c_{N+1})=f(c_1)-f(D_N)\longrightarrow0.
\]

Conversely, assume normality at zero for infima computed in $Q$.
If $x_n\uparrow_Qx\in E$, then
\[
 r_n=x_n/x\in E,\quad r_{n+1}\le r_n,\qquad
 \inf_Qr_n=(\sup_Qx_n)/x=x/x=0,
\]
\[
 x_n+r_n=x,\qquad
 f(x_n)=f(x)-f(r_n)\longrightarrow f(x).
\]
If $x_n\downarrow_Qx\in E$, then inverse right translation gives
\[
 r_n=x_n\backslash x\in E,\quad r_{n+1}\le r_n,\qquad
 \inf_Qr_n=(\inf_Qx_n)\backslash x=x\backslash x=0,
\]
\[
 r_n+x=x_n,\qquad
 f(x_n)=f(r_n)+f(x)\longrightarrow f(x).
\]
For an ordered orthogonal sequence with partial sums
$s_N\uparrow_Qs\in E$, these increasing-limit formulas give
\[
 \sum_{j=1}^Nf(a_j)=f(s_N)\longrightarrow f(s),
\]
which is sigma-additivity.

If $c_n\in E$ and $\inf_Lc_n=0$, then
\[
 \{d\in E:d\le c_n\text{ for all }n\}=\{0\}.
\]
Indeed every such $d$ satisfies $d\le\inf_Lc_n=0$.
Hence $\inf_Ec_n=0$, and intrinsic normality implies
\eqref{pd:eq-ambient-normality}.
\end{proof}

\subsection{Sectional Sasaki transport}

For $x,y\in L$, define
\begin{equation}\label{pd:eq-sasaki-definition}
 c=x\vee y,\qquad u=x/c,\qquad v=u\vee y,\qquad
 T_x(y)=v\backslash u,
\end{equation}
and put
\begin{equation}\label{pd:eq-kernel-definition}
 K_x(y)=x\wedge((x\vee y)\backslash y).
\end{equation}
The defining sums are
\begin{equation}\label{pd:eq-sasaki-sums}
 x+u=c,\qquad T_x(y)+u=v.
\end{equation}
Since $v\le c$, right translation reflects order and gives
$T_x(y)\le x$.

\begin{lemma}[Kernel identities and order test]\label{pd:sasaki-kernel}
For all $x,y\in L$,
\begin{equation}\label{pd:eq-sasaki-kernel}
 K_x(y)+T_x(y)=x,\qquad
 T_x(y)=K_x(y)/x,\qquad K_x(y)=x\backslash T_x(y).
\end{equation}
If $z\le x$ and $r=x\backslash z$, then
\begin{equation}\label{pd:eq-sasaki-order-test}
 T_x(y)\le z
 \quad\Longleftrightarrow\quad
 r+y\text{ exists and }r+y\le x\vee y.
\end{equation}
\end{lemma}

\begin{proof}
Let $c=x\vee y$, $u=x/c$, $v=u\vee y$, $t=v\backslash u$,
and $k=c\backslash v$. Then
\[
 (k+t)+u=k+(t+u)=k+v=c=x+u,
 \qquad k+t=x.
\]
Also $x+u=c$ implies $c\backslash u=x$, and hence
\[
 k=c\backslash(u\vee y)
   =(c\backslash u)\wedge(c\backslash y)
   =x\wedge((x\vee y)\backslash y)=K_x(y).
\]
Thus
\[
 K_x(y)+T_x(y)=x,\qquad
 K_x(y)=x\backslash T_x(y),\qquad
 T_x(y)=K_x(y)/x.
\]
For $z\le x$ and $r=x\backslash z$, we have $r\le x$ and
\[
 \begin{aligned}
 T_x(y)\le z
 &\ \Longleftrightarrow\
 x\backslash z\le x\backslash T_x(y)\\
 &\ \Longleftrightarrow\
 r\le x\wedge((x\vee y)\backslash y)\\
 &\ \Longleftrightarrow\
 r\le(x\vee y)\backslash y\\
 &\ \Longleftrightarrow\
 r+y\text{ exists and }r+y\le x\vee y.
 \end{aligned}
\]
\end{proof}

\begin{lemma}\label{pd:disjoint-sum}
If $a+b$ exists and $a\wedge b=0$, then $a+b=a\vee b$.
\end{lemma}

\begin{proof}
Set $c=a+b$, $j=a\vee b$, and $d=c\backslash j$.
Since $a,b\le j\le c$,
\[
 d+j=c,\qquad d+b\le c=a+b,\qquad d+a\le c.
\]
Right translation in the middle inequality gives $d\le a$.
Let $e$ satisfy $d+a=a+e$, and let $t=d/a$. Then
\[
 d+t=a,\qquad
 d+(d+t)=(d+t)+e=d+(t+e),\qquad
 d+t=t+e=a.
\]
The pseudo-effect algebra axiom applied to $t+e$ supplies $h$
with $t+e=e+h$. Thus $e\le a$. Left translation in
\[
 a+e=d+a\le c=a+b
\]
gives $e\le b$. Therefore
\[
 0\le e\le a\wedge b=0,\qquad d+a=a=0+a,\qquad d=0,
 \qquad c=d+j=j.
\]
\end{proof}

\begin{theorem}[Monotonicity and preservation of joins]
\label{pd:sasaki-joins}
For fixed $x$, the map $T_x$ is increasing and preserves every
existing nonempty join. Moreover, $T_x(d)=d$ whenever $d\le x$.
\end{theorem}

\begin{proof}
For $y\le Y$, define
\[
 h=x\vee y,\quad C=x\vee Y,\quad
 \alpha=C\backslash h,\quad q=h\backslash y,\quad
 \beta=C\backslash Y,\quad r=x\wedge\beta.
\]
We have $h\vee Y=C$, and therefore
\[
 \alpha\wedge\beta
 =(C\backslash h)\wedge(C\backslash Y)
 =C\backslash(h\vee Y)=0,
 \qquad \alpha\wedge r=0.
\]
The inequalities $r\le x\le h$ and the identity $\alpha+h=C$
give existence of $\alpha+r$. Lemma~\ref{pd:disjoint-sum} yields
$\alpha+r=\alpha\vee r$. Furthermore,
\[
 C\backslash y=(C\backslash h)+(h\backslash y)=\alpha+q,
 \qquad \alpha\le\alpha+q,\qquad
 r\le\beta\le C\backslash y=\alpha+q.
\]
Hence
\[
 \alpha+r=\alpha\vee r\le\alpha+q,\qquad r\le q,
 \qquad K_x(Y)=r\le x\wedge q=K_x(y).
\]
Sectional antitonicity gives
\[
 T_x(y)=K_x(y)/x\le K_x(Y)/x=T_x(Y).
\]

Let $Y=\bigvee_{i\in I}y_i$, with $I\ne\varnothing$.
Then $T_x(y_i)\le T_x(Y)$ for all $i$. If
$T_x(y_i)\le z$ for all $i$, define $z_0=z\wedge x$ and
$r=x\backslash z_0$. For every $i$,
\[
 T_x(y_i)\le z_0,\qquad
 r+y_i\text{ exists},\qquad
 r+y_i\le x\vee y_i\le x\vee Y
\]
by \eqref{pd:eq-sasaki-order-test}. Thus
\[
 y_i\le r^\sim\quad(i\in I),\qquad Y\le r^\sim,\qquad
 r+Y=\bigvee_{i\in I}(r+y_i)\le x\vee Y.
\]
The order test now gives $T_x(Y)\le z_0\le z$. This proves
\[
 T_x\!\left(\bigvee_{i\in I}y_i\right)
       =\bigvee_{i\in I}T_x(y_i).
\]
For $d\le x$,
\[
 x\vee d=x,\qquad x/(x\vee d)=x/x=0,\qquad
 T_x(d)=(0\vee d)\backslash0=d.
\]
\end{proof}

\begin{theorem}[One-sided absorption]\label{pd:sasaki-absorption}
If $d\le x$, then
\[
 T_d(T_x(y))\le T_d(y).
\]
\end{theorem}

\begin{proof}
Let
\[
 D=d\vee y,\quad U=x\vee y=x\vee D,\quad
 a=T_x(y),\quad h=T_x(D)=T_x(d)\vee T_x(y)=d\vee a,
\]
and set
\[
 k=x\backslash a=K_x(y),\quad
 k'=x\backslash h=K_x(D),\quad
 s=h\backslash a,\quad r=K_d(y).
\]
Since $a\le h\le x$,
\[
 k=(x\backslash h)+(h\backslash a)=k'+s.
\]
Put $A_0=U\backslash D$ and $B_0=D\backslash y$. Then
\[
 U\backslash y=A_0+B_0,\qquad
 k'=x\wedge(U\backslash D)\le A_0,\qquad
 r=d\wedge(D\backslash y)\le B_0.
\]
Heredity gives $k'+r\le U\backslash y$.
Also $r\le d\le h$ and $k'+h=x$, whence
\[
 k'+r\le x,\qquad
 k'+r\le x\wedge(U\backslash y)=k=k'+s,\qquad r\le s.
\]
Therefore
\[
 r+a\text{ exists},\qquad
 r+a\le s+a=h=d\vee a.
\]
Finally $r=K_d(y)=d\backslash T_d(y)$, so
\eqref{pd:eq-sasaki-order-test} with context $d$, input $a$,
and bound $T_d(y)$ gives
\[
 T_d(T_x(y))=T_d(a)\le T_d(y).
\]
\end{proof}

\begin{lemma}[Fixed-point transfer]\label{pd:sasaki-fixed-point}
If $e\le T_x(y)$, then $T_e(y)=e$.
\end{lemma}

\begin{proof}
Let $u=x/(x\vee y)$, $v=u\vee y$, $t=v\backslash u$,
and $q=(e\vee y)\backslash y$. Since $e\le t\le v$ and $y\le v$,
\[
 e\vee y\le v,\qquad
 q\le v\backslash y.
\]
Sectional De Morgan gives
\[
 t\wedge(v\backslash y)
 =(v\backslash u)\wedge(v\backslash y)
 =v\backslash(u\vee y)=v\backslash v=0.
\]
Consequently,
\[
 0\le K_e(y)=e\wedge q
 \le t\wedge(v\backslash y)=0,\qquad
 T_e(y)=K_e(y)/e=0/e=e.
\]
\end{proof}

\begin{theorem}[Fixed and varying contexts]\label{pd:sasaki-disappearance}
Suppose $w_n\downarrow0$ in $L$ and $y_n\le w_n$.
For each fixed $x\in L$, there are $b_n\downarrow0$ with
$T_x(y_n)\le b_n$. If
\[
 x_{n+1}=T_{x_n}(y_n),
\]
then $x_n\downarrow0$.
\end{theorem}

\begin{proof}
Fix $x$, and set
\[
 C_n=x\vee w_n,\qquad a_n=x/C_n,\qquad b_n=a_n\vee w_n.
\]
The lattice and translation formulas give
\[
 C_n\downarrow x\vee0=x,\qquad
 a_n\downarrow x/x=0,\qquad b_n\downarrow0\vee0=0.
\]
For $c_n=x\vee y_n$ and $u_n=x/c_n$, we have
\[
 c_n\le C_n,\qquad u_n\le a_n,\qquad
 T_x(y_n)+u_n=u_n\vee y_n,
\]
\[
 T_x(y_n)\le u_n\vee y_n\le a_n\vee w_n=b_n.
\]

Now suppose $x_{n+1}=T_{x_n}(y_n)$. The defining bound
$T_z(y)\le z$ gives $x_{n+1}\le x_n$.
Let $e=\bigwedge_nx_n$. For every $n$,
\[
 e\le x_{n+1}=T_{x_n}(y_n),\qquad T_e(y_n)=e
\]
by Lemma~\ref{pd:sasaki-fixed-point}.
The fixed-context construction with $x=e$ gives
\[
 e=T_e(y_n)\le
 \bigl(e/(e\vee w_n)\bigr)\vee w_n\downarrow0.
\]
Thus $e=0$ and $x_n\downarrow0$.
\end{proof}

\begin{proposition}[Crossed estimate]\label{pd:crossed-estimate}
Let $P$ be an additive cone in an Abelian group $X$, with
$v\preceq w$ meaning $w-v\in P$. If $f:E\to P$ is additive
on a $D$-sublattice, then
\[
 f(x)+f(y)-f(x\vee y)\preceq f(T_x(y))
 \qquad(x,y\in E).
\]
The cone need not be pointed, and $f$ need not be modular.
\end{proposition}

\begin{proof}
If $a\le b$ in $E$, then
\[
 f(b)-f(a)=f(a/b)\in P,\qquad f(a)\preceq f(b).
\]
Let $c=x\vee y$, $u=x/c$, and $v=u\vee y$. All these
elements and $T_x(y)$ belong to $E$. Their sums give
\[
 f(c)=f(x)+f(u),\qquad f(v)=f(T_x(y))+f(u).
\]
Since $y\le v$,
\[
 f(v)-f(y)\in P,\qquad
 f(T_x(y))-\bigl(f(x)+f(y)-f(c)\bigr)
 =f(v)-f(y)\in P.
\]
\end{proof}

\subsection{Total residuals and sequential closure}

\begin{theorem}[Total residuals]\label{pd:total-residuals}
The operations
\[
 \Delta_L(y,x)=y\wedge((y\vee x)\backslash x),\qquad
 \Delta_R(y,x)=y\wedge(x/(y\vee x))
\]
are increasing in $y$, decreasing in $x$, and jointly sequentially
order-continuous. For $x\le y$,
\begin{equation}\label{pd:eq-total-recovery}
 \Delta_L(y,x)=y\backslash x,\qquad
 \Delta_R(y,x)=x/y.
\end{equation}
For either operation $\Delta$, all four monotone-limit formulas hold:
\begin{align*}
 y_n\uparrow y&\ \Longrightarrow\
     \Delta(y_n,x)\uparrow\Delta(y,x),&
 y_n\downarrow y&\ \Longrightarrow\
     \Delta(y_n,x)\downarrow\Delta(y,x),\\
 x_n\uparrow x&\ \Longrightarrow\
     \Delta(y,x_n)\downarrow\Delta(y,x),&
 x_n\downarrow x&\ \Longrightarrow\
     \Delta(y,x_n)\uparrow\Delta(y,x).
\end{align*}
\end{theorem}

\begin{proof}
For $x\le y$, both $y\backslash x$ and $x/y$ are below $y$.
Hence
\[
 \Delta_L(y,x)=y\wedge(y\backslash x)=y\backslash x,\qquad
 \Delta_R(y,x)=y\wedge(x/y)=x/y.
\]
If $y\le Y$ and $r=\Delta_L(y,x)$, then
\[
 r\le y\le Y,\qquad
 r+x\le y\vee x\le Y\vee x,\qquad
 r\le(Y\vee x)\backslash x.
\]
Thus $r\le\Delta_L(Y,x)$. For $x\le X$,
\[
 \Delta_L(y,x)=K_y(x)=y\backslash T_y(x),\qquad
 T_y(x)\le T_y(X),
\]
\[
 \Delta_L(y,X)=y\backslash T_y(X)
       \le y\backslash T_y(x)=\Delta_L(y,x).
\]

If $y_n\uparrow y$, put $c_n=y_n\vee x$ and $c=y\vee x$.
Then
\[
 c_n\uparrow c,\qquad c_n\backslash x\uparrow c\backslash x,
 \qquad
 \Delta_L(y_n,x)=y_n\wedge(c_n\backslash x)
       \uparrow y\wedge(c\backslash x)=\Delta_L(y,x).
\]
If $y_n\downarrow y$, the formulas are
\[
 c_n=y_n\vee x\downarrow c=y\vee x,\qquad
 c_n\backslash x\downarrow c\backslash x,
\]
\[
 \Delta_L(y_n,x)=y_n\wedge(c_n\backslash x)
       \downarrow y\wedge(c\backslash x)=\Delta_L(y,x).
\]
For $x_n\uparrow x$, Theorem~\ref{pd:sasaki-joins} gives
\[
 T_y(x_n)\uparrow T_y(x),\qquad
 \Delta_L(y,x_n)=y\backslash T_y(x_n)
       \downarrow y\backslash T_y(x)=\Delta_L(y,x).
\]

For $x_n\downarrow x$, define
\[
 c_n=y\vee x_n\downarrow c=y\vee x,\qquad
 u_n=y/c_n\downarrow u=y/c,
\]
\[
 v_n=u_n\vee x_n\downarrow v=u\vee x,\qquad
 t_n=T_y(x_n)\downarrow t=\bigwedge_nT_y(x_n).
\]
Monotonicity gives $T_y(x)\le t$. Since $t_n+u_n=v_n$,
\[
 t\le t_n,\quad u\le u_n
 \quad\Longrightarrow\quad
 t+u\text{ exists and }t+u\le v_n\quad(n\ge1).
\]
Therefore
\[
 t+u\le\bigwedge_nv_n=v,\qquad
 t\le v\backslash u=T_y(x)\le t.
\]
It follows that
\[
 T_y(x_n)\downarrow T_y(x),\qquad
 \Delta_L(y,x_n)=y\backslash T_y(x_n)
       \uparrow y\backslash T_y(x)=\Delta_L(y,x).
\]

For the opposite algebra, associativity is the implication
\[
 (a\boxplus b)\boxplus c=c+(b+a)
       =(c+b)+a=a\boxplus(b\boxplus c).
\]
Its complements satisfy
\[
 a^\sim\boxplus a=a+a^\sim=1,\qquad
 a\boxplus a^-=a^-+a=1,
\]
so $a^-_{\rm op}=a^\sim$ and $a^\sim_{\rm op}=a^-$.
If $a\boxplus b=b+a$ exists, the original third axiom supplies
$d,e$ with
\[
 b+a=a+d=e+b=d\boxplus a=b\boxplus e.
\]
If $1\boxplus a=a+1$ or $a\boxplus1=1+a$ exists, the
original last axiom gives $a=0$.
Finally,
\[
 a\le_{\rm op}b
 \ \Longleftrightarrow\
 c+a=b\text{ for some }c
 \ \Longleftrightarrow\ a\le b.
\]
The second equivalence uses the third axiom to interchange
the position of $a$ in a defined sum. Formula
\eqref{pd:eq-opposite} gives
\[
 \Delta_L^{\rm op}(y,x)
 =y\wedge((y\vee x)\backslash_{\rm op}x)
 =y\wedge(x/(y\vee x))=\Delta_R(y,x).
\]
Thus each of the preceding polarity implications and each of
the four limit identities for $\Delta_L$ in the opposite
algebra is precisely its stated counterpart for $\Delta_R$
in $L$.

For the joint assertion, let
\[
 a_n\uparrow y,\quad b_n\downarrow y,\quad
 c_n\uparrow x,\quad d_n\downarrow x,\qquad
 a_n\le y_n\le b_n,\quad c_n\le x_n\le d_n.
\]
For either $\Delta$,
\[
 L_n:=\Delta(a_n,d_n)\le\Delta(y_n,x_n)
       \le U_n:=\Delta(b_n,c_n),
 \qquad L_n\uparrow,\quad U_n\downarrow.
\]
Also $L_n\le\Delta(y,x)\le U_n$. For $n\ge k$,
\[
 \Delta(a_k,d_n)\le L_n,\qquad
 U_n\le\Delta(b_k,c_n).
\]
The already proved fixed-coordinate limits give
\[
 \Delta(a_k,x)=\bigvee_{n\ge k}\Delta(a_k,d_n)
       \le\bigvee_nL_n,
 \qquad
 \bigwedge_nU_n
       \le\bigwedge_{n\ge k}\Delta(b_k,c_n)=\Delta(b_k,x).
\]
Taking the supremum or infimum in $k$,
\[
 \Delta(y,x)=\bigvee_k\Delta(a_k,x)
       \le\bigvee_nL_n\le\Delta(y,x),
\]
\[
 \Delta(y,x)\le\bigwedge_nU_n
       \le\bigwedge_k\Delta(b_k,x)=\Delta(y,x).
\]
Hence $L_n\uparrow\Delta(y,x)$ and
$U_n\downarrow\Delta(y,x)$.
\end{proof}

For a $D$-sublattice $E$, define
\[
 O(E)=\{x\in L:\text{some sequence }e_n\in E
                         \text{ satisfies }e_n\onull x\}.
\]

\begin{proposition}[Order-hull closure]\label{pd:order-hull}
The set $O(E)$ is a $D$-sublattice of $L$.
\end{proposition}

\begin{proof}
For $a\in E$, the constant row $a_n=a$ proves $a\in O(E)$.
Let $a_n,b_n\in E$ satisfy $a_n\onull a$ and $b_n\onull b$.
Lemma~\ref{pd:sigma-continuity} gives
\[
 a_n\wedge b_n\onull a\wedge b,\qquad
 a_n\vee b_n\onull a\vee b.
\]
Since both rows are $E$-valued, $a\wedge b,a\vee b\in O(E)$.
Theorem~\ref{pd:total-residuals} gives
\[
 \Delta_L(b_n,a_n)\onull\Delta_L(b,a),\qquad
 \Delta_R(b_n,a_n)\onull\Delta_R(b,a).
\]
If $a\le b$, their limits are $b\backslash a$ and $a/b$.
Thus both actual residuals belong to $O(E)$. In particular,
\[
 a^-=1\backslash a\in O(E),\qquad a^\sim=a/1\in O(E).
\]
If $a+b=c$ exists with $a,b\in O(E)$, then
\[
 b\le a^\sim,\qquad r=b/a^\sim\in O(E),\qquad b+r=a^\sim,
\]
\[
 c+r=(a+b)+r=a+(b+r)=a+a^\sim=1,\qquad
 r=c^\sim,\qquad c=r^-\in O(E).
\]
Every closure required of a $D$-sublattice has now been proved.
\end{proof}

For the initial $D$-sublattice $A$, write
\[
 A^\sigma=\{\bigvee_na_n:a_n\in A,\ a_n\uparrow\},\qquad
 A^\delta=\{\bigwedge_na_n:a_n\in A,\ a_n\downarrow\}.
\]

\begin{lemma}[Mixed sequential hulls]\label{pd:mixed-hulls}
The class $A^\sigma$ is closed under countable joins and finite
meets. The class $A^\delta$ is closed under countable meets and
finite joins. Both complements interchange these two classes.
For $\Delta=\Delta_L$ or $\Delta_R$,
\[
 \Delta(A^\sigma,A^\delta)\subseteq A^\sigma,\qquad
 \Delta(A^\delta,A^\sigma)\subseteq A^\delta.
\]
In particular, if $\ell\in A^\delta$, $u\in A^\sigma$, and
$\ell\le u$, then both $\ell/u$ and $u\backslash\ell$ belong
to $A^\sigma$.
\end{lemma}

\begin{proof}
Let $U_i=\bigvee_ka_{i,k}$ with $a_{i,k}\in A$ and
$a_{i,k}\uparrow_k$. Define
\[
 b_n=\bigvee_{i=1}^na_{i,n}\in A,\qquad U=\bigvee_iU_i.
\]
Then
\[
 b_n\le b_{n+1}\le U,\qquad
 a_{i,k}\le a_{i,n}\le b_n\quad(n\ge\max\{i,k\}).
\]
Consequently
\[
 U_i=\bigvee_ka_{i,k}\le\bigvee_nb_n\quad(i\ge1),\qquad
 U\le\bigvee_nb_n\le U.
\]
Thus $b_n\uparrow U$, and $U\in A^\sigma$.
For $a_n\uparrow U$ and $b_n\uparrow V$ in $A$,
\[
 a_n\wedge b_n\uparrow U\wedge V
\]
by \eqref{pd:eq-diagonal-meet}, so $U\wedge V\in A^\sigma$.

For $\ell_i=\bigwedge_kb_{i,k}$ with $b_{i,k}\in A$ and
$b_{i,k}\downarrow_k$, define
\[
 c_n=\bigwedge_{i=1}^nb_{i,n}\in A,\qquad
 \ell=\bigwedge_i\ell_i.
\]
Then
\[
 \ell\le c_{n+1}\le c_n,\qquad
 c_n\le b_{i,n}\le b_{i,k}\quad(n\ge\max\{i,k\}),
\]
\[
 \bigwedge_nc_n\le\bigwedge_kb_{i,k}=\ell_i\quad(i\ge1),
 \qquad
 \ell\le\bigwedge_nc_n\le\bigwedge_i\ell_i=\ell.
\]
Thus $c_n\downarrow\ell$, and $\ell\in A^\delta$.
For $a_n\downarrow\ell$ and $b_n\downarrow r$ in $A$,
\[
 a_n\vee b_n\downarrow\ell\vee r
\]
by \eqref{pd:eq-diagonal-join}, so $\ell\vee r\in A^\delta$.

The complement representations are
\[
 \begin{array}{ll}
 a_n\uparrow U:
   &a_n^-\downarrow U^-,\quad a_n^\sim\downarrow U^\sim,\\
 b_n\downarrow\ell:
   &b_n^-\uparrow\ell^-,\quad b_n^\sim\uparrow\ell^\sim.
 \end{array}
\]
All four rows belong to $A$. Hence
\[
 (A^\sigma)^-,(A^\sigma)^\sim\subseteq A^\delta,\qquad
 (A^\delta)^-,(A^\delta)^\sim\subseteq A^\sigma.
\]
The identities $(x^-)^\sim=x=(x^\sim)^-$ turn these
inclusions into the asserted interchange of classes.

For $a_n\uparrow U$ and $b_n\downarrow\ell$, both total
residuals satisfy
\[
 \Delta(a_n,b_n)\le\Delta(a_{n+1},b_{n+1}),\qquad
 \Delta(a_n,b_n)\onull\Delta(U,\ell).
\]
Thus $\Delta(a_n,b_n)\uparrow\Delta(U,\ell)$ and
$\Delta(U,\ell)\in A^\sigma$.
For $a_n\downarrow\ell$ and $b_n\uparrow U$,
\[
 \Delta(a_{n+1},b_{n+1})\le\Delta(a_n,b_n),\qquad
 \Delta(a_n,b_n)\onull\Delta(\ell,U),
\]
so $\Delta(\ell,U)\in A^\delta$.
When $\ell\le u$,
\[
 \ell/u=\Delta_R(u,\ell)\in A^\sigma,\qquad
 u\backslash\ell=\Delta_L(u,\ell)\in A^\sigma.
\]
\end{proof}

\begin{lemma}[Finite-expression brackets]\label{pd:word-brackets}
Let $P$ be a finite expression formed from $\wedge,\vee$, both
complements, the total residuals, and the constants $0,1$.
Let $U_i\in A^\sigma$ and $a_{i,n}\uparrow U_i$ in $A$,
for $1\le i\le r$. Put
$x=P(U_1,\ldots,U_r)$ and
$p_n=P(a_{1,n},\ldots,a_{r,n})$.
There are
\[
 \ell_n\in A^\delta,\quad u_n\in A^\sigma,\quad
 \ell_n\le x,p_n\le u_n,\quad
 \ell_n\uparrow x,\quad u_n\downarrow x.
\]
For every $k\ge n$ there are source endpoints
$\ell_{n,k},u_{n,k}\in A$ with
\[
 \ell_{n,k}\downarrow_k\ell_n,\qquad
 u_{n,k}\uparrow_k u_n,\qquad
 \ell_{n,k}\le u_{n,k}.
\]
Consequently,
\[
 \ell_{n,k}/u_{n,k}\uparrow_k\ell_n/u_n\in A^\sigma,
 \qquad \ell_n/u_n\downarrow_n0.
\]
\end{lemma}

\begin{proof}
For vectors $\boldsymbol a=(a_1,\ldots,a_r)$ and
$\boldsymbol b=(b_1,\ldots,b_r)$ with $a_i\le b_i$, define
$L_P(\boldsymbol a,\boldsymbol b)$ and
$U_P(\boldsymbol a,\boldsymbol b)$ recursively. With the arguments
suppressed on the right, the complete recursion is
\[
\begin{array}{c|cc}
 P&L_P&U_P\\ \hline
 X_i&a_i&b_i\\
 0&0&0\\
 1&1&1\\
 Q\wedge R&L_Q\wedge L_R&U_Q\wedge U_R\\
 Q\vee R&L_Q\vee L_R&U_Q\vee U_R\\
 Q^-&(U_Q)^-&(L_Q)^-\\
 Q^\sim&(U_Q)^\sim&(L_Q)^\sim\\
 \Delta_L(Q,R)&\Delta_L(L_Q,U_R)&\Delta_L(U_Q,L_R)\\
 \Delta_R(Q,R)&\Delta_R(L_Q,U_R)&\Delta_R(U_Q,L_R).
\end{array}                                                    \tag{*}
\]
For every $\boldsymbol t$ with $a_i\le t_i\le b_i$, we claim
\begin{equation}\label{pd:eq-word-enclosure}
 L_P(\boldsymbol a,\boldsymbol b)
 \le P(\boldsymbol t)
 \le U_P(\boldsymbol a,\boldsymbol b).
\end{equation}
The rows for variables and constants prove the initial cases.
If the claim holds for $Q,R$, write
$q=Q(\boldsymbol t)$, $r=R(\boldsymbol t)$.
For $\circ=\wedge,\vee$,
\[
 L_Q\circ L_R\le q\circ r\le U_Q\circ U_R.
\]
For either complement $\kappa\in\{-,\sim\}$,
\[
 (U_Q)^\kappa\le q^\kappa\le(L_Q)^\kappa.
\]
For $\Delta=\Delta_L,\Delta_R$,
\[
 \Delta(L_Q,U_R)
 \le\Delta(q,U_R)\le\Delta(q,r)
 \le\Delta(U_Q,r)\le\Delta(U_Q,L_R).
\]
These are every inductive case in \((*)\).

The recursion also has the interval-inclusion property
\begin{equation}\label{pd:eq-word-inclusion}
 \boldsymbol a\le\boldsymbol a'\le\boldsymbol b'
       \le\boldsymbol b
 \ \Longrightarrow\
 L_P(\boldsymbol a,\boldsymbol b)
 \le L_P(\boldsymbol a',\boldsymbol b')
 \le U_P(\boldsymbol a',\boldsymbol b')
 \le U_P(\boldsymbol a,\boldsymbol b).
\end{equation}
Here vector inequalities are coordinatewise. For the induction,
write $l_Q\le l'_Q\le u'_Q\le u_Q$ and
$l_R\le l'_R\le u'_R\le u_R$ for the already established
subexpression bounds. Meets and joins satisfy
\[
 l_Q\circ l_R\le l'_Q\circ l'_R,\qquad
 u'_Q\circ u'_R\le u_Q\circ u_R.
\]
Complements satisfy
\[
 (u_Q)^\kappa\le(u'_Q)^\kappa,\qquad
 (l'_Q)^\kappa\le(l_Q)^\kappa.
\]
For a total residual the two bounds are
\[
 \Delta(l_Q,u_R)
 \le\Delta(l'_Q,u_R)\le\Delta(l'_Q,u'_R),
\]
\[
 \Delta(u'_Q,l'_R)
 \le\Delta(u_Q,l'_R)\le\Delta(u_Q,l_R).
\]
The middle comparison in \eqref{pd:eq-word-inclusion} follows
from \eqref{pd:eq-word-enclosure}. This proves the inclusion
property for every expression.

For later limits, the expressions $L_P$ and $U_P$ are jointly
sequentially order-continuous in their $2r$ arguments.
Indeed the variables and constants have this property.
For any already constructed rows
$l_{Q,j}\onull l_Q$, $u_{Q,j}\onull u_Q$,
$l_{R,j}\onull l_R$, $u_{R,j}\onull u_R$, the recursive rows
have the following explicit limits:
\[
\begin{array}{c|cc}
 &\hbox{lower row and limit}&\hbox{upper row and limit}\\ \hline
 \wedge&
 l_{Q,j}\wedge l_{R,j}\onull l_Q\wedge l_R&
 u_{Q,j}\wedge u_{R,j}\onull u_Q\wedge u_R\\
 \vee&
 l_{Q,j}\vee l_{R,j}\onull l_Q\vee l_R&
 u_{Q,j}\vee u_{R,j}\onull u_Q\vee u_R\\
 \kappa&
 (u_{Q,j})^\kappa\onull(u_Q)^\kappa&
 (l_{Q,j})^\kappa\onull(l_Q)^\kappa\\
 \Delta&
 \Delta(l_{Q,j},u_{R,j})\onull\Delta(l_Q,u_R)&
 \Delta(u_{Q,j},l_{R,j})\onull\Delta(u_Q,l_R).
\end{array}
\]
The first two rows are Lemma~\ref{pd:sigma-continuity}.
For the complement row, brackets
$v_j\uparrow v$, $w_j\downarrow v$ give
$w_j^\kappa\uparrow v^\kappa$ and
$v_j^\kappa\downarrow v^\kappa$.
The last row is Theorem~\ref{pd:total-residuals}.
Structural induction on \((*)\) proves the assertion.

Let $\boldsymbol a_n=(a_{1,n},\ldots,a_{r,n})$ and
$\boldsymbol U=(U_1,\ldots,U_r)$, and define
\[
 \ell_n=L_P(\boldsymbol a_n,\boldsymbol U),\qquad
 u_n=U_P(\boldsymbol a_n,\boldsymbol U).
\]
The membership induction in \((*)\) uses the following
inclusions from Lemma~\ref{pd:mixed-hulls}:
\[
\begin{array}{c|cc}
 &\hbox{lower endpoints}&\hbox{upper endpoints}\\ \hline
 \wedge,\vee&A^\delta\circ A^\delta\subseteq A^\delta&
              A^\sigma\circ A^\sigma\subseteq A^\sigma\\
 \kappa&(A^\sigma)^\kappa\subseteq A^\delta&
           (A^\delta)^\kappa\subseteq A^\sigma\\
 \Delta&\Delta(A^\delta,A^\sigma)\subseteq A^\delta&
           \Delta(A^\sigma,A^\delta)\subseteq A^\sigma .
\end{array}
\]
The initial lower endpoints lie in $A\subseteq A^\delta$,
and the initial upper endpoints lie in $A^\sigma$.
Thus $\ell_n\in A^\delta$, $u_n\in A^\sigma$.
Equations \eqref{pd:eq-word-enclosure} and
\eqref{pd:eq-word-inclusion} give
\[
 \ell_n\le x,p_n\le u_n,\qquad
 \ell_n\le\ell_{n+1},\qquad u_{n+1}\le u_n.
\]
At equal endpoint vectors, induction in \((*)\) gives
$L_P(\boldsymbol U,\boldsymbol U)
 =P(\boldsymbol U)
 =U_P(\boldsymbol U,\boldsymbol U)$.
The order-continuity assertion just proved therefore gives
\[
 \ell_n\onull L_P(\boldsymbol U,\boldsymbol U)=x,\qquad
 u_n\onull U_P(\boldsymbol U,\boldsymbol U)=x.
\]
Together with monotonicity, these are $\ell_n\uparrow x$
and $u_n\downarrow x$.

For fixed $n$ and $k\ge n$, define
\[
 \ell_{n,k}=L_P(\boldsymbol a_n,\boldsymbol a_k),\qquad
 u_{n,k}=U_P(\boldsymbol a_n,\boldsymbol a_k).
\]
Every operation in \((*)\) preserves $A$, so
$\ell_{n,k},u_{n,k}\in A$. Since
$\boldsymbol a_n\le\boldsymbol a_k\le\boldsymbol a_{k+1}$,
\eqref{pd:eq-word-inclusion} gives
\[
 \ell_{n,k+1}\le\ell_{n,k}\le u_{n,k}\le u_{n,k+1}.
\]
Continuity of the same two expressions in their second vector
gives
\[
 \ell_{n,k}\onull L_P(\boldsymbol a_n,\boldsymbol U)=\ell_n,
 \qquad
 u_{n,k}\onull U_P(\boldsymbol a_n,\boldsymbol U)=u_n.
\]
Thus $\ell_{n,k}\downarrow_k\ell_n$ and
$u_{n,k}\uparrow_k u_n$. Finally,
\[
 \ell_{n,k}/u_{n,k}
 =\Delta_R(u_{n,k},\ell_{n,k})
 \uparrow_k\Delta_R(u_n,\ell_n)=\ell_n/u_n,
\]
\[
 \ell_n/u_n=\Delta_R(u_n,\ell_n)
 \downarrow_n\Delta_R(x,x)=x/x=0.
\]
The first row consists of elements of $A$, so
$\ell_n/u_n\in A^\sigma$.
\end{proof}

Every element of the $D$-sublattice generated by a set has a finite
expression of the preceding kind. Indeed, actual residuals are
recovered by the total residuals. A defined sum is recovered from
them and complementation by
\[
 a+b=\bigl(b/a^\sim\bigr)^-.
\]

\begin{theorem}[Sequential generation]\label{pd:transfinite-hull}
Let $E_0=A$, let $E_{\alpha+1}=O(E_\alpha)$, and at a nonzero
limit ordinal let $E_\lambda=\bigcup_{\alpha<\lambda}E_\alpha$.
Then
\[
 \sigma_L(A)=\bigcup_{\alpha<\omega_1}E_\alpha,
\]
where $\sigma_L(A)$ is the least $D$-sublattice containing $A$
and closed under ambient countable joins and meets.
\end{theorem}

\begin{proof}
The inclusions $E_\alpha\subseteq E_{\alpha+1}$ follow from
constant representing sequences. Transfinite induction gives
$E_\alpha\subseteq E_\beta$ for $\alpha\le\beta$.
Proposition~\ref{pd:order-hull} proves that a successor stage
is a $D$-sublattice when its predecessor is one.

At a nonzero limit $\lambda$, any finite list
$x_1,\ldots,x_r\in E_\lambda$ has ordinals
$\alpha_i<\lambda$ with $x_i\in E_{\alpha_i}$.
For $\beta=\max_{1\le i\le r}\alpha_i<\lambda$,
\[
 x_1,\ldots,x_r\in E_\beta.
\]
Their finite meets, joins, complements, comparable residuals,
and defined sums belong to $E_\beta\subseteq E_\lambda$.
Also $0,1\in E_0\subseteq E_\lambda$.
This proves by transfinite induction that every $E_\alpha$
is a $D$-sublattice.

Let $F=\bigcup_{\alpha<\omega_1}E_\alpha$ and
$x_n\in F$. Define
\[
 \alpha_n=\min\{\alpha<\omega_1:x_n\in E_\alpha\},\qquad
 \gamma=\sup_{n\ge1}(\alpha_n+1).
\]
The set $\bigcup_n(\alpha_n+1)$ is countable, so
$\gamma<\omega_1$, and $x_n\in E_\gamma$ for every $n$.
Set
\[
 p_N=\bigvee_{n=1}^Nx_n\in E_\gamma,\qquad
 q_N=\bigwedge_{n=1}^Nx_n\in E_\gamma.
\]
The universal bounds give
\[
 p_N\uparrow\bigvee_{n\ge1}x_n,\qquad
 q_N\downarrow\bigwedge_{n\ge1}x_n.
\]
For example, $z\ge p_N$ for every $N$ is equivalent to
$z\ge x_n$ for every $n$, and $z\le q_N$ for every $N$
is equivalent to $z\le x_n$ for every $n$.
Hence both countable extrema belong to
$O(E_\gamma)=E_{\gamma+1}\subseteq F$.
The finite-list argument also makes $F$ a $D$-sublattice.
It contains $A$, so $\sigma_L(A)\subseteq F$.

For the reverse inclusion put $S=\sigma_L(A)$.
If $z_n\in S$ and $z_n\onull z$ in $L$, then
\[
 r_N=\bigwedge_{n\ge N}z_n\in S,\qquad
 z=\bigvee_Nr_N\in S
\]
by \eqref{pd:eq-canonical-brackets}.
Thus $O(S)\subseteq S$.
The induction begins with $E_0=A\subseteq S$.
If $E_\alpha\subseteq S$, then
\[
 E_{\alpha+1}=O(E_\alpha)\subseteq O(S)\subseteq S.
\]
If $E_\alpha\subseteq S$ for every $\alpha<\lambda$, then
$E_\lambda=\bigcup_{\alpha<\lambda}E_\alpha\subseteq S$.
Therefore $F\subseteq S$.
\end{proof}

\begin{theorem}[Monotone-class generation]\label{pd:monotone-class}
The least subset of $L$ containing $A$ and closed under
ambient increasing sequential suprema and decreasing
sequential infima is $\sigma_L(A)$.
\end{theorem}

\begin{proof}
Let $M$ denote the least subset in the statement.
For each
\[
 T\in\{\wedge,\vee,\Delta_L,\Delta_R\}
\]
and fixed $a\in A$, define
\[
 C_a^T=\{x\in L:T(x,a)\in M\}.
\]
Every such $T$ is increasing in its first coordinate.
For $x_n\in C_a^T$,
\[
 x_n\uparrow x
 \ \Longrightarrow\
 T(x_n,a)\uparrow T(x,a)
 \ \Longrightarrow\ T(x,a)\in M,
\]
\[
 x_n\downarrow x
 \ \Longrightarrow\
 T(x_n,a)\downarrow T(x,a)
 \ \Longrightarrow\ T(x,a)\in M.
\]
The second implications use the two defining closures of $M$.
Also
\[
 b\in A\quad\Longrightarrow\quad
 T(b,a)\in A\subseteq M\quad\Longrightarrow\quad b\in C_a^T.
\]
Thus $C_a^T$ is monotone closed and contains $A$, so
\begin{equation}\label{pd:eq-first-monotone-pass}
 M\subseteq C_a^T\qquad(a\in A).
\end{equation}

Fix $x\in M$ and define
\[
 D_x^T=\{y\in L:T(x,y)\in M\}.
\]
For $T=\wedge,\vee$, its closure equations are
\[
 y_n\uparrow y\ \Longrightarrow\
 T(x,y_n)\uparrow T(x,y),\qquad
 y_n\downarrow y\ \Longrightarrow\
 T(x,y_n)\downarrow T(x,y).
\]
For $T=\Delta_L,\Delta_R$, they are
\[
 y_n\uparrow y\ \Longrightarrow\
 T(x,y_n)\downarrow T(x,y),\qquad
 y_n\downarrow y\ \Longrightarrow\
 T(x,y_n)\uparrow T(x,y).
\]
In all four cases, $y_n\in D_x^T$ implies $T(x,y)\in M$.
Moreover, \eqref{pd:eq-first-monotone-pass} gives
\[
 a\in A\quad\Longrightarrow\quad x\in C_a^T
 \quad\Longrightarrow\quad T(x,a)\in M.
\]
Therefore $A\subseteq D_x^T$ and $M\subseteq D_x^T$.
We have proved $T(M,M)\subseteq M$ for each of the four
operations.

For $\kappa\in\{-,\sim\}$ let
$B_\kappa=\{x\in L:x^\kappa\in M\}$. Its closure equations are
\[
 x_n\uparrow x\ \Longrightarrow\
 x_n^\kappa\downarrow x^\kappa,\qquad
 x_n\downarrow x\ \Longrightarrow\
 x_n^\kappa\uparrow x^\kappa.
\]
Since $a^\kappa\in A$ for $a\in A$, the set $B_\kappa$
contains $A$. Thus $M\subseteq B_\kappa$.
For $a\le b$ in $M$,
\[
 b\backslash a=\Delta_L(b,a)\in M,\qquad
 a/b=\Delta_R(b,a)\in M.
\]
For a defined sum $a+b=c$ with $a,b\in M$,
\[
 b\le a^\sim,\qquad
 r=\Delta_R(a^\sim,b)=b/a^\sim\in M,\qquad
 c=r^-\in M.
\]
Together with $0,1\in A\subseteq M$, this proves that $M$
is a $D$-sublattice.

For any sequence $x_n\in M$, the finite operations already
proved give
\[
 p_N=\bigvee_{n=1}^Nx_n\in M,\qquad
 q_N=\bigwedge_{n=1}^Nx_n\in M.
\]
By the defining closures of $M$,
\[
 \bigvee_{n\ge1}x_n=\bigvee_Np_N\in M,\qquad
 \bigwedge_{n\ge1}x_n=\bigwedge_Nq_N\in M.
\]
Thus $\sigma_L(A)\subseteq M$.
Conversely, $\sigma_L(A)$ contains $A$ and both monotone
sequential extrema, so minimality of $M$ gives
$M\subseteq\sigma_L(A)$.
\end{proof}

\begin{theorem}[Unrestricted uniqueness]\label{pd:uniqueness}
Let $X$ be a separated Abelian convergence group. Two
additive $\sigma$-additive maps
\[
 f,g:\sigma_L(A)\longrightarrow X
\]
that agree on $A$ agree on all of $\sigma_L(A)$.
No positivity, modularity, control, or exhaustivity
assumption on the competing extensions is required.
\end{theorem}

\begin{proof}
Put $S=\sigma_L(A)$. If $x_n\uparrow x$ in $S$, define
\[
 d_1=x_1,\qquad d_{n+1}=x_n/x_{n+1}.
\]
The partial-sum induction is
\[
 d_1=x_1,\qquad
 (d_1+\cdots+d_n)+d_{n+1}
 =x_n+(x_n/x_{n+1})=x_{n+1}.
\]
Thus $(d_n)$ is ordered orthogonal and
$\bigvee_n(d_1+\cdots+d_n)=x$. Sigma-additivity gives
\[
 f(x_n)=\sum_{j=1}^nf(d_j)\longrightarrow f(x),\qquad
 g(x_n)=\sum_{j=1}^ng(d_j)\longrightarrow g(x).
\]
If $x_n\downarrow x$, then $x_n^-\uparrow x^-$, so
\[
 f(x_n)=f(1)-f(x_n^-)\longrightarrow f(1)-f(x^-)=f(x),
\]
\[
 g(x_n)=g(1)-g(x_n^-)\longrightarrow g(1)-g(x^-)=g(x).
\]

Let $C=\{x\in S:f(x)=g(x)\}$. For any monotone sequence
$x_n\in C$ with ambient limit $x$, countable closure of $S$
gives $x\in S$. The preceding limits give
\[
 f(x_n)=g(x_n)\quad(n\ge1),\qquad
 f(x_n)\longrightarrow f(x),\qquad
 g(x_n)\longrightarrow g(x).
\]
Separation of $X$ yields $f(x)=g(x)$, hence $x\in C$.
Therefore $C$ is monotone closed. Since $A\subseteq C$,
Theorem~\ref{pd:monotone-class} gives $S\subseteq C$.
\end{proof}

\begin{proposition}[Stable range and inherited modularity]
\label{pd:stable-range}
Let $f:S\to X$ be an additive $\sigma$-additive extension
of $m:A\to X$, where $S=\sigma_L(A)$.
Then
\[
 f(S)\subseteq[m(A)]_X.
\]
If $P\subseteq X$ is a limit-stable additive cone with
$m(A)\subseteq P$, then $f(S)\subseteq P$, and, for
$e=m(1)$,
\[
 f(S)\subseteq\{v\in X:v\in P,\ e-v\in P\}.
\]
If $\pi:X\to Y$ is a continuous homomorphism into a separated
Abelian convergence group and $\pi m$ is modular, then
$\pi f$ is modular on $S$.
\end{proposition}

\begin{proof}
Let $C\subseteq X$ be limit-stable and contain $m(A)$, and put
\[
 B_C=\{x\in S:f(x)\in C\}.
\]
For a monotone sequence $x_n\in B_C$ with ambient limit $x$,
\[
 x\in S,\qquad f(x_n)\longrightarrow f(x),\qquad
 f(x_n)\in C\quad(n\ge1).
\]
Limit stability gives $f(x)\in C$, hence $x\in B_C$.
Also $a\in A$ implies $f(a)=m(a)\in C$.
Theorem~\ref{pd:monotone-class} gives $S=B_C$.
Consequently
\[
 f(S)\subseteq
 \bigcap\{C\subseteq X:C\text{ is limit-stable and }m(A)\subseteq C\}
 =[m(A)]_X.
\]
Taking $C=P$ gives $f(x)\in P$ for every $x\in S$. For every
such $x$,
\[
 x/1\in S,\qquad x+(x/1)=1,\qquad
 e-f(x)=f(x/1)\in P.
\]
This proves the interval assertion.

Let $h=\pi f$. For a defined sum,
\[
 h(a+b)=\pi(f(a)+f(b))=h(a)+h(b).
\]
For ordered partial sums $s_N\uparrow s$,
\[
 h(s_N)=\pi(f(s_N))\longrightarrow\pi(f(s))=h(s).
\]
Thus $h$ is additive and sigma-additive. In particular it is
monotone-sequentially continuous by
Lemma~\ref{pd:ambient-normality}. Define its modular defect by
\[
 D_h(x,y)=h(x\vee y)+h(x\wedge y)-h(x)-h(y).
\]
For fixed $a\in S$ and $x_n\uparrow x$,
\[
 x_n\vee a\uparrow x\vee a,\qquad
 x_n\wedge a\uparrow x\wedge a,
\]
\[
 h(x_n\vee a)\longrightarrow h(x\vee a),\quad
 h(x_n\wedge a)\longrightarrow h(x\wedge a),\quad
 h(x_n)\longrightarrow h(x),
\]
and hence $D_h(x_n,a)\to D_h(x,a)$.
For $x_n\downarrow x$, the corresponding formulas are
\[
 x_n\vee a\downarrow x\vee a,\qquad
 x_n\wedge a\downarrow x\wedge a,
\]
\[
 h(x_n\vee a)\longrightarrow h(x\vee a),\quad
 h(x_n\wedge a)\longrightarrow h(x\wedge a),\quad
 h(x_n)\longrightarrow h(x),
\]
which again give $D_h(x_n,a)\to D_h(x,a)$.

For $a\in A$, let $C_a=\{x\in S:D_h(x,a)=0\}$.
If a monotone row belongs to $C_a$, its defects form the
constant zero sequence. The preceding limit formulas and
separation give $D_h(x,a)=0$ at its limit. Hence $C_a$
is monotone closed. Source modularity gives
\[
 b\in A\quad\Longrightarrow\quad
 D_h(b,a)=D_{\pi m}(b,a)=0.
\]
Thus $A\subseteq C_a$, and the monotone-class theorem gives
$S=C_a$ for every $a\in A$.

For fixed $x\in S$, let $E_x=\{y\in S:D_h(x,y)=0\}$.
Since $D_h(x,y)=D_h(y,x)$, the displayed limit identities
give
\[
 y_n\uparrow y\text{ or }y_n\downarrow y
 \quad\Longrightarrow\quad
 D_h(x,y_n)\longrightarrow D_h(x,y).
\]
Therefore $E_x$ is monotone closed. The first pass gives
\[
 a\in A\quad\Longrightarrow\quad x\in C_a
 \quad\Longrightarrow\quad D_h(x,a)=0,
\]
so $A\subseteq E_x$. A second application of
Theorem~\ref{pd:monotone-class} yields $S=E_x$.
Hence $D_h(x,y)=0$ for every $x,y\in S$.
\end{proof}

\section{Uncontrolled modular extension on pseudo-\texorpdfstring{$D$}{D}-lattices}
\label{mod:section}

Throughout this section, $L$ is a $\sigma$-complete, lattice-$\sigma$-continuous
pseudo-$D$-lattice and $A\subseteq L$ is a $D$-sublattice. In particular,
$0,1\in A$. All extrema in the source are computed in $L$.
The group $Y$ is a separated Abelian net-convergence group, with convergence
$\theta$. No order on $Y$ is assumed. We write
\[
 A^\sigma=\{\sup_n a_n:a_n\in A,\ a_n\uparrow\},\qquad
 A^\delta=\{\inf_n a_n:a_n\in A,\ a_n\downarrow\}.
\]
The residual convention is
$(b\backslash a)+a=b$ and $a+(a/b)=b$ for $a\le b$.

Let $n:A\to Y$ be additive and modular, so that
\begin{equation}\label{mod:valuation}
 n(a\vee b)+n(a\wedge b)=n(a)+n(b).
\end{equation}
The source continuity assumption is ambient sequential normality:
\begin{equation}\label{mod:source-normality}
 a_j\in A,\quad a_j\downarrow_L0
 \quad\Longrightarrow\quad n(a_j)\tto{\theta}0.
\end{equation}
By Lemma~\ref{pd:ambient-normality}, intrinsic $\sigma$-additivity implies
\eqref{mod:source-normality}. The latter assumption is sufficient below.
For $s\in L$, put
\[
 V_A(s)=\{n(a):a\in A,\ a\le s\},\qquad
 \Gamma_A(s)=[V_A(s)]_\theta,\qquad N=[n(A)]_\theta.
\]
We assume that $N$ is complete, that
\begin{equation}\label{mod:range-assumptions}
 N-N\subseteq D=-D,\qquad
 \mathsf F(D),\qquad\mathsf S_{\mathrm{all}}(D),\qquad\mathsf U(N)
\end{equation}
hold for the original convergence, and that $n$ is uniformly hereditarily
exhaustive:
\begin{equation}\label{mod:source-he}
 (h_j)\subseteq A\text{ ordered orthogonal}
 \quad\Longrightarrow\quad (V_A(h_j))_j\Rightarrow0.
\end{equation}
The range principles in \eqref{mod:range-assumptions} are those in
Definitions~\ref{def:finite-response} and \ref{gen:summation}, and uniform
closure in Definition~\ref{gen:uniform}. They concern $Y$, $N$, $D$,
and their original convergence.
They do not assert the existence of an extension of $n$.

\subsection{Hereditary interval errors}

For an additive modular map $v:E\to Y$ on a $D$-sublattice, define
\begin{equation}\label{mod:interval-definition}
 I_v(a,b)=\{v(t)-v(s):a\wedge b\le s\le t\le a\vee b,\ s,t\in E\}.
\end{equation}
The subscript is omitted when the map is fixed. Every such set contains $0$.
If $v(E)\subseteq N$, it is a subset of $D$. In expressions involving
Minkowski sums, $rB$ denotes the sum of $r$ independent copies of $B$.
It does not denote a pointwise integer multiple unless the context states this.

\begin{lemma}[Interval calculus]\label{mod:interval-calculus}
Let $v:E\to Y$ be additive and modular. The following assertions hold.
\begin{enumerate}
\item For $a\le b$,
\begin{equation}\label{mod:interval-residual}
 I_v(a,b)=\{v(c):c\in E,\ c\le a/b\}
          =\{v(c):c\in E,\ c\le b\backslash a\}.
\end{equation}
\item For arbitrary $a,b,c\in E$,
\begin{equation}\label{mod:lattice-contractions}
 I_v(a\wedge c,b\wedge c)\subseteq I_v(a,b),\qquad
 I_v(a\vee c,b\vee c)\subseteq I_v(a,b).
\end{equation}
If $l\le c\le u$, then
\begin{equation}\label{mod:interval-splitting}
 I_v(l,u)\subseteq I_v(l,c)+I_v(c,u).
\end{equation}
Consequently, for every finite family,
\begin{equation}\label{mod:finite-lattice-errors}
 I_v\left(\bigvee_i a_i,\bigvee_i b_i\right)
 \subseteq\sum_i I_v(a_i,b_i),\qquad
 I_v\left(\bigwedge_i a_i,\bigwedge_i b_i\right)
 \subseteq\sum_i I_v(a_i,b_i).
\end{equation}
\item Left and right complementation preserve interval-value sets. The same
is true of a residual map with one endpoint fixed, on its natural interval.
\item Put $J_v(a,b)=I_v(a,b)-I_v(a,b)$. If $a\le b$ and $a'\le b'$, then
\begin{equation}\label{mod:residual-errors}
 \begin{split}
 I_v(b\backslash a,b'\backslash a')&\subseteq
       8J_v(a,a')+4J_v(b,b'),\\
 I_v(a/b,a'/b')&\subseteq
       8J_v(a,a')+4J_v(b,b').
 \end{split}
\end{equation}
Thus each fixed finite expression in lattice operations, complements and
the total residuals of Theorem~\ref{pd:total-residuals} admits an interval-error
bound by finitely many signed copies of the input interval errors.
\end{enumerate}
\end{lemma}

\begin{proof}
For $a\le s\le t\le b$, the element $s/t$ is below $a/b$ and
$v(t)-v(s)=v(s/t)$. Conversely, if $c\le a/b$, the sum $a+c$ exists and
belongs to $[a,b]$, and $v(c)=v(a+c)-v(a)$. The left-residual identity has
the following calculation: $c\le b\backslash a$ implies
$a\le c+a\le b$ and $v(c)=v(c+a)-v(a)$, while
$a\le s\le t\le b$ implies $t\backslash s\le b\backslash a$ and
$v(t)-v(s)=v(t\backslash s)$. This proves \eqref{mod:interval-residual}.

First suppose $a\le b$ and $a\wedge c\le s\le t\le b\wedge c$.
The equalities $s\wedge a=t\wedge a=a\wedge c$ and modularity give
\[
 v(t)-v(s)=v(t\vee a)-v(s\vee a).
\]
Both elements on the right belong to $[a,b]$. This proves the meet
contraction for comparable endpoints. If $a\vee c\le s\le t\le b\vee c$,
then $s\vee b=t\vee b=b\vee c$, and
\[
 v(t)-v(s)=v(t\wedge b)-v(s\wedge b).
\]
The right side belongs to $I_v(a,b)$. For arbitrary $a,b$, the interval
between $a\wedge c$ and $b\wedge c$ is contained in
$[(a\wedge b)\wedge c,(a\vee b)\wedge c]$. The comparable-endpoint result
applies to this larger interval. The join assertion is obtained from
$[(a\wedge b)\vee c,(a\vee b)\vee c]$.

For $l\le s\le t\le u$, modularity at $c$ gives
\[
 v(t)-v(s)=
 \bigl(v(t\wedge c)-v(s\wedge c)\bigr)
 +\bigl(v(t\vee c)-v(s\vee c)\bigr).
\]
The two terms belong to $I_v(l,c)$ and $I_v(c,u)$, respectively.
For \eqref{mod:finite-lattice-errors}, let $l_i=a_i\wedge b_i$ and
$u_i=a_i\vee b_i$, $1\le i\le r$. For $\circ\in\{\wedge,\vee\}$ put
\[
 v_j=\mathop{\circ}_{i=1}^r z_{i,j},\qquad
 z_{i,j}=\begin{cases}u_i&i\le j,\\l_i&i>j.\end{cases}
\]
Then $v_0\le\cdots\le v_r$, both original outputs lie in $[v_0,v_r]$,
and contraction in the $j$th coordinate gives
$I_v(v_{j-1},v_j)\subseteq I_v(l_j,u_j)=I_v(a_j,b_j)$.
Induction in \eqref{mod:interval-splitting} gives
\[
 I_v(v_0,v_r)\subseteq\sum_{j=1}^r I_v(v_{j-1},v_j)
 \subseteq\sum_{j=1}^r I_v(a_j,b_j).
\]

The residual maps with one endpoint fixed are order isomorphisms or
order anti-isomorphisms between their natural intervals. Their defining
additivity identities give, for $a\le s\le t$,
\[
 v(t\backslash a)-v(s\backslash a)=v(t)-v(s),\qquad
 v(a/t)-v(a/s)=v(t)-v(s).
\]
For $s\le t\le c$ they give
\[
 v(c\backslash s)-v(c\backslash t)=v(t)-v(s),\qquad
 v(s/c)-v(t/c)=v(t)-v(s).
\]
The interval maps are onto their corresponding intervals, so each
inclusion of interval-value sets supplied by these equations is an
equality. Setting $c=1$ proves the complement assertions.

For the last assertion, define
\[
 C_v(a,b)=\{v(a\wedge z)-v(b\wedge z):z\in E\}.
\]
Put $p=a\wedge b$, $s=(a\wedge z)\vee p$, and $t=(b\wedge z)\vee p$.
Then $s,t\in[p,a\vee b]$. Since
$(a\wedge z)\wedge p=(b\wedge z)\wedge p=p\wedge z$, modularity gives
\[
 v(s)-v(t)=v(a\wedge z)-v(b\wedge z).
\]
Both $v(s)-v(s\wedge t)$ and $v(t)-v(s\wedge t)$ belong to $I_v(a,b)$.
Subtracting them proves
\begin{equation}\label{mod:C-I-first}
 C_v(a,b)\subseteq J_v(a,b).
\end{equation}
For $r\in[p,a\vee b]$, modularity gives the intermediate identities
\[
 \begin{split}
 v(r)&=v(a\wedge r)+v(a\vee r)-v(a),\\
 v(a\vee r)&=v(a\vee b)+v(b\wedge(a\vee r))-v(b),\\
 v(a\vee b)&=v(a)+v(b)-v(p).
 \end{split}
\]
Substitution yields
\[
 v(r)=v(a)+C_{a,b}(a\wedge r)-C_{a,b}(a\vee r),
 \qquad C_{a,b}(z)=v(a\wedge z)-v(b\wedge z).
\]
Indeed, $b\wedge a\wedge r=p$ and
$(a\vee r)\vee b=a\vee b$. Subtracting this identity for two comparable
elements of $[p,a\vee b]$ proves
\begin{equation}\label{mod:C-I-second}
 I_v(a,b)\subseteq2C_v(a,b)-2C_v(a,b).
\end{equation}
For every finite path $x_0,\ldots,x_r$, the defining differences telescope:
\[
 C_v(x_0,x_r)\subseteq\sum_{j=1}^r C_v(x_{j-1},x_j).
\]
Apply this identity and \eqref{mod:C-I-first}--\eqref{mod:C-I-second} to
\[
 b\backslash a,\quad b\backslash(a\wedge a'),\quad
 b'\backslash(a\wedge a'),\quad b'\backslash a'.
\]
The first and third interval errors lie in $I_v(a,a')$, by fixed-endpoint
residual preservation and interval inclusion. The middle one equals
$I_v(b,b')$. This gives the first inclusion in
\eqref{mod:residual-errors}. Define $E^{\mathrm{op}}$ by
$x+_{\mathrm{op}}y=y+x$. Its order, meet, join, and interval-value
sets are unchanged, while
\[
 b\backslash_{\mathrm{op}}a=a/b,\qquad
 a/_{\mathrm{op}}b=b\backslash a.
\]
The map $v$ remains additive because
$v(x+_{\mathrm{op}}y)=v(y)+v(x)=v(x)+v(y)$, and its modular identity
is unchanged. The first proved inclusion, written in $E^{\mathrm{op}}$,
is therefore the explicit formula
\[
 I_v(a/b,a'/b')\subseteq8J_v(a,a')+4J_v(b,b'),
\]
which is the second inclusion in \eqref{mod:residual-errors}.
For either total residual $\Delta$, its formula and the preceding bounds give
\begin{equation}\label{mod:total-residual-error}
 \begin{split}
 I_v(\Delta(y,x),\Delta(y',x'))
 &\subseteq I_v(y,y')+8J_v(x,x')+4J_v(y\vee x,y'\vee x')\\
 &\subseteq5J_v(y,y')+12J_v(x,x').
 \end{split}
\end{equation}
The identity $J_v(a,b)=-J_v(a,b)$ is used in the second line.
\end{proof}

\begin{lemma}[Explicit finite-word constants]\label{mod:word-constants}
Let $P$ be a finite word in variables $X_1,\ldots,X_r$, constants, the
lattice operations, complements, and the two total residuals. Define
nonnegative integers $w_i(P)$ recursively by
\[
 \begin{array}{ll}
 w_i(X_j)=\mathbf1_{i=j},&w_i(c)=0,\\
 w_i(Q\wedge T)=w_i(Q\vee T)=w_i(Q)+w_i(T),
 &w_i(Q^-)=w_i(Q^\sim)=w_i(Q),\\
 w_i(\Delta(Q,T))=10w_i(Q)+24w_i(T).
 \end{array}
\]
For two input tuples $a,b\in E^r$,
\begin{equation}\label{mod:word-constant-bound}
 J_v(P(a),P(b))\subseteq\sum_{i=1}^r w_i(P)J_v(a_i,b_i),\qquad
 I_v(P(a),P(b))\subseteq\sum_{i=1}^r w_i(P)J_v(a_i,b_i).
\end{equation}
There is also an outward-interval version. Start with intervals
$[l_i,u_i]$, and compute an interval $[L_Q,U_Q]$ at each subword by
applying meet and join to the corresponding endpoints, swapping endpoints
under complementation, and setting
\[
 (L_{\Delta(Q,T)},U_{\Delta(Q,T)})
   =(\Delta(L_Q,U_T),\Delta(U_Q,L_T)).
\]
Then
\begin{equation}\label{mod:word-outward-bound}
 J_v(L_P,U_P)\subseteq\sum_iw_i(P)J_v(l_i,u_i).
\end{equation}
This assertion permits repeated occurrences of a variable with different
polarities.
\end{lemma}
\begin{proof}
For a variable or a constant the asserted inclusion is its defining
identity. Taking the difference of the two sides of
\eqref{mod:finite-lattice-errors} gives
$J_v(Q(a)\circ T(a),Q(b)\circ T(b))\subseteq
J_v(Q(a),Q(b))+J_v(T(a),T(b))$ for $\circ=\wedge,\vee$.
Complementation preserves $I_v$ and $J_v$.
Taking differences in \eqref{mod:total-residual-error} gives
$J_v(\Delta(Q(a),T(a)),\Delta(Q(b),T(b)))\subseteq
10J_v(Q(a),Q(b))+24J_v(T(a),T(b))$.
These four displayed recursions prove the first assertion by induction
on the number of operation symbols in $P$. The second uses $0\in I_v$
and hence $I_v\subseteq J_v$.
For outward endpoints the recursions are, explicitly,
\[
 \begin{split}
 J_v(L_{Q\circ T},U_{Q\circ T})
   &\subseteq J_v(L_Q,U_Q)+J_v(L_T,U_T)\quad(\circ=\wedge,\vee),\\
 J_v(U_Q^-,L_Q^-)&=J_v(L_Q,U_Q),\qquad
 J_v(U_Q^\sim,L_Q^\sim)=J_v(L_Q,U_Q),\\
 J_v(\Delta(L_Q,U_T),\Delta(U_Q,L_T))
   &\subseteq10J_v(L_Q,U_Q)+24J_v(U_T,L_T)\\
   &=10J_v(L_Q,U_Q)+24J_v(L_T,U_T).
 \end{split}
\]
These formulas and the variable and constant cases prove
\eqref{mod:word-outward-bound} by induction on the word tree.
\end{proof}

\begin{lemma}[Indexed stable-hull assembly]\label{mod:hull-bookkeeping}
Assume $\mathsf S_{\mathrm{all}}(D)$. If
$0\in U_\alpha\subseteq D$ is decreasing and null, then
\begin{equation}\label{mod:one-row}
 \bigl(D\cap[U_\alpha]_\theta\bigr)_\alpha\Rightarrow0.
\end{equation}
For each $i\ge1$, let $(U_{i,\alpha})_{\alpha\in J_i}$ be a decreasing
null row on a directed domain $J_i$, with $0\in U_{i,\alpha}\subseteq D$
for every $\alpha\in J_i$. Let $m_i\ge1$ be an integer and let
$\epsilon_{i,k}\in\{-1,1\}$ for $1\le k\le m_i$. Set
\[
 C=\{(i,k):i\ge1,\ 1\le k\le m_i\},\quad
 P_C=\prod_{(i,k)\in C}J_i,\quad
 W_{i,k,\alpha}=\epsilon_{i,k}U_{i,\alpha},
\]
\[
 B_\psi=\bigcup_{F\subseteq C\text{ finite}}
                   \sum_{(i,k)\in F}W_{i,k,\psi(i,k)},\qquad
 H_\psi=D\cap[B_\psi]_\theta.
\]
Then $H_\psi\Rightarrow0$. The diagonal map
$\delta:\prod_iJ_i\to P_C$, $\delta(\phi)(i,k)=\phi(i)$, is order
preserving and cofinal, so $H_{\delta(\phi)}\Rightarrow0$.
Countably indexed collections of countably many rows can be assembled
into one countable collection. Finite sums of the individual relative
hulls lie in the stable hull of the assembled raw sums.
\end{lemma}
\begin{proof}
For \eqref{mod:one-row}, use $U$ as row $1$ and $\{0\}$ on $\mathbb N$
as every row $i\ge2$. The resulting raw sum is $U_{\phi(1)}$. Its
evaluation tails are exactly those of the family in
\eqref{mod:one-row}.

For the copied rows, $W_{i,k,\alpha}\subseteq D$ because $D=-D$.
Each row is decreasing, contains zero, and is null by continuity of
inversion. An explicit enumeration of $C$ is obtained by arranging its
elements in increasing order of
$\pi(i,k)=\frac{(i+k-2)(i+k-1)}2+k$.
Let $e:\mathbb N\to C$ be this enumeration. Then
\[
 \psi\longmapsto(\psi(e(j)))_{j\ge1}
\]
is an order isomorphism of the corresponding product domains. Finite
subsets of $C$ correspond bijectively to finite subsets of $\mathbb N$,
and the raw sum sets agree under this map. The assertion
$H_\psi\Rightarrow0$ is therefore exactly the summation principle.

For each $i$, fix a finite-upper-bound selector
\[
 \beta_i:J_i^{m_i}\to J_i,\qquad
 \beta_i(\alpha_1,\ldots,\alpha_{m_i})\ge\alpha_k\quad(1\le k\le m_i),
 \qquad \beta_i(\alpha,\ldots,\alpha)=\alpha.
\]
Directedness gives such selectors. For $\psi\in P_C$, define
$\beta(\psi)(i)=\beta_i(\psi(i,1),\ldots,\psi(i,m_i))$.
Then $\delta(\beta(\psi))\ge\psi$. This proves cofinality of $\delta$.
The maps $\beta_i$ need not preserve order. The inequalities
\[
 U_{i,\beta(\psi)(i)}\subseteq U_{i,\psi(i,k)}
 \quad(1\le k\le m_i),\qquad \beta\delta(\phi)=\phi
\]
are the only properties used when endpoints of a certificate are
evaluated at $\beta(\psi)$.
Their tail property is explicit: for every $\theta\in\prod_iJ_i$,
\[
 \psi\ge\delta(\theta)
 \quad\Longrightarrow\quad
 \beta(\psi)(i)\ge\psi(i,1)\ge\theta(i)\quad\text{for every }i.
\]
Thus endpoint nets known to converge at the original indices still
converge after composition with $\beta$. This assertion uses the displayed
tail inclusion and does not require monotonicity of $\beta$.

For rows with labels $(n,i)\in\mathbb N^2$, the displayed bijection
$\pi(n,i)$ reindexes the product and its finite subsets. More general
countable occurrence labels are enumerated by their increasing $\pi$
codes after insertion of zero rows. For a fixed product index write
$A_n$ for the raw sums in block $n$, and $A_*$ for the assembled raw
sums. If $F\subseteq\mathbb N$ is finite, then
$\sum_{n\in F}A_n\subseteq A_*$, since the row labels in different
blocks are disjoint. Stable-hull addition gives
\begin{equation}\label{mod:hull-sums}
 [B_1]_\theta+\cdots+[B_r]_\theta
 \subseteq[B_1+\cdots+B_r]_\theta.
\end{equation}
For clarity, \eqref{mod:hull-sums} uses no stability of $D$. If
$K=[B+C]_\theta$, then $K-c$ is stable and contains $B$ for each
$c\in C$. Thus $[B]_\theta+c\subseteq K$. For each $b\in[B]_\theta$,
the stable set $K-b$ contains $C$, so $b+[C]_\theta\subseteq K$.
Induction gives the finite formula. Hence
\[
 \sum_{n\in F}(D\cap[A_n]_\theta)
 \subseteq\left[\sum_{n\in F}A_n\right]_\theta
 \subseteq[A_*]_\theta.
\]
If a value on the left is a difference of two points of $N$, it also
lies in $D$ and therefore lies in $D\cap[A_*]_\theta$.
For a net $z_\lambda\in D\cap[A_*]_\theta$ with limit $z\in N-N$,
stability gives $z\in[A_*]_\theta$ and the hypothesis $N-N\subseteq D$
gives $z\in D\cap[A_*]_\theta$. These two separate inclusions justify
every passage to a limit in a relative hull below.
\end{proof}

\begin{lemma}[Monotone interval tails]\label{mod:source-monotone-tails}
If $(a_\alpha)_{\alpha\in\Lambda}$ is an increasing net in $A$, then
\[
 T_\gamma=
 \bigcup_{\alpha,\beta\ge\gamma}I_n(a_\alpha,a_\beta)
 \quad\text{satisfies}\quad (T_\gamma)_\gamma\Rightarrow0.
\]
The same assertion holds for decreasing nets. Consequently their value
nets are Cauchy and converge to points of $N$.
\end{lemma}

\begin{proof}
The sets $T_\gamma$ decrease, contain $0$, and lie in $D$. Their evaluation
domain and net are
\[
 \mathcal E=\{(\gamma,z):\gamma\in\Lambda,\ z\in T_\gamma\},\quad
 (\gamma,z)\le(\gamma',z')\iff\gamma\le\gamma',\quad
 e(\gamma,z)=z.
\]
In particular $\operatorname{Tail}_e(\gamma,0)=T_\gamma$.
For every $(\gamma,z)\in\mathcal E$, fix witnesses
\[
 R(\gamma,z)=(\alpha_{\gamma,z},\beta_{\gamma,z},s_{\gamma,z},t_{\gamma,z})
\]
with $\alpha_{\gamma,z},\beta_{\gamma,z}\ge\gamma$,
$a_\alpha\wedge a_\beta\le s_{\gamma,z}\le t_{\gamma,z}
 \le a_\alpha\vee a_\beta$ and $z=n(t_{\gamma,z})-n(s_{\gamma,z})$.
Fix $\gamma_*\in\Lambda$ and a selector $\rho$ taking each finite
nonempty subset of $\Lambda$ to an upper bound. Define the strategy by
\[
 \sigma(\varnothing)=(\gamma_*,0),\qquad
 \sigma(h)=\left(\rho\left(\{\gamma_r\}\cup
       \{\alpha_{\gamma_r,z},\beta_{\gamma_r,z}:z\in S_r\}\right),0\right)
\]
for a nonempty legal history ending with $((\gamma_r,w_r),S_r)$.
In a play let $\gamma_j$ be the first coordinate of its $j$th move.
The witnesses and the definition of $\rho$ give, for every $z\in S_j$,
\[
 a_{\gamma_j}\le s_{\gamma_j,z}\le t_{\gamma_j,z}\le a_{\gamma_{j+1}},
 \qquad z\in
 I_n(a_{\gamma_j},a_{\gamma_{j+1}})
   =V_A(a_{\gamma_j}/a_{\gamma_{j+1}}).
\]
Writing $h_j=a_{\gamma_j}/a_{\gamma_{j+1}}$, residual associativity gives
\[
 h_1+\cdots+h_k=a_{\gamma_1}/a_{\gamma_{k+1}}\quad(k\ge1).
\]
Assumption \eqref{mod:source-he} makes the response evaluation null in
every play. Finite-response reflection proves $(T_\gamma)_\gamma\Rightarrow0$.

For arbitrary endpoints,
$n(a)-n(b)=(n(a)-n(a\wedge b))-(n(b)-n(a\wedge b))\in I_n(a,b)-I_n(a,b)$.
Thus $\{n(a_\alpha)-n(a_\beta):\alpha,\beta\ge\gamma\}
\subseteq T_\gamma-T_\gamma\Rightarrow0$. Completeness of $N$ gives
the value limit. If $a_\alpha$ decreases, $a_\alpha^-$ increases and
$I_n(a_\alpha,a_\beta)=I_n(a_\alpha^-,a_\beta^-)$, while
$n(a_\alpha)=n(1)-n(a_\alpha^-)$. These equalities prove both conclusions
for decreasing nets.
\end{proof}

\begin{lemma}[Uniform collapse of an order-null Cauchy sequence]
\label{mod:uniform-collapse}
Suppose $(r_j)\subseteq A$ satisfies
\[
 r_j\longrightarrow_o0\text{ in }L,\qquad
 U_k:=\bigcup_{p,q\ge k}I_n(r_p,r_q)\Rightarrow0.
\]
Then $(V_A(r_j))_j\Rightarrow0$.
\end{lemma}

\begin{proof}
The families $U_k$ decrease and contain $0$. For
$\phi\in\mathbb N^{\mathbb N}$ set
\[
 Q_\phi=D\cap
 \left[\bigcup_{F\subseteq\mathbb N\text{ finite}}
                    \sum_{j\in F}U_{\phi(j)}\right]_\theta.
\]
The principle $\mathsf S_{\mathrm{all}}(D)$ gives $Q_\phi\Rightarrow0$.
Fix $\phi$ and an arbitrary integer $k\ge\phi(1)$. There is an increasing
sequence $k_1=k<k_2<\cdots$ such that
$k_j,k_{j+1}\ge\phi(j)$ for all $j$. For example,
$k_{j+1}=1+\max\{k_j,\phi(j),\phi(j+1)\}$ has this property.
Put $c_m=\bigwedge_{j\le m}r_{k_j}$ and
$q_N=\bigvee_{j\ge N}r_j$. Order nullity gives $q_N\downarrow0$.
Since $k_m\uparrow\infty$ and $c_m\le r_{k_m}\le q_{k_m}$,
\[
 0\le\bigwedge_m c_m\le\bigwedge_m q_{k_m}=0.
\]
Hence $c_m\downarrow_L0$.

Let $s\in A$ satisfy $s\le r_k$. Then $s\wedge c_m\downarrow_L0$ and
$s\wedge c_1=s$. For $j\ge1$, meet contraction gives
\[
 n(s\wedge c_j)-n(s\wedge c_{j+1})
 \in I_n(r_{k_j},r_{k_{j+1}})\subseteq U_{\phi(j)}.
\]
The reason is $c_j\le r_{k_j}$ and
$c_{j+1}=c_j\wedge r_{k_{j+1}}$. Telescoping and
\eqref{mod:source-normality} place $n(s)$ in the stable hull defining
$Q_\phi$. Also $n(s)\in N\subseteq D$, because $0\in N$.
The integer $k$ and the fragment $s$ were arbitrary. Therefore
\[
 \bigcup_{k\ge\phi(1)}V_A(r_k)\subseteq Q_\phi.
\]
These are containments of entire evaluation tails. Nullity of $Q_\phi$
proves the assertion.
\end{proof}

\subsection{The finite \texorpdfstring{$D$}{D}-sublattice generated by monotone limits}

Call a source sequence $(a_j)$ interval-Cauchy when its tail interval
families in Lemma~\ref{mod:uniform-collapse} are null. For interval-Cauchy
rows $a_{i,j}$, put $T_{i,N}=\bigcup_{p,q\ge N}I_n(a_{i,p},a_{i,q})$.
Lemma~\ref{mod:word-constants} gives the actual tail inclusion
\[
 \bigcup_{p,q\ge N}I_n(P(a_{1,p},\ldots,a_{r,p}),P(a_{1,q},\ldots,a_{r,q}))
 \subseteq\sum_{i=1}^r w_i(P)(T_{i,N}-T_{i,N})\Rightarrow0.
\]
The final nullity uses continuity of each finite group sum and inversion.
Value differences belong to the difference of two interval tails. The
value sequence is therefore Cauchy and converges in $N$.

\begin{lemma}[Equality of source representations]\label{mod:representation}
Let $(a_j)$ and $(b_j)$ be interval-Cauchy source sequences with the same
ambient order limit. Then
\[
 n(a_j)-n(b_j)\tto{\theta}0.
\]
\end{lemma}

\begin{proof}
The sequences
\[
 r_j=(a_j\wedge b_j)/a_j,\qquad
 s_j=(a_j\wedge b_j)/b_j
\]
are interval-Cauchy by \eqref{mod:residual-errors}. Sequential continuity
of the total residuals shows that both are order null.
Lemma~\ref{mod:uniform-collapse} gives $n(r_j)\to0$ and $n(s_j)\to0$.
The identity
$n(a_j)-n(b_j)=n(r_j)-n(s_j)$ completes the proof.
\end{proof}

Let $K$ be the $D$-sublattice generated by $A^\sigma$. An element of $K$
is a finite expression in elements of $A^\sigma$, using the lattice
operations, complements, and total residuals. Each input $u_i\in A^\sigma$
has a source representation $a_{i,j}\uparrow u_i$. Substitution of these
representations gives an interval-Cauchy source sequence converging in
order to the value of the expression. Define $n_0$ on that value to be
the limit of its source values.

\begin{proposition}[The finite extension]\label{mod:finite-extension}
The preceding definition gives a well-defined additive modular map
$n_0:K\to N$ extending $n$. Write $n_\sigma=n_0|_{A^\sigma}$ and
$n_\delta=n_0|_{A^\delta}$. The following properties hold.
\begin{enumerate}
\item If $a_j\uparrow u$ in the source, then
\begin{equation}\label{mod:exhaustion-shell}
 n(a_j)\to n_\sigma(u),\qquad
 \Gamma_A(a_j/u)\Rightarrow0.
\end{equation}
The map $n_\sigma$ preserves increasing sequential limits in $A^\sigma$.
\item For $x\in K$ and $s\in A^\sigma$,
\begin{equation}\label{mod:finite-domination}
 x\le s\quad\Longrightarrow\quad n_0(x)\in\Gamma_A(s).
\end{equation}
Consequently, if $l\in A^\delta$, $u\in A^\sigma$ and $l\le u$, then
\begin{equation}\label{mod:mixed-interval-domination}
 I_{n_0}(l,u)\subseteq\Gamma_A(l/u).
\end{equation}
\item Every $x\in K$ has brackets
\begin{equation}\label{mod:finite-brackets}
 l_j\in A^\delta,\quad u_j\in A^\sigma,\quad
 l_j\uparrow x,\quad u_j\downarrow x,\quad
 \Gamma_A(l_j/u_j)\Rightarrow0.
\end{equation}
\item If $(p_j)\subseteq A$ is a finite-expression representation of $x$,
then $I_{n_0}(x,p_j)\Rightarrow0$. For source endpoints $a,b$,
\begin{equation}\label{mod:finite-interval-density}
 I_{n_0}(a,b)\subseteq[I_n(a,b)]_\theta.
\end{equation}
\end{enumerate}
\end{proposition}

\begin{proof}
Lemma~\ref{mod:representation} proves independence of every source
representation and every finite expression. The constant row $a_j=a$
has limit value $n(a)$, so $n_0|_A=n$. If $p_j,q_j$ represent $x,y$,
the word-tail bound proves interval Cauchyness of $p_j\vee q_j$ and
$p_j\wedge q_j$. Their order limits are $x\vee y$ and $x\wedge y$.
Consequently
\[
 \begin{split}
 n(p_j\vee q_j)+n(p_j\wedge q_j)&=n(p_j)+n(q_j),\\
 n_0(x\vee y)+n_0(x\wedge y)&=n_0(x)+n_0(y).
 \end{split}
\]
The second line is the limit of the first in the convergence group.
If $x\le y$ in $K$, take source
representations $p_j\to_o x$ and $q_j\to_o y$. Then $p_j\wedge q_j$
represents $x$ and $(p_j\wedge q_j)/q_j$ represents $x/y$. Passing to
limits in
\[
 n(q_j)=n(p_j\wedge q_j)+n((p_j\wedge q_j)/q_j)
\]
proves $n_0(y)=n_0(x)+n_0(x/y)$. This gives additivity for every defined
sum. Complementation also gives
$n_\delta(l)=n(1)-n_\sigma(l^-)$.

For a source exhaustion $a_j\uparrow u$, the residuals
$a_j/a_k$, $k\ge j$, increase to $a_j/u$. If $c\in A$ is below $a_j/u$,
then $c\wedge(a_j/a_k)\uparrow c$. Ambient normality and residual
additivity imply continuity at this increasing source limit. Therefore
\[
 \Gamma_A(a_j/u)\subseteq
 D\cap\left[\bigcup_{p,q\ge j}I_n(a_p,a_q)\right]_\theta.
\]
Lemma~\ref{mod:source-monotone-tails} and \eqref{mod:one-row} prove
\eqref{mod:exhaustion-shell}.

To prove \eqref{mod:finite-domination}, let $p_j$ represent $x$ and
$a_j\uparrow s$ be a source exhaustion. The source sequence
$p_j\wedge a_j$ is interval-Cauchy and converges in order to $x\wedge s=x$.
Lemma~\ref{mod:representation} gives
$n(p_j\wedge a_j)\to n_0(x)$. Each term belongs to $V_A(s)$, proving
\eqref{mod:finite-domination}. If $l\le c\le d\le u$ in $K$, then
$c/d\le l/u$. The mixed-hull lemma, Lemma~\ref{pd:mixed-hulls}, places
$l/u$ in $A^\sigma$. Formula \eqref{mod:finite-domination} applied to
$c/d$ proves \eqref{mod:mixed-interval-domination}.

Suppose $u_j\uparrow u$ in $A^\sigma$, and let $a_{i,k}\uparrow_k u_i$
be source exhaustions. The source sequence
$p_j=\bigvee_{i\le j}a_{i,j}$ increases to $u$ and satisfies $p_j\le u_j$.
By additivity and \eqref{mod:finite-domination},
\[
 n_\sigma(u_j)-n(p_j)\in\Gamma_A(p_j/u).
\]
The right side is null by \eqref{mod:exhaustion-shell}. Since
$n(p_j)\to n_\sigma(u)$, increasing continuity follows.

For \eqref{mod:finite-brackets}, fix a finite expression for $x$ in inputs
$u_i\in A^\sigma$, with $a_{i,j}\uparrow u_i$. The finite-expression
brackets of Lemma~\ref{pd:word-brackets}, applied to the input intervals
$[a_{i,j},u_i]$, give $l_j\in A^\delta$, $u_j\in A^\sigma$ with
$l_j\uparrow x$ and $u_j\downarrow x$. For fixed $j$, substitute
$a_{i,k}$ for $u_i$, with $k\ge j$, in this bracket construction. Its
source endpoints $l_{j,k},u_{j,k}$ satisfy
\[
 l_{j,k}\downarrow_k l_j,\qquad u_{j,k}\uparrow_k u_j,\qquad
 l_{j,k}\le u_{j,k}.
\]
Consequently $g_{j,k}=l_{j,k}/u_{j,k}\uparrow_k g_j=l_j/u_j$.
The recursive proof of Lemma~\ref{mod:word-constants}, applied to lower
and upper interval endpoints at each operation, gives
\[
 I_n(l_{j,k},u_{j,k})=V_A(g_{j,k})
 \subseteq\sum_{i=1}^r w_i(P)(T_{i,j}-T_{i,j})=:B_j,
\]
where the decreasing null families are
\[
 T_{i,j}=\bigcup_{p,q\ge j}I_n(a_{i,p},a_{i,q}).
\]
For $c\in A$ below $g_j$,
$c\wedge g_{j,k}\uparrow c$, so $n(c)\in[B_j]_\theta$. It follows that
\[
 \Gamma_A(g_j)\subseteq D\cap[B_j]_\theta.
\]
For each $i$, expand $w_i(P)(T_{i,j}-T_{i,j})$ into $w_i(P)$ rows
$T_{i,j}$ and $w_i(P)$ rows $-T_{i,j}$. These rows, rather than the sets
$J_v$, are the admissible subsets of $D$.
Lemma~\ref{mod:hull-bookkeeping} makes their independently indexed
relative hull null. The common index $j$ is its cofinal diagonal because
there are only finitely many rows. This proves nullity of the right side.
This proves \eqref{mod:finite-brackets}.

For completeness, hereditary interval approximation is verified directly.
Let $a,b\in K$ have source representations $a_k,b_k$, and let
$a\wedge b\le s\le t\le a\vee b$ with $s,t\in K$. Represent $s,t$ by
source rows $s_k,t_k$ and define
\[
 d_k=a_k\wedge b_k,\quad e_k=a_k\vee b_k,\quad
 s'_k=(s_k\vee d_k)\wedge e_k,\quad
 t'_k=(t_k\vee d_k)\wedge e_k,\quad
 p_k=s'_k\wedge t'_k,\quad q_k=t'_k.
\]
Then $d_k\le p_k\le q_k\le e_k$. Sequential continuity of the lattice
operations gives $p_k\to_o s\wedge t=s$ and $q_k\to_o t$.
The explicit word-tail bound above makes $p_k,q_k$ interval-Cauchy.
Representation independence gives $n(p_k)\to n_0(s)$ and
$n(q_k)\to n_0(t)$. For every $j$, all $k\ge j$ satisfy
$n(q_k)-n(p_k)\in\bigcup_{k\ge j}I_n(a_k,b_k)$. Taking the limit gives
\begin{equation}\label{mod:finite-hereditary-approximation}
 I_{n_0}(a,b)\subseteq
 \left[\bigcup_{k\ge j}I_n(a_k,b_k)\right]_\theta.
\end{equation}
With constant endpoint representations this gives
\eqref{mod:finite-interval-density}. With $a=x$, $b=p_j$, and the source
representation $(p_k)$ for $x$, it gives, uniformly for $j\ge N$,
\[
 I_{n_0}(x,p_j)\subseteq
 D\cap\left[\bigcup_{p,q\ge N}I_n(p_p,p_q)\right]_\theta.
\]
One-row hull control proves the remaining assertion.
\end{proof}

\subsection{A transfer lemma for increasing-chain tails}

The next lemma replaces the disjointization argument available on a Boolean
algebra. It uses the range principle $\mathsf U(N)$ on four families of
points of $N$, not on a family of uncontrolled formal error sums.

\begin{lemma}[Transfer of interval exhaustivity]\label{mod:tail-transfer}
Let $B\subseteq E$ be $D$-sublattices and let $v:E\to N$ be additive and
modular. Suppose that every increasing sequence $(b_j)\subseteq B$ satisfies
\begin{equation}\label{mod:base-chain-tails}
 \left(\bigcup_{m\ge N}I_v(b_N,b_m)\right)_N\Rightarrow0,
\end{equation}
where the intervals in \eqref{mod:base-chain-tails} are computed in $B$.
Suppose also that
\begin{equation}\label{mod:base-interval-domination}
 I_v^E(a,b)\subseteq[I_v^B(a,b)]_\theta\qquad(a,b\in B).
\end{equation}
For every $x\in E$, suppose there are a directed domain $J_x$, points
$p_{x,\alpha}\in B$, and decreasing null sets
$0\in U_{x,\alpha}\subseteq D$ such that
\begin{equation}\label{mod:transfer-approximation}
 I_v^E(x,p_{x,\alpha})\subseteq U_{x,\alpha}.
\end{equation}
Then every increasing sequence $(x_j)\subseteq E$ satisfies
\begin{equation}\label{mod:extended-chain-tails}
 Z_N:=\bigcup_{m\ge N}I_v^E(x_N,x_m)\Rightarrow0.
\end{equation}
\end{lemma}

\begin{proof}
For each $i$ use the row supplied for $x_i$. On the product of their
domains, put
\[
 p_i^\phi=p_{x_i,\phi(i)},\qquad
 y_j^\phi=\bigvee_{i\le j}p_i^\phi\in B,
\]
and
\[
 A_\phi=\bigcup_{F\subseteq\mathbb N\text{ finite}}
                   \sum_{i\in F}U_{x_i,\phi(i)},\qquad
 Q_\phi=D\cap[A_\phi]_\theta.
\]
The family $Q_\phi$ is null. Since $x_j=\bigvee_{i\le j}x_i$,
\eqref{mod:finite-lattice-errors} gives
\begin{equation}\label{mod:corrected-chain-error}
 I_v^E(x_j,y_j^\phi)\subseteq Q_\phi
 \qquad\text{for every }j.
\end{equation}
Let
\[
 H_\phi=D\cap[A_\phi-A_\phi]_\theta.
\]
Duplicating the rows with opposite signs proves $H_\phi\Rightarrow0$.

Fix $r\le m$ and $x_r\le s\le t\le x_m$. Define
\[
 p=(s\vee y_r^\phi)\wedge y_m^\phi,\qquad
 q=(t\vee y_r^\phi)\wedge y_m^\phi.
\]
Then $y_r^\phi\le p\le q\le y_m^\phi$. To estimate the first endpoint,
insert $s\wedge y_m^\phi$ between $s$ and $p$. Meet and join contractions
give
\[
 v(p)-v(s)\in
 I_v^E(x_r,y_r^\phi)-I_v^E(x_m,y_m^\phi)\subseteq H_\phi.
\]
The second endpoint satisfies the separately recorded identity
\[
 v(q)-v(t)=\bigl(v(q)-v(t\wedge y_m^\phi)\bigr)
                -\bigl(v(t)-v(t\wedge y_m^\phi)\bigr)
 \in I_v^E(x_r,y_r^\phi)-I_v^E(x_m,y_m^\phi)\subseteq H_\phi.
\]
The final inclusion
uses \eqref{mod:hull-sums} and the fact that the displayed value differences
belong to $N-N\subseteq D$. These estimates are uniform in $r,m,s,t$.

For fixed $\phi$, define the decreasing family
\[
 B_N^\phi=\bigcup_{m\ge N}I_v^B(y_N^\phi,y_m^\phi).
\]
The base-chain hypothesis gives $B_N^\phi\Rightarrow0$.
For every $r\ge N$, $m\ge r$, and the indicated $s,t$, the constructed
pair satisfies $y_N^\phi\le y_r^\phi\le p\le q\le y_m^\phi$, whence
\[
 v(q)-v(p)\in I_v^E(y_N^\phi,y_m^\phi)
 \subseteq[I_v^B(y_N^\phi,y_m^\phi)]_\theta
 \subseteq[B_N^\phi]_\theta.
\]
Since $v(q)-v(p)\in N-N\subseteq D$, all these values lie in
$D\cap[B_N^\phi]_\theta\Rightarrow0$ by \eqref{mod:one-row}.

Here is an explicit application of $\mathsf U(N)$. Use the fixed index set
\[
 \mathcal J=\{(r,m,s,t):r\le m,\ x_r\le s\le t\le x_m,\ s,t\in E\}.
\]
For $j=(r,m,s,t)\in\mathcal J$ define the four point families by
\[
 (X_N(j),Y_N(j),X_N^\phi(j),Y_N^\phi(j))=
 \begin{cases}
 (v(t),v(s),v(q),v(p))&r\ge N,\\
 (0,0,0,0)&r<N.
 \end{cases}
\]
All four point families lie in $N$. The preceding calculations give
\[
 \{X_N^\phi(j)-X_N(j):N\ge1,j\in\mathcal J\}\subseteq H_\phi,
 \quad
 \{Y_N^\phi(j)-Y_N(j):N\ge1,j\in\mathcal J\}\subseteq H_\phi.
\]
For fixed $\phi$, their approximating differences are uniformly null in
$N$. The two approximation errors are uniformly contained in
$H_\phi\cup(-H_\phi)$, and hence are null as $\phi$ varies. The conclusion
of $\mathsf U(N)$ makes the target differences uniformly null. Their range
at row $N$ is exactly $Z_N$: the case $r=N$ supplies all its elements, and
$r>N$ supplies only subintervals of $[x_N,x_m]$. This proves
\eqref{mod:extended-chain-tails}.
\end{proof}

\begin{corollary}\label{mod:finite-chain-tails}
Every increasing sequence in $K$ has null hereditary interval tails.
\end{corollary}
\begin{proof}
Apply Lemma~\ref{mod:tail-transfer} with $B=A$, $E=K$, and $v=n_0$.
Lemma~\ref{mod:source-monotone-tails} supplies the base tails.
Equation \eqref{mod:finite-interval-density} supplies interval domination.
The decreasing relative hulls in the last paragraph of the proof of
Proposition~\ref{mod:finite-extension} supply the approximation rows.
\end{proof}

\subsection{Source certificates and the regulated domain}

\begin{definition}\label{mod:regulated-definition}
An element $x\in L$ is regulated if there are decreasing null rows
$0\in U_{i,\alpha}\subseteq D$, indexed by arbitrary directed domains
$J_i$, and, for $\phi\in\prod_iJ_i$, elements
$l_\phi\in A^\delta$, $u_\phi\in A^\sigma$ such that
\begin{equation}\label{mod:regulated-certificate}
 l_\phi\le x\le u_\phi,\qquad
 \Gamma_A(l_\phi/u_\phi)\subseteq Q_\phi,
\end{equation}
where
\[
 Q_\phi=D\cap
 \left[\bigcup_{F\subseteq\mathbb N\text{ finite}}
                        \sum_{i\in F}U_{i,\phi(i)}\right]_\theta.
\]
Write $R$ for the class of regulated elements.
\end{definition}

Only source fragment values occur in this definition. Neither an extension
value nor the convergence of an approximation net is part of the certificate.

\begin{proposition}[Values and finite operations on $R$]
\label{mod:regulated-finite}
The set $R$ is a $D$-sublattice containing $K$. There is an additive modular
map $\nu:R\to N$ extending $n_0$, characterized by
\begin{equation}\label{mod:certificate-value}
 n_\sigma(u_\phi)\tto{\theta}\nu(x)
\end{equation}
for every certificate of $x$. For every fixed certificate index,
\begin{equation}\label{mod:certificate-fixed-error}
 n_\sigma(u_\phi)-\nu(x)\in\Gamma_A(l_\phi/u_\phi),\qquad
 \nu(x)-n_\delta(l_\phi)\in\Gamma_A(l_\phi/u_\phi).
\end{equation}
Furthermore,
\begin{equation}\label{mod:regulated-interval-density}
 I_\nu(a,b)\subseteq[I_{n_0}(a,b)]_\theta\qquad(a,b\in K),
\end{equation}
and
\begin{equation}\label{mod:regulated-source-domination}
 x\in R,\quad x\le s\in A^\sigma
 \quad\Longrightarrow\quad\nu(x)\in\Gamma_A(s).
\end{equation}
\end{proposition}

\begin{proof}
The brackets in \eqref{mod:finite-brackets} certify every point of $K$,
using the single null row $\Gamma_A(l_j/u_j)$ and zero rows thereafter.
For a certificate of $x$ and two indices $\phi,\psi$, put
$p=u_\phi\wedge u_\psi$. Since both lower endpoints lie below $x$, we have
$l_\phi\le p\le u_\phi$ and $l_\psi\le p\le u_\psi$. Therefore
\[
 n_\sigma(u_\phi)-n_\sigma(u_\psi)
 \in\Gamma_A(l_\phi/u_\phi)-\Gamma_A(l_\psi/u_\psi)
 \subseteq Q_\phi-Q_\psi.
\]
The right side is uniformly null on the independent product. Completeness
of $N$ gives \eqref{mod:certificate-value}. For a second certificate
$(l'_\chi,u'_\chi,Q'_\chi)$, the element $p=u_\phi\wedge u'_\chi$ satisfies
$l_\phi\le p\le u_\phi$ and $l'_\chi\le p\le u'_\chi$. Hence
\[
 n_\sigma(u_\phi)-n_\sigma(u'_\chi)\in Q_\phi-Q'_\chi\Rightarrow0.
\]
The independent difference net also converges to the difference of its
two limits. Separation proves equality of those limits. If $x\in K$, then
$n_\sigma(u_\phi)-n_0(x)\in\Gamma_A(l_\phi/u_\phi)$, so this definition
extends $n_0$.

Fix $\phi$ and let $p_\psi=u_\phi\wedge u_\psi$. The difference
$n_\sigma(u_\psi)-n_0(p_\psi)$ belongs to the $\psi$ certificate error.
Thus $n_0(p_\psi)\to\nu(x)$. The elements $p_\psi$ lie in
$[l_\phi,u_\phi]$, so \eqref{mod:mixed-interval-domination} and stability
of $\Gamma_A(l_\phi/u_\phi)$ prove both assertions in
\eqref{mod:certificate-fixed-error}.

We next verify finite-operation closure. Let $x_1,\ldots,x_r\in R$
have certificates with row labels $(i,j)$, and let $P$ be a fixed word.
At a product index $\phi$ write their brackets as $l_i\le x_i\le u_i$.
Define a pair $(L_Q,U_Q)$ for every subword $Q$ recursively:
\[
 \begin{array}{c|c}
 Q&(L_Q,U_Q)\\ \hline
 X_i&(l_i,u_i)\\
 c&(c,c)\\
 Q_1\circ Q_2&(L_{Q_1}\circ L_{Q_2},U_{Q_1}\circ U_{Q_2})\quad(\circ=\wedge,\vee)\\
 Q_1^-&(U_{Q_1}^-,L_{Q_1}^-)\\
 Q_1^\sim&(U_{Q_1}^\sim,L_{Q_1}^\sim)\\
 \Delta(Q_1,Q_2)&(\Delta(L_{Q_1},U_{Q_2}),\Delta(U_{Q_1},L_{Q_2})).
 \end{array}
\]
Here constants are source elements. For each row of the table,
the monotonicities give $L_Q\le Q(x_1,\ldots,x_r)\le U_Q$.
The mixed-hull identities give $L_Q\in A^\delta$, $U_Q\in A^\sigma$.
Set $l=L_P$, $u=U_P$, and $G_i=\Gamma_A(l_i/u_i)$.
The recursions in Lemma~\ref{mod:word-constants} give
\begin{equation}\label{mod:regulated-word-width}
 I_{n_0}(l,u)\subseteq\sum_{i=1}^r w_i(P)(G_i-G_i).
\end{equation}
For a source fragment $c\le l/u$, the element $l+c$ belongs to $K\cap[l,u]$,
and
\[
 n(c)=n_0(l+c)-n_0(l)\in I_{n_0}(l,u).
\]
Writing $A_i$ for the raw sums of the $i$th certificate gives
$G_i\subseteq D\cap[A_i]_\theta$. Thus
\[
 \Gamma_A(l/u)\subseteq
 D\cap\left[\sum_{i=1}^r w_i(P)(A_i-A_i)\right]_\theta.
\]
To index this certificate independently, label each original row
$(i,j)$ by $(i,j,k,\epsilon)$ with
$1\le k\le\max\{1,w_i(P)\}$ and $\epsilon\in\{-1,1\}$.
Rows with $w_i(P)=0$ are replaced by zero rows. Use the signed rows
$\epsilon U_{i,j,\alpha}$, the index bijections and the selectors
$\beta_{i,j}$ of Lemma~\ref{mod:hull-bookkeeping}. Evaluate each input
bracket at these selected upper indices and then use the displayed
table for $L_P,U_P$. Decrease of the rows gives exactly the preceding
containment at the independent product index. This is a certificate
of Definition~\ref{mod:regulated-definition} for $P(x_1,\ldots,x_r)$.
This applies to both complements, both lattice operations, and the total
residuals. Their recovery identities give the defined residuals. Defined
sums are expressed by a residual and a complement. Hence $R$ is a
$D$-sublattice.

Let $P_0$ denote the original product of the input certificate domains
and let $P_C$ denote the independent-copy product used for this word.
The certificate value theorem first gives convergence on $P_C$:
\[
 n_0(U_P(\beta(\psi)))\longrightarrow\nu(P(x_1,\ldots,x_r)).
\]
Its restriction along the order-preserving cofinal map
$\delta:P_0\to P_C$ satisfies $\beta\delta(\phi)=\phi$, and hence
\[
 n_0(U_P(\phi))\longrightarrow\nu(P(x_1,\ldots,x_r))
       \quad(\phi\in P_0).
\]
Every word in the finitely many inputs has now been returned to this
original product $P_0$. For any endpoint tuple $e_i\in\{l_i,u_i\}$
the displayed table gives
$L_P\le P(e_1,\ldots,e_r)\le U_P$, so
\[
 n_0(U_P)-n_0(P(e_1,\ldots,e_r))\in I_{n_0}(L_P,U_P)\Rightarrow0.
\]
Subtracting the null difference from the preceding limit yields
\begin{equation}\label{mod:word-value-limit}
 n_0(P(e_1,\ldots,e_r))\to\nu(P(x_1,\ldots,x_r)).
\end{equation}
For upper endpoints $p,q$ of $x,y$, respectively, the identity
$n_0(p\vee q)+n_0(p\wedge q)=n_0(p)+n_0(q)$ and the four limits
in \eqref{mod:word-value-limit} give
$\nu(x\vee y)+\nu(x\wedge y)=\nu(x)+\nu(y)$.
For $x\le y$ in $R$, let $p$ and $q$ be their
independently indexed certificate upper endpoints. The expressions
$p\wedge q$ and $(p\wedge q)/q$ converge in value to $\nu(x)$ and
$\nu(x/y)$, respectively. The identity
\[
 n_0(q)=n_0(p\wedge q)+n_0((p\wedge q)/q)
\]
gives $\nu(y)=\nu(x)+\nu(x/y)$. This proves additivity.

It remains to prove the hereditary assertions without assuming normality
of $R$. Fix $a\le s\le t\le b$ with $a,b\in K$ and $s,t\in R$.
Use independent certificate upper endpoints $u_\psi^s,u_\chi^t$, and put
\[
 w_\psi=u_\psi^s\wedge b,\qquad z_\chi=u_\chi^t\wedge b,
 \qquad p_{\psi,\chi}=w_\psi\wedge z_\chi,\quad
 q_{\psi,\chi}=z_\chi.
\]
These points lie in $K\cap[a,b]$, with $p_{\psi,\chi}\le q_{\psi,\chi}$.
Since $l_\psi^s\le s\le b$, the difference between the values of
$u_\psi^s$ and $w_\psi$ lies in the $\psi$ error. Thus
$n_0(w_\psi)\to\nu(s)$. Also $z_\chi\ge t\ge s\ge l_\psi^s$, so
$p_{\psi,\chi}\ge l_\psi^s$ and
\[
 n_0(w_\psi)-n_0(p_{\psi,\chi})
       \in I_{n_0}(l_\psi^s,u_\psi^s).
\]
This error is null uniformly in $\chi$. Hence
$n_0(p_{\psi,\chi})\to\nu(s)$ and
$n_0(q_{\psi,\chi})\to\nu(t)$. Taking their difference proves
\eqref{mod:regulated-interval-density}. For noncomparable $a,b$, apply
the comparable statement to $a\wedge b,a\vee b$.

Finally, if $x\le s\in A^\sigma$, the clipped upper endpoints
$u_\phi\wedge s$ lie in $K$ and have values converging to $\nu(x)$,
because $l_\phi\le x\le s$. Each value belongs to $\Gamma_A(s)$ by
\eqref{mod:finite-domination}. Stability proves
\eqref{mod:regulated-source-domination}.
\end{proof}

\begin{corollary}\label{mod:regulated-chain-tails}
For every increasing sequence $(x_j)\subseteq R$,
\[
 \left(\bigcup_{m\ge N}I_\nu(x_N,x_m)\right)_N\Rightarrow0.
\]
\end{corollary}
\begin{proof}
Apply Lemma~\ref{mod:tail-transfer} with $B=K$, $E=R$, and $v=\nu$.
Corollary~\ref{mod:finite-chain-tails} supplies the base tails, and
\eqref{mod:regulated-interval-density} supplies interval domination.
For a certificate of $x$, its upper endpoint $u_\phi\in K$ satisfies
\[
 I_\nu(x,u_\phi)\subseteq I_\nu(l_\phi,u_\phi)
 \subseteq[I_{n_0}(l_\phi,u_\phi)]_\theta
 \subseteq\Gamma_A(l_\phi/u_\phi)\subseteq Q_\phi.
\]
These are precisely the decreasing null approximation rows required in
\eqref{mod:transfer-approximation}.
\end{proof}

\subsection{Countable closure, normality, and the double process}

\begin{proposition}[Countable closure and monotone continuity]
\label{mod:countable-closure}
The set $R$ is a $\sigma$-$D$-sublattice of $L$, and $\nu$ preserves ambient
increasing and decreasing sequential limits in $R$. In particular,
$\nu$ is ambiently $\sigma$-additive.
\end{proposition}

\begin{proof}
Let $b_j\in R$ increase, and put $b=\sup_j b_j$ in $L$.
For each $j$ fix a certificate with rows $U_{j,i,\alpha}$ and domains
$J_{j,i}$. Write
\[
 P=\prod_{(j,i)\in\mathbb N^2}J_{j,i},\qquad
 A_\phi=\bigcup_{F\subseteq\mathbb N^2\text{ finite}}
                      \sum_{(j,i)\in F}U_{j,i,\phi(j,i)},\qquad
 Q_\phi=D\cap[A_\phi]_\theta.
\]
The map $(j,i)\mapsto\pi(j,i)$ in Lemma~\ref{mod:hull-bookkeeping}
identifies this family with the summation axiom, so $Q_\phi\Rightarrow0$.
Let $l_j^\phi,u_j^\phi$ be the brackets at the restriction
$\phi|_{\{j\}\times\mathbb N}$, and put
$G_j^\phi=\Gamma_A(l_j^\phi/u_j^\phi)$.
For every finite $F\subseteq\mathbb N$ the disjoint row labels give
\begin{equation}\label{mod:countable-block-hull}
 \sum_{j\in F}G_j^\phi\subseteq[A_\phi]_\theta,\qquad
 D\cap\sum_{j\in F}G_j^\phi\subseteq Q_\phi.
\end{equation}
Define
\[
 U_m^\phi=\bigvee_{j\le m}u_j^\phi,\qquad
 U^\phi=\bigvee_j u_j^\phi\in A^\sigma,\qquad
 L_N^\phi=l_N^\phi\in A^\delta.
\]
For $m\ge N$,
\[
 L_N^\phi\le b_N\le b_m\le U_m^\phi.
\]
Put $Z_N=\bigcup_{m\ge N}I_\nu(b_N,b_m)$. These are decreasing null
sets by Corollary~\ref{mod:regulated-chain-tails}.
Interval splitting gives
\begin{equation}\label{mod:countable-split}
 I_\nu(L_N^\phi,U_m^\phi)
 \subseteq I_\nu(L_N^\phi,b_N)+I_\nu(b_N,b_m)
                         +I_\nu(b_m,U_m^\phi).
\end{equation}
The first term lies in the $N$th certificate error. Since
$b_m=\bigvee_{j\le m}b_j$, finite join contraction places the third term
in the sum of the first $m$ certificate errors. Thus the right side of
\eqref{mod:countable-split} lies in $Z_N+Q_\phi+Q_\phi$.

Fix $N,\phi$ and $a\in A$ with $a\le L_N^\phi/U^\phi$. The residuals
\[
 r_m=L_N^\phi/U_m^\phi\in A^\sigma\qquad(m\ge N)
\]
increase to $L_N^\phi/U^\phi$. Therefore
$a_m=a\wedge r_m\in A^\sigma$ increases to $a$. The elements
$L_N^\phi+a_m$ belong to $K\cap[L_N^\phi,U_m^\phi]$, and
\[
 n_\sigma(a_m)=\nu(L_N^\phi+a_m)-\nu(L_N^\phi)
                  \in Z_N+Q_\phi+Q_\phi.
\]
Increasing continuity of $n_\sigma$, already proved on $A^\sigma$, gives
\begin{equation}\label{mod:countable-shell}
 \Gamma_A(L_N^\phi/U^\phi)
       \subseteq[Z_N+Q_\phi+Q_\phi]_\theta.
\end{equation}
This argument has used no normality assertion on $R$.

For an independently indexed certificate use
\[
 \widehat P=\mathbb N\times\prod_{j,i\ge1,\ k\in\{1,2\}}J_{j,i},
 \qquad
 \beta(\psi)(j,i)=\beta_{j,i}(\psi(j,i,1),\psi(j,i,2)).
\]
The row with label $0$ is $Z_N$, and the row with label $(j,i,k)$ is
$U_{j,i,\psi(j,i,k)}$. Denote their raw finite sums by $\widehat A_{N,\psi}$.
Because $\beta(\psi)(j,i)\ge\psi(j,i,k)$,
\[
 Z_N+A_{\beta(\psi)}+A_{\beta(\psi)}
 \subseteq\widehat A_{N,\psi}.
\]
Each term of the first copy of $A$ uses label $(j,i,1)$, and each term of
the second copy uses $(j,i,2)$. Zero in every row permits omitted terms.
Equations \eqref{mod:countable-shell} and \eqref{mod:hull-sums} give
\[
 \Gamma_A(L_N^{\beta(\psi)}/U^{\beta(\psi)})
 \subseteq D\cap[Z_N+A_{\beta(\psi)}+A_{\beta(\psi)}]_\theta
 \subseteq D\cap[\widehat A_{N,\psi}]_\theta.
\]
The last family is null by the explicit countable reindexing in
Lemma~\ref{mod:hull-bookkeeping}. Thus
$L_N^{\beta(\psi)}\le b\le U^{\beta(\psi)}$ is a certificate on its full
independent product domain. Every countable join
is the increasing limit of its finite partial joins. Complementation
then gives countable meets. Thus $R$ is a $\sigma$-$D$-sublattice.

The value sequence $\nu(b_m)$ is Cauchy by its interval-tail estimate.
Let its limit in $N$ be $z$. For every fixed $\phi$,
\[
 n_\sigma(U_m^\phi)-\nu(b_m)
 \in\sum_{j\le m}\Gamma_A(l_j^\phi/u_j^\phi)\subseteq[A_\phi]_\theta.
\]
The sequence $U_m^\phi$ increases to $U^\phi$ inside $A^\sigma$.
Passing to the limit gives
$n_\sigma(U^\phi)-z\in[A_\phi]_\theta$. The difference lies in
$N-N\subseteq D$, hence belongs to $Q_\phi$. It follows that
$n_\sigma(U^\phi)\to z$ as $\phi$ varies.
The certificate just constructed gives
$n_\sigma(U^{\beta(\psi)})\to\nu(b)$ on $\widehat P$.
The diagonal $(N,\phi)\mapsto(N,\delta(\phi))$ is order preserving and
cofinal, with $\beta\delta(\phi)=\phi$. Therefore
$n_\sigma(U^\phi)\to\nu(b)$ on $\mathbb N\times P$.
For every $(N_0,\phi_0)$ its value tail is
\[
 \{n_\sigma(U^\phi):N\ge N_0,\ \phi\ge\phi_0\}
   =\{n_\sigma(U^\phi):\phi\ge\phi_0\}.
\]
These are exactly the tails of the net indexed by $P$. The latter
converges both to $z$ and to $\nu(b)$, so separation gives $z=\nu(b)$.
Complementation proves preservation of decreasing limits.

If $(h_j)\subseteq R$ is ordered orthogonal and its ambient orthosum $h$
belongs to $R$, its partial sums increase to $h$. Additivity and the
increasing continuity just proved give
$\sum_{j\le m}\nu(h_j)\to\nu(h)$. This is ambient $\sigma$-additivity.
\end{proof}

\begin{proposition}[Uniform normality and exhaustivity]
\label{mod:relative-normality}
The map $\nu$ is uniformly hereditarily exhaustive. If $w_j\downarrow0$
in $L$, then
\begin{equation}\label{mod:external-rn}
 \bigl(\{\nu(a):a\in R,\ a\le w_j\}\bigr)_j\Rightarrow0.
\end{equation}
\end{proposition}

\begin{proof}
For ordered orthogonal $h_j\in R$, put $s_0=0$ and
$s_j=h_1+\cdots+h_j$ for $j\ge1$.
They form an increasing sequence, and
$I_\nu(s_{j-1},s_j)=\{\nu(a):a\in R,\ a\le h_j\}$ by
\eqref{mod:interval-residual}. Corollary~\ref{mod:regulated-chain-tails}
therefore proves uniform hereditary exhaustivity.

First let $c_j\in R$ decrease to zero. Its complements $x_j=c_j^-$
increase to $1$. Put $Z_N=\bigcup_{m\ge N}I_\nu(x_N,x_m)$.
For $x_N\le s\le t\le1$ in $R$ and $m\ge N$, the pairs
$s\wedge x_m\le t\wedge x_m$ lie in $[x_N,x_m]$. Lattice
$\sigma$-continuity and monotone continuity of $\nu$ give
\[
 s\wedge x_m\uparrow s,\quad t\wedge x_m\uparrow t,\quad
 \nu(t\wedge x_m)-\nu(s\wedge x_m)\in Z_N,
 \quad \nu(t)-\nu(s)\in[Z_N]_\theta.
\]
Thus $I_\nu(x_N,1)\subseteq D\cap[Z_N]_\theta\Rightarrow0$.
The identities
$I_\nu(x_N,1)=I_\nu(0,c_N)=\{\nu(a):a\in R,a\le c_N\}$ prove
\[
 \bigl(\{\nu(a):a\in R,\ a\le c_N\}\bigr)_N\Rightarrow0.
\]

Now let $w_j\downarrow0$ be arbitrary in $L$ and write
$V_j=\{\nu(a):a\in R,\ a\le w_j\}$. These are decreasing subsets of
$D$ containing zero. On the evaluation domain
$\mathcal E=\{(j,z):z\in V_j\}$, whose order compares the first
coordinate, define
\[
 \sigma(h)=(|h|+1,0).
\]
For every $(j,z)\in\mathcal E$ fix $a(j,z)\in R$ with
$a(j,z)\le w_j$ and $\nu(a(j,z))=z$.
In a play its $j$th response $S_j$ lies in $V_j$. Put
$b_j=\bigvee_{z\in S_j}a(j,z)$. Then
$b_j\in R$, $b_j\le w_j$, and
$S_j\subseteq\{\nu(a):a\in R,a\le b_j\}$.
Countable closure gives
$c_N=\bigvee_{j\ge N}b_j\in R$. Then $c_N\le w_N$, so $c_N\downarrow0$.
The preceding hereditary normality makes the response evaluations null.
Finite-response reflection proves \eqref{mod:external-rn}.
\end{proof}

\begin{proposition}[Directed monotone processes]\label{mod:directed-processes}
Every increasing or decreasing net in $R$ has a $\nu$-value limit in $N$.
Its hereditary interval tails are null. Consequently, if $E\subseteq R$
is a $D$-sublattice and $w\in L$, the process
\[
 \lim_{a\in E,\ a\le w}\nu(a)
\]
exists. If $c\in R$, the fixed-context process
\[
 \lim_{a\in E,\ a\le w}\nu(T_c(a))
\]
also exists whenever the indicated Sasaki terms belong to $R$.
No intrinsic countable-additivity assertion about the inclusion $E\subseteq L$
is needed.
\end{proposition}

\begin{proof}
Let $a_\lambda\in R$ increase on a directed set $\Lambda$. Set
\[
 T_\gamma=\bigcup_{\lambda,\lambda'\ge\gamma}I_\nu(a_\lambda,a_{\lambda'}),
 \qquad\mathcal E=\{(\gamma,z):z\in T_\gamma\}.
\]
For each $(\gamma,z)$ fix witnesses
$(\lambda_{\gamma,z},\lambda'_{\gamma,z},s_{\gamma,z},t_{\gamma,z})$
for membership in $T_\gamma$. Fix a finite-upper-bound selector $\rho$
on $\Lambda$ and an initial index $\gamma_*$. Define
\[
 \sigma(\varnothing)=(\gamma_*,0),\qquad
 \sigma(h)=\left(\rho\left(\{\gamma_r\}\cup
 \{\lambda_{\gamma_r,z},\lambda'_{\gamma_r,z}:z\in S_r\}\right),0\right)
\]
for a history ending at $((\gamma_r,w_r),S_r)$.
Every play satisfies
\[
 S_j\subseteq I_\nu(a_{\gamma_j},a_{\gamma_{j+1}})
   =\{\nu(c):c\in R,c\le h_j\},\qquad
 h_j=a_{\gamma_j}/a_{\gamma_{j+1}},\qquad
 \sum_{j=1}^kh_j=a_{\gamma_1}/a_{\gamma_{k+1}}.
\]
Uniform hereditary exhaustivity from
Proposition~\ref{mod:relative-normality} gives $(S_j)_j\Rightarrow0$.
Thus $\mathsf F(D)$ gives $T_\gamma\Rightarrow0$.
The pairwise value tails lie in $T_\gamma-T_\gamma$, which is null.
Completeness of $N$ supplies the limit.
For a decreasing net, the identities
$I_\nu(a_\lambda,a_{\lambda'})=I_\nu(a_\lambda^-,a_{\lambda'}^-)$ and
$\nu(a_\lambda)=\nu(1)-\nu(a_\lambda^-)$ reduce its two conclusions to
the increasing net $(a_\lambda^-)$.
The domain $\{a\in E:a\le w\}$ is nonempty and directed by finite joins.
Its identity map is increasing. The map $a\mapsto T_c(a)$ is increasing
by the Sasaki identities. Both asserted limits are therefore instances
of the monotone-net conclusion.
\end{proof}

\begin{proposition}[Contextual outer regularity]\label{mod:regularity}
Let $x,z\in R$, and let $B$ be a subset of $R$ containing $A^\sigma$ and
closed under finite meets. Ordered by reverse lattice order,
\begin{equation}\label{mod:contextual-outer}
 \nu(z\vee h)\tto{\theta}\nu(z\vee x)
 \qquad(h\in B,\ x\le h).
\end{equation}
For $z=0$, this is outer regularity over $B$. Complementation gives the
corresponding inner regularity over families containing $A^\delta$ and
closed under finite joins.
\end{proposition}

\begin{proof}
The cover domain is nonempty because $1\in A^\sigma$, and it is directed
under reverse order. Fix a certificate $l_\phi\le x\le u_\phi$ of $x$.
Its upper endpoint belongs to $B$. Every subsequent cover $h\le u_\phi$
satisfies
\[
 z\vee l_\phi\le z\vee x\le z\vee h\le z\vee u_\phi.
\]
Meet and join contractions and \eqref{mod:regulated-interval-density}
give
\[
 \nu(z\vee h)-\nu(z\vee x)
 \in I_\nu(l_\phi,u_\phi)
 \subseteq\Gamma_A(l_\phi/u_\phi)\subseteq Q_\phi.
\]
For every certificate threshold, an entire cover tail has this bound.
The family $Q_\phi$ is null, proving \eqref{mod:contextual-outer}.
\end{proof}

\begin{theorem}[Uncontrolled modular Sion extension]\label{mod:extension}\label{mod:main}
Let $L$ be a $\sigma$-complete, lattice-$\sigma$-continuous pseudo-$D$-lattice,
let $A\subseteq L$ be a $D$-sublattice, and let $Y$ be a separated Abelian
net convergence group. Let $n:A\to Y$ be additive, modular, ambiently
sequentially normal, and uniformly hereditarily exhaustive. Put
$N=[n(A)]_\theta$. Suppose that $N$ is complete and that a symmetric subset
$D\subseteq Y$ satisfies
\[
 N-N\subseteq D=-D,\qquad
 \mathsf F(D),\qquad \mathsf S_{\rm all}(D),\qquad \mathsf U(N).
\]
Then there is a unique ambiently sequentially normal additive extension
\[
 \eta:\sigma_L(A)\longrightarrow Y
\]
of $n$. Its values belong to $N$, and it is modular and $\sigma$-additive.
It preserves sequential order convergence in $L$.
It is uniformly hereditarily exhaustive and satisfies
\[
 w_j\downarrow0\text{ in }L
 \quad\Longrightarrow\quad
 \bigl(\{\eta(a):a\in\sigma_L(A),\ a\le w_j\}\bigr)_j\Rightarrow0.
\]
Its directed monotone processes and contextual outer regularity are those
of Propositions~\ref{mod:directed-processes} and \ref{mod:regularity}.
For every $x\in L$, the original Sion limits
\begin{equation}\label{mod:sion-limits}
 g(x)=\lim_{a\in A,\ a\le x}n(a),\qquad
 \mu(x)=\lim_{u\in A^\sigma,\ x\le u}g(u)
\end{equation}
exist in $N$. Moreover,
\[
 g|_{A^\sigma}=n_\sigma,\qquad
 \mu|_{\sigma_L(A)}=\eta.
\]
The uniqueness conclusion also applies to every $\sigma$-additive additive
extension into $Y$, without a modularity or range-hull assumption on the
competing extension.
\end{theorem}

\begin{proof}
Proposition~\ref{mod:countable-closure} gives
$\sigma_L(A)\subseteq R$. Restrict $\nu$ to this generated
$\sigma$-$D$-sublattice and call the restriction $\eta$.
All asserted continuity, modularity, and uniform hereditary properties
have been proved on $R$.

For sequential order continuity, let $a_j\to_o a$ in $\sigma_L(A)$,
with canonical brackets $p_j\uparrow a$ and $q_j\downarrow a$.
Their residuals $g_j=p_j/q_j$ decrease to zero. Both $p_j/a_j$ and
$p_j/a$ lie below $g_j$. Proposition~\ref{mod:relative-normality} gives
\[
 \eta(a_j)-\eta(p_j)=\eta(p_j/a_j)\to0,\qquad
 \eta(a)-\eta(p_j)=\eta(p_j/a)\to0.
\]
Subtraction proves $\eta(a_j)\to\eta(a)$.

Uniqueness is the monotone-class consequence of
Theorems~\ref{pd:monotone-class} and \ref{pd:uniqueness}.
Indeed, the equality set of two ambiently sequentially normal additive
extensions contains $A$ and is closed under increasing and decreasing
sequential limits. The monotone class generated by $A$ is
$\sigma_L(A)$. A $\sigma$-additive competing extension is ambiently
sequentially normal by Lemma~\ref{pd:ambient-normality}. Its modularity
and range membership are not prerequisites for this argument.

The inner limit $g(x)$ exists by Lemma~\ref{mod:source-monotone-tails},
applied to the source elements below $x$. It remains to identify this
limit on $A^\sigma$, since a source exhaustion need not be cofinal in that
directed set. Let $a_j\uparrow u\in A^\sigma$ and fix $c\in A$, $c\le u$.
Lattice-$\sigma$-continuity gives $c\wedge a_j\uparrow c$. Source
normality and modularity imply
\[
 n(c\vee a_j)-n(a_j)=n(c)-n(c\wedge a_j)\longrightarrow0.
\]
Thus $n(c\vee a_j)\to n_\sigma(u)$. Define, on
$C_u=\{c\in A:c\le u\}$ ordered by $\le$,
\[
 B_c=\{n(d):d\in A,\ c\le d\le u\},\qquad W_c=B_c-B_c.
\]
The inner net is eventually in $B_c$, and $n(c\vee a_j)\in B_c$ for
every $j$. Hence
\[
 g(u),n_\sigma(u)\in[B_c]_\theta,
 \qquad g(u)-n_\sigma(u)\in D\cap[W_c]_\theta\quad(c\in C_u).
\]
The families $W_c$ decrease and lie in $D$. Cauchyness of the inner net
means $(W_c)_{c\in C_u}\Rightarrow0$. Lemma~\ref{mod:hull-bookkeeping}
gives $D\cap[W_c]_\theta\Rightarrow0$. The constant net with value
$g(u)-n_\sigma(u)$ has every tail contained in these evaluation tails,
so it converges to zero. Separation gives $g(u)=n_\sigma(u)$.

The outer domain in \eqref{mod:sion-limits} is nonempty and directed by
finite meets. Its elements belong to $A^\sigma\subseteq R$, and
$g(u)=n_\sigma(u)=\nu(u)$. Proposition~\ref{mod:directed-processes}
therefore gives the outer limit for every $x\in L$.
For $x\in\sigma_L(A)$, Proposition~\ref{mod:regularity}, with $z=0$ and
$B=A^\sigma$, identifies that limit with $\eta(x)$.
\end{proof}

\begin{remark}\label{mod:domain-scope}
The theorem identifies the extension on $\sigma_L(A)$ and preserves
Sion's two prescribed directed processes on all of $L$. It does not
identify $\sigma_L(A)$ with an unrestricted equality cut between inner
and outer envelopes. The regulated domain $R$ is an auxiliary
$\sigma$-$D$-sublattice established by source certificates. No closure
assertion about a larger Carath\'eodory equality cut is used.
\end{remark}

\section{Relative range conditions and their realizations}
\label{sec:relative-range}

The conditions in this section concern a homomorphism, a cone, and a fixed interval in the range. They contain no source or extension data.

\begin{definition}[Relative range conditions]\label{rel:range}\label{range:relative}
Let $X$ and $Y$ be separated Abelian net convergence groups, let $\pi:X\to Y$ be a continuous homomorphism, and let $P\subseteq X$ be a limit-stable additive cone containing $0$. Thus $P+P\subseteq P$. Pointedness is not required. Write
\[
 x\preceq y\quad\Longleftrightarrow\quad y-x\in P,
 \qquad I_e=\{x\in X:0\preceq x\preceq e\}
\]
for a fixed $e\in P$. The following are the \emph{relative range conditions at $e$}.
\begin{enumerate}
\item[\textnormal{(R1)}] Every increasing net $(x_\alpha)$ in $I_e$ whose image $(\pi x_\alpha)$ converges in $Y$ converges in $X$.
\item[\textnormal{(R2)}] If nets on a common directed domain satisfy
\[
 0\preceq x_\alpha\preceq b_\alpha\preceq e,
 \qquad b_\alpha\longrightarrow0,
 \qquad \pi x_\alpha\longrightarrow0,
\]
then $x_\alpha\to0$.
\item[\textnormal{(R3)}] Let $T$ be a rooted tree of countable depth. Each node $t$ has a nonempty directed set $J_t$ of children, written $t^\frown\alpha$ for $\alpha\in J_t$. Suppose that at every node there is a decreasing null row
\[
 q_{t,\alpha}\in I_e,
 \qquad q_{t,\beta}\preceq q_{t,\alpha}\quad(\alpha\le\beta),
 \qquad q_{t,\alpha}\longrightarrow0.
\]
Let $d\in I_e\cap\ker\pi$. Assume that for every infinite branch
\[
 t_0=\varnothing,
 \qquad t_{n+1}=t_n^\frown\alpha_n
\]
there is a sequence $(r_n)_{n\ge1}$ in $I_e$ with $r_n\to0$ and
\begin{equation}\label{rel:tree-inequality}
 d\preceq r_n+\sum_{j=0}^{n-1}q_{t_j,\alpha_j}
 \qquad(n\ge1).
\end{equation}
Then $d=0$.
\end{enumerate}
The sequences of remainders in \textnormal{(R3)} may depend on the entire branch. Neither those remainders nor the individual row values are required to lie in $\ker\pi$. No bound by $e$ is imposed on the finite sums in \eqref{rel:tree-inequality}.
\end{definition}

For a convergent net $x_\alpha\in I_e$,
\[
 x_\alpha\to x,
 \quad e-x_\alpha\to e-x,
 \quad x_\alpha,e-x_\alpha\in P
 \quad\Longrightarrow\quad x,e-x\in P.
\]
Thus $I_e$ is limit-stable. If $x_\alpha$ decreases in $I_e$ and $\pi x_\alpha\to y$, then $e-x_\alpha$ increases in $I_e$ and $\pi(e-x_\alpha)\to\pi e-y$. Condition \textnormal{(R1)} therefore also applies to decreasing nets.

\begin{lemma}[Separation in the kernel]\label{rel:kernel-separation}
Under \textnormal{(R3)}, $P\cap(-P)\cap\ker\pi=\{0\}$.
\end{lemma}
\begin{proof}
For $d\in P\cap(-P)\cap\ker\pi$,
\[
 e-d\in P,
 \qquad d\in I_e\cap\ker\pi,
 \qquad d\preceq0.
\]
On the tree with one child at every node set $q_{t,\alpha}=0$ and $r_n=0$. Then $d\preceq r_n+\sum_{j<n}q_{t_j,\alpha_j}=0$ for every $n$. Condition \textnormal{(R3)} gives $d=0$.
\end{proof}

\begin{proposition}[The uncontrolled realization]\label{rel:uncontrolled}
For every separated Abelian net convergence group $X$, the data
\[
 Y=X,\qquad \pi=\operatorname{id}_X,\qquad P=X
\]
satisfy the relative range conditions at every $e\in X$.
\end{proposition}
\begin{proof}
Here $I_e=X$ and $\ker\pi=\{0\}$. The three implications are
\[
 \begin{array}{ll}
 \mathrm{(R1)}&\pi x_\alpha\to y\ \Longrightarrow\ x_\alpha\to y,\\
 \mathrm{(R2)}&\pi x_\alpha\to0\ \Longrightarrow\ x_\alpha\to0,\\
 \mathrm{(R3)}&d\in I_e\cap\ker\pi\ \Longrightarrow\ d=0.
 \end{array}
\]
\end{proof}

We next verify the ordered realization. For an Abelian lattice-ordered group $G$, a net $x_\alpha$ converges in order to $x$, written $x_\alpha\to_o x$, when there is a decreasing net $b_\gamma$ in $G^+$ with existing infimum $0$ such that
\begin{equation}\label{rel:order-convergence}
 \forall\gamma\ \exists\alpha_0\ \forall\alpha\ge\alpha_0:
 \quad |x_\alpha-x|\le b_\gamma.
\end{equation}
The index set of the bounds need not equal the index set of the original net. Throughout the ordered realizations we use this actual order convergence on arbitrary nets.

Recall that $G$ is \emph{super Dedekind complete} when it is Dedekind complete and every nonempty bounded-above subset has a countable subset with the same supremum. The analogous infimum statement follows by negation. A \emph{bounded regulator} is a family $(a_{i,k})_{i,k\ge1}$ in $G^+$ bounded by one element of $G^+$ and satisfying $a_{i,k}\downarrow_k0$ for every $i$. Weak $\sigma$-distributivity means
\begin{equation}\label{rel:wsd}
 \inf_{\varphi\in\N^\N}\sup_{i\ge1}a_{i,\varphi(i)}=0
\end{equation}
for every bounded regulator. These conventions agree with the order-convergence definitions used elsewhere in the paper.

\begin{lemma}[A finite capped estimate]\label{rel:dyadic-cap}
In every Abelian lattice-ordered group, if $e,b_1,b_2,\ldots\ge0$ and $F\subset\{1,2,\ldots\}$ is finite, then
\begin{equation}\label{rel:dyadic-inequality}
 e\wedge\sum_{i\in F}b_i
 \le\bigvee_{i\in F}(e\wedge2^i b_i).
\end{equation}
The empty sum and the empty join are both interpreted as $0$.
\end{lemma}
\begin{proof}
Finite lattice inequalities in an Abelian lattice-ordered group are detected by its order-separating family of totally ordered quotient groups, by the subdirect representation theorem~\cite[p.~224, Theorem~13]{Birkhoff48}. It is therefore enough to prove the assertion in a totally ordered Abelian group. Suppose $F\ne\varnothing$ and put
\[
 B=\max_{i\in F}(e\wedge2^i b_i),\qquad N=\max F.
\]
If $B=e$, the left side of \eqref{rel:dyadic-inequality} is at most $B$. If $B<e$, then $2^i b_i\le B$ for every $i\in F$. Integer multiplication and addition give
\[
 2^N\sum_{i\in F}b_i
 \le\left(\sum_{i\in F}2^{N-i}\right)B
 \le(2^N-1)B\le2^N B.
\]
Total order gives $\sum_{i\in F}b_i\le B$: otherwise $2^N(\sum_{i\in F}b_i-B)>0$, contradicting the displayed inequality. The inequality holds in every order-separating quotient and therefore in the original lattice-ordered group. No divisibility is used.
\end{proof}

\begin{proposition}[The positive realization]\label{rel:positive}
Let $G$ be a super Dedekind complete weakly $\sigma$-distributive Abelian lattice-ordered group with its order convergence. The data
\[
 X=G,\qquad Y=\{0\},\qquad \pi=0,\qquad P=G^+
\]
satisfy the relative range conditions at every $e\in G^+$.
\end{proposition}
\begin{proof}
If $x_\alpha\ge0$ and $x_\alpha\to_o x$, then
\[
 0\le x^-\le |x-x_\alpha|\to_o0,
 \qquad x^-=0.
\]
Thus $G^+$ is limit-stable. For \textnormal{(R1)}, let $0\le x_\alpha\uparrow\le e$. Dedekind completeness gives $x=\sup_\alpha x_\alpha\in[0,e]$, and
\[
 |x-x_\alpha|=x-x_\alpha\downarrow0.
\]
For \textnormal{(R2)}, let $0\le x_\alpha\le b_\alpha\le e$ and $b_\alpha\to_o0$. If $c_\gamma\downarrow0$ is a family of order bounds for $b_\alpha$, then
\[
 \forall\gamma\ \exists\alpha_\gamma\ \forall\alpha\ge\alpha_\gamma:
 \quad 0\le x_\alpha\le b_\alpha\le c_\gamma.
\]
This is $x_\alpha\to_o0$.

For \textnormal{(R3)}, let $T$, $q_{t,\alpha}$, and $d$ satisfy its premises. We construct a countably branching tree mapped into $T$. Its nodes are the finite words $s\in\N^{<\omega}$, including $\varnothing$. Put $t_\varnothing=\varnothing$. Suppose $t_s$ has been defined. The identities
\[
 0\le q_{t_s,\alpha}\le e,
 \qquad q_{t_s,\alpha}\downarrow_\alpha0
\]
and super Dedekind completeness give indices $\beta_{s,k}\in J_{t_s}$ such that $\inf_k q_{t_s,\beta_{s,k}}=0$. Directedness permits indices $\alpha_{s,k}$ satisfying
\[
 \alpha_{s,1}\ge\beta_{s,1},
 \qquad
 \alpha_{s,k+1}\ge\alpha_{s,k},\beta_{s,k+1}.
\]
Define
\[
 t_{s^\frown k}=t_s^\frown\alpha_{s,k},
 \qquad b_{s,k}=q_{t_s,\alpha_{s,k}}.
\]
Then $b_{s,k}\downarrow_k0$. The construction is performed successively at each finite level. Every branch in $\N^{<\omega}$ maps to a branch of $T$, even if distinct retained child numbers represent the same original index. No cofinality assertion about $(\alpha_{s,k})_k$ is needed.

Fix a bijection $\iota:\N^{<\omega}\to\N\setminus\{0\}$. Define
\[
 a_{\iota(s),k}=e\wedge2^{\iota(s)}b_{s,k},
 \qquad
 R_\varphi=\sup_{s\in\N^{<\omega}}
                 a_{\iota(s),\varphi(\iota(s))}
 \quad(\varphi\in\N^\N).
\]
This is a regulator bounded by $e$, so $\inf_\varphi R_\varphi=0$. For each $\varphi$, define
\[
 s_0=\varnothing,
 \qquad k_n=\varphi(\iota(s_n)),
 \qquad s_{n+1}=s_n^\frown k_n,
 \qquad v_n=\sum_{j=0}^{n-1}b_{s_j,k_j}.
\]
The indices $\iota(s_0),\ldots,\iota(s_{n-1})$ are distinct. Lemma~\ref{rel:dyadic-cap} gives
\[
 e\wedge v_n
 \le\bigvee_{j<n}\left(e\wedge2^{\iota(s_j)}b_{s_j,k_j}\right)
 \le R_\varphi.
\]
The original branch premise supplies $r_n\in[0,e]$ with $r_n\to_o0$ and $d\le r_n+v_n$. Therefore
\[
 0\le(d-r_n)^+\le e\wedge v_n\le R_\varphi,
 \qquad
 |(d-r_n)^+-d|\le r_n\to_o0.
\]
Order-limit preservation of inequalities yields $d\le R_\varphi$. Since $\varphi$ was arbitrary,
\[
 0\le d\le\inf_{\varphi\in\N^\N}R_\varphi=0.
\]
This proves \textnormal{(R3)}.
\end{proof}

\begin{definition}[Lattice-normed groups and sequential completeness]\label{rel:lattice-norm}
Let $F$ be an Abelian group and $G$ an Abelian lattice-ordered group. A lattice norm is a map $p:F\to G^+$ satisfying
\[
 p(x)=0\Longleftrightarrow x=0,
 \qquad p(-x)=p(x),
 \qquad p(x+y)\le p(x)+p(y).
\]
The notation $x_\alpha\to_{bo}x$ means $p(x_\alpha-x)\to_o0$. A sequence is bo-Cauchy when the double sequence $p(x_n-x_m)$ converges in order to $0$. Sequential bo-completeness means that every bo-Cauchy sequence has a bo-limit in $F$. In the vector-space case, $G$ is a real vector lattice and the usual homogeneity of $p$ is also imposed. The argument below uses only the displayed group properties.
\end{definition}

\begin{proposition}[The graph-cone realization]\label{rel:graph}
Let $(F,p,G)$ be a sequentially bo-complete lattice-normed group, where $G$ is super Dedekind complete and weakly $\sigma$-distributive. Give $F\times G$ the product of bo-convergence and order convergence. For $u\in F$ and $M\in G^+$ with $p(u)\le M$, put
\[
 X=F\times G,\qquad Y=\{0\},\qquad\pi=0,
 \qquad P=\{(x,r):p(x)\le r\},\qquad e=(u,M).
\]
These data satisfy the relative range conditions at $e$.
\end{proposition}
\begin{proof}
The inequalities
\[
 p((v_\alpha+w_\alpha)-(v+w))
 \le p(v_\alpha-v)+p(w_\alpha-w),
 \qquad p(-v_\alpha+v)=p(v_\alpha-v)
\]
prove continuity of the group operations. Quasi-subnets and finite mixtures preserve the order-null norm nets. If $v_\alpha\to_{bo}v$ and $v_\alpha\to_{bo}w$, then
\[
 0\le p(v-w)\le p(v-v_\alpha)+p(v_\alpha-w)\to_o0,
 \qquad v=w.
\]
Thus $F$ is a separated Abelian convergence group. Further,
\[
 p(v)\le r,\ p(w)\le s
 \quad\Longrightarrow\quad p(v+w)\le r+s,
\]
so $P+P\subseteq P$. If $(x_\alpha,r_\alpha)\to(x,r)$ and $(x_\alpha,r_\alpha)\in P$, then
\[
 0\le(p(x)-r)^+
 \le p(x-x_\alpha)+|r_\alpha-r|\to_o0.
\]
Hence $(x,r)\in P$.

For \textnormal{(R1)}, write the increasing net as $(x_\alpha,r_\alpha)_{\alpha\in I}$. Membership in $I_e$ and monotonicity give
\[
 \begin{gathered}
 p(x_\alpha)\le r_\alpha\le M,
 \qquad p(u-x_\alpha)\le M-r_\alpha,\\
 \alpha\le\beta\quad\Longrightarrow\quad
 p(x_\beta-x_\alpha)\le r_\beta-r_\alpha.
 \end{gathered}
\]
Put $r=\sup_{\alpha\in I}r_\alpha$. Super Dedekind completeness supplies indices $\beta_n$ with $\sup_n r_{\beta_n}=r$. Directedness supplies increasing $\alpha_n$ with $\alpha_n\ge\beta_1,\ldots,\beta_n$. Thus $r_{\alpha_n}\uparrow r$. For $m,n\ge N$,
\[
 p(x_{\alpha_m}-x_{\alpha_n})
 \le r_{\alpha_{\max(m,n)}}-r_{\alpha_{\min(m,n)}}
 \le r-r_{\alpha_N}\downarrow_N0.
\]
Sequential bo-completeness gives $x_{\alpha_n}\to_{bo}x$ for some $x\in F$. For every $\alpha\in I$ and $n\ge1$, an upper bound $\beta\ge\alpha,\alpha_n$ gives
\[
 \begin{split}
 p(x_\alpha-x_{\alpha_n})
 &\le p(x_\beta-x_\alpha)+p(x_\beta-x_{\alpha_n})\\
 &\le(r_\beta-r_\alpha)+(r_\beta-r_{\alpha_n})\\
 &\le(r-r_\alpha)+(r-r_{\alpha_n}).
 \end{split}
\]
Consequently
\[
 p(x_\alpha-x)
 \le(r-r_\alpha)+(r-r_{\alpha_n})+p(x_{\alpha_n}-x).
\]
The last two terms converge in order to $0$ as $n\to\infty$. Preservation of inequalities under order limits gives
\begin{equation}\label{rel:graph-net-bound}
 p(x_\alpha-x)\le r-r_\alpha.
\end{equation}
Since $r-r_\alpha\downarrow0$, the full net converges to $(x,r)$. The inequalities defining $I_e$ pass to this limit and give $p(x)\le r$ and $p(u-x)\le M-r$. Thus $(x,r)\in I_e$.

For \textnormal{(R2)}, write $x_\alpha=(v_\alpha,s_\alpha)$ and $b_\alpha=(w_\alpha,t_\alpha)$. The graph inequalities give
\[
 0\le p(v_\alpha)\le s_\alpha\le t_\alpha.
\]
Since $b_\alpha\to0$, one has $t_\alpha\to_o0$. The displayed bounds give $s_\alpha\to_o0$ and $p(v_\alpha)\to_o0$, which is $x_\alpha\to0$.

For \textnormal{(R3)}, write
\[
 d=(v,a),\qquad q_{t,\alpha}=(w_{t,\alpha},b_{t,\alpha}),
 \qquad r_n=(z_n,c_n).
\]
The second coordinates satisfy
\[
 \begin{gathered}
 0\le a,b_{t,\alpha},c_n\le M,
 \qquad b_{t,\alpha}\downarrow_\alpha0,
 \qquad c_n\to_o0,\\
 a\le c_n+\sum_{j<n}b_{t_j,\alpha_j}
 \quad\text{on every branch}.
 \end{gathered}
\]
Proposition~\ref{rel:positive} gives $a=0$. Since $p(v)\le a$, one has $v=0$ and $d=0$.
\end{proof}

\begin{proposition}[Products]\label{rel:products}
For each $j\in J$, let $(X_j,Y_j,\pi_j,P_j,e_j)$ satisfy the relative range conditions. With product convergence, the data
\[
 X=\prod_jX_j,
 \quad Y=\prod_jY_j,
 \quad\pi=\prod_j\pi_j,
 \quad P=\prod_jP_j,
 \quad e=(e_j)_j
\]
also satisfy those conditions.
\end{proposition}
\begin{proof}
The identities
\[
 P+P\subseteq\prod_j(P_j+P_j)\subseteq P,
 \qquad I_e=\prod_j I_{e_j},
 \qquad\ker\pi=\prod_j\ker\pi_j
\]
hold coordinatewise. If $x_\alpha\in P$ and $x_\alpha\to x$, then $x_{\alpha,j}\to x_j\in P_j$ for every $j$, so $x\in P$.

For \textnormal{(R1)}, write $\pi x_\alpha\to y=(y_j)_j$. For every $j$,
\[
 x_{\alpha,j}\in I_{e_j},\quad
 x_{\alpha,j}\text{ increasing},\quad
 \pi_jx_{\alpha,j}\to y_j
 \quad\Longrightarrow\quad
 x_{\alpha,j}\to x_j\in I_{e_j}.
\]
Therefore $x_\alpha\to(x_j)_j\in I_e$ in product convergence. For \textnormal{(R2)}, its premises give, for every $j$,
\[
 0\preceq_jx_{\alpha,j}\preceq_jb_{\alpha,j}\preceq_je_j,
 \quad b_{\alpha,j}\to0,
 \quad\pi_jx_{\alpha,j}\to0.
\]
The factor condition gives $x_{\alpha,j}\to0$ for every $j$, hence $x_\alpha\to0$.

For \textnormal{(R3)}, fix $j$. On its unchanged tree, the projected data satisfy
\[
 \begin{gathered}
 d_j\in I_{e_j}\cap\ker\pi_j,
 \qquad q_{t,\alpha,j}\downarrow_\alpha0,\\
 \forall\text{ branches}\ \exists(r_{n,j})_n\subseteq I_{e_j}:
 \quad r_{n,j}\to0,
 \quad d_j\preceq_jr_{n,j}+\sum_{k<n}q_{t_k,\alpha_k,j}.
 \end{gathered}
\]
The factor condition yields $d_j=0$. Thus $d=(d_j)_j=0$.
\end{proof}

\begin{proposition}[Closed subgroups]\label{rel:subgroups}
Let $(X,Y,\pi,P,e)$ satisfy the relative range conditions. If $H\subseteq X$ is a limit-stable subgroup containing $e$, then
\[
 (H,Y,\pi|_H,P\cap H,e)
\]
satisfies them with the induced convergence on $H$.
\end{proposition}
\begin{proof}
Put $P_H=P\cap H$. Then
\[
 P_H+P_H\subseteq P_H,
 \qquad I_e^H=H\cap I_e,
 \qquad\ker(\pi|_H)=H\cap\ker\pi.
\]
Limits in $H$ are the limits in $X$ whose limit point belongs to $H$. Hence $P_H$ is limit-stable. For \textnormal{(R1)}, an increasing $H$-valued net in $I_e^H$ with convergent projected net has an ambient limit $x\in I_e$ by \textnormal{(R1)} in $X$. Limit stability of $H$ gives $x\in H$, hence $x\in I_e^H$. For \textnormal{(R2)}, the three hypotheses in $H$ are the three hypotheses in $X$. The ambient conclusion $x_\alpha\to0$, together with $0\in H$, is convergence in $H$. For \textnormal{(R3)}, each row, branch remainder, and distinguished point in $H$ belongs to the corresponding ambient interval, and
\[
 d\in\ker(\pi|_H)\ \Longrightarrow\ d\in\ker\pi.
\]
The ambient tree condition therefore yields $d=0$.
\end{proof}

\begin{proposition}[Order products beyond super Dedekind completeness]\label{rel:order-products}
Let $(G_j)_{j\in J}$ be super Dedekind complete weakly $\sigma$-distributive Abelian lattice-ordered groups. Give $G=\prod_jG_j$ its lattice order and its actual order convergence. Then $(G,\{0\},0,G^+,e)$ satisfies the relative range conditions for every $e\ge0$.
\end{proposition}
\begin{proof}
For every nonempty bounded-above $B\subseteq G$,
\[
 \sup_G B=\bigl(\sup_{G_j}\{b_j:b\in B\}\bigr)_{j\in J}.
\]
Thus $G$ is Dedekind complete. If $0\le x_\alpha\uparrow\le e$, then $x=\sup_\alpha x_\alpha$ satisfies $x-x_\alpha\downarrow0$, proving \textnormal{(R1)}. For \textnormal{(R2)}, each order bound $c_\gamma$ for $b_\alpha\to_o0$ also bounds $x_\alpha$ on the corresponding tail, since $0\le x_\alpha\le b_\alpha$.

If $v_\alpha\to_o v$ in $G$, its order bounds $c_\gamma\downarrow0$ satisfy
\[
 c_{\gamma,j}\downarrow0,
 \qquad
 \forall\gamma\ \exists\alpha_\gamma\ \forall\alpha\ge\alpha_\gamma:
 |v_{\alpha,j}-v_j|\le c_{\gamma,j}.
\]
Thus every coordinate projection preserves order convergence. The projected premises of \textnormal{(R3)} hold in $G_j$ for every $j$, so Proposition~\ref{rel:positive} gives $d_j=0$ for every $j$. Hence $d=0$.

For an uncountable set $J$, the example $G=\mathbb R^J$ is not super Dedekind complete. The singleton indicators $\mathbf1_{\{j\}}$ have supremum $\mathbf1_J$, while a countable subfamily has a supremum supported on only countably many coordinates. Nevertheless the preceding argument proves all three relative range conditions for its order convergence. This assertion does not identify product convergence with order convergence on arbitrary nets.
\end{proof}

The product and subgroup results provide mixed realizations with a nontrivial modular homomorphism. For example, if $Y$ is any separated Abelian convergence group and $G$ is a positive realization above, then
\[
 X=Y\times G,\qquad
 P=Y\times G^+,\qquad
 \pi(y,g)=y
\]
satisfies the relative range conditions. The same is true on every limit-stable subgroup of $Y\times G$, with the restricted cone and projection. A splitting of that subgroup or of its projection is not required.

\section{Relative extension over a modular homomorphism}
\label{sec:relative-lifting}

The modular theorem applies without an order or a positive control on its
range. The present section treats a different assertion. An additive
measure taking values in a stable cone is lifted over its modular image
under a continuous homomorphism. The cone and the three relative range
principles enter this lifting assertion explicitly.

Throughout this section, \(L\) is a \(\sigma\)-complete
lattice-\(\sigma\)-continuous pseudo-\(D\)-lattice, \(A\) is a
\(D\)-sublattice of \(L\), and \(S=\sigma_L(A)\).
Let \(\pi:X\to Y\) be a continuous homomorphism between separated
abelian convergence groups. Let \(P\subseteq X\) be a stable additive
cone containing zero, and let \(e\in P\). Write
\[
 x\preceq y\ \Longleftrightarrow\ y-x\in P,
 \qquad I_e=\{x\in X:0\preceq x\preceq e\}.
\]
The relative principles \(\mathrm{(R1)}\), \(\mathrm{(R2)}\), and
\(\mathrm{(R3)}\) are those of Definition~\ref{rel:range}, on this
single interval \(I_e\).

Let \(m:A\to P\) be additive and ambiently normal, with \(m(1)=e\).
Assume that \(n=\pi m\) is modular and uniformly hereditarily
exhaustive. Put \(N=[n(A)]_Y\), and assume that \(N\) is complete
and that
\[
 N-N\subseteq D=-D,\qquad
 \mathsf F(D),\qquad \mathsf S_{\rm all}(D),\qquad\mathsf U(N)
 \tag{L1}\label{lift:modular-data}
\]
hold. Since \(0=n(0)\in N\),
\[
 N=N-\{0\}\subseteq N-N\subseteq D.
\]
For \(a_j\in A\), \(a_j\downarrow0\) in \(L\),
\[
 m(a_j)\to0\Longrightarrow n(a_j)=\pi m(a_j)\to0.
\]
Thus \(n\) is ambiently normal.
Theorem~\ref{mod:extension} supplies the modular extension
\[
 \eta:S\longrightarrow N,\qquad \eta|_A=n.
 \tag{L2}\label{lift:eta}
\]
We use its ambient relative normality, its monotone directed
processes, and its contextual regularity from
Propositions~\ref{mod:relative-normality},
\ref{mod:directed-processes}, and~\ref{mod:regularity}.
All these assertions concern the already constructed map \(\eta\).
They also give sequential order continuity. Indeed, if
\(a_j\to_o a\) in \(S\), its canonical brackets satisfy
\(p_j\uparrow a\), \(q_j\downarrow a\), and
\[
 0\le p_j/a_j\le p_j/q_j\downarrow0,\qquad
 \eta(a_j)-\eta(p_j)=\eta(p_j/a_j)\longrightarrow0.
\]
Ambient normality gives
\(\eta(a)-\eta(p_j)=\eta(p_j/a)\to0\).
Thus \(\eta(a_j)\to\eta(a)\).

The interval \(I_e\) is stable. If \(x,y\in I_e\) and
\(x\preceq y\), then
\[
 y-x\in I_e,
 \qquad e-(y-x)=(e-y)+x\in P.
 \tag{L3}\label{lift:interval-difference}
\]
Limits therefore preserve every cone inequality used below.
Principle \(\mathrm{(R1)}\) also holds for decreasing nets in
\(I_e\), by applying its increasing version to \(e-x_\alpha\).
For every positive additive map \(f:E\to P\), where \(E\) is a
\(D\)-sublattice containing \(1\) and \(f(1)=e\), additivity gives
\[
 e-f(a)=f(a/1)\in P,\qquad f(E)\subseteq I_e.
 \tag{L4}\label{lift:range-bound}
\]

\begin{lemma}[Relative inner capacities]\label{lift:capacities}
Let \(A\subseteq E\subseteq S\) be a \(D\)-sublattice. Let
\(f:E\to I_e\) be positive and additive, with
\(\pi f=\eta|_E\), and assume that \(f\) is ambiently normal.
For every \(w\in L\), the net
\[
 \tau_E(w)=\lim_{\substack{a\in E\\a\le w}}f(a)
 \tag{L5}\label{lift:capacity}
\]
exists in \(I_e\). The index set has its lattice order. The map
\(\tau_E\) is increasing, and
\[
 0\preceq f(a)\preceq\tau_E(w)\quad(a\in E,\ a\le w).
 \tag{L6}\label{lift:capacity-bound}
\]
For every sequence \(w_N\downarrow0\) in \(L\),
\[
 \tau_E(w_N)\longrightarrow0,\qquad
 \{f(a):a\in E,\ a\le w_N\}\Rightarrow0.
 \tag{L7}\label{lift:relative-normality}
\]
\end{lemma}

\begin{proof}
For \(w\in L\), define the directed set
\[
 J(w)=\{a\in E:a\le w\},\qquad
 a\le_{J(w)}b\ \Longleftrightarrow\ a\le_L b.
\]
Its directedness and the required monotonicity are
\[
 0\in J(w),\qquad
 a,b\in J(w)\Longrightarrow a\vee b\in J(w),\qquad
 a\le b\Longrightarrow f(b)-f(a)=f(a/b)\in P.
\]
Proposition~\ref{mod:directed-processes}, with domain \(E\)
and bound \(w\), supplies
\(\kappa_w=\lim_{a\in J(w)}\eta(a)\in N\).
Thus the data for \(\mathrm{(R1)}\) are
\[
 x_a=f(a)\in I_e,\qquad
 a\le b\Longrightarrow x_a\preceq x_b,\qquad
 \pi x_a=\eta(a)\longrightarrow\kappa_w.
\]
Its conclusion is the existence of
\(\tau_E(w)=\lim_{a\in J(w)}x_a\in I_e\).
For fixed \(a\in E\), \(a\le w\), the tail indexed by
\(\{b\in E:a\le b\le w\}\) has every value above \(f(a)\).
Stability of \(P\) proves \eqref{lift:capacity-bound}.
If \(w\le v\), all the values defining \(\tau_E(w)\) lie
below \(\tau_E(v)\). Passage to their limit proves
\(\tau_E(w)\preceq\tau_E(v)\).

Fix \(w_N\downarrow0\). For \(c\in E\), define
\[
 \rho_c(N)=
 \lim_{\substack{a\in E\\a\le w_N}}f(T_ca).
 \tag{L8}\label{lift:rho}
\]
The transport \(T_c\) is increasing and preserves finite joins.
Its values belong to \(E\cap[0,c]\), by
Lemma~\ref{pd:sasaki-kernel} and
Theorem~\ref{pd:sasaki-joins}.
The application of \(\mathrm{(R1)}\) in \eqref{lift:rho}
has directed domain \(J(w_N)\) and values
\[
 x_a=f(T_ca)\in I_e,\quad
 a\le b\Longrightarrow x_a\preceq x_b,\quad
 \pi x_a=\eta(T_ca)\longrightarrow
 \kappa_{c,N}:=\lim_{a\in J(w_N)}\eta(T_ca)\in N.
\]
The last limit is
Proposition~\ref{mod:directed-processes}, instantiated at \(E,w_N,c\).
Since \(J(w_{N+1})\subseteq J(w_N)\) and \(T_ca\le c\),
passage to the defining limits gives
\[
 0\preceq \rho_c(N+1)\preceq\rho_c(N)\preceq f(c).
 \tag{L9}\label{lift:rho-bounds}
\]

To compute the projected limit of \eqref{lift:rho-bounds}, set
\[
 v_N=c/(c\vee w_N),\qquad b_N=v_N\vee w_N.
\]
The residual calculus gives \(v_N\downarrow0\) and
\(b_N\downarrow0\). The defining identity for \(T_c\) gives
\[
 T_ca\le \bigl(c/(c\vee a)\bigr)\vee a\le b_N
 \quad(a\le w_N).
\]
Consequently the decreasing families
\[
 V_N^c=\{\eta(T_ca):a\in E,\ a\le w_N\}
\]
are uniformly null by Proposition~\ref{mod:relative-normality}.
They contain zero and are contained in \(N\subseteq D\).
Continuity of \(\pi\) and stability of \(N\) give
\[
 \pi\rho_c(N)
 =\lim_{a\in J(w_N)}\eta(T_ca)
 \in N\cap[V_N^c]_Y\subseteq D\cap[V_N^c]_Y.
\]
In \(\mathsf S_{\rm all}(D)\), use
\[
 U_{1,N}=V_N^c,\qquad U_{i,k}=\{0\}\ (i\ge2).
\]
Its relative hull conclusion is
\((D\cap[V_N^c]_Y)_N\Rightarrow0\), so
\(\pi\rho_c(N)\to0\).
For the next application of \(\mathrm{(R1)}\), the values
\[
 z_N=e-\rho_c(N)\in I_e,\qquad
 z_N\preceq z_{N+1},\qquad
 \pi z_N\longrightarrow\pi e
\]
have a limit \(z_c\in I_e\). Set \(\gamma_c=e-z_c\).
Then
\[
 \gamma_c=\lim_N\rho_c(N)
 \in I_e\cap\ker\pi,\qquad
 0\preceq\gamma_c\preceq f(c).
 \tag{L10}\label{lift:gamma}
\]

Fix \(N\), \(c\in E\), and \(a\in E\) with \(a\le w_N\).
Put
\[
 d=T_ca,\qquad q=\rho_c(N)-f(d)\in I_e.
\]
For \(K\ge N\) and \(b\in E\), \(b\le w_K\), preservation
of joins gives
\[
 d\vee T_cb=T_c(a\vee b),\qquad
 f(d\vee T_cb)\preceq\rho_c(N).
\]
The crossed estimate and absorption from
Proposition~\ref{pd:crossed-estimate} and
Theorem~\ref{pd:sasaki-absorption}
give
\[
 f(d)+f(T_cb)-\rho_c(N)
 \preceq f(T_d(T_cb))
 \preceq f(T_db).
\]
For each \(K\ge N\), continuity of addition and stability of
\(P\), on the common index set \(J(w_K)\), give
\[
 f(d)+\rho_c(K)-\rho_c(N)\preceq\rho_d(K).
\]
The limit \(K\to\infty\) gives
\[
 f(d)+\gamma_c-\rho_c(N)\preceq\gamma_d,
\]
or equivalently
\[
 \gamma_c\preceq\gamma_d+q.
 \tag{L11}\label{lift:gamma-edge}
\]

Define a rooted tree and its contexts by
\[
 \mathcal T=\bigcup_{k\ge0}\prod_{j=1}^kJ(w_j),\qquad
 c_\varnothing=1,\qquad
 c_{h^\frown a}=T_{c_h}a
 \quad\bigl(h\in\mathcal T,\ |h|=k,\ a\in J(w_{k+1})\bigr).
\]
The empty product is \(\{\varnothing\}\). For a node \(h\) of
length \(k\), its directed child set is \(J(w_{k+1})\), and
its error row is
\[
 q_{h,a}=\rho_{c_h}(k+1)-f(c_{h^\frown a}).
\]
For every such node,
\[
 q_{h,a}\in I_e,\qquad
 a\le b\Longrightarrow q_{h,b}\preceq q_{h,a},\qquad
 \lim_{a\in J(w_{k+1})}q_{h,a}=0.
\]
For an arbitrary branch \((a_j)_{j\ge1}\in\prod_{j\ge1}J(w_j)\),
put
\[
 h_0=\varnothing,\quad h_N=(a_1,\ldots,a_N),\quad
 c_N=c_{h_N},\quad q_N=q_{h_{N-1},a_N},\quad r_N=f(c_N).
\]
The hypotheses of Theorem~\ref{pd:sasaki-disappearance} are
\[
 a_N\le w_N,\quad w_N\downarrow0,\quad
 c_0=1,\quad c_N=T_{c_{N-1}}a_N.
\]
Its conclusion and ambient normality are
\[
 c_N\downarrow0,\qquad r_N\in I_e,\qquad r_N\longrightarrow0.
\]
Equation~\eqref{lift:gamma-edge} gives, for each \(j\ge1\),
\(\gamma_{c_{j-1}}\preceq\gamma_{c_j}+q_j\).
Adding these \(N\) inequalities gives
\[
 \gamma_1\preceq\gamma_{c_N}+\sum_{j=1}^{N}q_j
 \preceq r_N+\sum_{j=1}^{N}q_j.
\]
Thus \(\mathrm{(R3)}\) applies to
\[
 d=\gamma_1\in I_e\cap\ker\pi,\quad
 (\mathcal T,(J(w_{|h|+1}))_h,(q_{h,a})_{h,a}),\quad
 (r_N)_N
\]
for every branch, and gives \(\gamma_1=0\).
Since \(T_1a=a\), \eqref{lift:rho} gives
\(\rho_1(N)=\tau_E(w_N)\). This proves the first conclusion
of \eqref{lift:relative-normality}.

For the second conclusion, define the directed preordered set
\[
 \mathcal J=\{(N,a):N\ge1,\ a\in E,\ a\le w_N\},\qquad
 (N,a)\le_{\mathcal J}(K,b)\ \Longleftrightarrow\ N\le K.
\]
On \(\mathcal J\), define \(x_{N,a}=f(a)\) and
\(b_{N,a}=\tau_E(w_N)\). All premises of \(\mathrm{(R2)}\) are
\[
 0\preceq x_{N,a}\preceq b_{N,a}\preceq e,\qquad
 b_{N,a}\longrightarrow0,\qquad
 \pi x_{N,a}=\eta(a)\longrightarrow0.
\]
The last convergence is Proposition~\ref{mod:relative-normality}
for the bounds \(w_N\). Hence \(x_{N,a}\to0\), which is
the asserted uniform evaluation limit.
\end{proof}

\begin{lemma}[Hereditary exhaustivity is inherited]\label{lift:he}
Let \(E\) be any \(D\)-sublattice with \(A\subseteq E\subseteq S\),
and let \(f:E\to I_e\) be positive and additive, with
\(\pi f=\eta|_E\). Then \(f\) is uniformly hereditarily exhaustive.
Ambient normality of \(f\) is not required for this assertion.
\end{lemma}

\begin{proof}
Let \((h_j)\) be an ordered orthogonal sequence in \(E\).
Set \(s_0=0\). Its partial sums
\(s_N=\sum_{j=1}^Nh_j\) belong to \(E\).
Their supremum \(s\) belongs to \(S\), and
\[
 \pi f(s_N)=\eta(s_N)\longrightarrow\eta(s).
\]
The sequence \(f(s_N)\) increases in \(I_e\).
Principle \(\mathrm{(R1)}\) makes it convergent in \(X\), whence
\[
 f(h_N)=f(s_N)-f(s_{N-1})\longrightarrow0.
\]
Set
\[
 \mathcal J_h=\{(N,a):N\ge1,\ a\in E,\ a\le h_N\},\qquad
 (N,a)\le(K,b)\Longleftrightarrow N\le K.
\]
On this directed preorder, the data for \(\mathrm{(R2)}\) are
\[
 x_{N,a}=f(a),\quad b_{N,a}=f(h_N),\quad
 0\preceq x_{N,a}\preceq b_{N,a}\preceq e,\quad
 b_{N,a}\to0,\quad \pi x_{N,a}=\eta(a)\to0.
\]
The last convergence is hereditary exhaustivity of \(\eta\)
for the ordered orthogonal sequence \((h_N)\).
Thus \((x_{N,a})_{\mathcal J_h}\to0\).
\end{proof}

\begin{lemma}[Extension to one order hull]\label{lift:successor}
Under the hypotheses of Lemma~\ref{lift:capacities}, there is a
positive additive ambiently normal extension
\[
 f^o:O(E)\longrightarrow I_e,\qquad \pi f^o=\eta|_{O(E)}.
\]
For every \(E\)-valued sequence \(a_n\to_o x\) in \(L\),
\[
 f(a_n)\longrightarrow f^o(x).
 \tag{L12}\label{lift:order-limit}
\]
\end{lemma}

\begin{proof}
Let \(a_n\in E\), \(a_n\to_o x\). For \(k\ge N\), put
\[
 p_{N,k}=\bigwedge_{N\le r\le k}a_r,\qquad
 q_{N,k}=\bigvee_{N\le r\le k}a_r.
\]
Their limits in \(L\) satisfy
\[
 p_{N,k}\downarrow_k p_N=\bigwedge_{r\ge N}a_r,\qquad
 q_{N,k}\uparrow_k q_N=\bigvee_{r\ge N}a_r,\qquad
 p_N\uparrow x,\quad q_N\downarrow x.
 \tag{L13}\label{lift:successor-brackets}
\]
All these lattice limits belong to \(S\).
For each fixed \(N\), the two applications of
\(\mathrm{(R1)}\) have increasing values
\[
 e-f(p_{N,k})\in I_e,\qquad f(q_{N,k})\in I_e
 \quad(k\ge N)
\]
and projected limits
\[
 \pi e-\eta(p_{N,k})\longrightarrow\pi e-\eta(p_N),
 \qquad \eta(q_{N,k})\longrightarrow\eta(q_N).
\]
Subtracting the first resulting limit from \(e\) defines
\[
 l_N=\lim_k f(p_{N,k}),\qquad
 u_N=\lim_k f(q_{N,k}).
\]
For \(k\ge N+1\), the finite brackets satisfy
\[
 p_{N,k}\le p_{N+1,k}\le q_{N+1,k}\le q_{N,k}.
\]
Positivity of \(f\), followed by their limits in \(k\), gives
\[
 0\preceq l_N\preceq l_{N+1}\preceq u_{N+1}
 \preceq u_N\preceq e.
 \tag{L14}\label{lift:value-brackets}
\]
Theorem~\ref{pd:total-residuals}, with the increasing first
coordinate \(q_{N,k}\) and decreasing second coordinate
\(p_{N,k}\), gives
\[
 g_N=p_N/q_N\downarrow0,\qquad
 p_{N,k}/q_{N,k}\uparrow_k g_N.
\]
The finite residuals belong to \(E\) and lie below \(g_N\).
Additivity and Lemma~\ref{lift:capacities} imply
\[
 0\preceq f(q_{N,k})-f(p_{N,k})
 =f(p_{N,k}/q_{N,k})\preceq\tau_E(g_N).
\]
After passage to the limit in \(k\),
\[
 0\preceq\delta_N:=u_N-l_N\preceq\tau_E(g_N)\preceq e.
 \tag{L15}\label{lift:gap}
\]
For this application of \(\mathrm{(R2)}\), set
\[
 x_N=\delta_N,\qquad b_N=\tau_E(g_N).
\]
Equation~\eqref{lift:gap}, Lemma~\ref{lift:capacities}, and
\eqref{lift:successor-brackets} verify
\[
 0\preceq x_N\preceq b_N\preceq e,\qquad
 b_N\longrightarrow0,\qquad
 \pi x_N=\eta(q_N)-\eta(p_N)\longrightarrow0.
\]
Thus \(\delta_N\to0\).
Its decreasing property is
\[
 \delta_N-\delta_{N+1}
 =(u_N-u_{N+1})+(l_{N+1}-l_N)\in P.
\]

The increasing nets \(l_N\) and \(e-u_N\) lie in \(I_e\), and
\[
 \pi l_N=\eta(p_N)\to\eta(x),\qquad
 \pi(e-u_N)=\pi e-\eta(q_N)\to\pi e-\eta(x).
\]
Principle \(\mathrm{(R1)}\) gives limits \(l\) and \(e-u\),
with \(l,u\in I_e\).
The equality \(u_N-l_N=\delta_N\to0\) and separation give
\(l=u\). Denote this common value by \(f^o(x)\).
For every \(n\),
\[
 l_n\preceq f(a_n)\preceq u_n,\qquad
 0\preceq f(a_n)-l_n\preceq\delta_n\preceq e.
\]
Moreover,
\[
 \pi\bigl(f(a_n)-l_n\bigr)
 =\eta(a_n)-\eta(p_n)\longrightarrow0.
\]
Principle \(\mathrm{(R2)}\), with
\(x_n=f(a_n)-l_n\) and \(b_n=\delta_n\), gives
\(f(a_n)-l_n\to0\). Since \(l_n\to f^o(x)\),
this proves \eqref{lift:order-limit}.

For two representations \(a_n\to_o x\), \(b_n\to_o x\), set
\[
 t_{2n-1}=a_n,\qquad t_{2n}=b_n,\qquad
 L_n=\left(\bigwedge_{k\ge n}a_k\right)
       \wedge\left(\bigwedge_{k\ge n}b_k\right),\qquad
 U_n=\left(\bigvee_{k\ge n}a_k\right)
       \vee\left(\bigvee_{k\ge n}b_k\right).
\]
Then
\[
 L_n\uparrow x,\quad U_n\downarrow x,\quad
 r\ge2n-1\Longrightarrow L_n\le t_r\le U_n.
\]
Thus \(t_r\to_o x\). Its value sequence has a limit \(z\)
by the construction above. The cofinal maps
\(n\mapsto2n-1\) and \(n\mapsto2n\) give
\[
 f(a_n)\to z,\qquad f(b_n)\to z.
\]
Separation identifies the two representation limits.
For \(x=a\in E\), the representation \(a_n=a\) gives
\(f^o(a)=f(a)\). For every representation,
\[
 \pi f^o(x)=\lim_n\pi f(a_n)
           =\lim_n\eta(a_n)=\eta(x).
\]

Proposition~\ref{pd:order-hull} makes \(O(E)\) a \(D\)-sublattice.
For \(x\le y\) in \(O(E)\), let \(a_n,b_n\in E\) represent
\(x,y\). Then
\[
 c_n=a_n\wedge b_n\to_o x,\qquad
 r_n=c_n/b_n\to_o x/y.
\]
The second convergence is Theorem~\ref{pd:total-residuals}.
Taking limits in \(f(b_n)=f(c_n)+f(r_n)\) gives
\[
 f^o(y)=f^o(x)+f^o(x/y).
 \tag{L16}\label{lift:successor-additivity}
\]
Positivity is \(f^o(x)\in I_e\subseteq P\). For \(x\le y\),
\eqref{lift:successor-additivity} gives
\(f^o(y)-f^o(x)=f^o(x/y)\in P\).
If \(x+y=z\) is defined, the residual identity \(x/z=y\)
and \eqref{lift:successor-additivity} give
\[
 f^o(x+y)=f^o(z)=f^o(x)+f^o(x/z)=f^o(x)+f^o(y).
\]

It remains to establish ambient normality. Let
\(x_j\in O(E)\), \(x_j\downarrow0\) in \(L\).
The data for \(\mathrm{(R1)}\) are the increasing sequence
\[
 z_j=e-f^o(x_j)\in I_e,\qquad
 \pi z_j=\pi e-\eta(x_j)\longrightarrow\pi e.
\]
If \(z=\lim_jz_j\), set \(d=e-z\). Then
\[
 d=\lim_jf^o(x_j)\in I_e\cap\ker\pi,\qquad
 d\preceq f^o(x_j).
 \tag{L17}\label{lift:successor-d}
\]
For each \(j\), fix a representation
\(a_{j,k}\in E\), \(a_{j,k}\to_o x_j\).
Define
\[
 y_{j,k}=\bigwedge_{r\le j}a_{r,k},\qquad
 Q_{r,k}=\bigvee_{s\ge k}a_{r,s},\qquad
 w_j=\bigwedge_{r\le j}Q_{r,j}.
\]
Finite lattice continuity gives \(y_{j,k}\to_o x_j\).
The sequence \(w_j\) decreases. For each fixed \(r\) and
\(j\ge r\), \(w_j\le Q_{r,j}\), so
\[
 \inf_jw_j\le\inf_jQ_{r,j}=x_r.
\]
Since \(\inf_rx_r=0\), we have \(w_j\downarrow0\).
Also \(y_{j,k}\le w_j\) whenever \(k\ge j\).

For each fixed \(j\), the instantiated brackets are
\[
 L_{j,k}=\lim_{r\to\infty}
 f\!\left(\bigwedge_{k\le s\le r}y_{j,s}\right),\qquad
 U_{j,k}=\lim_{r\to\infty}
 f\!\left(\bigvee_{k\le s\le r}y_{j,s}\right),\qquad
 \delta_{j,k}=U_{j,k}-L_{j,k}.
\]
Equations~\eqref{lift:value-brackets}--\eqref{lift:gap},
applied with \(a_k=y_{j,k}\) and \(x=x_j\), give
\[
 \delta_{j,k}\in I_e,\quad
 \delta_{j,k+1}\preceq\delta_{j,k},\quad
 \delta_{j,k}\to0,\quad
 L_{j,k}\preceq f(y_{j,k})\preceq U_{j,k},\qquad
 L_{j,k}\preceq f^o(x_j)\preceq U_{j,k}.
\]
In particular,
\[
 f^o(x_j)\preceq f(y_{j,k})+\delta_{j,k}.
 \tag{L18}\label{lift:successor-row}
\]
Define
\[
 \mathcal T=\bigcup_{j\ge0}\prod_{i=1}^j\{k\in\N:k\ge i\},
 \qquad
 J_h=\{k\in\N:k\ge |h|+1\},\qquad
 q_{h,k}=\delta_{|h|+1,k}.
\]
Each \(J_h\) has its usual order. The displayed properties of
\(\delta_{j,k}\) verify all error-row premises of
\(\mathrm{(R3)}\). For an arbitrary branch \((k_j)\),
put \(r_j=f(y_{j,k_j})\in I_e\). The inequalities
\[
 k_j\ge j,\qquad y_{j,k_j}\le w_j,\qquad w_j\downarrow0
\]
and Lemma~\ref{lift:capacities} give \(r_j\to0\). Equations
\eqref{lift:successor-d} and~\eqref{lift:successor-row} give
\[
 d\preceq r_j+\delta_{j,k_j}
 \preceq r_j+\sum_{i=1}^j\delta_{i,k_i}.
\]
With \(d\in I_e\cap\ker\pi\) from
\eqref{lift:successor-d}, these are the branch premises of
\(\mathrm{(R3)}\). Its conclusion is \(d=0\).
Thus \(f^o\) is ambiently normal. Lemma~\ref{lift:capacities}
now also gives its uniform relative ambient normality.
\end{proof}

Define the domain hierarchy by
\[
 E_0=A,\qquad E_{\alpha+1}=O(E_\alpha),\qquad
 E_\lambda=\bigcup_{\alpha<\lambda}E_\alpha
 \quad(\lambda\text{ a nonzero limit ordinal}).
 \tag{L19}\label{lift:hierarchy}
\]
Each \(E_\alpha\) is a \(D\)-sublattice. We next verify the
continuity needed at a countable limit. This is the stage at
which a pointwise union of normal maps alone would not suffice.

\begin{lemma}[Contextual flattening at a countable limit]
\label{lift:limit}
Let \(\lambda<\omega_1\) be a nonzero limit ordinal.
Suppose that coherent positive additive ambiently normal maps
\[
 f_\alpha:E_\alpha\to I_e,\qquad
 \pi f_\alpha=\eta|_{E_\alpha}\quad(\alpha<\lambda)
\]
have been constructed, with \(f_0=m\).
Their union \(f_\lambda:E_\lambda\to I_e\)
is positive, additive, and ambiently normal. For every
\(y,z\in E_\lambda\),
\[
 f_\lambda(z\vee y)=
 \lim_{\substack{h\in A^\sigma\\y\le h}}
 f_\lambda(z\vee h),
 \tag{L20}\label{lift:contextual}
\]
where the covers are ordered by reverse lattice order.
\end{lemma}

\begin{proof}
For \(x\in E_\lambda\), put
\(\rho_\lambda(x)=\min\{\alpha<\lambda:x\in E_\alpha\}\).
For \(x,y\in E_\lambda\), let
\(\beta=\max\{\rho_\lambda(x),\rho_\lambda(y)\}<\lambda\).
Then \(x,y\in E_\beta\). If \(x+y\) is defined, it belongs
to \(E_\beta\), and
\[
 f_\lambda(x+y)=f_\beta(x+y)
 =f_\beta(x)+f_\beta(y)=f_\lambda(x)+f_\lambda(y).
\]
For every \(x\in E_\lambda\),
\[
 f_\lambda(x)=f_{\rho_\lambda(x)}(x)\in I_e,\qquad
 \pi f_\lambda(x)=\eta(x),\qquad f_\lambda(1)=f_0(1)=e.
\]
Set \(B=A^\sigma\). Lemma~\ref{pd:mixed-hulls} gives
\[
 1\in B\subseteq E_1,\qquad
 h,k\in B\Longrightarrow h\wedge k\in B,\qquad
 (h_i)\subseteq B\Longrightarrow\bigvee_i h_i\in B.
\]
For \(y\in E_\lambda\), define
\[
 H(y)=\{h\in B:y\le h\},\qquad
 h\le_{H(y)}k\Longleftrightarrow k\le_L h.
\]
The directedness is \(1\in H(y)\) and
\(h,k\le_{H(y)}h\wedge k\).
For fixed \(z\in E_\lambda\), the data for
\(\mathrm{(R1)}\) are
\[
 z_h=e-f_\lambda(z\vee h)\in I_e,\quad
 h\le_{H(y)}k\Longrightarrow z_h\preceq z_k,\quad
 \pi z_h=\pi e-\eta(z\vee h)
       \longrightarrow\pi e-\eta(z\vee y).
\]
The last limit is Proposition~\ref{mod:regularity}, with
the specified elements \(y,z\) and the class \(A^\sigma\).
Set \(q=e-\lim_{h\in H(y)}z_h\).
For every \(h\in H(y)\),
\[
 f_\lambda(z\vee h)-f_\lambda(z\vee y)
 =f_\lambda\bigl((z\vee y)/(z\vee h)\bigr)\in P.
\]
Passing to the limit, and then using the tail of the
cover net at a fixed \(h\), gives
\[
 f_\lambda(z\vee y)\preceq q
 \preceq f_\lambda(z\vee h)\quad(h\in B,\ y\le h),
 \qquad
 d=q-f_\lambda(z\vee y)\in I_e\cap\ker\pi.
 \tag{L21}\label{lift:context-d}
\]

For \(\alpha<\lambda\), let \(\mathcal Q_\alpha\) denote
\[
 \forall y\in E_\alpha\ \forall z\in E_\lambda:\quad
 \lim_{h\in H(y)}f_\lambda(z\vee h)=f_\lambda(z\vee y).
\]
For \(y\in E_0=A\), \(y\in H(y)\) and
\(y\le_{H(y)}h\Rightarrow h=y\), so \(\mathcal Q_0\) holds.
For a limit \(\beta<\lambda\),
\[
 E_\beta=\bigcup_{\alpha<\beta}E_\alpha,\qquad
 \bigl(\forall\alpha<\beta:\mathcal Q_\alpha\bigr)
 \Longrightarrow\mathcal Q_\beta.
\]
For the successor step, let \(y\in O(E_\alpha)\), where the
assertion is known for every element of \(E_\alpha\) and every
context in \(E_\lambda\). Fix \(e_n\in E_\alpha\) with
\(e_n\to_o y\), and put
\[
 \ell_n=\bigwedge_{r\ge n}e_r\uparrow y,\qquad
 u_n=\bigvee_{r\ge n}e_r\downarrow y.
\]
Set \(\zeta=\rho_\lambda(z)\) and
\[
 \theta=\max\{\zeta,\alpha+1,1\}<\lambda.
 \tag{L22}\label{lift:fixed-stage}
\]
The elements \(y,z,\ell_n,u_n\), and every element of \(B\)
belong to \(E_\theta\).

Set \(C_0=z\vee y\).
For \(j\ge0\), let \(\mathcal T_{2j}\) be the set of histories
\[
 h=((n_1,v_1),\ldots,(n_j,v_j)),\qquad
 0<n_1<\cdots<n_j,\quad v_i\in H(e_{n_i}).
\]
The empty history is the root. Define
\[
 C_h=C_0\vee\bigvee_{i=1}^jv_i,\qquad
 k_h=\begin{cases}0,&j=0,\\n_j,&j\ge1.\end{cases}
\]
The empty join is zero. Put
\[
 \mathcal T_{2j+1}
 =\{(h,n):h\in\mathcal T_{2j},\ n\in\N,\ n>k_h\},
 \qquad \mathcal T=\bigsqcup_{k\ge0}\mathcal T_k.
\]
The parent maps are
\[
 \operatorname{par}(h,n)=h,\qquad
 \operatorname{par}\bigl(h^\frown((n,v))\bigr)=(h,n).
\]
At an even node \(h\), the directed child set is
\(J_h=\{n\in\N:n>k_h\}\), with its usual order.
The child \((h,n)\) has context
\(W_{h,n}=C_h\vee e_n\). Define its incoming error by
\[
 q_{h,n}^{\rm rep}
 =f_\theta\bigl(C_h/(C_h\vee u_n)\bigr).
 \tag{L23}\label{lift:representation-error}
\]
For every even node \(h\),
\[
 C_h\in E_\theta,\quad y\le C_h,\quad
 C_h\vee u_n\downarrow C_h,\quad
 C_h/(C_h\vee u_n)\downarrow0.
\]
Ambient normality at the fixed stage \(\theta\) gives
\[
 q_{h,n}^{\rm rep}\in I_e,\quad
 n\le m\Longrightarrow q_{h,m}^{\rm rep}\preceq q_{h,n}^{\rm rep},
 \quad q_{h,n}^{\rm rep}\longrightarrow0.
\]
The actual edge increment is bounded by
\[
 0\preceq f_\theta(W_{h,n})-f_\theta(C_h)
 =f_\theta\bigl(C_h/(C_h\vee e_n)\bigr)
 \preceq q^{\rm rep}_{h,n}.
 \tag{L24}\label{lift:representation-edge}
\]
At the odd node \((h,n)\), the child set is \(H(e_n)\).
For \(v\in H(e_n)\), the child is the even history
\(h^\frown((n,v))\), with context \(C_h\vee v\).
Its error is the actual increment
\[
 q^{\rm cov}_{h,n,v}
 =f_\theta(C_h\vee v)-f_\theta(C_h\vee e_n).
 \tag{L25}\label{lift:cover-error}
\]
The induction hypothesis \(\mathcal Q_\alpha\), evaluated at
\(e_n\in E_\alpha\) and \(C_h\in E_\theta\subseteq E_\lambda\),
gives
\[
 \lim_{v\in H(e_n)} f_\theta(C_h\vee v)
 =f_\theta(C_h\vee e_n).
\]
Therefore
\[
 q^{\rm cov}_{h,n,v}\in I_e,\quad
 v\le_{H(e_n)}w\Longrightarrow
 q^{\rm cov}_{h,n,w}\preceq q^{\rm cov}_{h,n,v},\quad
 q^{\rm cov}_{h,n,v}\longrightarrow0.
\]
Membership in \(I_e\) uses
\eqref{lift:interval-difference}, since
\(C_h\vee e_n\le C_h\vee v\).

An arbitrary branch is a sequence \((n_j,v_j)_{j\ge1}\) with
\(n_{j+1}>n_j\) and \(v_j\in H(e_{n_j})\).
Write \(h_j=((n_1,v_1),\ldots,(n_j,v_j))\) and
define, for \(j\ge0\),
\[
 C_{2j}=C_{h_j},\qquad
 C_{2j+1}=C_{h_j}\vee e_{n_{j+1}},\qquad
 q_{2j+1}=q^{\rm rep}_{h_j,n_{j+1}},\qquad
 q_{2j+2}=q^{\rm cov}_{h_j,n_{j+1},v_{j+1}}.
\]
Since \(e_{n_{j+1}}\le v_{j+1}\),
\[
 C_{2j}\le C_{2j+1}\le C_{2j+2},\qquad
 0\preceq f_\theta(C_k)-f_\theta(C_{k-1})\preceq q_k.
\]
Set \(h_\infty=\bigvee_jv_j\in B\).
For every \(j\),
\[
 \ell_{n_j}\le e_{n_j}\le v_j\le h_\infty.
\]
Thus \(y\le h_\infty\), and
\[
 C_k\uparrow C_\infty=z\vee h_\infty.
 \tag{L26}\label{lift:context-limit}
\]
All the terms and the limit in \eqref{lift:context-limit}
belong to the single stage \(E_\theta\) fixed in
\eqref{lift:fixed-stage}. Its ambient normality, applied to
\(C_k/C_\infty\downarrow0\), gives
\[
 r_k=f_\theta(C_\infty)-f_\theta(C_k)
 =f_\theta(C_k/C_\infty)\longrightarrow0,
 \qquad r_k\in I_e.
\]
The complete list of branch premises for \(\mathrm{(R3)}\) is
\[
 d\in I_e\cap\ker\pi,\qquad r_k\in I_e,\qquad r_k\to0,
\]
together with the decreasing null error rows proved above and
the following inequality, for every \(k\ge1\):
\[
 d\preceq f_\theta(C_\infty)-f_\theta(C_0)
 =r_k+\sum_{j=1}^k
       \bigl(f_\theta(C_j)-f_\theta(C_{j-1})\bigr)
 \preceq r_k+\sum_{j=1}^kq_j.
\]
Principle \(\mathrm{(R3)}\) gives \(d=0\), proving
\(\mathcal Q_{\alpha+1}\).
The transfinite induction gives
\(\mathcal Q_\alpha\) for every \(\alpha<\lambda\), and hence
\eqref{lift:contextual} on \(E_\lambda\).
The only normal map used in this induction was \(f_\theta\),
where \(\theta<\lambda\).

We now deduce normality of the union from a fixed earlier stage.
For \(x\in E_\lambda\), apply \eqref{lift:contextual} to
\(y=x^{-}\) and \(z=0\). For a cover \(h\in B\),
\(x^{-}\le h\), its right complement \(a=h^\sim\) satisfies
\[
 a\in E_1,\qquad a\le x,\qquad
 f_1(a)=e-f_\lambda(h)\longrightarrow
 e-f_\lambda(x^-)=f_\lambda(x).
\]
Every such value is bounded above by \(\tau_{E_1}(x)\).
Stability of \(P\) gives
\[
 0\preceq f_\lambda(x)\preceq\tau_{E_1}(x).
 \tag{L27}\label{lift:fixed-capacity}
\]
For \(w_N\downarrow0\) in \(L\), set
\[
 \mathcal J_\lambda=\{(N,x):N\ge1,\ x\in E_\lambda,\ x\le w_N\},
 \qquad (N,x)\le(K,y)\Longleftrightarrow N\le K.
\]
The local \(\mathrm{(R2)}\) data are
\[
 a_{N,x}=f_\lambda(x),\quad b_{N,x}=\tau_{E_1}(w_N),\quad
 0\preceq a_{N,x}\preceq b_{N,x}\preceq e,\quad
 b_{N,x}\to0,\quad \pi a_{N,x}=\eta(x)\to0.
\]
The upper-bound convergence is Lemma~\ref{lift:capacities}
for the already normal map \(f_1\).
The projected convergence is
Proposition~\ref{mod:relative-normality} for the bounds \(w_N\).
Hence \((a_{N,x})_{\mathcal J_\lambda}\to0\).
For \(x_N\in E_\lambda\), \(x_N\downarrow0\), its values
form an evaluation selection with \(w_N=x_N\), so
\(f_\lambda(x_N)\to0\).
\end{proof}

\begin{theorem}[Relative Sion extension]\label{rel:lifting}
\label{lift:extension}
Let \(L\) be a \(\sigma\)-complete, lattice-\(\sigma\)-continuous
pseudo-\(D\)-lattice and let \(A\subseteq L\) be a \(D\)-sublattice.
Let \(\pi:X\to Y\) be a continuous homomorphism of separated Abelian
net convergence groups, and let \(P\subseteq X\) be a limit-stable
additive cone containing zero. Let \(m:A\to P\) be additive and
ambiently normal, and put \(e=m(1)\).
Assume \(\mathrm{(R1)}\)--\(\mathrm{(R3)}\) on \(I_e\), as defined
in Definition~\ref{rel:range}. Suppose that \(n=\pi m\) is modular and
uniformly hereditarily exhaustive, that \(N=[n(A)]_Y\) is complete,
and that a symmetric subset \(D\subseteq Y\) satisfies
\[
 N-N\subseteq D=-D,\qquad
 \mathsf F(D),\qquad\mathsf S_{\rm all}(D),\qquad\mathsf U(N).
\]
Let \(\eta\) be the modular extension of \(n\) supplied by
Theorem~\ref{mod:main}.
There is a positive additive extension
\[
 \bar m:S=\sigma_L(A)\longrightarrow I_e
\]
which is sequentially order-continuous and \(\sigma\)-additive,
with \(\pi\bar m=\eta\).
It is uniformly hereditarily exhaustive and satisfies
\[
 w_N\downarrow0\text{ in }L
 \quad\Longrightarrow\quad
 \{\bar m(a):a\in S,\ a\le w_N\}\Rightarrow0.
 \tag{L28}\label{lift:final-rn}
\]
It is the unique additive \(\sigma\)-additive extension of \(m\)
to \(S\), among all such \(X\)-valued extensions.

The original Sion processes
\[
 g(x)=\lim_{\substack{a\in A\\a\le x}}m(a),\qquad
 \mu(x)=\lim_{\substack{u\in A^\sigma\\x\le u}}g(u)
 \tag{L29}\label{lift:sion-processes}
\]
exist for every \(x\in L\). Their inner index sets have their
lattice order and their outer index sets have reverse lattice
order. They satisfy
\[
 g(u)=\bar m(u)\quad(u\in A^\sigma),\qquad
 \mu|_S=\bar m.
 \tag{L30}\label{lift:sion-identification}
\]
More generally, for \(x,z\in S\),
\[
 \bar m(z\vee x)
 =\lim_{\substack{u\in A^\sigma\\x\le u}}
       \bar m(z\vee u).
 \tag{L31}\label{lift:final-context}
\]
\end{theorem}

\begin{proof}
For \(\alpha<\omega_1\), the recursion invariant is
\[
 \begin{gathered}
 f_\alpha:E_\alpha\to I_e\text{ is positive and additive},
 \qquad f_\alpha|_A=m,\qquad \pi f_\alpha=\eta|_{E_\alpha},\\
 a_j\in E_\alpha,\ a_j\downarrow0\text{ in }L
       \Longrightarrow f_\alpha(a_j)\to0,\\
 \beta<\alpha\Longrightarrow f_\alpha|_{E_\beta}=f_\beta.
 \end{gathered}
\]
For \(\alpha=0\), these are the hypotheses on \(m\).
Given the invariant at \(\alpha\), define
\[
 f_{\alpha+1}=(f_\alpha)^o
\]
by Lemma~\ref{lift:successor}, instantiated at
\(E=E_\alpha\), \(f=f_\alpha\), and the fixed projection \(\eta\).
Its conclusions verify all three lines of the invariant.
At a nonzero countable limit \(\lambda\), define
\[
 f_\lambda(x)=f_\alpha(x)
 \quad\text{for any }\alpha<\lambda\text{ with }x\in E_\alpha.
\]
Coherence makes this value independent of \(\alpha\).
Lemma~\ref{lift:limit}, with this family and \(f_0=m\),
verifies ambient normality. Finite additivity, positivity,
the projection identity, and all restriction identities
hold in a common earlier stage. This proves the invariant
for every \(\alpha<\omega_1\).
Put
\[
 E_*=\bigcup_{\alpha<\omega_1}E_\alpha,\qquad
 f_*=\bigcup_{\alpha<\omega_1}f_\alpha.
\]
For \(a\in E_*\), let
\(\rho(a)=\min\{\alpha<\omega_1:a\in E_\alpha\}\).
For a sequence \((a_n)\subseteq E_*\), define
\[
 \gamma=\sup_{n\ge1}\bigl(\rho(a_n)+1\bigr)<\omega_1,\qquad
 b_k=\bigvee_{n\le k}a_n,\qquad c_k=\bigwedge_{n\le k}a_n.
\]
Then
\[
 a_n,b_k,c_k\in E_\gamma,\qquad
 b_k\uparrow\bigvee_na_n,\quad c_k\downarrow\bigwedge_na_n,
 \qquad
 \bigvee_na_n,\ \bigwedge_na_n\in E_{\gamma+1}\subseteq E_*.
\]
Every finite partial algebraic operation on \(E_*\) belongs to
the stage whose index is the maximum of the finitely many input
ranks. Therefore \(E_*\) is a \(\sigma\)-\(D\)-sublattice and
\(S\subseteq E_*\).
For the reverse inclusion, the recursion gives
\[
 E_0=A\subseteq S,\qquad
 \bigl(\forall\alpha<\lambda:E_\alpha\subseteq S\bigr)
       \Longrightarrow E_\lambda\subseteq S.
\]
At a successor stage, closure of \(S\) under countable meets
and joins gives
\[
 a_n\to_o a
 \quad\Longrightarrow\quad
 a=\bigvee_N\bigwedge_{n\ge N}a_n.
\]
Consequently \(E_\alpha\subseteq S\) for all \(\alpha<\omega_1\),
and \(E_*=S\).

Let \(\bar m=f_*\). Its positivity, additivity, source
restriction, and projection satisfy
\[
 \bar m(S)\subseteq I_e,\qquad
 \bar m|_A=m,\qquad \pi\bar m=\eta.
\]
For \(a_n\downarrow0\) in \(L\), with \(a_n\in S\), the ordinal
\(\gamma=\sup_n(\rho(a_n)+1)\) satisfies
\[
 a_n\in E_\gamma,\qquad
 \bar m(a_n)=f_\gamma(a_n)\longrightarrow0.
\]
Thus the instantiated hypotheses of Lemma~\ref{lift:capacities}
are \(E=S\), \(f=\bar m\), \(\pi f=\eta\), and this ambient
normality. Its conclusion is \eqref{lift:final-rn}.
Lemma~\ref{lift:he}, at \(E=S\), \(f=\bar m\), gives uniform
hereditary exhaustivity.

For an ordered orthogonal sequence \((a_j)\subseteq S\), let
\(s_N=\sum_{j=1}^Na_j\uparrow a\in S\).
The residuals \(s_N/a\downarrow0\) belong to \(S\), and
\[
 \bar m(a)-\sum_{j=1}^N\bar m(a_j)
 =\bar m(s_N/a)\longrightarrow0.
\]
Thus \(\bar m\) is \(\sigma\)-additive.
For \(a_n\to_o a\) in \(S\), put
\[
 p_n=\bigwedge_{j\ge n}a_j,\qquad q_n=\bigvee_{j\ge n}a_j.
\]
Then
\[
 p_n\uparrow a,\quad q_n\downarrow a,\quad
 p_n/a\downarrow0,\quad a/q_n\downarrow0,
\]
and ambient normality gives
\[
 \bar m(a)-\bar m(p_n)=\bar m(p_n/a)\to0,\qquad
 \bar m(q_n)-\bar m(a)=\bar m(a/q_n)\to0.
\]
The local \(\mathrm{(R2)}\) data are
\[
 x_n=\bar m(a_n)-\bar m(p_n),\qquad
 b_n=\bar m(q_n)-\bar m(p_n).
\]
They satisfy
\[
 0\preceq \bar m(a_n)-\bar m(p_n)
 \preceq \bar m(q_n)-\bar m(p_n)\preceq e.
\]
and
\[
 b_n\longrightarrow0,\qquad
 \pi x_n=\eta(a_n)-\eta(p_n)\longrightarrow0.
\]
Hence \(x_n\to0\) by \(\mathrm{(R2)}\), and
\(\bar m(a_n)=x_n+\bar m(p_n)\to\bar m(a)\).

Let \(v:S\to X\) be any additive \(\sigma\)-additive extension
of \(m\), and define
\[
 C=\{x\in S:v(x)=\bar m(x)\}.
\]
Then \(A\subseteq C\). For \(x_j\uparrow x\) in \(S\), put
\[
 x_0=0,\qquad h_j=x_{j-1}/x_j,\qquad
 \sum_{j=1}^Nh_j=x_N,\qquad \sum_{j\ge1}h_j=x.
\]
The \(\sigma\)-additivity of both maps gives
\[
 v(x)=\lim_Nv(x_N),\qquad
 \bar m(x)=\lim_N\bar m(x_N).
\]
For \(x_j\downarrow x\), right complementation gives
\(x_j^\sim\uparrow x^\sim\), and hence
\[
 \begin{aligned}
 v(x_j)=v(1)-v(x_j^\sim)\to v(1)-v(x^\sim)=v(x),\\
 \bar m(x_j)=\bar m(1)-\bar m(x_j^\sim)
     \to\bar m(1)-\bar m(x^\sim)=\bar m(x).
 \end{aligned}
\]
It follows by separation that \(C\) is closed under increasing
and decreasing ambient sequence limits. Theorem~\ref{pd:monotone-class},
instantiated at the initial \(D\)-sublattice \(A\), states
\[
 \bigcap\{M\subseteq L:A\subseteq M,\ M
       \text{ is closed under both monotone sequence limits}\}=S.
\]
Thus \(S\subseteq C\), proving \(v=\bar m\).

The inner limit in \eqref{lift:sion-processes} is
Lemma~\ref{lift:capacities} for \(E=A\).
For \(u\in A^\sigma\), let \(a_n\in A\), \(a_n\uparrow u\).
All values \(m(a)\), \(a\in A\), \(a\le u\), lie below
\(\bar m(u)\), so \(g(u)\preceq\bar m(u)\).
Also \(m(a_n)\preceq g(u)\). For every \(n\),
\[
 0\preceq g(u)-m(a_n)
 \preceq\bar m(u)-m(a_n)\preceq e.
\]
For \(\mathrm{(R2)}\), set
\[
 x_n=g(u)-m(a_n),\qquad b_n=\bar m(u)-m(a_n).
\]
The modular inner identity in Theorem~\ref{mod:extension} gives
\[
 \pi g(u)=\lim_{a\in A,\ a\le u}n(a)=\eta(u).
\]
Therefore the remaining premises are
\[
 b_n\to0,\qquad
 \pi x_n=\eta(u)-n(a_n)\to0.
\]
Principle \(\mathrm{(R2)}\) gives \(x_n\to0\). Since
\(m(a_n)\to\bar m(u)\), separation identifies
\(g(u)=\bar m(u)\).

For arbitrary \(x\in L\), define
\[
 H(x)=\{u\in A^\sigma:x\le u\},\qquad
 u\le_{H(x)}v\Longleftrightarrow v\le_Lu.
\]
Then \(1\in H(x)\) and \(u,v\le_{H(x)}u\wedge v\).
Proposition~\ref{mod:directed-processes}, applied to the
decreasing lattice net \(u\in H(x)\subseteq S\), supplies
\(\kappa_x=\lim_{u\in H(x)}\eta(u)\).
The increasing values
\[
 z_u=e-\bar m(u)\in I_e,\qquad
 \pi z_u=\pi e-\eta(u)\to\pi e-\kappa_x
\]
satisfy \(\mathrm{(R1)}\). Thus
\[
 \mu(x)=e-\lim_{u\in H(x)}z_u
       =\lim_{u\in H(x)}\bar m(u)
       =\lim_{u\in H(x)}g(u)
\]
exists in \(I_e\).

For \(x,z\in S\), set
\[
 \delta=\max\{\rho(x),\rho(z),1\},\qquad
 \lambda=\delta+\omega<\omega_1.
\]
Then \(\lambda\) is a limit ordinal, \(x,z\in E_\lambda\),
and \(A^\sigma\subseteq E_1\subseteq E_\lambda\).
Equation~\eqref{lift:contextual}, with \(y=x\) and the context
\(z\), gives
\[
 \bar m(z\vee x)
 =f_\lambda(z\vee x)
 =\lim_{u\in H(x)}f_\lambda(z\vee u)
 =\lim_{u\in H(x)}\bar m(z\vee u).
\]
Putting \(z=0\) proves
\(\bar m(x)=\lim_{u\in H(x)}g(u)=\mu(x)\).
\end{proof}

\begin{remark}[Scope of the measurable domain]\label{lift:domain-scope}
The asserted extension domain is the generated
\(\sigma\)-\(D\)-sublattice \(S\).
The existence of both Sion processes on \(L\) does not assert
that an unrestricted equality cut of their envelopes is a
\(\sigma\)-\(D\)-sublattice.
Equation~\eqref{lift:final-context} is a proved contextual
regularity formula on \(S\), with the original upper class
\(A^\sigma\).
\end{remark}

\section{Explicit extension corollaries}
\label{sec:pseudo-corollaries}

Throughout this section, $L$ is a $\sigma$-complete lattice-$\sigma$-continuous pseudo-$D$-lattice, $A\subseteq L$ is a $D$-sublattice containing $0,1$, and $S=\sigma_L(A)$. Source joins, meets, and order limits are computed in $L$.

\begin{lemma}[The zero projected data]\label{app:zero-package}
For $n:A\to\{0\}$, $n(a)=0$, the modular range assumptions hold with $N=D=\{0\}$.
\end{lemma}
\begin{proof}
Every net in $N$ is constant at $0$, so $N$ is stable and complete. Moreover,
\[
 \begin{gathered}
 n(a+b)=0=n(a)+n(b),\\
 n(a\vee b)+n(a\wedge b)=0=n(a)+n(b),\\
 a_k\downarrow0\quad\Longrightarrow\quad n(a_k)=0\longrightarrow0,\\
 \{n(a):a\in A,\ a\le h_k\}=\{0\}\Rightarrow0,\qquad
 N-N=D=-D=\{0\}.
 \end{gathered}
\]
For $\mathsf F(D)$, every tested net $u:\Lambda\to D$ satisfies $u_\alpha=0$. For $\mathsf S_{\rm all}(D)$, every row, finite-sum family, and relative stable hull is $\{0\}$:
\[
 U_{i,\alpha}=\{0\},\qquad
 A_\varphi=\{0\},\qquad
 D\cap[A_\varphi]=\{0\}\Rightarrow0.
\]
For $\mathsf U(N)$, all four point families and all their differences equal $0$. Thus its concluding evaluation is constant at $0$.
\end{proof}

\begin{corollary}[Uncontrolled modular extension]\label{app:uncontrolled}
Let $X$ be a separated Abelian net convergence group, and let $m:A\to X$ be an additive modular ambiently normal measure that is uniformly hereditarily exhaustive. Intrinsic $\sigma$-additivity is a sufficient source hypothesis for ambient normality. Put $N=[m(A)]_X$. Assume that $N$ is Cauchy complete and that there is a symmetric set $D\subseteq X$ such that
\[
 N-N\subseteq D,
 \qquad \mathsf F(D),
 \qquad \mathsf S_{\rm all}(D),
 \qquad \mathsf U(N).
\]
Then $m$ has a unique additive $\sigma$-additive extension $\bar m:S\to X$. The extension is modular and sequentially order continuous. It is uniformly hereditarily exhaustive and uniformly relatively ambiently normal. The original Sion processes exist for every $x\in L$:
\begin{equation}\label{app:uncontrolled-sion}
 g(x)=\lim_{\substack{a\in A\\a\le x}}m(a),
 \qquad
 \mu(x)=\lim_{\substack{u\in A^\sigma\\x\le u}}g(u).
\end{equation}
The first directed family is ordered increasingly and the second decreasingly. One has $\mu|_S=\bar m$, and
\begin{equation}\label{app:uncontrolled-context}
 \bar m(z\vee x)
 =\lim_{\substack{u\in A^\sigma\\x\le u}}\bar m(z\vee u)
 \qquad(x,z\in S).
\end{equation}
No order or positive controller on $X$ is assumed.
\end{corollary}
\begin{proof}
In Theorem~\ref{mod:main} set
\[
 Y=X,\qquad n=m,\qquad [n(A)]_Y=N.
\]
The source assumptions are additivity, modularity, ambient normality, and uniform hereditary exhaustivity, all stipulated here. Its range assumptions are exactly completeness of $N$, $N-N\subseteq D=-D$, and the three displayed principles. The theorem gives $\bar m$, \eqref{app:uncontrolled-sion}, and \eqref{app:uncontrolled-context}, with unrestricted $\sigma$-additive uniqueness.

The relative theorem has the identical specialization
\[
 \pi=\operatorname{id}_X,
 \qquad P=X,
 \qquad e=m(1),
 \qquad I_e=X,
 \qquad\pi m=m.
\]
Here $m(A)\subseteq P$ and $e-m(a)=m(a/1)\in P$. Proposition~\ref{rel:uncontrolled} verifies \textnormal{(R1)}--\textnormal{(R3)}. The quotient range assumptions remain the original assumptions on $N,D$.

If intrinsic $\sigma$-additivity is assumed instead of ambient normality, then
\[
 a_k\in A,\quad a_k\downarrow_L0
 \quad\Longrightarrow\quad
 \inf_A a_k=0
 \quad\Longrightarrow\quad m(a_k)\to0
\]
by Lemma~\ref{pd:ambient-normality}. Thus this alternative source hypothesis supplies exactly the required normality.
\end{proof}

\begin{corollary}[Positive extension on pseudo-$D$-lattices]\label{app:positive}
Let $G$ be a super Dedekind complete weakly $\sigma$-distributive Abelian lattice-ordered group. Every finite positive $\sigma$-additive measure $\nu:A\to G^+$ has a unique additive $\sigma$-additive extension $\bar\nu:S\to G$. This extension is positive and sequentially order continuous. With $M=\nu(1)$, its values belong to $[0,M]$. The Sion processes exist in order convergence for every $x\in L$, and on $S$ they give $\bar\nu$. In particular,
\begin{equation}\label{app:positive-envelopes}
 \begin{split}
 g_\nu(x)&=\sup\{\nu(a):a\in A,\ a\le x\},\\
 \mu_\nu(x)&=\inf\{g_\nu(u):u\in A^\sigma,\ x\le u\},\\
 \bar\nu(z\vee x)
 &=\inf\{\bar\nu(z\vee u):u\in A^\sigma,\ x\le u\}
 \quad(x,z\in S).
 \end{split}
\end{equation}
No modularity of $\nu$ is required.
\end{corollary}
\begin{proof}
Set
\[
 X=G,\quad P=G^+,\quad e=M=\nu(1),\quad
 Y=\{0\},\quad\pi=0,\quad n=\pi\nu=0.
\]
For every $a\in A$,
\[
 \nu(a)\ge0,
 \qquad M-\nu(a)=\nu(a/1)\ge0,
 \qquad \nu(a)\in I_M=[0,M].
\]
If $a_k\in A$ and $a_k\downarrow_L0$, then $\inf_Aa_k=0$ and Lemma~\ref{pd:ambient-normality} gives $\nu(a_k)\downarrow0$. Thus the source is ambiently normal. Proposition~\ref{rel:positive} proves \textnormal{(R1)}--\textnormal{(R3)}, and Lemma~\ref{app:zero-package} proves every projected source and range condition. Theorem~\ref{rel:lifting} therefore gives the positive extension, its continuity, and its uniqueness among all additive $\sigma$-additive extensions.

For a bounded increasing net $(v_\alpha)$ in $G$, $v=\sup_\alpha v_\alpha$ satisfies $v-v_\alpha\downarrow0$. For a bounded decreasing net, $v_\alpha-\inf_\beta v_\beta\downarrow0$. The source inner net is increasing, the outer net is decreasing, and the contextual cover net is decreasing. Applying these two identities to the three limits in Theorem~\ref{rel:lifting} gives precisely \eqref{app:positive-envelopes}.
\end{proof}

For commutative $D$-lattices, Corollary~\ref{app:positive} is the corrected positive Barbieri-type specialization proved here, in connection with~\cite{Barbieri14}. Its asserted domain is $\sigma_L(A)$. The pseudo-$D$-lattice statement requires neither modularity of $\nu$ nor commutativity of the partial addition.

\begin{corollary}[Positive product ranges]\label{app:positive-products}
The conclusion of Corollary~\ref{app:positive} remains valid when $G$ is any product of super Dedekind complete weakly $\sigma$-distributive Abelian lattice-ordered groups, with the order convergence of the product group. In particular, it holds for $G=\mathbb R^J$ for an arbitrary index set $J$.
\end{corollary}
\begin{proof}
Write $G=\prod_{j\in J}G_j$ and $M=\nu(1)$. Then
\[
 0\le\nu(a)\le M,
 \qquad a_k\downarrow_L0\Longrightarrow\nu(a_k)\to_o0
\]
by positivity and source $\sigma$-additivity. Proposition~\ref{rel:order-products} gives \textnormal{(R1)}--\textnormal{(R3)} for the actual order convergence on $G$. With $\pi=0$, Lemma~\ref{app:zero-package} gives the complete projected package. Theorem~\ref{rel:lifting} applies. Since suprema and infima of bounded families in $G$ are computed coordinatewise, the directed limits have exactly the forms in \eqref{app:positive-envelopes}.
\end{proof}

\begin{corollary}[Lattice-normed extension with its source controller]\label{app:lattice-normed}
Let $(F,p,G)$ be a sequentially bo-complete lattice-normed group, with $G$ super Dedekind complete and weakly $\sigma$-distributive. Let $m:A\to F$ be additive, and let $\nu:A\to G^+$ be a finite positive $\sigma$-additive measure satisfying
\begin{equation}\label{app:source-domination}
 p(m(a))\le\nu(a)\qquad(a\in A).
\end{equation}
Then $m$ is bo-$\sigma$-additive and has a unique additive bo-$\sigma$-additive extension $\bar m:S\to F$. If $\bar\nu$ is the positive extension in Corollary~\ref{app:positive}, then
\begin{equation}\label{app:extended-domination}
 p(\bar m(x))\le\bar\nu(x)\qquad(x\in S).
\end{equation}
The Sion processes for $m$ and $\nu$ exist in their original convergences on every element of $L$. Their restrictions to $S$ are $\bar m$ and $\bar\nu$, and both satisfy the contextual formula \eqref{app:uncontrolled-context}. Neither modularity of $m$ nor decomposability of the lattice norm is required.
\end{corollary}
\begin{proof}
Let $(a_n)$ be an ordered orthogonal sequence in $A$ with intrinsic sum $a\in A$. Put $s_N=\sum_{n\le N}a_n$ and $r_N=s_N/a\in A$. Then $s_N+r_N=a$ and
\[
 p\left(m(a)-\sum_{n\le N}m(a_n)\right)
 =p(m(r_N))
 \le\nu(r_N)
 =\nu(a)-\sum_{n\le N}\nu(a_n)\downarrow0.
\]
This proves bo-$\sigma$-additivity of $m$. If $a_n\in A$ and $a_n\downarrow_L0$, Lemma~\ref{pd:ambient-normality} gives
\[
 0\le p(m(a_n))\le\nu(a_n)\downarrow0,
 \qquad (m(a_n),\nu(a_n))\longrightarrow(0,0).
\]

Define the source graph measure
\[
 \widetilde m(a)=(m(a),\nu(a)),\qquad
 \widetilde e=(m(1),\nu(1)).
\]
For each defined sum $a+b$,
\[
 \widetilde m(a+b)
 =(m(a)+m(b),\nu(a)+\nu(b))
 =\widetilde m(a)+\widetilde m(b).
\]
The preceding limits prove its $\sigma$-additivity and ambient normality. For the graph cone $P=\{(v,r):p(v)\le r\}$,
\[
 p(m(a))\le\nu(a)\quad\Longrightarrow\quad\widetilde m(a)\in P,
\]
and
\[
 \widetilde e-\widetilde m(a)
 =\widetilde m(a/1)\in P,
\]
so $\widetilde m(A)\subseteq I_{\widetilde e}$. Proposition~\ref{rel:graph} proves \textnormal{(R1)}--\textnormal{(R3)} at $\widetilde e$, using sequential bo-completeness. For $Y=\{0\}$ and $\pi=0$, Lemma~\ref{app:zero-package} proves the modularity, hereditary exhaustivity, completeness, reflection, summation, and closure assumptions of the projected data. Theorem~\ref{rel:lifting} yields
\[
 \overline{\widetilde m}(x)=(m_1(x),\nu_1(x)).
\]
The coordinate projections are continuous homomorphisms, so for every ordered orthogonal sequence with sum $x\in S$,
\[
 \begin{split}
 m_1(x)&=\operatorname{pr}_F\lim_N\sum_{n\le N}
            \overline{\widetilde m}(x_n)
          =\lim_N\sum_{n\le N}m_1(x_n),\\
 \nu_1(x)&=\operatorname{pr}_G\lim_N\sum_{n\le N}
            \overline{\widetilde m}(x_n)
          =\lim_N\sum_{n\le N}\nu_1(x_n).
 \end{split}
\]
Their restrictions to $A$ are $m$ and $\nu$. Corollary~\ref{app:positive} therefore gives $\nu_1=\bar\nu$. Since $\overline{\widetilde m}(x)\in P$,
\[
 p(m_1(x))\le\nu_1(x)=\bar\nu(x)\qquad(x\in S).
\]
Set $\bar m=m_1$.

If $v:S\to F$ is any additive bo-$\sigma$-additive extension of $m$, the map $V=(v,\bar\nu)$ satisfies
\[
 V|_A=\widetilde m,
 \qquad V(x+y)=V(x)+V(y),
 \qquad V\Bigl(\sum_nx_n\Bigr)=\lim_N\sum_{n\le N}V(x_n).
\]
The unrestricted uniqueness in Theorem~\ref{rel:lifting} gives $V=\overline{\widetilde m}$ and $v=\bar m$. This argument places no positivity or domination assumption on $V$.

Write $g_{\widetilde m},\mu_{\widetilde m}$ for the graph processes. Continuity of the coordinate projections gives, for every $x\in L$,
\[
 \begin{gathered}
 g_m(x)=\operatorname{pr}_F g_{\widetilde m}(x),\qquad
 g_\nu(x)=\operatorname{pr}_G g_{\widetilde m}(x),\\
 \mu_m(x)=\operatorname{pr}_F\mu_{\widetilde m}(x),\qquad
 \mu_\nu(x)=\operatorname{pr}_G\mu_{\widetilde m}(x).
 \end{gathered}
\]
Projection of the graph contextual identity gives its two stated coordinate identities.
\end{proof}

In Corollary~\ref{app:lattice-normed}, $\nu$ is source data. The extended inequality \eqref{app:extended-domination} is obtained only after the graph measure has been extended.

\begin{lemma}[Topological range verification]\label{app:topological-range}
Every Hausdorff Abelian topological group $X$ satisfies $\mathsf F(X)$ and $\mathsf S_{\rm all}(X)$. For every nonempty $N\subseteq X$, it satisfies $\mathsf U(N)$. Stable hulls for its net convergence are topological closures.
\end{lemma}
\begin{proof}
For $B\subseteq X$, topological closure is stable under net limits, so $[B]_X\subseteq\overline B$. If $x\in\overline B$, the pairs
\[
 I_x=\{(U,b):U\text{ a zero-neighborhood},\ b\in B\cap(x+U)\},
\]
ordered by $(U,b)\le(V,c)$ when $V\subseteq U$, form a directed domain. The net $(U,b)\mapsto b$ belongs to $B$ and converges to $x$. Hence $x\in[B]_X$ and $[B]_X=\overline B$.

For $\mathsf F(X)$, suppose that $u:\Lambda\to X$ does not converge to $0$. There is a zero-neighborhood $V$ such that
\[
 \forall\alpha\in\Lambda:\quad T_u(\alpha)\setminus V\ne\varnothing.
\]
For any proposed strategy $\sigma$, the recursive prescriptions
\[
 \alpha_n=\sigma((\alpha_j,\{x_j\})_{j<n}),
 \qquad x_n\in T_u(\alpha_n)\setminus V
\]
define a legal play. Its responses $S_n=\{x_n\}$ satisfy $S_n\cap V=\varnothing$ for every $n$, so $(S_n)_n\not\Rightarrow0$. Thus a strategy whose every response evaluation is null can exist only when $u\to0$.

For $\mathsf S_{\rm all}(X)$, let $V$ be a zero-neighborhood. Continuity of addition gives a symmetric open zero-neighborhood $U$ with $U+U\subseteq V$. If $x\in\overline U$, then $(x+U)\cap U\ne\varnothing$, whence $x\in U-U=U+U$. Therefore
\[
 W_0=\overline U\subseteq V
\]
is a closed zero-neighborhood. Recursively obtain symmetric zero-neighborhoods $W_i$ with $W_i+W_i\subseteq W_{i-1}$ for $i\ge1$. Since $0\in W_i$,
\[
 W_i\subseteq W_{i-1},\qquad
 \sum_{i=k}^mW_i\subseteq W_{k-1}\quad(1\le k\le m).
\]
The second inclusion is descending induction on $k$: its induction step is
\[
 W_k+\sum_{i=k+1}^mW_i\subseteq W_k+W_k\subseteq W_{k-1}.
\]
For the rows $U_{i,\alpha}\Rightarrow0$ in the definition of $\mathsf S_{\rm all}$, there exist indices $\alpha_i$ such that
\[
 \forall\alpha\ge\alpha_i:\quad U_{i,\alpha}\subseteq W_i.
\]
Put $\varphi_0(i)=\alpha_i$. For every $\varphi\ge\varphi_0$ and nonempty finite $F\subseteq\N$, insertion of zeros gives
\[
 \sum_{i\in F}U_{i,\varphi(i)}
 \subseteq\sum_{i=1}^{\max F}W_i\subseteq W_0.
\]
The empty sum is $0\in W_0$. Consequently
\[
 A_\varphi\subseteq W_0,
 \qquad [A_\varphi]_X=\overline{A_\varphi}\subseteq W_0\subseteq V
 \qquad(\varphi\ge\varphi_0).
\]
This is $(X\cap[A_\varphi]_X)_\varphi\Rightarrow0$.

For $\mathsf U(N)$, use the point families in its definition and fix a zero-neighborhood $V$. Let $W=-W$ be a zero-neighborhood with $W+W+W\subseteq V$. The two uniform approximation hypotheses yield $\beta_0$ with
\[
 \forall n\ge1\ \forall j\in J:
 \quad x_n^{\beta_0}(j)-x_n(j)\in W,
 \quad y_n^{\beta_0}(j)-y_n(j)\in W.
\]
The $\beta_0$-row hypothesis yields an integer $n_0$ with
\[
 \forall n\ge n_0\ \forall j\in J:
 \quad x_n^{\beta_0}(j)-y_n^{\beta_0}(j)\in W.
\]
For those indices,
\[
 \begin{split}
 x_n(j)-y_n(j)
 &=(x_n(j)-x_n^{\beta_0}(j))
   +(x_n^{\beta_0}(j)-y_n^{\beta_0}(j))
   +(y_n^{\beta_0}(j)-y_n(j))\\
 &\in W+W+W\subseteq V.
 \end{split}
\]
Thus $(\{x_n(j)-y_n(j):j\in J\})_n\Rightarrow0$.
\end{proof}

\begin{corollary}[Topological modular extension]\label{app:topological}
Let $X$ be a Hausdorff Abelian topological group, and let $m:A\to X$ be an additive modular ambiently normal measure. Suppose that $m$ is exhaustive on ordered orthogonal source sequences and that the closed source range $\overline{m(A)}$ is Cauchy complete. Then every conclusion of Corollary~\ref{app:uncontrolled} holds. Intrinsic $\sigma$-additivity may replace ambient normality in the source assumptions. In particular, the completeness assumption holds when $X$ itself is complete.
\end{corollary}
\begin{proof}
Let $(h_n)$ be an ordered orthogonal sequence in $A$ and put
\[
 V_n=\{m(a):a\in A,\ a\le h_n\}.
\]
If $(V_n)_n\not\Rightarrow0$, there is a zero-neighborhood $U$ with
\[
 \forall N\ \exists n\ge N\ \exists a\in A:
 \quad a\le h_n,\quad m(a)\notin U.
\]
Recursion gives $n_1<n_2<\cdots$ and $a_k\le h_{n_k}$ with $m(a_k)\notin U$. Put $b_i=a_k$ for $i=n_k$ and $b_i=0$ otherwise. Monotonicity of partial addition gives, for every $N$,
\[
 b_1+\cdots+b_N\text{ exists},
 \qquad b_1+\cdots+b_N\le h_1+\cdots+h_N.
\]
Thus $(a_k)_k$ is ordered orthogonal. Exhaustivity gives $m(a_k)\to0$, contradicting $m(a_k)\notin U$ for every $k$. We have proved the quantified hereditary statement
\[
 \forall U\ni0\ \exists N\ \forall n\ge N\ \forall a\in A:
 \quad a\le h_n\Longrightarrow m(a)\in U.
\]

Set $N=[m(A)]_X$ and $D=X$. Lemma~\ref{app:topological-range} gives
\[
 N=\overline{m(A)},\qquad
 N-N\subseteq D=-D,\qquad
 \mathsf F(D),\quad\mathsf S_{\rm all}(D),\quad\mathsf U(N).
\]
Completeness of $N$ is the stated hypothesis. Additivity, modularity, and ambient normality of $m$ are also stated hypotheses, and uniform hereditary exhaustivity was just proved. Corollary~\ref{app:uncontrolled} applies. If $X$ is complete, every Cauchy net in the closed set $N$ has an ambient limit belonging to $N$, so $N$ is complete.
\end{proof}

For complete $X$, Corollary~\ref{app:topological} gives the generated-domain topological modular extension of Weber~\cite{Weber18} and its Avallone--De Simone--Vitolo specializations~\cite{Avallone06,Avallone16}. The domain asserted here is $\sigma_L(A)$, with the original Sion limits.

\begin{corollary}[Mixed ranges and inherited algebraic constraints]\label{app:mixed}
Let $Y$ be a separated Abelian net convergence group, let $G$ be a positive realization from Proposition~\ref{rel:positive} or Proposition~\ref{rel:order-products}, and let $H$ be a limit-stable subgroup of $Y\times G$. Suppose
\[
 m=(n,\nu):A\longrightarrow H
\]
is additive and ambiently normal, with $\nu(a)\ge0$ for every $a\in A$. Assume that $n$ is modular and uniformly hereditarily exhaustive, and that its stable range $N=[n(A)]_Y$ is complete with
\[
 N-N\subseteq D=-D,
 \qquad\mathsf F(D),
 \qquad\mathsf S_{\rm all}(D),
 \qquad\mathsf U(N).
\]
Then $m$ has a unique additive $\sigma$-additive extension $\bar m:S\to H$. Its first coordinate is modular, its second coordinate is positive, and all Sion and contextual conclusions of Theorem~\ref{rel:lifting} hold in $H$.
\end{corollary}
\begin{proof}
Set
\[
 P=(Y\times G^+)\cap H,
 \qquad\pi(y,g)=y,
 \qquad e=(n(1),\nu(1)).
\]
For $a\in A$,
\[
 m(a)\in H,\quad\nu(a)\ge0
 \quad\Longrightarrow\quad m(a)\in P,
 \qquad e-m(a)=m(a/1)\in P.
\]
Thus $m(A)\subseteq I_e$. On $Y$, Proposition~\ref{rel:uncontrolled} gives the data $(Y,Y,\operatorname{id},Y,n(1))$. On $G$, Proposition~\ref{rel:positive} or Proposition~\ref{rel:order-products} gives $(G,\{0\},0,G^+,\nu(1))$. Proposition~\ref{rel:products} gives their product, whose projected range $Y\times\{0\}$ is identified with $Y$. Proposition~\ref{rel:subgroups} restricts these data to $H$. Hence \textnormal{(R1)}--\textnormal{(R3)} hold for $(H,Y,\pi,P,e)$.

The projected source is $\pi m=n$, and continuity of $\pi$ gives
\[
 a_k\downarrow_L0\quad\Longrightarrow\quad
 m(a_k)\to0\quad\Longrightarrow\quad n(a_k)\to0.
\]
Its modularity and uniform hereditary exhaustivity, the completeness of $N=[n(A)]_Y$, and the three range principles on $N,D$ are exactly the stated hypotheses. The source map $m$ is additive and ambiently normal by assumption. Theorem~\ref{rel:lifting} therefore gives $\bar m:S\to H$, its unrestricted uniqueness, and its Sion identities. Its range lies in $P$, so its second coordinate is positive. Its first coordinate is the modular projected extension. Every value remains in $H$.
\end{proof}

\section{The Boolean regulated extension}\label{gen:core-section}\label{sec:regulated}
The source in this section is a Boolean ring in a sigma-complete Boolean algebra, with the data of Definition~\ref{gen:data}. The construction uses block reflection and retains the full regulated domain. It has weaker reflection requirements than Theorem~\ref{mod:main}.

The data of this section are those of Definition~\ref{gen:data}. We assume hereditary
exhaustivity, completeness of $M$, and
\begin{equation}\label{gen:standing}
 \mathsf B_\omega(D),\qquad \mathsf S_{\mathfrak I}(D),
 \qquad \mathsf U(M),\qquad \mathbb N\in\mathfrak I.
\end{equation}
Write $[A]=\clo_\theta A$ and $A_\alpha\Rightarrow0$ for convergence
of the evaluation net to zero. All hulls in this section use $\theta$.
In particular, $M-M\subseteq D$ does not assert stability or additive
closure of $D$.

\begin{definition}[Regulated elements]\label{gen:regulated}
For decreasing null rows $0\in U_{i,\alpha}\subseteq D$, with
domains $I_i\in\mathfrak I$, put
\begin{equation}\label{gen:Q-definition}
 A_\varphi=\bigcup_{F\subseteq\mathbb N\ {\rm finite}}
            \sum_{i\in F}U_{i,\varphi(i)},\qquad
 Q_\varphi=D\cap[A_\varphi],\qquad \varphi\in\prod_i I_i.
\end{equation}
An element $b\in\mathbb B$ is regulated if there are such rows and
$\ell_\varphi\in\mathcal H_\delta$, $u_\varphi\in\mathcal H_\sigma$ with
\begin{equation}\label{gen:certificate}
 \ell_\varphi\le b\le u_\varphi,\qquad
 \Gamma_m(u_\varphi\setminus\ell_\varphi)\subseteq Q_\varphi.
\end{equation}
These data form a certificate. The class of elements admitting such certificates is
$\mathcal R=\mathcal R_{\mathfrak I}(m,D)$.
\end{definition}

These certificates use only source fragment values. No extension
value or Cauchy condition on an approximation net enters their
definition. We will prove that the upper and lower approximation
nets are Cauchy and have a common limit.

We now state our main theorem of this section, and the main ingredient of the extension procedure.

\begin{theorem}[Intrinsic regulated extension]\label{gen:extension}
Under \eqref{gen:standing}, $\mathcal R$ is a Boolean
$\sigma$-subalgebra containing $\sigma(\mathcal H)$, and $m$ has a
countably additive extension $\bar m:\mathcal R\to M$. There are
canonical maps $m_\sigma:\mathcal H_\sigma\to M$ and
$m_\delta:\mathcal H_\delta\to M$ for which
\begin{align}
 m_\sigma(u)&\tto{\theta}\bar m(b)
 &&(u\in\mathcal H_\sigma,\ b\le u),\label{gen:full-upper}\\
 m_\delta(\ell)&\tto{\theta}\bar m(b)
 &&(\ell\in\mathcal H_\delta,\ \ell\le b),\label{gen:full-lower}
\end{align}
where the first domain is ordered by reverse inclusion and the
second by inclusion. The extension is unique among countably
additive extensions satisfying \eqref{gen:full-lower}. Its
restriction to $\sigma(\mathcal H)$ is unique among all countably
additive extensions of $m$.
\end{theorem}

We shall prove it in a series of lemmas, and collect the final result later. We begin with the following.

\begin{lemma}[Source classes and finite cover estimates]
\label{gen:source-classes}\label{gen:cover-hull}
The class $\mathcal H_\sigma$ is closed under countable joins,
finite meets and differences by source elements. The class
$\mathcal H_\delta$ is closed under finite joins and meets.
$0,1$ and $\mathcal H$ belong to both classes. If
$U\in\mathcal H_\sigma$, $L\in\mathcal H_\delta$, then
$U\setminus L\in\mathcal H_\sigma$. Furthermore,
\begin{equation}\label{gen:finite-cover-estimate}
 \Gamma_m\left(\bigvee_{i=1}^r s_i\right)
 \subseteq\left[\sum_{i=1}^r\Gamma_m(s_i)\right]
 \qquad(s_i\in\mathcal H_\sigma).
\end{equation}
If $s_{i,n}$ decrease and $\Gamma_m(s_{i,n})\Rightarrow0$ for
$1\le i\le r$, then $\Gamma_m(\bigvee_{i=1}^r s_{i,n})\Rightarrow0$.
\end{lemma}
\begin{proof}
For $t=\bigvee_j(c\wedge a_j)$, the inequalities
$a_j\le c^c\vee t$ give
$c\wedge\bigvee_j a_j\le t\le c\wedge\bigvee_j a_j$.
Thus, for $U_i=\bigvee_n a_{i,n}$ and $a\in\mathcal H$,
\[
 \bigvee_iU_i=\bigvee_{i,n}a_{i,n},\qquad
 U_1\wedge U_2=\bigvee_{n,k}(a_{1,n}\wedge a_{2,k}),\qquad
 U_i\setminus a=\bigvee_n(a_{i,n}\setminus a).
\]
All right-hand terms belong to $\mathcal H$, and $\mathbb N^2$
is countable. For a source exhaustion $h_n\uparrow1$,
$a^c=\bigvee_n(h_n\setminus a)\in\mathcal H_\sigma$.
If $L_i=U_i^c$ with $U_i\in\mathcal H_\sigma$, then
\[
 L_1\vee L_2=(U_1\wedge U_2)^c\in\mathcal H_\delta,
 \qquad L_1\wedge L_2=(U_1\vee U_2)^c\in\mathcal H_\delta.
\]
Constant source sequences put $0$ and every source element in
$\mathcal H_\sigma$, and the fixed exhaustion puts $1$ there.
The formula for $a^c$ above gives $a=(a^c)^c\in\mathcal H_\delta$.
Also $0=1^c$ and $1=0^c$ belong to $\mathcal H_\delta$.
Finally $U\setminus L=U\wedge L^c\in\mathcal H_\sigma$.

For increasing source exhaustions $a_{i,n}\uparrow s_i$ and
$c\in\mathcal H$ below $\bigvee_i s_i$, put
\[
 c_n=c\wedge\bigvee_{i=1}^r a_{i,n},\qquad
 d_{i,n}=c\wedge\left(a_{i,n}\setminus\bigvee_{j<i}a_{j,n}\right).
\]
Then $c_n\uparrow c$, $d_{i,n}\le s_i$, and
$d_{i,n}\wedge d_{j,n}=0$ for $i\ne j$. Moreover
\[
 c_n=\bigvee_{i=1}^r d_{i,n},\qquad
 m(c_n)=\sum_{i=1}^r m(d_{i,n})\in\sum_{i=1}^r\Gamma_m(s_i),
 \qquad m(c_n)\tto\theta m(c).
\]
The last limit uses source countable additivity on the disjoint
increments of $(c_n)$. Hence
$V_m(\bigvee_i s_i)\subseteq[\sum_i\Gamma_m(s_i)]$.
Idempotence proves \eqref{gen:finite-cover-estimate}.
For the final assertion use
the rows $\Gamma_m(s_{i,n})$ and zero rows thereafter in
$\mathsf S_{\mathfrak I}(D)$. Their relative hulls are
\[
 B_{k_1,\ldots,k_r}
   =D\cap\left[\sum_{i=1}^r\Gamma_m(s_{i,k_i})\right].
\]
They form a null evaluation family. Given $k_1,\ldots,k_r$,
every $n\ge\max_i k_i$ satisfies
$\Gamma_m(\bigvee_i s_{i,n})\subseteq B_{k_1,\ldots,k_r}$.
The corresponding evaluation net is a quasi-subnet of the null
evaluation of $(B_{k_1,\ldots,k_r})$.
\end{proof}

\begin{lemma}[Exhaustion limits]\label{gen:sigma-construction}
For $U\in\mathcal H_\sigma$ and any source exhaustion $a_n\uparrow U$,
$m(a_n)$ converges to a value $m_\sigma(U)\in M$ independent of the
exhaustion, and
\begin{equation}\label{gen:exhaustion-tail}
 \Gamma_m(U\setminus a_n)\Rightarrow0.
\end{equation}
The resulting map extends $m$, is finitely additive on disjoint
elements, is modular, and is continuous on increasing sequences in
$\mathcal H_\sigma$.
\end{lemma}
\begin{proof}
Set $a_0=0$, $h_n=a_n\setminus a_{n-1}$, and let $T_N=T_N(m,h)$
be the null families of Lemma~\ref{gen:ut}. For $p,q\ge N$,
\[
 m(a_p)-m(a_q)\in T_{N+1}\cup(-T_{N+1}).
\]
For thresholds $N_1,N_2$, the independent difference tail at
$(N,N)$, $N=\max(N_1,N_2)$, is contained in
$T_{N_1}\cup(-T_{N_2})$, a tail of the mixed null evaluation.
Thus $(m(a_p)-m(a_q))_{(p,q)}\tto\theta0$, and completeness of $M$
gives $m(a_n)\tto\theta z_a\in M$.

Put $W_N=\{m(c):c\in\mathcal H,\ c\le a_j\setminus a_N
\text{ for some }j\ge N\}$. Then $W_N\subseteq T_{N+1}$,
$W_N\Rightarrow0$, and the $W_N$ decrease. If
$c\in\mathcal H$, $c\le U\setminus a_N$, then
$c\wedge a_j\uparrow c$ and $m(c\wedge a_j)\in W_N$.
Thus $m(c)\in[W_N]$, so
\[
 \Gamma_m(U\setminus a_N)\subseteq D\cap[W_N]\Rightarrow0
\]
by Lemma~\ref{gen:relative}. For another exhaustion $b_k\uparrow U$
with limit $z_b$, source continuity gives
\[
 m(b_k\setminus a_N)
 =m(b_k)-m(b_k\wedge a_N)\longrightarrow z_b-m(a_N).
\]
Each left-hand value lies in $V_m(U\setminus a_N)\subseteq[W_N]$.
The limit belongs to $M-M\subseteq D$, hence to $D\cap[W_N]$.
Consequently $z_b-m(a_N)\in D\cap[W_N]\Rightarrow0$, so
$m(a_N)\tto\theta z_b$ and separation gives $z_a=z_b$.

For source exhaustions $a_n\uparrow U$, $b_n\uparrow V$, passage
to limits in
\[
 m(a_n\vee b_n)+m(a_n\wedge b_n)=m(a_n)+m(b_n)
\]
gives $m_\sigma(U\vee V)+m_\sigma(U\wedge V)
=m_\sigma(U)+m_\sigma(V)$. If $U\wedge V=0$, the second term is zero.
A constant exhaustion gives $m_\sigma|_{\mathcal H}=m$.
Every term in a source exhaustion of $W\in\mathcal H_\sigma$
belongs to $V_m(W)$. Stability therefore gives
$m_\sigma(W)\in\Gamma_m(W)$.
If $a\in\mathcal H$ and $a\le U$, these facts imply
\begin{equation}\label{gen:sigma-difference}
 m_\sigma(U)-m(a)=m_\sigma(U\setminus a)
                  \in\Gamma_m(U\setminus a).
\end{equation}
For $U_n\uparrow U$, write $a_{i,k}\uparrow_k U_i$ in the source
and set $p_n=\bigvee_{i\le n}a_{i,n}$. Then $p_n\uparrow U$,
$p_n\le U_n$, and
\[
 m_\sigma(U_n)-m(p_n)
  \in\Gamma_m(U_n\setminus p_n)
  \subseteq\Gamma_m(U\setminus p_n)\Rightarrow0.
\]
The identity $m_\sigma(U_n)=m(p_n)+(m_\sigma(U_n)-m(p_n))$
and independent addition give $m_\sigma(U_n)\tto\theta m_\sigma(U)$.
\end{proof}

\begin{lemma}[The finite Boolean stage]
\label{gen:delta-construction}\label{gen:source-approximation}
\label{gen:K-construction}
Put $\mathcal K=\operatorname{alg}(\mathcal H_\sigma)$. Every
$E\in\mathcal K$ has sequences $p_n\in\mathcal H$,
$c_n\in\mathcal H_\sigma$ satisfying
\begin{equation}\label{gen:source-approximation-data}
 c_n\downarrow,\qquad E\triangle p_n\le c_n,
 \qquad\Gamma_m(c_n)\Rightarrow0.
\end{equation}
For every pair satisfying these conditions, $m(p_n)$ has the same
limit $m_0(E)\in M$. The map $m_0:\mathcal K\to M$ is finitely
additive, extends $m_\sigma$, and satisfies
\begin{equation}\label{gen:K-domination}
 E\le s\in\mathcal H_\sigma\quad\Longrightarrow\quad
 m_0(E)\in\Gamma_m(s).
\end{equation}
For $L\in\mathcal H_\delta$,
$m_0(L)=m_\delta(L):=m_\sigma(1)-m_\sigma(L^c)\in M$.
\end{lemma}
\begin{proof}
Represent $E=P(U_1,\ldots,U_r)$ by a finite Boolean expression.
For source exhaustions $a_{i,n}\uparrow U_i$ and $h_n\uparrow1$, put
\[
 p_n=h_n\wedge P(a_{1,n},\ldots,a_{r,n}),\qquad
 c_n=h_n^c\vee\bigvee_{i=1}^r(U_i\setminus a_{i,n}).
\]
Restriction of every complement in $P$ to $h_n$ gives a ring
expression. More explicitly, the recursive rules
\[
 \begin{gathered}
 (1)_h=h,\qquad(0)_h=0,\qquad(a)_h=h\wedge a,\qquad
 (Q^c)_h=h\setminus Q_h,\\
 (Q\vee R)_h=Q_h\vee R_h,\qquad (Q\wedge R)_h=Q_h\wedge R_h.
 \end{gathered}
\]
define source expressions with $P_h=h\wedge P$ by induction.
In particular $p_n\in\mathcal H$. The inequalities
\[
 a^c\triangle(a')^c=a\triangle a',\qquad
 (a\vee b)\triangle(a'\vee b')
 \le(a\triangle a')\vee(b\triangle b'),
\]
\[
 (a\wedge b)\triangle(a'\wedge b')
 \le(a\triangle a')\vee(b\triangle b')
\]
give $P(U)\triangle P(a)\le\bigvee_i(U_i\triangle a_i)$ by
induction on the same expression. Restriction to $h_n$ adds only
the error $h_n^c$, proving $E\triangle p_n\le c_n$.
Each family $\Gamma_m(U_i\setminus a_{i,n})$ is null by
\eqref{gen:exhaustion-tail}, and the exhaustion of $1$ gives
$\Gamma_m(h_n^c)\Rightarrow0$. All the shell sequences decrease.
Lemma~\ref{gen:cover-hull}, applied to these $r+1$ rows, gives
$\Gamma_m(c_n)\Rightarrow0$. This proves
\eqref{gen:source-approximation-data}.

For any such approximation and $p,q\ge N$,
\[
 m(p_p)-m(p_q)
 =m(p_p\setminus p_q)-m(p_q\setminus p_p)
 \in\Gamma_m(c_N)-\Gamma_m(c_N).
\]
For independent thresholds $N_1,N_2$, the right-hand side at
$N=\max(N_1,N_2)$ is contained in
$\Gamma_m(c_{N_1})-\Gamma_m(c_{N_2})$, a tail of the independent
null difference evaluation. This proves Cauchyness, and
completeness gives a limit in $M$.
If $(q_n,d_n)$ is a second approximation, set
$H_n=\Gamma_m(c_n\vee d_n)\Rightarrow0$. Since
$p_n\triangle q_n\le c_n\vee d_n$,
\[
 m(p_n)-m(q_n)
 =m(p_n\setminus q_n)-m(q_n\setminus p_n)\in H_n-H_n.
\]
Writing $m(p_n)\to z$ and $m(q_n)\to z'$, independent addition
and diagonal tail comparison give $m(p_n)-m(q_n)\to z-z'$.
The displayed inclusion gives convergence to zero, so $z=z'$.

For disjoint $E,F$, source approximations $(p_n,c_n),(q_n,d_n)$
satisfy
\[
 F\setminus(q_n\setminus p_n)
 \le(F\setminus q_n)\vee(F\wedge p_n)
 \le d_n\vee(p_n\setminus E)\le c_n\vee d_n,
\]
\[
 (q_n\setminus p_n)\setminus F\le q_n\setminus F\le d_n,
 \qquad(E\vee F)\triangle(p_n\vee q_n)\le c_n\vee d_n.
\]
Thus $q_n\setminus p_n$ approximates $F$, and $p_n\vee q_n$
approximates $E\vee F$. Passage to limits in
$m(p_n\vee q_n)=m(p_n)+m(q_n\setminus p_n)$ proves finite
additivity. If $E\in\mathcal H_\sigma$ and $p_n\uparrow E$, then
$c_n=E\setminus p_n$ satisfies \eqref{gen:source-approximation-data},
so $m_0(E)=\lim_n m(p_n)=m_\sigma(E)$.
For $L\in\mathcal H_\delta$, finite additivity on $L\vee L^c=1$
gives $m_0(L)=m_\sigma(1)-m_\sigma(L^c)\in M$.
For later use, finite additivity also gives modularity on $\mathcal K$.
The three disjoint elements $E\setminus F$, $E\wedge F$, and
$F\setminus E$ have union $E\vee F$. Expanding their values gives
\[
 m_0(E\vee F)+m_0(E\wedge F)=m_0(E)+m_0(F).
\]

Finally, if $E\le s$ and $b_n\uparrow s$ is a source exhaustion,
then $q_n=p_n\wedge b_n\le s$ and
$E\triangle q_n\le c_n\vee(s\setminus b_n)$. These are source
approximations, so $m(q_n)\to m_0(E)$. Stability of $\Gamma_m(s)$,
which contains every $m(q_n)$, proves \eqref{gen:K-domination}.
\end{proof}

We next transfer the hereditary tail estimates from the source to
the finite Boolean stage. The following uniform-closure argument
will also be used on the regulated domain. In both applications
the approximating values belong to $M$.

\begin{lemma}[Transfer of hereditary tails]\label{gen:tail-transfer}
Let $v:\mathcal B\to M$ be finitely additive on a Boolean ring,
and let $b_n\in\mathcal B$ be disjoint. Suppose that for a directed
parameter $\beta$ there are disjoint $a_n^\beta\in\mathcal B$ and
decreasing null sets $0\in Q_\beta\subseteq D$ such that, with
$a_F^\beta=\bigvee_{n\in F}a_n^\beta$, every finite nonempty $F$
and $c\in\mathcal B$, $c\le\bigvee_{n\in F}b_n$, satisfy
\[
 v(c)-v(c\wedge a_F^\beta)\in Q_\beta.
\]
Suppose also that, for fixed $\beta$, the evaluations of
$v(c\wedge a_F^\beta)$ over these pairs with $\min F\ge N$
converge to zero as $N\to\infty$. Under $\mathsf U(M)$,
\[
 Z_N=\{v(c):c\in\mathcal B,\ c\le\bigvee_{n\in F}b_n,
          \ F\subseteq\{N,N+1,\ldots\}\text{ finite}\}
 \Rightarrow0.
\]
\end{lemma}
\begin{proof}
Put $I=\{(F,c):\varnothing\ne F\subseteq\mathbb N\text{ finite},
c\in\mathcal B,\ c\le\bigvee_{n\in F}b_n\}$. Define
\[
 x_N(F,c)=\begin{cases}v(c),&\min F\ge N,\\0,&\min F<N,\end{cases}
 \quad
 x_N^\beta(F,c)=\begin{cases}v(c\wedge a_F^\beta),&\min F\ge N,\\
 0,&\min F<N.\end{cases}
\]
Their values lie in $M$, including zero, and their approximating
evaluations are null in $N$ by hypothesis. With the two other
families identically zero, the remaining premise of $\mathsf U(M)$
is verified by the explicit inclusion
\[
 \bigcup_{\beta\ge\beta_0}
 \{x_N^\beta(F,c)-x_N(F,c):N\ge1,\ (F,c)\in I\}
 \subseteq-Q_{\beta_0}.
\]
Here $0\in Q_\beta$ covers the inactive pairs. Since
$-Q_\beta\Rightarrow0$, $\mathsf U(M)$ applied with
$y_N=y_N^\beta=0$ gives
$(\{x_N(F,c):(F,c)\in I\})_N\Rightarrow0$.
Its range at $N$ is $Z_N$: each nonempty allowed support occurs
in $I$, and zero occurs at $(\{N\},0)$.
\end{proof}

\begin{lemma}[Hereditary tails on the finite stage]\label{gen:UT-K}
For disjoint $E_n\in\mathcal K$,
\[
 Z_N^{\mathcal K}=\{m_0(c):c\in\mathcal K,\
 c\le\bigvee_{n\in F}E_n,\
 F\subseteq\{N,N+1,\ldots\}\text{ finite}\}\Rightarrow0.
\]
\end{lemma}
\begin{proof}
For each $n$, Lemma~\ref{gen:source-approximation} supplies
$(p_{n,k},c_{n,k})_{k\ge1}$ with
$E_n\triangle p_{n,k}\le c_{n,k}\downarrow$ and
$\Gamma_m(c_{n,k})\Rightarrow0$. Define, for $\varphi\in\mathbb N^{\mathbb N}$,
\[
 Q_\varphi=D\cap\left[
   \bigcup_{J\subseteq\mathbb N\text{ finite}}
       \sum_{n\in J}\Gamma_m(c_{n,\varphi(n)})\right].
\]
These sets decrease and satisfy $Q_\varphi\Rightarrow0$ by
$\mathsf S_{\mathfrak I}(D)$. The finite-cover estimate gives
\begin{equation}\label{gen:K-selected-shells}
 \Gamma_m\left(\bigvee_{n\in J}c_{n,\varphi(n)}\right)
 \subseteq Q_\varphi\qquad(J\subseteq\mathbb N\text{ finite}).
\end{equation}
Put $d_n^\varphi=p_{n,\varphi(n)}\setminus\bigvee_{i<n}p_{i,\varphi(i)}$.
They are disjoint source elements. Disjointness of the $E_n$ gives
\[
 E_n\setminus d_n^\varphi
 \le (E_n\setminus p_{n,\varphi(n)})
       \vee\bigvee_{i<n}(p_{i,\varphi(i)}\setminus E_i)
 \le\bigvee_{i\le n}c_{i,\varphi(i)}.
\]
Writing $d_F^\varphi=\bigvee_{n\in F}d_n^\varphi$, for
$\varnothing\ne F\subseteq\mathbb N$ finite and
$c\in\mathcal K$, $c\le\bigvee_{n\in F}E_n$, one obtains
\[
 c\setminus d_F^\varphi\le
 \bigvee_{n\in F}(E_n\setminus d_n^\varphi)
 \le\bigvee_{i\le\max F}c_{i,\varphi(i)}.
\]
Consequently \eqref{gen:K-domination} and \eqref{gen:K-selected-shells}
give $m_0(c)-m_0(c\wedge d_F^\varphi)
=m_0(c\setminus d_F^\varphi)\in Q_\varphi$.
For fixed $\varphi$, the source tail sets $T_N^\varphi=T_N(m,d^\varphi)$
are null by Lemma~\ref{gen:ut}. If $\min F\ge N$, then
\[
 m_0(c\wedge d_F^\varphi)\in\Gamma_m(d_F^\varphi)
 \subseteq D\cap[T_N^\varphi].
\]
The families $D\cap[T_N^\varphi]$ decrease and have null evaluation
by Lemma~\ref{gen:relative}. For every integer $N_0$,
\[
 \bigcup_{N\ge N_0}
 \{m_0(c\wedge d_F^\varphi):\varnothing\ne F\subseteq\mathbb N
       \text{ finite},\ \min F\ge N,\
                  c\in\mathcal K,\ c\le\bigvee_{n\in F}E_n\}
 \subseteq D\cap[T_{N_0}^\varphi].
\]
Thus the approximating evaluations satisfy the fixed-parameter premise
of Lemma~\ref{gen:tail-transfer}. Apply that lemma with
$\mathcal B=\mathcal K$, $v=m_0$, $b_n=E_n$,
$\beta=\varphi$, and $a_n^\beta=d_n^\varphi$.
The previously established $Q_\varphi$ inclusion supplies its other
approximation premise, so its conclusion is the stated nullity of
$Z_N^{\mathcal K}$.
\end{proof}

\begin{lemma}[Regulated values and finite additivity]
\label{gen:K-regulated}\label{gen:R-values}\label{gen:R-finite}
One has $\mathcal K\subseteq\mathcal R$. Every certificate of $b$
has a Cauchy upper-value net $m_\sigma(u_\varphi)$, whose limit
$\bar m(b)\in M$ is independent of the certificate. The class
$\mathcal R$ is a Boolean algebra, $\bar m$ is finitely additive
and extends $m_0$. For every fixed certificate index,
\begin{align}
 m_\sigma(u_\varphi)-\bar m(b)&\in
       \Gamma_m(u_\varphi\setminus\ell_\varphi),\label{gen:upper-error}\\
 \bar m(b)-m_\delta(\ell_\varphi)&\in
       \Gamma_m(u_\varphi\setminus\ell_\varphi).\label{gen:lower-error}
\end{align}
\end{lemma}
\begin{proof}
For \eqref{gen:source-approximation-data}, set $u_n=p_n\vee c_n$,
$\ell_n=p_n\setminus c_n$. Then $\ell_n\le E\le u_n$,
$u_n\setminus\ell_n=c_n$, and
$\ell_n^c=p_n^c\vee c_n\in\mathcal H_\sigma$. The single row
$\Gamma_m(c_n)$, with zero rows thereafter, certifies $E$.

For a certificate of $b$, put $s_\varphi=u_\varphi\setminus\ell_\varphi$.
For any two indices, finite additivity on $\mathcal K$ gives
\[
 m_\sigma(u_\varphi)-m_\sigma(u_\psi)
 =m_0(u_\varphi\setminus u_\psi)-m_0(u_\psi\setminus u_\varphi)
 \in Q_\varphi-Q_\psi,
\]
since $u_\varphi\setminus u_\psi\le s_\varphi$ and the reverse
difference is below $s_\psi$. For thresholds $\varphi_0,\psi_0$,
every difference with $\varphi\ge\varphi_0$, $\psi\ge\psi_0$ lies
in $Q_{\varphi_0}-Q_{\psi_0}$. Thus the independent difference net
is null. Completeness gives an upper limit $z\in M$.
For a second certificate $(\ell'_\chi,u'_\chi,Q'_\chi)$ and its
upper limit $z'$, the inequalities
\[
 u_\varphi\setminus u'_\chi\le s_\varphi,\qquad
 u'_\chi\setminus u_\varphi\le u'_\chi\setminus\ell'_\chi
\]
give
\[
 m_\sigma(u_\varphi)-m_\sigma(u'_\chi)\in Q_\varphi-Q'_\chi.
\]
Its tail at $(\varphi_0,\chi_0)$ lies in
$Q_{\varphi_0}-Q'_{\chi_0}$, so the independent net converges to zero.
It also converges to $z-z'$ by independent addition, and hence $z=z'$.
For $b\in\mathcal K$,
$m_\sigma(u_\varphi)-m_0(b)=m_0(u_\varphi\setminus b)\in Q_\varphi$,
so this limit extends $m_0$.

Fix $\varphi$ and put $t_\psi=u_\psi\wedge u_\varphi$.
The estimate
$m_\sigma(u_\psi)-m_\sigma(t_\psi)\in Q_\psi$ gives
$m_\sigma(t_\psi)\to\bar m(b)$. The inequalities $\ell_\varphi\le b\le t_\psi\le u_\varphi$ give
\[
 \begin{aligned}
 m_0(u_\varphi\setminus t_\psi)
 &=m_\sigma(u_\varphi)-m_\sigma(t_\psi)
   \tto\theta m_\sigma(u_\varphi)-\bar m(b),\\
 m_0(t_\psi\setminus\ell_\varphi)
 &=m_\sigma(t_\psi)-m_\delta(\ell_\varphi)
   \tto\theta\bar m(b)-m_\delta(\ell_\varphi).
 \end{aligned}
\]
Both left-hand values belong to $\Gamma_m(s_\varphi)$ by
\eqref{gen:K-domination}. Its stability gives
\eqref{gen:upper-error}--\eqref{gen:lower-error}.

The pair $((u_\varphi)^c,(\ell_\varphi)^c)$ sandwiches $b^c$
and has shell $(\ell_\varphi)^c\setminus(u_\varphi)^c=s_\varphi$,
so $b^c\in\mathcal R$.
For certificates of $b,c$, set
$L=\ell_\varphi^b\vee\ell_\psi^c$ and
$U=u_\varphi^b\vee u_\psi^c$. They sandwich $b\vee c$ and
$U\setminus L\le s_\varphi^b\vee s_\psi^c$.
If the two certificates have rows $U^b_{i,\alpha}$ and
$U^c_{i,\beta}$, index a new family by
$W_{2i-1,\alpha}=U^b_{i,\alpha}$ and
$W_{2i,\beta}=U^c_{i,\beta}$. The selector for this family is
the pair $(\varphi,\psi)$ of the original selectors.
Lemma~\ref{gen:flatten} and the finite-cover
estimate give
\[
 \Gamma_m(U\setminus L)
 \subseteq D\cap[\Gamma_m(s_\varphi^b)+\Gamma_m(s_\psi^c)]
 \subseteq D\cap[Q_\varphi^b+Q_\psi^c]
 \subseteq Q^*_{\varphi,\psi},
\]
where $Q^*$ is the relative closed sum for the interleaved rows.
The constant certificates for $0$ and $1$ are already supplied by
$\mathcal K\subseteq\mathcal R$. The complement and join constructions
therefore make $\mathcal R$ a Boolean algebra, with
$b\wedge c=(b^c\vee c^c)^c\in\mathcal R$.
If $b\wedge c=0$, then
$u_\varphi^b\wedge u_\psi^c\le s_\varphi^b\vee s_\psi^c$, and
$m_\sigma(u_\varphi^b\wedge u_\psi^c)\in Q^*_{\varphi,\psi}\Rightarrow0$.
The union certificate has upper values
$m_\sigma(u_\varphi^b\vee u_\psi^c)\to\bar m(b\vee c)$.
Thus the identity
\[
 m_\sigma(u_\varphi^b\vee u_\psi^c)
 =m_\sigma(u_\varphi^b)+m_\sigma(u_\psi^c)
       -m_\sigma(u_\varphi^b\wedge u_\psi^c)
\]
gives $\bar m(b\vee c)=\bar m(b)+\bar m(c)$.
\end{proof}

\begin{lemma}[Domination and hereditary tails on the regulated stage]
\label{gen:K-hull}\label{gen:R-domination}\label{gen:UT-R}
For $E\in\mathcal K$ put
$\Gamma_0(E)=[\{m_0(F):F\in\mathcal K,\ F\le E\}]$.
If $c\in\mathcal R$ and $c\le E$, then $\bar m(c)\in\Gamma_0(E)$.
In particular,
\begin{equation}\label{gen:R-source-domination}
 c\le s\in\mathcal H_\sigma\quad\Longrightarrow\quad
 \bar m(c)\in\Gamma_m(s).
\end{equation}
For disjoint $b_n\in\mathcal R$,
\[
 Z_N^{\mathcal R}=\{\bar m(c):c\in\mathcal R,\
 c\le\bigvee_{n\in F}b_n,\
 F\subseteq\{N,N+1,\ldots\}\text{ finite}\}\Rightarrow0.
\]
\end{lemma}
\begin{proof}
For a certificate of $c\le E$, put
$s_\varphi=u_\varphi\setminus\ell_\varphi$. Then
$u_\varphi\setminus E\le s_\varphi$ and
\[
 m_\sigma(u_\varphi)-m_0(u_\varphi\wedge E)
 =m_0(u_\varphi\setminus E)\in\Gamma_m(s_\varphi)\subseteq Q_\varphi.
\]
Hence $m_0(u_\varphi\wedge E)\to\bar m(c)$.
Every $u_\varphi\wedge E\in\mathcal K$ is below $E$, so
stability of $\Gamma_0(E)$ gives $\bar m(c)\in\Gamma_0(E)$. For $E=s\in\mathcal H_\sigma$,
\eqref{gen:K-domination} gives $\Gamma_0(s)\subseteq\Gamma_m(s)$.

For each $b_n$ fix a certificate with rows $U_{n,i,\alpha}$.
Write $\Theta\in\prod_{n,i}I_{n,i}$ for all its indices,
$u_n^\Theta,\ell_n^\Theta$ for the selected sandwiches, and
$s_n^\Theta=u_n^\Theta\setminus\ell_n^\Theta$.
Flatten these rows and denote their relative closed sums by
$Q_\Theta$. Then $Q_\Theta\Rightarrow0$ and
\begin{equation}\label{gen:R-selected-shells}
 \Gamma_m\left(\bigvee_{n\in J}s_n^\Theta\right)
 \subseteq Q_\Theta\qquad(J\subseteq\mathbb N\text{ finite}).
\end{equation}
Define disjoint $d_n^\Theta=u_n^\Theta\setminus\bigvee_{i<n}u_i^\Theta
\in\mathcal K$. For $n\in F$, disjointness gives
\[
 b_n\setminus d_n^\Theta
 \le\bigvee_{i<n}(b_n\wedge u_i^\Theta)
 \le\bigvee_{i<n}(u_i^\Theta\setminus b_i)
 \le\bigvee_{i<n}s_i^\Theta.
\]
Hence $c\setminus d_F^\Theta\le\bigvee_{i\le\max F}s_i^\Theta$
whenever $c\le\bigvee_{n\in F}b_n$. By
\eqref{gen:R-source-domination} and \eqref{gen:R-selected-shells},
$\bar m(c)-\bar m(c\wedge d_F^\Theta)\in Q_\Theta$.
For fixed $\Theta$, define
$T_N^\Theta=\{m_0(E):E\in\mathcal K,\ E\le d_F^\Theta,
F\subseteq\{N,N+1,\ldots\}\text{ finite}\}$.
Lemma~\ref{gen:UT-K} gives $T_N^\Theta\Rightarrow0$.
For $\min F\ge N$,
\[
 \bar m(c\wedge d_F^\Theta)
 \in\Gamma_0(d_F^\Theta)\subseteq D\cap[T_N^\Theta].
\]
For every integer $N_0$, the approximating values at all indices
$N\ge N_0$ lie in $D\cap[T_{N_0}^\Theta]$.
Lemma~\ref{gen:relative} makes these bounding families null, proving
the fixed-$\Theta$ approximation premise of Lemma~\ref{gen:tail-transfer}.
Apply that lemma with $\mathcal B=\mathcal R$, $v=\bar m$,
$a_n^\beta=d_n^\Theta$, and $\beta=\Theta$.
The error $\bar m(c)-\bar m(c\wedge d_F^\Theta)$ has already
been placed in $Q_\Theta$ for every admissible pair $(F,c)$.
Its conclusion is the asserted nullity of $Z_N^{\mathcal R}$.
Only finite operations on $\mathcal R$ have been used.
\end{proof}

\begin{proposition}[Countable closure and additivity]
\label{gen:R-sigma-closure}\label{gen:R-countable-additivity}
The algebra $\mathcal R$ is closed under countable joins and
$\bar m$ is countably additive.
\end{proposition}
\begin{proof}
Let $b_n\in\mathcal R$ be disjoint and $b=\bigvee_n b_n\in\mathbb B$.
For each $n$, fix certificate rows $U_{n,i,\alpha}$.
For $\Theta\in\prod_{n,i}I_{n,i}$ denote their selected approximants
by $\ell_n^\Theta,u_n^\Theta$, put
$s_n^\Theta=u_n^\Theta\setminus\ell_n^\Theta$, and let $Q_\Theta$
be the flattened relative hull. Lemma~\ref{gen:flatten} gives
$Q_\Theta\Rightarrow0$ and \eqref{gen:R-selected-shells}. Set
\[
 b^m=\bigvee_{n\le m}b_n,\quad
 U_m^\Theta=\bigvee_{n\le m}u_n^\Theta,\quad
 U^\Theta=\bigvee_nu_n^\Theta,\quad
 L_N^\Theta=\bigvee_{n\le N}\ell_n^\Theta.
\]
Then $L_N^\Theta\in\mathcal H_\delta$, $U^\Theta\in\mathcal H_\sigma$
and $L_N^\Theta\le b\le U^\Theta$. Put
\[
 Z_N=\{\bar m(c):c\in\mathcal R,\ c\le\bigvee_{n\in F}b_n,
       \ F\subseteq\{N+1,N+2,\ldots\}\text{ finite}\}.
\]
Then $0\in Z_{N+1}\subseteq Z_N\subseteq M\subseteq D$, and
Lemma~\ref{gen:UT-R} gives $Z_N\Rightarrow0$.

Fix $N,\Theta$ and a source element $a\le U^\Theta\setminus L_N^\Theta$.
Define
\[
 A_m=a\wedge U_m^\Theta,\qquad
 c_m=A_m\wedge(b^m\setminus b^N),\qquad e_m=A_m\setminus c_m.
\]
All these elements already belong to $\mathcal R$, since
$A_m\in\mathcal H_\sigma$ and only finite operations occur.
Moreover $\bar m(c_m)\in Z_N$. The Boolean bounds
\[
 e_m\le(A_m\setminus b^m)\vee(A_m\wedge b^N),\quad
 U_m^\Theta\setminus b^m\le\bigvee_{i\le m}s_i^\Theta,\quad
 A_m\wedge b^N\le b^N\setminus L_N^\Theta
                    \le\bigvee_{i\le N}s_i^\Theta
\]
give $e_m\le\bigvee_{i\le\max(m,N)}s_i^\Theta$. Hence
$\bar m(e_m)\in Q_\Theta$ and
$m_\sigma(A_m)\in Z_N+Q_\Theta$.
Since $A_m\uparrow a$, increasing continuity of $m_\sigma$ gives
$m_\sigma(A_m)\to m(a)$. Taking stable hulls over all such $a$ yields
\begin{equation}\label{gen:countable-shell-hull}
 \Gamma_m(U^\Theta\setminus L_N^\Theta)
 \subseteq[Z_N+Q_\Theta].
\end{equation}
Write $Q_\Theta=D\cap[A_\Theta]$ for its flattened raw sums.
Append the single sequential row $Z_N$ to those rows. Their new
raw sums are $Z_N+A_\Theta$ and their relative hulls
$Q'_{N,\Theta}=D\cap[Z_N+A_\Theta]$ are null by
$\mathsf S_{\mathfrak I}(D)$. The stable-hull sum identity gives
\[
 [Z_N+Q_\Theta]\subseteq[[Z_N]+[A_\Theta]]
                       \subseteq[Z_N+A_\Theta].
\]
The left-hand side of \eqref{gen:countable-shell-hull} belongs to
$M\subseteq D$, hence lies in $Q'_{N,\Theta}$. The sandwiches
$(L_N^\Theta,U^\Theta)$ therefore certify $b$. For arbitrary $e_n\in\mathcal R$, the elements
$b_n=e_n\setminus\bigvee_{j<n}e_j\in\mathcal R$ are disjoint and
$\bigvee_n b_n=\bigvee_n e_n$. Thus $\bigvee_n e_n\in\mathcal R$.

For $p,q\ge N$,
$\bar m(b^p)-\bar m(b^q)\in Z_N\cup(-Z_N)$.
For any independent thresholds $N_1,N_2$, the difference tail at
$(N,N)$, $N=\max(N_1,N_2)$, lies in
$Z_{N_1}\cup(-Z_{N_2})$. Therefore the independent difference net
is null, and completeness gives $\bar m(b^m)\to z\in M$. For fixed $\Theta$,
\[
 m_\sigma(U_m^\Theta)-\bar m(b^m)
 =\bar m(U_m^\Theta\setminus b^m)\in Q_\Theta\subseteq[A_\Theta].
\]
Increasing continuity gives $m_\sigma(U_m^\Theta)\to m_\sigma(U^\Theta)$.
Stability of $[A_\Theta]$ therefore puts $m_\sigma(U^\Theta)-z$
in that hull. The difference also belongs to $M-M\subseteq D$, so
\begin{equation}\label{gen:raw-before-D}
 m_\sigma(U^\Theta)-z\in Q_\Theta.
\end{equation}
Here the limit lies in $[A_\Theta]\cap D$ without any stability
claim for $D$. Since $Q_\Theta\Rightarrow0$,
\eqref{gen:raw-before-D} gives $m_\sigma(U^\Theta)\to z$.
The upper net of the constructed certificate, indexed by
$(N,\Theta)$, has at $(N_0,\Theta_0)$ the value tail
\[
 \{m_\sigma(U^\Theta):N\ge N_0,\ \Theta\ge\Theta_0\}
 =\{m_\sigma(U^\Theta):\Theta\ge\Theta_0\}.
\]
It therefore converges both to $z$ and, by its certificate, to
$\bar m(b)$. Separation gives $z=\bar m(b)$, proving
$\bar m(b)=\lim_m\sum_{n\le m}\bar m(b_n)$.
\end{proof}

\begin{proposition}[Regularity and uniqueness]
\label{gen:regularity}\label{gen:uniqueness}
The full approximation nets and the uniqueness assertions of
Theorem~\ref{gen:extension} hold.
\end{proposition}
\begin{proof}
Put $\mathcal U_b=\{u\in\mathcal H_\sigma:b\le u\}$ and
$\mathcal L_b=\{\ell\in\mathcal H_\delta:\ell\le b\}$.
Both are nonempty by a certificate. If $u,v\in\mathcal U_b$ and
$\ell,h\in\mathcal L_b$, then
$u\wedge v\in\mathcal U_b$ and $\ell\vee h\in\mathcal L_b$.
These are upper bounds in the specified directed orders.
For a certificate and an upper approximant
$b\le u\le u_\varphi$, finite additivity and domination give
\[
 m_\sigma(u)-\bar m(b)=\bar m(u\setminus b)
 \in\Gamma_m(u_\varphi\setminus\ell_\varphi)\subseteq Q_\varphi.
\]
Thus for every $\varphi$ the upper difference net has a value tail
contained in $Q_\varphi$, namely its tail at $u_\varphi$.
It is a quasi-subnet of the null evaluation of $(Q_\varphi)$,
which proves \eqref{gen:full-upper}. For $\ell_\varphi\le\ell\le b$,
\[
 \bar m(b)-m_\delta(\ell)=\bar m(b\setminus\ell)
 \in\Gamma_m(u_\varphi\setminus\ell_\varphi)\subseteq Q_\varphi.
\]
The lower difference net has its tail at $\ell_\varphi$ contained
in $Q_\varphi$, proving \eqref{gen:full-lower}. These two tail
inclusions use each certificate index separately.

For any countably additive extension $n$ and source exhaustion
$a_k\uparrow U$, put $a_0=0$ and $d_k=a_k\setminus a_{k-1}$.
Then $\bigvee_kd_k=U$, $\sum_{j\le k}n(d_j)=n(a_k)=m(a_k)$, and
\[
 n(U)=\lim_k n(a_k)=\lim_km(a_k)=m_\sigma(U).
\]
In particular $n(1)=m_\sigma(1)$, and for $L\in\mathcal H_\delta$,
$n(L)=n(1)-n(L^c)=m_\delta(L)$. If $n$ has full lower regularity,
then $n(b)=\lim_{\ell\le b}m_\delta(\ell)=\bar m(b)$.

For two extensions $n_1,n_2$ on $\sigma(\mathcal H)$, their
equality class $\mathcal L$ contains $1$ and $\mathcal H$.
If $b\in\mathcal L$, then $n_1(b^c)=n_1(1)-n_1(b)
=n_2(1)-n_2(b)=n_2(b^c)$. For disjoint $b_j\in\mathcal L$,
\[
 n_1\left(\bigvee_jb_j\right)
 =\lim_k\sum_{j\le k}n_1(b_j)
 =\lim_k\sum_{j\le k}n_2(b_j)
 =n_2\left(\bigvee_jb_j\right).
\]
Thus $\mathcal L$ is a Dynkin system containing the
intersection-closed family $\mathcal P=\mathcal H\cup\{1\}$.

Let $\mathcal D$ be the smallest Dynkin system containing
$\mathcal P$. In every Dynkin system, $a\le b$ implies closure
under $b\setminus a=(b^c\vee a)^c$, since $b^c\wedge a=0$.
For fixed $p\in\mathcal P$, put
$\mathcal D_p=\{b\in\mathcal D:p\wedge b\in\mathcal D\}$.
It contains $1$ because $p\wedge1=p\in\mathcal D$, and it
contains $\mathcal P$ because that family is intersection-closed.
For $b\in\mathcal D_p$, both $p$ and $p\wedge b$ belong to
$\mathcal D$. Nested-difference closure gives
$p\setminus(p\wedge b)\in\mathcal D$. Thus the first identity below
places $b^c$ in $\mathcal D_p$.
For a disjoint sequence $(b_j)\subseteq\mathcal D_p$, the elements
$p\wedge b_j$ are disjoint members of $\mathcal D$.
Their join belongs to $\mathcal D$, and the second identity places
$\bigvee_jb_j$ in $\mathcal D_p$:
\[
 p\wedge b^c=p\setminus(p\wedge b),\qquad
 p\wedge\bigvee_j b_j=\bigvee_j(p\wedge b_j)
\]
Consequently $\mathcal D_p$ is a Dynkin system containing
$\mathcal P$, and minimality gives $\mathcal D_p=\mathcal D$.
For fixed $b\in\mathcal D$, let
$\mathcal E_b=\{a\in\mathcal D:a\wedge b\in\mathcal D\}$.
Then $1\wedge b=b\in\mathcal D$ and $\mathcal P\subseteq\mathcal E_b$
by $\mathcal D_p=\mathcal D$. If $a\in\mathcal E_b$, then
$b\wedge a\in\mathcal D$ and $b\wedge a\le b$, so
$b\wedge a^c=b\setminus(b\wedge a)\in\mathcal D$.
If $(a_j)\subseteq\mathcal E_b$ is disjoint, then
$(b\wedge a_j)$ is disjoint in $\mathcal D$ and
\[
 b\wedge\bigvee_ja_j=\bigvee_j(b\wedge a_j)\in\mathcal D.
\]
Thus $\mathcal E_b$ is a Dynkin system containing $\mathcal P$,
and minimality gives $\mathcal E_b=\mathcal D$.
Every intersection of two elements of $\mathcal D$ therefore belongs
to $\mathcal D$. For $a,b\in\mathcal D$, the formula
$a\vee b=(a^c\wedge b^c)^c$ gives finite-union closure.
For $a_n\in\mathcal D$,
$d_n=a_n\setminus\bigvee_{j<n}a_j\in\mathcal D$ are disjoint and
$\bigvee_na_n=\bigvee_nd_n\in\mathcal D$. Therefore
$\sigma(\mathcal H)=\sigma(\mathcal P)\subseteq\mathcal D
\subseteq\mathcal L$.
\end{proof}

\begin{proof}[Completion of Theorem~\ref{gen:extension}]
Lemmas~\ref{gen:sigma-construction}, \ref{gen:K-construction} and
\ref{gen:R-values} construct
$m\subseteq m_\sigma\subseteq m_0\subseteq\bar m$ with values in $M$.
Proposition~\ref{gen:R-countable-additivity} gives
$\sum_{n\le N}\bar m(b_n)\to\bar m(\bigvee_n b_n)$ for disjoint
$b_n\in\mathcal R$, and
$\mathcal H\subseteq\mathcal K\subseteq\mathcal R
\Longrightarrow\sigma(\mathcal H)\subseteq\mathcal R$.
Proposition~\ref{gen:regularity} supplies
\eqref{gen:full-upper}--\eqref{gen:full-lower} and both uniqueness
assertions.
\end{proof}

\begin{proposition}[Transport and the intrinsic difference set]
\label{gen:transport}\label{gen:minimal-D}
Let $S$ be a subgroup containing $M$, with inherited convergence,
and $T:S\to Y$ a continuous homomorphism into a separated convergence
group. Then $T\bar m$ is countably additive and has full upper and
lower approximation nets $Tm_\sigma,Tm_\delta$.
If $S=X$, $Tm$ satisfies Theorem~\ref{gen:extension} with difference
set $D_Y$, and $T(D)\subseteq D_Y$, then
\begin{equation}\label{gen:transport-identification}
 \mathcal R_{\mathfrak I}(m,D)\subseteq\mathcal R_{\mathfrak I}(Tm,D_Y),
 \qquad\overline{Tm}(b)=T\bar m(b).
\end{equation}
If all set-sized directed row domains are admitted, then $D_0=M-M$
satisfies the same structural assumptions and
\begin{equation}\label{gen:minimal-D-equality}
 \mathcal R_{\mathrm{all}}(m,D)=\mathcal R_{\mathrm{all}}(m,D_0).
\end{equation}
\end{proposition}
\begin{proof}
For disjoint $b_n\in\mathcal R$, continuity and additivity give
\[
 \sum_{n\le N}T\bar m(b_n)
 =T\!\left(\sum_{n\le N}\bar m(b_n)\right)
 \longrightarrow T\bar m\!\left(\bigvee_n b_n\right).
\]
The same continuity applied to \eqref{gen:full-upper} and
\eqref{gen:full-lower} gives the asserted approximation limits.
If $S=X$ and $C\subseteq Y$ is stable, then every convergent net
in $T^{-1}(C)$ has its limit in $T^{-1}(C)$.
Taking $C=[T(A)]_Y$ yields $[A]\subseteq T^{-1}(C)$, and hence
$T([A])\subseteq[T(A)]_Y$. A source certificate with raw sums
$A_\varphi$ has image rows $T(U_{i,\alpha})$ and raw sums
$T(A_\varphi)$. Its shell satisfies
\[
 V_{Tm}(u_\varphi\setminus\ell_\varphi)
 =T(V_m(u_\varphi\setminus\ell_\varphi))
 \subseteq[T(A_\varphi)]_Y.
\]
Taking the target stable hull and using its membership in
$M_Y\subseteq D_Y$ gives a target certificate. For $a_n\uparrow U$ in the source,
$(Tm)_\sigma(U)=\lim_nT(m(a_n))=T(m_\sigma(U))$.
The full upper approximants therefore converge both to
$\overline{Tm}(b)$ and to $T\bar m(b)$, proving
\eqref{gen:transport-identification}.

Since $0\in M$, one has $M\subseteq D_0=M-M=-D_0\subseteq D$.
Proposition~\ref{gen:inheritance} supplies
$\mathsf B_\omega(D_0)$ and $\mathsf S_{\rm all}(D_0)$.
The assumptions on $M$ remain unchanged. For any raw sums $A_\varphi$,
$D_0\cap[A_\varphi]\subseteq D\cap[A_\varphi]$, so every
$D_0$-certificate is a $D$-certificate. Conversely, for a
$D$-certificate put $W_\varphi=M\cap[A_\varphi]$.
Then
$0\in W_\varphi\subseteq Q_\varphi$, $\varphi\le\psi\Rightarrow
W_\psi\subseteq W_\varphi$, and $W_\varphi\Rightarrow0$.
Each $W_\varphi$ is stable as an intersection of stable sets.
The certificate has
$\Gamma_m(u_\varphi\setminus\ell_\varphi)
\subseteq M\cap[A_\varphi]=W_\varphi$. Their domain
$\prod_i I_i$ is admissible when all directed row domains are
admitted. Use $W_\varphi$ as one row and zero rows thereafter.
Since $W_\varphi\subseteq M\subseteq D_0$, its relative closed sum
is $D_0\cap[W_\varphi]=W_\varphi$, giving a $D_0$-certificate.
\end{proof}

\begin{remark}\label{gen:sequential-D-warning}
The last reduction does not assert the same equality for exclusively
sequential rows: their countable selector product need not have a
countable cofinal subset. The extension theorem itself has used
only the original row domains and one additional sequential row.
\end{remark}

We now construct a complete convergence which is not induced by a
topology and satisfies the range principles of
Theorem~\ref{gen:extension}. The same example has block reflection
but fails matrix reflection, separating the two conditions even
among complete convergence groups.

Let $G=(\mathbb Z/2\mathbb Z)^{\mathbb N^2}$, with coordinatewise
addition, and put
\[
 \mathcal J=\{J\subseteq\mathbb N^2:
       J\cap(\{r\}\times\mathbb N)\text{ is finite for every }r\},
 \qquad \operatorname{supp}x=\{(r,k):x(r,k)\ne0\}.
\]
Define $x_\alpha\tto\rho x$ when both conditions hold:
\begin{equation}\label{inv:rowfinite-model}
 \begin{aligned}
 &\forall(r,k)\ \exists\alpha(r,k)\ \forall\alpha\geq\alpha(r,k):
       x_\alpha(r,k)=x(r,k),\\
 &\exists J\in\mathcal J\ \exists\alpha_0\ \forall\alpha\geq\alpha_0:
       \operatorname{supp}(x_\alpha-x)\subseteq J.
 \end{aligned}
\end{equation}
Thus errors eventually have their supports in one set which meets
each physical row in finitely many coordinates.

\begin{proposition}[A complete non-topological range]
\label{inv:rowfinite-basic}\label{inv:rowfinite-separation}
\label{inv:topological-reflection}
The convergence $\rho$ is a separated, complete abelian group
convergence. With $M=D=G$, it satisfies
$\mathsf B_\omega(D)$, $\mathsf S_{\mathrm{all}}(D)$ and
$\mathsf U(M)$, but fails $\mathsf R_\omega(D)$ and
$\mathsf F(D)$. No topology induces $\rho$.
\end{proposition}
\begin{proof}
For $J,K\in\mathcal J$ and every $r$,
$(J\cup K)\cap(\{r\}\times\mathbb N)$ is finite.
Constants have empty error support. For a quasi-subnet $v\qsub u$,
each value-tail inclusion transfers both eventual support
containment and eventual equality at a fixed coordinate.
If $u,v\to x$ have eventual support bounds $J,K$, their mixed net
has bound $J\cup K$. At coordinate $q$, a product threshold above
both coordinate thresholds makes every mixed value equal to $x(q)$. For independent sums,
\[
 \operatorname{supp}((x_\alpha+y_\beta)-(x+y))
 \subseteq\operatorname{supp}(x_\alpha-x)
          \cup\operatorname{supp}(y_\beta-y),
\]
and both coordinates stabilize on a product tail. Inversion is the identity. Thus constants, quasi-subnets, finite
mixing, independent addition and inversion obey the required axioms.
If a net converges to $x$ and $y$, every coordinate is eventually
equal to both $x(q)$ and $y(q)$, so $x=y$.
For every $J\subseteq\mathbb N^2$, the subgroup
$G_J=\{x:\operatorname{supp}x\subseteq J\}$ is limit-stable,
since every coordinate outside $J$ remains zero at a limit.
In particular, every coordinate kernel is stable.

If $(x_\alpha)$ is Cauchy, its independent difference net is
$\rho$-null. There are $J\in\mathcal J$ and a common upper threshold $\alpha_0$
for the two difference indices such that
$\operatorname{supp}(x_\alpha-x_\beta)\subseteq J$ whenever
$\alpha,\beta\geq\alpha_0$. At each coordinate, the same Cauchy
condition gives eventual equality of all pairs of values. Denote
the resulting eventual value by $x(r,k)$. Outside $J$, every
$x_\alpha$, $\alpha\geq\alpha_0$, agrees with $x_{\alpha_0}$,
and its eventual value is the same. Thus
$\operatorname{supp}(x_\alpha-x)\subseteq J$ on this tail, while
every coordinate stabilizes to $x$. This proves completeness.

For block reflection let $(A_n)$ satisfy its premises. The partition
into singleton blocks gives $(A_n)_n\Rightarrow0$. Therefore, on
some tail $n\geq N_0$, all members of $A_n$ have supports in a
fixed $J\in\mathcal J$. For finite
$F\subseteq\{N_0,N_0+1,\ldots\}$ and $x_n\in A_n$,
$\operatorname{supp}(\sum_{n\in F}x_n)\subseteq J$. For each coordinate $q$, there is $N(q)$
such that every member of $A_n$ vanishes at $q$ for $n\geq N(q)$.
All finite sums from that tail also vanish there. These are the
two conditions for the evaluation of the finite tail sums to be
$\rho$-null, proving $\mathsf B_\omega(D)$.

For closed summation, consider decreasing null rows
$0\in U_{i,\alpha}\subseteq G$ on arbitrary domains $I_i$.
For each $i$, there are $a_i^0\in I_i$, $J_i\in\mathcal J$ with
$U_{i,\alpha}\subseteq G_{J_i}$ for $\alpha\geq a_i^0$.
The set $J_i\cap(\{1,\ldots,i\}\times\mathbb N)$ is finite.
Coordinate nullity and finite directedness supply $a_i^1\geq a_i^0$
such that every member of $U_{i,a_i^1}$ vanishes there. Set
\[
 J_i'=J_i\setminus(\{1,\ldots,i\}\times\mathbb N),\qquad
 J=\bigcup_iJ_i'.
\]
For fixed $r$, only the indices $i<r$ contribute to
$J\cap(\{r\}\times\mathbb N)$, and
\[
 J\cap(\{r\}\times\mathbb N)
 =\bigcup_{i<r}\bigl(J_i'\cap(\{r\}\times\mathbb N)\bigr)
 \quad\text{is finite}.
\]
Hence $J\in\mathcal J$. If $a\geq(a_i^1)_i$, every selected
finite sum belongs to $G_J$, and its stable hull belongs there
as well. For a fixed coordinate $q=(r,k)$, only the finitely many
rows $i<r$ can contribute at $q$. Coordinate nullity gives
$a_i^2\geq a_i^1$ such that $x(q)=0$ for every
$x\in U_{i,a_i^2}$ when $i<r$. Put
$a_i^2=a_i^1$ for $i\geq r$. Above $(a_i^2)_i$, the raw sums
vanish at $q$. The stable coordinate kernel contains their stable
hulls. Those hulls consequently have support in $J$ on one tail
and vanish eventually at every fixed coordinate. Their evaluation
is $\rho$-null, proving $\mathsf S_{\mathrm{all}}(D)$.

For $\mathsf U(M)$, use its notation with the additional parameter
$j$. Uniform nullity of the two errors gives a common threshold
$\beta_0$ and $J_x,J_y\in\mathcal J$ such that their supports
are contained in these sets for all $n,j$ and
$\beta\geq\beta_0$. At $\beta_0$, the approximating difference
admits $J_0\in\mathcal J$ and $N_0\in\N$ such that its supports
are contained in $J_0$ for every $n\ge N_0$ and every $j$. The identity
\[
 x_n(j)-y_n(j)
 =(x_n(j)-x_n^\beta(j))
 +(x_n^\beta(j)-y_n^\beta(j))
 +(y_n^\beta(j)-y_n(j))
\]
places the target supports in $J_x\cup J_0\cup J_y$ on that tail.
For a fixed coordinate $q$, uniform nullity of the two errors
gives $\beta_1\geq\beta_0$ where both vanish at $q$ for all $n,j$.
At this fixed parameter the approximating differences vanish at
$q$ for every $n\ge N_q$ and every $j$, for some integer $N_q$. The same
identity gives coordinate nullity of the target evaluation and
proves $\mathsf U(M)$.

For failure of matrix reflection let $e_{r,k}$ be the coordinate
unit vector and define
\[
 A_{N,L}=\{0\}\cup\{e_{r,k}:r\geq N,\ k\leq L\}.
\]
The sets decrease in $N$ and increase in $L$. For every selector
$\varphi\in\mathbb N^{\mathbb N}$, their selected supports lie in
\[
 J_\varphi=\{(r,k): k\leq\max_{1\leq N\leq r}\varphi(N)\}
 \in\mathcal J.
\]
Every coordinate in physical row $r$ vanishes for $N>r$.
Thus $(A_{N,\varphi(N)})_N\Rightarrow0$ for every $\varphi$.
However,
\[
 T_N=\bigcup_LA_{N,L}
 =\{0\}\cup\{e_{r,k}:r\geq N,\ k\in\mathbb N\}
\]
has no evaluation tail supported in a member of $\mathcal J$:
its supports contain the whole physical row $N$. Hence
$(T_N)_N\not\Rightarrow0$, disproving $\mathsf R_\omega(D)$.
Proposition~\ref{gen:reflection-block} then disproves
$\mathsf F(D)$.

Finally, every topological group satisfies matrix reflection.
If decreasing $T_N=\bigcup_LA_{N,L}$ are not null, some
neighbourhood $V$ of zero satisfies $T_N\setminus V\ne\varnothing$
for every $N$. There are $x_N\notin V$ and $L_N$ with
$x_N\in A_{N,L_N}$. The selector $\varphi(N)=L_N$ gives an
evaluation with values outside $V$ at every stage, contradicting
its assumed nullity. Therefore the failure just proved prevents
any topology from inducing $\rho$.
\end{proof}

\section{The all-subset double process and its splitting domain}
\label{sec:inner-boolean}\label{sec:outer-extension}\label{sec:sion-short}

The regulated extension of the previous section was constructed from certificates defined by
source fragments. We now use finite-response reflection to construct
inner and outer limits for every subset, including sets outside the
regulated domain, thus getting the analogue of the classical procedure of \cite{Sion}. We work now with
$\mathbb B=\mathcal P(\Omega)$ and the data of
Theorem~\ref{gen:extension}, write $m=\tau$, and retain its notation
$M,D,\theta,\Gamma_m,m_\sigma,m_\delta,m_0$.
In this section the row domains are unrestricted and finite-response
reflection $\mathsf F(D)$ is assumed. Thus the structural hypotheses are
$\mathsf F(D),\mathsf S_{\rm all}(D),\mathsf U(M)$. The block condition
required in Theorem~\ref{gen:extension} is supplied by reflection.
Put $\mathcal K=\operatorname{alg}(\mathcal H_\sigma)$ and
\[
 V(A)=\{\tau H:H\in\mathcal H, H\subseteq A\},\qquad
 \Gamma(A)=[V(A)]_\theta,\qquad
 \Delta(A)=D\cap[V(A)-V(A)]_\theta.
\]
Here $\Gamma(A)\subseteq M\subseteq D$, whereas no stability of $D$
is required. The notation $c_D(W)=D\cap[W]_\theta$ will be used only
after taking the indicated stable hull.

\begin{theorem}[All-subset Sion--Carath\'eodory refinement]
\label{thm:main}
There are unique maps $g,\mu:\mathcal P(\Omega)\to M$ such that
\begin{equation}\label{eq:double-process}
 g(A)=\lim_{H\in\mathcal H,\ H\subseteq A}\tau H,
 \qquad
 \mu(A)=\lim_{B\in\mathcal H_\sigma,\ A\subseteq B}g(B),
\end{equation}
where inner indices increase and outer indices decrease.
They satisfy $g(H)=\mu(H)=\tau H$ for $H\in\mathcal H$ and
\[
 A_n\uparrow A\ \Longrightarrow\ \mu(A_n)\tto{\theta}\mu(A),
 \qquad
 \mathcal M_\mu=\mathcal R_{\rm all}(\tau,D),
\]
where
\[
 \mathcal M_\mu=
 \{E\subseteq\Omega:\mu(T)=\mu(T\cap E)+\mu(T\setminus E)
                         \text{ for every }T\subseteq\Omega\}.
\]
On this sigma-algebra $\mu=\bar m$. Consequently $\mu$ is countably
additive there, both full approximation nets of
Theorem~\ref{gen:extension} converge to its values, and its restriction
to $\sigma(\mathcal H)$ is the unique countably additive extension.
For every $A\subseteq\Omega$,
\[
 \mu(A)=\lim_{\substack{T\in\mathcal M_\mu\\ A\subseteq T}}\mu(T),
\]
where the measurable supersets of $A$ are ordered by reverse inclusion.
Moreover
\begin{equation}\label{eq:outer-comparison}
 g(A)\in\Gamma(A),\quad
 \mu(A)\in\bigcap_{S\in\mathcal H_\sigma,\ A\subseteq S}\Gamma(S),
 \quad
 \mu(A)-\mu(E)\in
 \bigcap_{S\in\mathcal H_\sigma,\ A\triangle E\subseteq S}\Delta(S).
\end{equation}
\end{theorem}

We first use finite responses to control residual fragments. The next
lemma requires convergence on disjoint source sequences, with no Cauchy
assumption on either process.

\begin{lemma}[Residual convergence]\label{lem:residual}
Let $\mathscr R$ be a ring and $W:\mathscr R\to\mathcal P(D)$ be
monotone, with $0\in W(E)$ for every $E$. Suppose
$\ev_nW(E_n)\tto{\theta}0$ for every disjoint sequence in $\mathscr R$.
If $\varnothing\ne\mathscr A\subseteq\mathscr R$ is closed under finite
unions, then
\[
 R_C=\bigcup_{A\in\mathscr A}W(A\setminus C),\qquad
 \ev_{C\in\mathscr A}R_C\tto{\theta}0,
\]
with $\mathscr A$ ordered by inclusion.
\end{lemma}
\begin{proof}
For $C\subseteq C'$ in $\mathscr A$, monotonicity gives
$R_{C'}\subseteq R_C$. Since $0\in R_C$, the evaluation net $u$
has $T_u(C,0)=R_C$. Fix $C_*\in\mathscr A$. For each
$(C,x)$ with $C\in\mathscr A$ and $x\in R_C$, fix
$a(C,x)\in\mathscr A$ with $x\in W(a(C,x)\setminus C)$.
Define a strategy by $\alpha_1=(C_*,0)$ and, after a legal
finite response $\varnothing\ne S_n\subseteq R_{C_n}$, by
\[
 C_{n+1}=C_n\cup\bigcup_{x\in S_n}a(C_n,x),\qquad
 \alpha_{n+1}=(C_{n+1},0),\qquad E_n=C_{n+1}\setminus C_n.
\]
Finite-union closure gives $C_{n+1}\in\mathscr A$. In every play,
\[
 E_n\in\mathscr R,\qquad E_n\cap E_j=\varnothing\ (n\ne j),\qquad
 a(C_n,x)\setminus C_n\subseteq E_n\quad(x\in S_n).
\]
Consequently $S_n\subseteq W(E_n)$, and the assumed nullity of
$(W(E_n))_n$ gives $(S_n)_n\Rightarrow0$ by tail containment.
The strategy therefore satisfies the premise of $\mathsf F(D)$,
which gives $u\tto{\theta}0$.
\end{proof}

\begin{proof}[Proof of Theorem~\ref{thm:main}]
We first construct the limits. For fixed $A\subseteq\Omega$, the
family $\mathcal H^-(A)=\{H\in\mathcal H:H\subseteq A\}$ is
nonempty and closed under finite unions. The preceding lemma gives
\[
 R^-_A(H_0)=\bigcup_{K\in\mathcal H^-(A)}V(K\setminus H_0)
 \Rightarrow0.
\]
If $H,K\supseteq H_0$ are in this family, then
\[
 \tau H-\tau K
 =\tau(H\setminus K)-\tau(K\setminus H)
 \in R^-_A(H_0)-R^-_A(H_0).
\]
For arbitrary thresholds $H_1,H_2$, put $H_0=H_1\cup H_2$.
Every pair $(H,K)\ge(H_0,H_0)$ has its difference in
$R^-_A(H_0)-R^-_A(H_0)\subseteq R^-_A(H_1)-R^-_A(H_2)$.
The sets on the right are the tails of an independent difference
of null evaluations. These quantified tail inclusions make the
inner difference net null. Thus the inner
net is Cauchy in $M$. Completeness gives its limit $g(A)\in M$.
Its values belong to $V(A)\subseteq\Gamma(A)$, so stability gives
$g(A)\in\Gamma(A)$.
For $H\in\mathcal H$ the inner net is eventually constant at
$\tau H$. If $H_n\uparrow B\in\mathcal H_\sigma$, the exhaustion
estimate of Lemma~\ref{gen:sigma-construction} gives
$\Gamma(B\setminus H_n)\Rightarrow0$. For fixed $n$, restrict the
inner net to $K\in\mathcal H$ with $H_n\subseteq K\subseteq B$.
For any original index $K_0\subseteq B$, the source element
$K_0\cup H_n$ is an index of this restriction. Every restricted
index $K\supseteq K_0\cup H_n$ satisfies $K\supseteq K_0$.
Thus the restricted net is a quasi-subnet and still converges to $g(B)$.
On this restricted domain,
\[
 \tau K-\tau H_n=\tau(K\setminus H_n)\in V(B\setminus H_n).
\]
Hence $g(B)-\tau H_n\in c_D(V(B\setminus H_n))$ and
$\tau H_n\to g(B)$. Separation identifies
\begin{equation}\label{eq:strong-inner-sigma-identity}
 g(B)=m_\sigma(B)=m_0(B)\qquad(B\in\mathcal H_\sigma).
\end{equation}

Write $\nu=m_0$ and
$V_\nu(E)=\{\nu F:F\in\mathcal K,F\subseteq E\}$.
For disjoint $E_n\in\mathcal K$, the inclusion
$V_\nu(E_n)\subseteq Z_n^{\mathcal K}$ in
Lemma~\ref{gen:UT-K} gives $(V_\nu(E_n))_n\Rightarrow0$. For
$\mathcal H^+(A)=\{B\in\mathcal H_\sigma:A\subseteq B\}$ define
\[
 R_A(B)=\bigcup_{L\in\mathcal H^+(A),\ L\subseteq B}
                      V_\nu(B\setminus L),\qquad
 \mathscr R_A(B)=c_D(R_A(B)).
\]
Apply Lemma~\ref{lem:residual} to
$\{\Omega\setminus B:B\in\mathcal H^+(A)\}\subseteq\mathcal K$.
For arbitrary covers $L,B$ of $A$,
\[
 (\Omega\setminus L)\setminus(\Omega\setminus B)
 =B\setminus L=B\setminus(L\cap B),\qquad L\cap B\in\mathcal H^+(A).
\]
Thus the residual family from the lemma equals $R_A(B)$. If
$A\subseteq B'\subseteq B$, then, for $L\subseteq B'$ covering $A$,
$B'\setminus L\subseteq B\setminus L$. Hence
$R_A(B')\subseteq R_A(B)$, and both residual families decrease in
the reverse-inclusion direction. One-row hull control gives
\begin{equation}\label{eq:outer-residual-null}
 R_A(B)\Rightarrow0,\qquad \mathscr R_A(B)\Rightarrow0.
\end{equation}
For covers $B_1,B_2\subseteq B_0$ of $A$, finite additivity yields
\[
 g(B_1)-g(B_2)
 =\nu(B_1\setminus B_2)-\nu(B_2\setminus B_1)
 \in R_A(B_0)-R_A(B_0).
\]
For every cover $B_0$, all pairs in the product tail indexed by
$(B_0,B_0)$ have their displayed difference in
$R_A(B_0)-R_A(B_0)$. Equation~\eqref{eq:outer-residual-null}
and continuity of independent subtraction make this difference family null.
The outer net is therefore Cauchy in $M$, and completeness gives $\mu(A)\in M$.
For $B\in\mathcal H_\sigma$, the covers of $B$ contained in $B$
consist of $B$ alone. Thus $\mu(B)=g(B)$. Restricting covers of $T$ to those
contained in $B$, and taking limits in
$g(B)-g(L)=\nu(B\setminus L)$, proves
\begin{equation}\label{eq:first-excess}
 A\subseteq T\subseteq B\in\mathcal H_\sigma
 \quad\Longrightarrow\quad
 g(B)-\mu(T)\in\mathscr R_A(B).
\end{equation}
For an original cover $L_0$ of $T$, the element $L_0\cap B$
is a cover of $T$ contained in $B$. Every cover
$L\subseteq L_0\cap B$ belongs to the original tail below $L_0$.
The restricted outer net therefore converges to $\mu(T)$.
For every index in this restriction, $\nu(B\setminus L)\in R_A(B)$.
Its limit belongs to $[R_A(B)]_\theta$ and independently to
$M-M\subseteq D$, proving the displayed inclusion.

We record a second excess estimate. If $E\in\mathcal K$ and
$E\subseteq B\setminus A$, let $(p_k,c_k)$ be a source approximation
of $\Omega\setminus E$ from Lemma~\ref{gen:K-construction}, and set
\[
 O_k=\bigcap_{j\le k}(p_j\cup c_j),\qquad I_k=\Omega\setminus O_k.
\]
Then $O_k\in\mathcal H_\sigma$, $I_k\subseteq E$,
$E\setminus I_k\subseteq c_k$, and
$\nu(E)-\nu(I_k)\in\Gamma(c_k)\Rightarrow0$.
Since $A\cap E=\varnothing$ and $I_k\subseteq E\subseteq B$,
\[
 A\subseteq L_k:=B\cap O_k\subseteq B,\qquad
 B\setminus L_k=I_k,\qquad \nu(I_k)\in R_A(B).
\]
The null error above gives $\nu(I_k)\to\nu(E)$.
Stability of $[R_A(B)]_\theta$ and $\nu(E)\in M\subseteq D$ give
\begin{equation}\label{eq:boolean-excess}
 E\in\mathcal K, E\subseteq B\setminus A
 \quad\Longrightarrow\quad \nu(E)\in\mathscr R_A(B).
\end{equation}

Next we prove the comparison estimates. Fix a cover $S$ of $A$.
For every original cover $B_0$, the intersection $B_0\cap S$
is a cover of $A$ contained in $S$. Every cover below that
intersection lies in the original tail below $B_0$. Thus the
restricted outer net has limit $\mu(A)$. Its values satisfy
$g(B)\in\Gamma(B)\subseteq\Gamma(S)$, so stability gives
$\mu(A)\in\Gamma(S)$. Intersecting over $S$
gives the second inclusion in \eqref{eq:outer-comparison}. If $F,G\in\mathcal K$ lie in $S$, source
approximations truncated by a ring exhaustion of $S$ give
\[
 \nu(F)-\nu(G)\in D\cap[V(S)-V(S)]_\theta=\Delta(S).
\]
Indeed, if $(p_k,c_k)$ approximates $F$ and $h_k\uparrow S$ is a
ring exhaustion, then
\[
 t_k=c_k\cup(S\setminus h_k),\qquad
 F\triangle(p_k\cap h_k)\subseteq t_k,\qquad
 \Gamma(t_k)\Rightarrow0.
\]
The last nullity is supplied by Lemmas~\ref{gen:cover-hull}
and~\ref{gen:sigma-construction}. Finite additivity and
\eqref{gen:K-domination} give
\[
 \nu(F)-\tau(p_k\cap h_k)\in\Gamma(t_k)-\Gamma(t_k)\Rightarrow0.
\]
Thus $\tau(p_k\cap h_k)\to\nu(F)$. The construction for $G$ gives
$\tau(q_k\cap h_k)\to\nu(G)$, and their independent differences lie
in $V(S)-V(S)$.
Now suppose $A\triangle E\subseteq S\in\mathcal H_\sigma$.
For covers $U$ of $A$ and $V$ of $E$, put
\[
 W=(U\cap V)\cup S,\qquad U'=U\cap W,\qquad V'=V\cap W.
\]
These are countable ring unions. Since $A\setminus S=E\setminus S$,
we have $A\subseteq U'$ and $E\subseteq V'$.
On the product of the two reverse-inclusion cover domains, the
maps $(U,V)\mapsto U',V'$ are monotone. For any fixed covers
$U_0,V_0$, all $U\subseteq U_0$, $V\subseteq V_0$ satisfy
$U'\subseteq U_0$, $V'\subseteq V_0$. Hence both transformed
nets have the original outer limits. Furthermore
\[
 W\setminus U',W\setminus V'\subseteq S,\qquad
 g(U')-g(V')=\nu(W\setminus V')-\nu(W\setminus U')\in\Delta(S).
\]
Passage to the limit first puts $\mu(A)-\mu(E)$ in the stable hull
$[V(S)-V(S)]_\theta$. Independently its membership in $M-M\subseteq D$
proves the last inclusion in \eqref{eq:outer-comparison}.

We now establish splitting and identify the domain. Fix $H\in\mathcal H$
and $T\subseteq\Omega$, and define the reverse-inclusion domains
\[
 \begin{split}
 \mathcal C_1&=\{B\in\mathcal H_\sigma:T\cap H\subseteq B\subseteq H\},\\
 \mathcal C_2&=\{B\in\mathcal H_\sigma:T\setminus H\subseteq B\subseteq\Omega\setminus H\}.
 \end{split}
\]
They are nonempty and closed under finite intersections.
For any cover $Q_1$ of $T\cap H$, the element $Q_1\cap H$ lies in
$\mathcal C_1$, and every $B_1\subseteq Q_1\cap H$ lies in the
original cover tail below $Q_1$. For any cover $Q_2$ of
$T\setminus H$, the element $Q_2\setminus H$ lies in
$\mathcal C_2$, and every $B_2\subseteq Q_2\setminus H$ lies in
the original tail below $Q_2$. Thus the restricted value nets
converge to $\mu(T\cap H)$ and $\mu(T\setminus H)$.
For every cover $Q$ of $T$, the product index
$(Q\cap H,Q\setminus H)$ belongs to $\mathcal C_1\times\mathcal C_2$.
Above this index, $T\subseteq B_1\cup B_2\subseteq Q$.
Hence the union-value net is a quasi-subnet of the outer net for $T$.
Since $B_1\cap B_2=\varnothing$, finite additivity gives
\[
 g(B_1\cup B_2)=g(B_1)+g(B_2)
 \ \longrightarrow\ \mu(T\cap H)+\mu(T\setminus H),
\]
The preceding product-tail inclusion makes the left side converge
to $\mu(T)$, proving the splitting identity for $H$.
For $E\in\mathcal K$, a source approximation $(p_k,c_k)$ satisfies
$E\triangle p_k\subseteq c_k$ and $\Gamma(c_k)\Rightarrow0$.
The row $V(c_k)-V(c_k)$ decreases, contains zero, and is null
by subtraction. It lies in $M-M\subseteq D$. One-row hull control
gives $\Delta(c_k)\Rightarrow0$. Since
\[
 (T\cap E)\triangle(T\cap p_k)\subseteq c_k,\qquad
 (T\setminus E)\triangle(T\setminus p_k)\subseteq c_k,
\]
the last estimate in \eqref{eq:outer-comparison} gives
\[
 \mu(T\cap p_k)\to\mu(T\cap E),\qquad
 \mu(T\setminus p_k)\to\mu(T\setminus E),\qquad
 \mu(p_k)\to\mu(E).
\]
The identities for the source sets $p_k$ and separation give
$\mu(T)=\mu(T\cap E)+\mu(T\setminus E)$. Also
$\mu(p_k)=\tau p_k\to\nu(E)$, so $\mu(E)=\nu(E)$.
Thus $\mathcal K\subseteq\mathcal M_\mu$ and $\mu|_{\mathcal K}=\nu$.
For $E\in\mathcal M_\mu$, the defining identity can be written
\[
 \mu(T)=\mu(T\setminus E^c)+\mu(T\cap E^c)
       =\mu(T\cap E^c)+\mu(T\setminus E^c).
\]
Thus $E^c\in\mathcal M_\mu$.
For $E,F\in\mathcal M_\mu$, splitting $T\cap E$ by $F$
and splitting $T\setminus(E\cap F)$ by $E$ gives
\[
 \mu(T)=\mu(T\cap E\cap F)
       +\mu((T\cap E)\setminus F)+\mu(T\setminus E),
 \qquad
 \mu(T\setminus(E\cap F))
       =\mu((T\cap E)\setminus F)+\mu(T\setminus E).
\]
Hence $E\cap F\in\mathcal M_\mu$.
The equality $\mu(\varnothing)=\tau(\varnothing)=0$ gives the
splitting identities for $\varnothing$ and $\Omega$.
The formula $E\cup F=(E^c\cap F^c)^c$ then proves finite-union closure.
Thus this class is a Boolean algebra. For
disjoint $E,F$ in it, the splitting identity with $T=E\cup F$
gives $\mu(E\cup F)=\mu(E)+\mu(F)$.

Let $E\in\mathcal M_\mu$. If $F\in\mathcal M_\mu$ lies in
$B\setminus A$, \eqref{eq:first-excess} and splitting at $B$ give
$\mu(F)=g(B)-\mu(B\setminus F)\in\mathscr R_A(B)$.
For covers $U$ of $E$ and $V$ of $E^c$, put $L=V^c$ and $S=U\cap V$.
For every $H\in\mathcal H$ contained in $S$,
\[
 \tau H=\mu(H\cap E)+\mu(H\setminus E)
 \in\mathscr R_{E^c}(V)+\mathscr R_E(U).
\]
Therefore
\[
 \Gamma(S)\subseteq
 D\cap[\mathscr R_{E^c}(V)+\mathscr R_E(U)]_\theta\Rightarrow0.
\]
Indeed, the two residual rows decrease, contain zero, and are
null by \eqref{eq:outer-residual-null}. Closed summation applied
to these two rows and zero rows bounds the displayed hull by its
null relative finite-sum family. The stable, decreasing null row $\Gamma(U\cap V)$,
together with zero rows, is itself a certificate for
$L\subseteq E\subseteq U$: its relative closed finite-sum family is
$D\cap[\Gamma(U\cap V)]_\theta=\Gamma(U\cap V)$.
Thus $E\in\mathcal R_{\rm all}(\tau,D)$.

Conversely, let a regulated $E$ have certificate
$L_\varphi\subseteq E\subseteq U_\varphi$,
$\Gamma(S_\varphi)\subseteq Q_\varphi$, where
$S_\varphi=U_\varphi\setminus L_\varphi$ and $Q_\varphi\Rightarrow0$.
Since $D=-D$, both $Q_\varphi$ and $-Q_\psi$ are admissible null rows.
Their selector domain $P$ is set-sized and directed, and all directed
row domains are admitted in this section. The two-row summation
principle gives
\[
 F_{\varphi,\psi}=D\cap[Q_\varphi-Q_\psi]_\theta\Rightarrow0.
\]
Indeed, the finite sums from the two rows form $Q_\varphi-Q_\psi$
because both rows contain zero. For any product threshold
$(\varphi_0,\psi_0)$, an upper bound $\chi_0\in P$ of
$\varphi_0,\psi_0$ satisfies
\[
 \bigcup_{\chi\ge\chi_0}D\cap[Q_\chi-Q_\chi]_\theta
 \subseteq F_{\varphi_0,\psi_0}.
\]
Thus the diagonal evaluation is a quasi-subnet of the null product
evaluation. Moreover $V(S_\varphi)\subseteq Q_\varphi$ gives the
second inclusion below after subtraction, stable closure, and
intersection with $D$:
\[
 Z_\varphi=D\cap[Q_\varphi-Q_\varphi]_\theta\Rightarrow0,
 \qquad \Delta(S_\varphi)\subseteq Z_\varphi.
\]
For every $T$, comparison therefore shows that
\[
 \mu(T\cap L_\varphi)-\mu(T\cap E)\to0,\qquad
 \mu(T\setminus L_\varphi)-\mu(T\setminus E)\to0.
\]
Since $L_\varphi\in\mathcal K$ splits $T$, their sum is the constant
$\mu(T)-\mu(T\cap E)-\mu(T\setminus E)$. Separation makes it zero.
This proves $\mathcal M_\mu=\mathcal R_{\rm all}(\tau,D)$.
The full upper net of Theorem~\ref{gen:extension} is the defining
outer net by \eqref{eq:strong-inner-sigma-identity}. Thus
$\mu=\bar m$ on the common domain. Theorem~\ref{gen:extension} therefore supplies, on this same
domain, sigma-closure, countable additivity and both full
approximation limits. Its uniqueness assertion applies to the
restriction to $\sigma(\mathcal H)$.

Fix $A\subseteq\Omega$. Its measurable supersets form a directed set
under reverse inclusion, since $\Omega\in\mathcal M_\mu$ and
$\mathcal M_\mu$ is closed under finite intersections. Each
$B\in\mathcal H_\sigma$ containing $A$ belongs to $\mathcal M_\mu$.
For every measurable $T$ with $A\subseteq T\subseteq B$,
\eqref{eq:first-excess} applied to $A$ and to $T$ gives
\[
 \mu(T)-\mu(A)=(g(B)-\mu(A))-(g(B)-\mu(T))
 \in\mathscr R_A(B)-\mathscr R_A(B).
\]
These difference sets are null by \eqref{eq:outer-residual-null}
and continuity of independent subtraction. For each $B$ the inclusion
holds on the entire tail of measurable supersets contained in $B$.
Tail containment gives $\mu(T)\tto{\theta}\mu(A)$ on the stated domain.

It remains to prove increasing continuity on arbitrary sets.
Let $A_n\uparrow A$. Consider the families
$\beta=(B,(B_i)_{i\ge1})$, with $A\subseteq B$,
$A_i\subseteq B_i\subseteq B$ and $B,B_i\in\mathcal H_\sigma$,
with reverse inclusion in every coordinate. This is a nonempty directed
domain. The family with $B=\Omega$ and all $B_i=\Omega$ is an index.
For two indices $(B,(B_i))$ and $(C,(C_i))$, the family
$(B\cap C,(B_i\cap C_i)_i)$ is an upper bound and satisfies all
cover constraints. Put
$D_n^\beta=\bigcup_{i\le n}B_i$ and
$D^\beta=\bigcup_iB_i$.
For each fixed $\beta$, put $C_1^\beta=D_1^\beta$ and
$C_n^\beta=D_n^\beta\setminus D_{n-1}^\beta$ for $n\ge2$.
These sets belong to $\sigma(\mathcal H)$, are disjoint, and satisfy
$\bigcup_{i\le n}C_i^\beta=D_n^\beta$ and
$\bigcup_iC_i^\beta=D^\beta$.
On $\mathcal H_\sigma$, $g=m_\sigma=\bar m$.
Countable additivity therefore gives
$g(D_n^\beta)=\sum_{i\le n}\bar m(C_i^\beta)\to g(D^\beta)$.
Since $A\subseteq D^\beta\subseteq B$, the first excess estimate
and residual monotonicity give
\[
 g(D^\beta)-\mu(A)\in\mathscr R_A(D^\beta)
 \subseteq\mathscr R_A(B).
\]
For $i<n$, the disjoint sets
$E_{i,n}=(B_i\setminus\bigcup_{j<i}B_j)\setminus B_n$
lie in $B_i\setminus A_i$, since $A_i\subseteq A_n\subseteq B_n$.
Equations~\eqref{eq:first-excess}--\eqref{eq:boolean-excess} give
\[
 g(D_n^\beta)-\mu(A_n)
 =\sum_{i<n}\nu(E_{i,n})+g(B_n)-\mu(A_n)
 \in D\cap\sum_{i\le n}\mathscr R_{A_i}(B_i).
\]
For $\Phi=(B_i)_i$, closed summation gives the null row
\[
 Q_\Phi=D\cap\left[
 \bigcup_{F\subseteq\mathbb N\ {\rm finite}}
       \sum_{i\in F}\mathscr R_{A_i}(B_i)\right]_\theta.
\]
Thus every $g(D_n^\beta)-\mu(A_n)$ lies in $Q_\Phi$, uniformly
in $n$. The other errors lie in $\mathscr R_A(B)$, independently
of $n$. Both conclusions hold on the joint domain of $\beta$:
for each $B_0$ covering $A$ and each family $L_i$ covering $A_i$,
\[
 \beta_0=(B_0,(L_i\cap B_0)_i)
 \quad\Longrightarrow\quad
 \beta\ge\beta_0\ \Longrightarrow\ B\subseteq B_0,
 \quad B_i\subseteq L_i\ (i\ge1).
\]
The intersections are admissible since $A_i\subseteq A\subseteq B_0$.
With $\Lambda=(L_i)_i$, the preceding bounds give the entire-tail
inclusions
\[
 \begin{split}
 \bigcup_{\beta\ge\beta_0}
   \{g(D^\beta)-\mu(A):n\ge1\}
   &\subseteq\mathscr R_A(B_0),\\
 \bigcup_{\beta\ge\beta_0}
   \{g(D_n^\beta)-\mu(A_n):n\ge1\}
   &\subseteq Q_\Lambda.
 \end{split}
\]
Here the dummy index $n$ in the first line records the uniform family
required by $\mathsf U(M)$. Every threshold $B_0$ and every selector
threshold $\Lambda$ can be used in this construction. Thus both error
evaluations are quasi-subnets of their respective null bounding
evaluations. All four point families below belong to $M$.
Apply $\mathsf U(M)$ with a singleton auxiliary set $J$ to
\[
 (x_n^\beta,y_n^\beta)=(g(D^\beta),g(D_n^\beta)),\qquad
 (x_n,y_n)=(\mu(A),\mu(A_n)).
\]
It gives $\mu(A)-\mu(A_n)\to0$. Separation of the limits in
\eqref{eq:double-process} proves uniqueness of both processes.
\end{proof}

\begin{corollary}[Topological group recovery]\label{class:topological}
Let $X$ be a Hausdorff abelian topological group, let
$\Omega\in\mathcal H_\sigma$, and let $\tau:\mathcal H\to X$ be
countably additive and exhaustive: $\tau H_n\to0$ whenever the $H_n$
are pairwise disjoint. If $\overline{\tau(\mathcal H)}$ is Cauchy
complete, every conclusion of Theorem~\ref{thm:main} holds in $X$.
\end{corollary}
\begin{proof}
Set $D=X$ and let $\theta$ be the given topological convergence.
In a topological space, a closed set contains every limit of a net
with values in it. Conversely, if $x\in\overline C$, the domain
\[
 \{(U,c):U\text{ is a neighborhood of }x,\ c\in C\cap U\},
\]
ordered by reverse inclusion of $U$, is nonempty. For any two indices
$(U,c)$ and $(V,d)$, the inclusion $x\in\overline C$ gives
$e\in C\cap U\cap V$, and $(U\cap V,e)$ dominates both indices.
This directed domain carries the net $(U,c)\mapsto c$.
For every neighborhood $V$ of $x$, its tail at an index with $U\subseteq V$
is contained in $V$. This net converges to $x$.
A limit-stable set $C$ must therefore contain $\overline C$.
The stable sets are exactly the closed sets, and hence
$M=[\tau(\mathcal H)]_\theta=\overline{\tau(\mathcal H)}$ is
complete by hypothesis, and $M-M\subseteq X$. Suppose hereditary
exhaustivity fails on a disjoint sequence $(H_n)$. There are a
zero-neighborhood $V$, indices $n_1<n_2<\cdots$, and source sets
$K_j\subseteq H_{n_j}$ with $\tau K_j\notin V$. The $K_j$ are
disjoint, whereas exhaustivity gives $\tau K_j\to0$, a contradiction.

For $\mathsf F(X)$, let $u$ be a nonnull net. There is a
zero-neighborhood $V$ such that
$T_u(\alpha)\setminus V\ne\varnothing$ for every index $\alpha$.
Given any strategy, define recursively
\[
 \alpha_n=\sigma((\alpha_j,\{x_j\})_{j<n}),\qquad
 x_n\in T_u(\alpha_n)\setminus V.
\]
This legal play has a nonnull response evaluation $(\{x_n\})_n$.
Consequently no strategy satisfies the premise of $\mathsf F(X)$
for a nonnull net, proving that principle.

For $\mathsf S_{\rm all}(X)$, fix a zero-neighborhood $V$.
Let $W$ be symmetric and open with $W+W\subseteq V$.
For $x\in\overline W$, $(x+W)\cap W\ne\varnothing$, and hence
$x\in W-W=W+W$. Thus $W_0=\overline W\subseteq V$ is a closed
zero-neighborhood. Recursively take symmetric zero-neighborhoods
$W_n$ with $W_n+W_n\subseteq W_{n-1}$ for $n\ge1$.
Since $0\in W_n$, also $W_n\subseteq W_{n-1}$. For every $m$,
backward induction on $j=m,m-1,\ldots,1$ gives
\[
 \sum_{n=j}^mW_n\subseteq W_{j-1}.
\]
Indeed, the case $j=m$ is the preceding nesting, and
$W_j+\sum_{n=j+1}^mW_n\subseteq W_j+W_j\subseteq W_{j-1}$.
For each decreasing null row $(U_{n,\alpha})_{\alpha\in I_n}$,
there is $\alpha_n$ with
$U_{n,\alpha}\subseteq W_n$ whenever $\alpha\ge\alpha_n$.
For every selector $\varphi\ge(\alpha_n)_n$, zero insertion gives
\[
 A_\varphi\subseteq W_0,\qquad
 [A_\varphi]_\theta\subseteq W_0\subseteq V.
\]
This verifies closed summation on arbitrary row domains.

For $\mathsf U(M)$, use its notation and fix a zero-neighborhood
$V$. Let $W=-W$ satisfy $W+W+W\subseteq V$. The two uniform error
premises supply one $\beta$ such that, for every $n\ge1$ and $j\in J$,
\[
 x_n^\beta(j)-x_n(j)\in W,\qquad
 y_n^\beta(j)-y_n(j)\in W.
\]
The null $\beta$-row supplies $N$ with
$x_n^\beta(j)-y_n^\beta(j)\in W$ for every $n\ge N$, $j\in J$.
For those indices,
\[
 x_n(j)-y_n(j)
 =(x_n(j)-x_n^\beta(j))
 +(x_n^\beta(j)-y_n^\beta(j))
 +(y_n^\beta(j)-y_n(j))\in W+W+W\subseteq V.
\]
This proves $\mathsf U(M)$. Countable additivity and
$\Omega\in\mathcal H_\sigma$ are the remaining source hypotheses.
Theorem~\ref{thm:main} now applies.
\end{proof}

This corollary recovers Sion's topological-group double process
\cite{Sion}. The measurable-cover limit in Theorem~\ref{thm:main}
gives its outer regularity for every $A\subseteq\Omega$.
We next verify reflection and closed summation for order
convergence and determine the distributive laws associated with
summation. For applications concerned only with the regulated
extension, these conditions can be verified without constructing
the all-subset process.

\section{Order convergence and the distributive conditions}
\label{ord:section}\label{inv:section}\label{sec:order-short}

We now verify the range principles for order convergence. Closed
summation also admits a converse: its validity on every bounded
interval determines a distributive law. Throughout this section $E$ is
an abelian lattice-ordered group, and $|x|=x\vee(-x)$. The lattice-group
identities used below are recorded in \cite{Birkhoff}.

\begin{definition}\label{def:orderconv}\label{ent:order-data}
A net $x:\Lambda\to E$ converges in order to $x_0$, written
$x_\alpha\onull x_0$, if some decreasing net
$r_\gamma\downarrow0$ satisfies
\[
 \forall\gamma\ \exists\alpha_\gamma\ \forall\alpha\geq\alpha_\gamma:
 \quad |x_\alpha-x_0|\leq r_\gamma.
\]
The domains of $x$ and $r$ are independent. Dedekind completeness means
that every nonempty upper-bounded subset has a supremum. The countable
sup property means that every nonempty subset having a supremum has a countable
subset with the same supremum.
\end{definition}

\begin{definition}\label{def:wsd}\label{cmp:strong-wsd}
A bounded regulator is a matrix $0\leq a_{n,k}\leq h$ satisfying
$a_{n,k}\downarrow_k0$ for every $n$. A Dedekind-complete group is
weakly sigma-distributive, abbreviated WSD, if
\begin{equation}\label{ord:wsd}
 \inf_{\varphi\in\N^\N}\sup_n a_{n,\varphi(n)}=0
\end{equation}
for every bounded regulator. It is weakly
$(\sigma,\infty)$-distributive if, for every family of decreasing
net rows $a_{n,\alpha}\downarrow_{\alpha\in A_n}0$ with one common
positive bound,
\begin{equation}\label{cmp:strong-wsd-equation}
 \inf_{a\in\prod_n A_n}\sup_n a_{n,a_n}=0.
\end{equation}
All selector products carry the coordinatewise order.
\end{definition}

Order convergence is a separated group convergence. Tail refinement
preserves each regulating bound. Finite mixing is bounded by the sum
of the two regulators. For independent regulators,
$\inf_{\gamma,\delta}(r_\gamma+s_\delta)=0$ by translation
invariance of infima, so addition preserves convergence. Negation
preserves absolute differences. Two limits satisfy
$|x-y|\leq r_\gamma+s_\delta$ at every pair of regulating indices
and therefore coincide.

In a Dedekind-complete group, order nullity of a bounded net has the
tail-envelope criterion
\begin{equation}\label{ord:filter-envelope}
 x_\alpha\onull0
 \quad\Longleftrightarrow\quad
 u_\alpha:=\sup_{\beta\geq\alpha}|x_\beta|\downarrow0.
\end{equation}
Indeed, an order regulator bounds a suitable $u_\alpha$ at every
regulating index. Conversely, $(u_\alpha)$ itself regulates $x$.
For $C=[-h,h]$ and $D=C-C$,
\begin{equation}\label{ord:carrier-bound}
 0\in C=-C,\qquad C\subseteq D\subseteq[-2h,2h],\qquad\clo_o C=C.
\end{equation}
For the last equality, $x_\alpha\in C$ and
$|x_\alpha-x|\leq r_\gamma$ eventually imply
$(|x|-h)^+\leq r_\gamma$ for every $\gamma$, hence $|x|\leq h$.

\begin{lemma}[Capped sums]\label{ord:cap}\label{ord:aggregation}
For $h,r_1,\ldots,r_m\in E_+$,
\begin{equation}\label{ord:cap-formula}
 h\wedge\sum_{n=1}^m r_n
 \leq\bigvee_{n=1}^m(h\wedge2^n r_n).
\end{equation}
Consequently, if $E$ is Dedekind complete and WSD, and
$a_{n,k}\downarrow_k0$ without a common bound, then
\begin{equation}\label{ord:aggregation-q}
 q_\varphi:=\sup_m\left(h\wedge\sum_{n=1}^m a_{n,\varphi(n)}\right)
 \downarrow_{\varphi\in\N^\N}0.
\end{equation}
For net rows the same assertion holds under
weak $(\sigma,\infty)$-distributivity.
\end{lemma}
\begin{proof}
Put $b=\bigvee_{n\leq m}(h\wedge2^nr_n)$, $c=h-b\geq0$,
and $d_n=(2^nr_n-b)^+$. Lattice distributivity gives
\[
 d_n\wedge c=((2^nr_n\wedge h)-b)^+=0\qquad(1\leq n\leq m).
\]
For $u,v,c\geq0$ and $w=(u+v)\wedge c$,
\[
 w=(w\wedge u)+(w-u)^+,\qquad
 w\wedge u\leq u\wedge c,\qquad (w-u)^+\leq v\wedge c.
\]
Hence $(u+v)\wedge c\leq(u\wedge c)+(v\wedge c)$, and
$\bigl(\sum_{n=1}^m2^{m-n}d_n\bigr)\wedge c=0$.
For $S=\sum_{n=1}^m r_n$,
\[
 (S-b)^+\leq\bigl(2^m(S-b)+b\bigr)^+
 =\left(\sum_{n=1}^m2^{m-n}(2^nr_n-b)\right)^+
 \leq\sum_{n=1}^m2^{m-n}d_n.
\]
Thus $(h\wedge S-b)^+=c\wedge(S-b)^+=0$, proving
\eqref{ord:cap-formula}.

For a row domain $A_n$, put $c_{n,\alpha}=h\wedge2^na_{n,\alpha}$.
Translation preserves infima, and the diagonal of $A_n^{2^n}$ is
cofinal. Therefore
\[
 \inf_{\alpha\in A_n}2^na_{n,\alpha}
 =\inf_{(\alpha_1,\ldots,\alpha_{2^n})\in A_n^{2^n}}
       \sum_{j=1}^{2^n}a_{n,\alpha_j}=0,
 \qquad c_{n,\alpha}\downarrow_\alpha0.
\]
Each capped finite sum lies in $[0,h]$, so its supremum $q_a$ exists.
For every selector $a\in\prod_nA_n$,
\[
 0\leq q_a\leq\sup_n c_{n,a_n}\leq h,
 \qquad a\leq a'\Longrightarrow q_{a'}\leq q_a.
\]
The asserted distributive law gives
$0\leq\inf_aq_a\leq\inf_a\sup_nc_{n,a_n}=0$.
Taking $A_n=\N$ proves \eqref{ord:aggregation-q}.
\end{proof}

\begin{theorem}[The distributive content of closed summation]
\label{cmp:strong-wsd-equivalence}
Let $E$ be Dedekind complete. Closed bounded summation for arbitrary
directed rows on every symmetric order interval is equivalent to
weak $(\sigma,\infty)$-distributivity. Closed bounded summation for
sequential rows on every symmetric interval is equivalent to WSD.
\end{theorem}
\begin{proof}
Assume the corresponding distributive law. For decreasing null rows
$0\in U_{n,\alpha}\subseteq[-H,H]$, define
\[
 b_{n,\alpha}=\sup_{u\in U_{n,\alpha}}|u|,\qquad
 K_a=\bigcup_{m\geq1}\sum_{n=1}^mU_{n,a_n},\qquad
 Q_a=[-H,H]\cap\clo_oK_a.
\]
The first supremum exists in $[0,H]$. A row regulator
$r_{n,\gamma}\downarrow0$ supplies
\[
 \forall\gamma\ \exists\alpha(n,\gamma)\
 \forall\alpha\geq\alpha(n,\gamma):\quad
 0\leq b_{n,\alpha}\leq r_{n,\gamma}.
\]
Hence $b_{n,\alpha}\downarrow_\alpha0$. Zero insertion identifies
$K_a$ with the union of all finite selected sums.
Lemma~\ref{ord:aggregation} gives
\[
 q_a=\sup_m\left(H\wedge\sum_{n=1}^m b_{n,a_n}\right)\downarrow0,
 \qquad K_a\subseteq L_a:=\{z:|z|\wedge H\leq q_a\}.
\]
If $z_\delta\in L_a$ and $z_\delta\onull z$ with regulator
$r_\gamma\downarrow0$, then for every $\gamma$ there is an index
$\delta_\gamma$ such that every $\delta\ge\delta_\gamma$ satisfies
\[
 \bigl|(|z|\wedge H)-(|z_\delta|\wedge H)\bigr|
 \leq|z-z_\delta|\leq r_\gamma,
 \qquad (|z|\wedge H-q_a)^+\leq r_\gamma.
\]
Taking the infimum over $\gamma$ proves $z\in L_a$. Thus
$\clo_oK_a\subseteq L_a$ and
$Q_a\subseteq[-q_a,q_a]$, which regulates $\ev_a(Q_a)$.

Conversely, let $0\leq b_{n,\alpha}\leq H$ decrease to zero in
each row, and put $U_{n,\alpha}=[0,b_{n,\alpha}]$. The bounds
$|u|\leq b_{n,\alpha}$ regulate the row evaluations.
Closed summation supplies $r_\gamma\downarrow0$ with
\[
 \forall\gamma\ \exists a(\gamma)\ \forall a\geq a(\gamma):
 \quad Q_a\subseteq[-r_\gamma,r_\gamma].
\]
For every $n$ and $a$,
\begin{equation}\label{ord:necessity-inclusion}
 b_{n,a_n}=0+\cdots+0+b_{n,a_n}\in K_a\cap[-H,H]\subseteq Q_a.
\end{equation}
All selected suprema belong to $[0,H]$. Consequently
\[
 0\leq\inf_a\sup_n b_{n,a_n}
 \leq\sup_n b_{n,a(\gamma)_n}\leq r_\gamma\quad(\gamma),
 \qquad\inf_a\sup_n b_{n,a_n}=0.
\]
Arbitrary domains give \eqref{cmp:strong-wsd-equation}, and
$A_n=\N$ gives \eqref{ord:wsd}.
\end{proof}

\begin{corollary}\label{ord:equivalence}\label{cmp:strong-wsd-csp}
In a Dedekind-complete group with the countable sup property, WSD,
weak $(\sigma,\infty)$-distributivity, and closed bounded summation
on every interval are equivalent.
\end{corollary}
\begin{proof}
For $b_\alpha\downarrow0$, the set $\{-b_\alpha\}$ has supremum
zero. The countable sup property supplies
$(\alpha'_k)_{k\geq1}$ with $\inf_kb_{\alpha'_k}=0$.
Set $\alpha_1=\alpha'_1$, and recursively take an upper bound
$\alpha_{k+1}$ of $\alpha_k$ and $\alpha'_{k+1}$. Then
\[
 \alpha_k\leq\alpha_{k+1},\qquad
 0\leq\inf_kb_{\alpha_k}\leq\inf_kb_{\alpha'_k}=0,
\]
so
\begin{equation}\label{eq:countable-regulator}
 b_{\alpha_k}\downarrow_k0.
\end{equation}
For bounded null rows, carry out this recursion with indices
$\alpha_{n,k}\in A_n$. WSD gives
\[
 0\leq\inf_{a\in\prod_n A_n}\sup_n b_{n,a_n}
 \leq\inf_{\varphi\in\N^\N}\sup_n b_{n,\alpha_{n,\varphi(n)}}=0.
\]
Weak $(\sigma,\infty)$-distributivity implies WSD by the specialization
$A_n=\N$. Theorem~\ref{cmp:strong-wsd-equivalence} identifies the
corresponding closed-summation conditions. The extraction uses
$\inf_kb_{\alpha_{n,k}}=0$, without requiring a cofinal sequence in
$A_n$.
\end{proof}

\begin{theorem}[The band determined by an interval]\label{inv:wsd-local}
Let $F$ be a Dedekind-complete real vector lattice and $h\in F_+$.
Put $B_h=\{h\}^{\perp\perp}$, where
$A^\perp=\{x:|x|\wedge|a|=0\ (a\in A)\}$.
Sequential closed summation on $[-2h,2h]$ holds if and only if
$B_h$ is WSD. No countable sup property is required.
\end{theorem}
\begin{proof}
A band in a Dedekind-complete vector lattice contains the supremum
of each of its positive subsets bounded above in $F$. Indeed, if
$0\leq z_\eta\in B_h$ and $z=\sup_\eta z_\eta$, then for
$w\in\{h\}^\perp$,
$z\wedge|w|=\sup_\eta(z_\eta\wedge|w|)=0$.
Its bounded positive infima also belong to it by solidity.
Thus these extrema agree with the extrema in $F$.

Suppose $B_h$ is WSD. For decreasing null interval rows
$U_{n,k}\subseteq[-2h,2h]$, the bounds
$b_{n,k}=\sup_{u\in U_{n,k}}|u|$ belong to $B_h$ and decrease to
zero. Lemma~\ref{ord:aggregation} in $B_h$ gives
\[
 q_\varphi=\sup_m\left(2h\wedge\sum_{n=1}^m
                     b_{n,\varphi(n)}\right)\downarrow0
 \quad\text{in }B_h\text{ and in }F.
\]
The proof of Theorem~\ref{cmp:strong-wsd-equivalence} establishes
order stability of $\{z:|z|\wedge2h\leq q_\varphi\}$ in $F$.
For $K_\varphi=\bigcup_m\sum_{n=1}^mU_{n,\varphi(n)}$, this gives
\[
 K_\varphi\subseteq\{z:|z|\wedge2h\leq q_\varphi\},\qquad
 [-2h,2h]\cap\clo_oK_\varphi\subseteq[-q_\varphi,q_\varphi].
\]
These bounds prove sequential closed summation.

Conversely, \eqref{ord:necessity-inclusion} applied to the interval
$[-2h,2h]$ gives
$\inf_\varphi\sup_na_{n,\varphi(n)}=0$ whenever
$0\leq a_{n,k}\leq2h$ and $a_{n,k}\downarrow_k0$.
For an arbitrary bounded regulator in $(B_h)_+$, define
$d=\inf_\varphi\sup_na_{n,\varphi(n)}\in(B_h)_+$.
The regulator $(a_{n,k}\wedge h)$ satisfies the preceding bound.
Finite meet preserves these existing suprema and infima, whence
\[
 d\wedge h
 =\inf_\varphi\left(\sup_na_{n,\varphi(n)}\wedge h\right)
 =\inf_\varphi\sup_n(a_{n,\varphi(n)}\wedge h)=0.
\]
Thus $d\in B_h\cap\{h\}^\perp=\{0\}$, proving WSD.
\end{proof}

The inverse assertions quantify over every null row in the stated
interval. They impose no necessity claim on a single prescribed
measure. For instance, $B_{(1,0)}=\mathbb R\times\{0\}$ in
$\mathbb R\times F_0$, regardless of the distributive properties
of $F_0$.

\begin{theorem}[Finite-response reflection]\label{ord:fg}
If $E$ is Dedekind complete, has the countable sup property, and is
WSD, then order convergence satisfies $\mathsf F$ on every bounded
interval.
\end{theorem}
\begin{proof}
Let $x:\Lambda\to[-H,H]$ and put
$s=\inf_\alpha\sup_{\beta\geq\alpha}|x_\beta|\in[0,H]$.
Suppose $s\ne0$, and fix a first-player strategy $\sigma$.
On the countable set $T=\bigcup_{m\geq0}\N^m$, set
$H_\varnothing=\varnothing$. At a constructed history $H_t$, put
$\alpha_t=\sigma(H_t)$. The countable sup property supplies
$z_{t,j}\in T_x(\alpha_t)$ satisfying
\[
 \sup_{j\geq1}|z_{t,j}|=\sup_{\beta\geq\alpha_t}|x_\beta|\geq s.
\]
For $k\geq1$ define
\[
 S_{t,k}=\{z_{t,j}:1\leq j\leq k\},\quad
 b_{t,k}=\bigvee_{j\leq k}|z_{t,j}|,\quad
 H_{t^\frown k}=H_t^\frown(\alpha_t,S_{t,k}).
\]
Every $S_{t,k}$ is nonempty, finite and contained in the requested
tail. Moreover
\[
 a_{t,k}=(s-b_{t,k})^+\downarrow_k
 (s-\sup_j|z_{t,j}|)^+=0,\qquad0\leq a_{t,k}\leq s.
\]
WSD, after an enumeration of $T$, gives
$\inf_{\varphi\in\N^T}\sup_ta_{t,\varphi(t)}=0$.
Since $s\ne0$, some $\varphi$ satisfies
$b:=\sup_ta_{t,\varphi(t)}\ne s$. Put $v=s-b>0$ and define
\[
 t_0=\varnothing,\qquad t_{n+1}=t_n^\frown\varphi(t_n),\qquad
 S_{n+1}=S_{t_n,\varphi(t_n)}\quad(n\geq0).
\]
This is a legal play against $\sigma$. The inequality
$s-b_{t_n,\varphi(t_n)}\leq a_{t_n,\varphi(t_n)}\leq b$ yields
\begin{equation}\label{ord:fg-obstruction}
 0<v\leq b_{t_n,\varphi(t_n)}
       =\bigvee_{z\in S_{n+1}}|z|\qquad(n\geq0).
\end{equation}
If $\ev_n(S_n)\onull0$, a regulator $r_\gamma\downarrow0$ would give
\[
 \forall\gamma\ \exists N_\gamma\ \forall n\geq N_\gamma
 \ \forall z\in S_n:\quad |z|\leq r_\gamma.
\]
Equation~\eqref{ord:fg-obstruction} then implies
$0<v\leq r_\gamma$ for every $\gamma$, a contradiction.
Thus a strategy with every response evaluation null requires $s=0$.
Criterion~\eqref{ord:filter-envelope} proves $x\onull0$.
\end{proof}

\begin{proposition}[Uniform closure]\label{ord:uc}
Every bounded interval in a Dedekind-complete group satisfies
$\mathsf U$ for order convergence, including the version uniform
over an additional index set.
\end{proposition}
\begin{proof}
Retain the additional nonempty index set $J$ in the definition of
$\mathsf U$. All points lie in an interval $[a,b]$, so
$|x_n(j)-y_n(j)|\leq b-a$.
Let $r^x_\gamma\downarrow0$ and $r^y_\eta\downarrow0$ regulate
the two uniform approximation errors. On their product domain put
$r_{(\gamma,\eta)}=r^x_\gamma+r^y_\eta\downarrow0$.
Directedness of the approximation domain provides thresholds
$\beta_{(\gamma,\eta)}$ above both original thresholds. Relabeling
this product regulator as $(r_\gamma)$ gives
\begin{equation}\label{ord:uc-uniform}
 |x_n^\beta(j)-x_n(j)|+|y_n^\beta(j)-y_n(j)|\leq r_\gamma
 \quad(\beta\geq\beta_\gamma,\ n\geq1,\ j\in J).
\end{equation}
Dedekind completeness defines
\[
 t_N=\sup_{n\geq N,\ j\in J}|x_n(j)-y_n(j)|\in[0,b-a],
 \qquad t=\inf_Nt_N.
\]
At $\beta_\gamma$ the approximating pair hypothesis supplies
$s_{\gamma,\delta}\downarrow_\delta0$ and $N_{\gamma,\delta}$ with
\[
 |x_n^{\beta_\gamma}(j)-y_n^{\beta_\gamma}(j)|
 \leq s_{\gamma,\delta}\quad(n\geq N_{\gamma,\delta},\ j\in J).
\]
The triangle inequality and \eqref{ord:uc-uniform} give
\[
 0\leq t\leq t_{N_{\gamma,\delta}}
 \leq r_\gamma+s_{\gamma,\delta},\qquad
 t\leq\inf_\delta(r_\gamma+s_{\gamma,\delta})=r_\gamma.
\]
Therefore $t\leq\inf_\gamma r_\gamma=0$, and
$|x_n(j)-y_n(j)|\leq t_N\downarrow_N0$ for $n\geq N$, uniformly
in $j\in J$. A singleton $J$ includes the unindexed assertion.
\end{proof}

\begin{theorem}[Order-Cauchy completeness]\label{ord:complete}
Every Dedekind-complete abelian lattice-ordered group is complete
for order-Cauchy nets, and each order-limit-stable subset is complete.
\end{theorem}
\begin{proof}
Let $r_\gamma\downarrow0$ satisfy
\begin{equation}\label{ord:cauchy-diameter}
 |x_\alpha-x_\beta|\leq r_\gamma
 \quad(\alpha,\beta\geq\lambda_\gamma).
\end{equation}
Fix $\gamma_0$ and put $\alpha_0=\lambda_{\gamma_0}$ and
$\Lambda_0=\{\alpha:\alpha\geq\alpha_0\}$. Then
\[
 x_{\alpha_0}-r_{\gamma_0}\leq x_\beta
 \leq x_{\alpha_0}+r_{\gamma_0}\qquad(\beta\in\Lambda_0).
\]
For $\alpha\in\Lambda_0$, Dedekind completeness therefore defines
\[
 \ell_\alpha=\inf_{\beta\geq\alpha}x_\beta,\quad
 u_\alpha=\sup_{\beta\geq\alpha}x_\beta,\quad
 \ell=\sup_{\alpha\in\Lambda_0}\ell_\alpha,\quad
 u=\inf_{\alpha\in\Lambda_0}u_\alpha.
\]
For $\alpha,\beta\in\Lambda_0$, an upper bound $\delta\geq\alpha,\beta$
satisfies $\ell_\alpha\leq x_\delta\leq u_\beta$, giving
$\ell\leq u$. For each $\gamma$, take
$\alpha\geq\alpha_0,\lambda_\gamma$. For $\beta,\delta\geq\alpha$,
\[
 x_\beta\leq x_\delta+r_\gamma
 \quad\Longrightarrow\quad
 u_\alpha\leq\ell_\alpha+r_\gamma
 \quad\Longrightarrow\quad0\leq u-\ell\leq r_\gamma.
\]
Thus $u=\ell=:x$. Whenever $\beta\geq\alpha\geq\alpha_0,\lambda_\gamma$,
\[
 \ell_\alpha\leq x,x_\beta\leq u_\alpha,\qquad
 |x_\beta-x|\leq u_\alpha-\ell_\alpha\leq r_\gamma.
\]
This proves $x_\alpha\onull x$. If all $x_\alpha$ belong to an
order-limit-stable subset $C$, its stability gives $x\in C$.
The same regulator verifies convergence in the induced structure on
$C$, proving its completeness.
\end{proof}

\begin{corollary}\label{ord:carriers}
In a Dedekind-complete WSD group with the countable sup property,
each interval $[-h,h]$ is complete and satisfies the finite-response,
closed summation, and uniform closure principles for order convergence.
\end{corollary}
\begin{proof}
For $C=[-h,h]$, equation~\eqref{ord:carrier-bound} gives
$\clo_oC=C$, and Theorem~\ref{ord:complete} makes $C$ complete.
The verified range principles on this interval are
\[
 \mathsf F(C)\quad\text{by Theorem~\ref{ord:fg}},\qquad
 \mathsf S_{\rm all}(C)\quad\text{by Corollary~\ref{ord:equivalence}},
 \qquad\mathsf U(C)\quad\text{by Proposition~\ref{ord:uc}}.
\]
Each statement holds for every $h\in E_+$, including the interval
$[-2h,2h]$ that contains $C-C$.
\end{proof}

Having verified the range principles, we turn to the source estimates
needed in the applications. We will derive these estimates from the
original charge and use completeness to construct the extension values.

\section{Positive extension without a lattice order}
\label{sec:volauf}\label{sec:ordered-short}

Volauf's extension theorem allows a partially ordered range whose order
need not be a lattice order. We recover it by representing the source
charge in a lattice-ordered group and then returning to the original
range through an order-continuous homomorphism. Using only $G$ and the
values of $\tau$ on $\HH$, we construct in
Lemmas~\ref{volfull:Gcomplete}--\ref{volfull:step-properties} a complete
Boolean algebra $\mathbb D$, an integer-step group $K$, a premeasure
$\lambda:\HH\to K$, and an order-continuous homomorphism $J:K\to G$
satisfying $J\lambda=\tau$. These lemmas verify the extension hypotheses,
with only directed extrema taken in $G$.

We then apply Theorem~\ref{gen:extension} to $\lambda$ and transport
its extension $\bar\lambda$ to $J\bar\lambda$ by
Proposition~\ref{gen:transport}. The restriction to $\sigma(\HH)$
recovers Volauf's extension~\cite[Theorem~3]{Volauf80}.
Theorem~\ref{vol:extension} identifies $J\bar\lambda$ with the double
process on its full splitting domain in the original partially ordered
group. The Boolean identities give the all-subset formulas and
positivity properties, while sigma-closure, countable additivity and
full regularity come from the general extension and transport results.
We conclude with the original process estimates and the direct
lattice-valued specializations in Corollary~\ref{volfull:compatibility},
where the criteria of Section~\ref{sec:order-short} apply in the
original range.

\begin{definition}\label{class:volauf-regular}
A partially ordered abelian group $G$ is \emph{monotone complete} if
every nonempty upward-directed upper-bounded subset has a supremum,
and \emph{order separable} if every nonempty subset possessing a
supremum contains a countable subset with that supremum. For a bounded
positive matrix $a_{n,k}\downarrow_k0$, let
\[
 \Phi(a)=\left\{\varphi\in\N^{\N}:
    \left\{\sum_{n\le m}a_{n,\varphi(n)}:m\ge1\right\}
    \text{ is upper bounded}\right\}.
\]
The group is \emph{regular} if $\Phi(a)\ne\varnothing$ and the only
positive common lower bound of
$s_\varphi=\sup_m\sum_{n\le m}a_{n,\varphi(n)}$, $\varphi\in\Phi(a)$,
is zero. This is the bounded-selector formulation of
\cite[p.~352]{Volauf80}.
\end{definition}

Here $x_\alpha\to_o x$ means that a net $r_\gamma\downarrow0$ in $G_+$
satisfies $-r_\gamma\le x_\alpha-x\le r_\gamma$ eventually in $\alpha$
for each $\gamma$. For order-Cauchy nets the same bounds hold for
$x_\alpha-x_\beta$ on a common tail in the two independent indices.
Throughout this section $G$ has these three properties, $\HH$ is an
algebra on $\Omega$, and $\tau:\HH\to G_+$ is finitely additive with
\begin{equation}\label{volfull:original-countadd}
 H_n\downarrow\varnothing\quad\Longrightarrow\quad\tau H_n\downarrow0.
\end{equation}
Put $h=\tau\Omega$. This condition is equivalent to Volauf's positive
countable-additivity condition: if $H=\dot\bigcup_nH_n\in\HH$, then
$\tau H-\sum_{n\le N}\tau H_n=\tau(H\setminus\bigcup_{n\le N}H_n)
\downarrow0$. Conversely, apply that identity to the disjoint differences
$H_n\setminus H_{n+1}$ of a decreasing sequence with empty intersection.
Countable additivity itself implies finite additivity by appending empty
sets to a finite disjoint family.

\begin{lemma}\label{volfull:Gcomplete}\label{volfull:regulator}
Order convergence in $G$ is separated, order intervals are limit-stable,
and every order-Cauchy net converges. For the matrices in
Definition~\ref{class:volauf-regular}, $s_\varphi\downarrow_\varphi0$.
\end{lemma}
\begin{proof}
Translations preserve existing extrema. For $r_\alpha\downarrow0$ and
$t_\beta\downarrow0$, a common lower bound $z$ of their sums satisfies
$z-r_\alpha\le\inf_\beta t_\beta=0$ for every $\alpha$, whence $z\le0$.
Thus $r_\alpha+t_\beta\downarrow_{(\alpha,\beta)}0$. In particular,
$2r_\alpha\downarrow0$. Two limits $x,y$ of one net satisfy
\[
 -r_\alpha-t_\beta\le x-y\le r_\alpha+t_\beta
 \quad(\alpha,\beta),\qquad x=y.
\]
If $a\le x_i\le b$ and $x_i\to_o x$, then
$a-r_\alpha\le x\le b+r_\alpha$ for every regulating index $\alpha$.
Taking the corresponding extrema gives $a\le x\le b$.
For $\varphi,\psi\in\Phi(a)$, their pointwise maximum $\chi$ belongs
to $\Phi(a)$ and $0\le s_\chi\le s_\varphi,s_\psi$.
Monotone completeness supplies $\inf_\varphi s_\varphi\ge0$.
Regularity makes this infimum zero.

Let $(x_i)_{i\in I}$ be order-Cauchy with regulator $r_\beta\downarrow0$.
Order separability applied to $\{-r_\beta\}$ supplies
$\beta'_k$ with $\inf_k r_{\beta'_k}=0$.
Recursively take $\beta_1\ge\beta'_1$ and
$\beta_{k+1}\ge\beta_k,\beta'_{k+1}$.
Then $q_k=r_{\beta_k}\downarrow0$. Apply regularity to the bounded
matrix $a_{j,k}=q_k$. A selector $\varphi\in\Phi(a)$ defines
\[
 c_j=q_{\varphi(j)},\qquad
 s=\sup_m\sum_{j\le m}c_j,\qquad
 t_j=s-\sum_{i<j}c_i\downarrow0,
 \qquad t_j-t_{j+1}=c_j.
\]
For each $j$, let $d_j\in I$ satisfy
$-c_j\le x_i-x_{i'}\le c_j$ whenever $i,i'\ge d_j$.
Recursively take $i_1\ge d_1$ and $i_{j+1}\ge i_j,d_{j+1}$.
The quantities $l_j=x_{i_j}-t_j$ and $u_j=x_{i_j}+t_j$ satisfy
\[
 l_{j+1}-l_j=x_{i_{j+1}}-x_{i_j}+c_j\ge0,
 \qquad u_{j+1}-u_j=x_{i_{j+1}}-x_{i_j}-c_j\le0.
\]
For $j\le k$ one has $l_j\le l_k\le u_k$, while for $k\le j$ one
has $l_j\le u_j\le u_k$. Thus $l_j\le u_k$ for all $j,k$.
Monotone completeness applied to the increasing family $(l_j)$ and
the decreasing family $(u_j)$ supplies their extrema, and
\[
 0\le\inf_j u_j-\sup_jl_j\le u_k-l_k=2t_k\quad(k\ge1).
\]
Put $x=\sup_jl_j=\inf_j u_j$. Then
$-t_j\le x_{i_j}-x\le t_j$, so $x_{i_j}\to_o x$.
For $i\ge i_j$, all $k\ge j$ satisfy
$-c_j\le x_i-x_{i_k}\le c_j$.
Stability of this interval gives
$-t_j\le-c_j\le x_i-x\le c_j\le t_j$.
This proves convergence of the full net.
\end{proof}

\begin{lemma}[Boolean representation of the original charge]
\label{volfull:boolean-values}\label{volfull:Dcomplete}
\label{volfull:Dorder}\label{volfull:Dwsd}
There are a complete Boolean algebra $\mathbb D$ with the countable sup
property, a Boolean homomorphism $\iota:\HH\to\mathbb D$, and a
positive finitely additive map $v:\mathbb D\to[0,h]$ such that
\[
 v\iota=\tau,\qquad v(d)=0\Longleftrightarrow d=0,
 \qquad d_\alpha\uparrow d\Longrightarrow v(d_\alpha)\uparrow v(d).
\]
The corresponding decreasing-net identity also holds, and
\begin{equation}\label{volfull:boolean-wsd}
 d_{n,k}\downarrow_k0\quad\Longrightarrow\quad
 \bigwedge_{\varphi\in\N^{\N}}\bigvee_nd_{n,\varphi(n)}=0.
\end{equation}
\end{lemma}
\begin{proof}
Set $d(H,K)=\tau(H\symdiff K)$. Positivity and finite additivity give
\[
 d(H,L)\le d(H,K)+d(K,L),\qquad
 -d(H,K)\le\tau H-\tau K\le d(H,K).
\]
Let $\mathbb D$ consist of $d$-Cauchy sequences in $\HH$, modulo
$d(H_n,K_n)\to_o0$. The triangle inequality proves equivalence.
Lemma~\ref{volfull:Gcomplete} defines
\[
 v([H_n])=o\text{-}\lim_n\tau H_n,\qquad
 \delta([H_n],[K_n])=o\text{-}\lim_nd(H_n,K_n).
\]
Both expressions are independent of representatives, since their
differences lie between the positive and negative of
$d(H_n,H'_n)+d(K_n,K'_n)$. In particular, the triangle inequality gives
\[
 -d(H,H')-d(K,K')\le d(H,K)-d(H',K')
 \le d(H,H')+d(K,K'),
\]
which proves that the defining distance sequence is order-Cauchy.
The estimates
\[
 d(H^c,K^c)=d(H,K),\quad
 d(H\cap L,K\cap M),\ d(H\cup L,K\cup M)
 \le d(H,K)+d(L,M)
\]
make coordinatewise Boolean operations well-defined. Constant sequences
define $\iota$, and
$\delta(x,y)=v(x\symdiff y)$ is a separated $G_+$-valued metric.
The source value inequality passes to limits and gives
\[
 -\delta(x,y)\le v(x)-v(y)\le\delta(x,y).
\]
Consequently $\delta$-convergence implies order convergence of the
corresponding $v$-values.
Passage to limits gives
$\delta(x\wedge z,y\wedge w)\le\delta(x,y)+\delta(z,w)$,
the corresponding union inequality, and
$\delta(x^c,y^c)=\delta(x,y)$. Hence, for independent convergent nets,
\[
 \delta(x_\alpha\wedge z_\beta,x\wedge z)
 \le\delta(x_\alpha,x)+\delta(z_\beta,z)\to_o0,
\]
and the union and complement inequalities give their corresponding
limits. Furthermore,
$v([H_n])=0\Longleftrightarrow\tau H_n\to_o0
\Longleftrightarrow[H_n]=[\varnothing]$.
Passing to the limit in
$\tau H_n+\tau K_n=\tau(H_n\cup K_n)+\tau(H_n\cap K_n)$
gives $v(x)+v(y)=v(x\vee y)+v(x\wedge y)$ and hence finite additivity.

We verify completeness, including the bounded approximation required
for a diagonal argument. For $x=[A_n]$, its Cauchy regulator has a
decreasing sequence $q_k\downarrow0$ extracted by order separability.
A bounded selector of
the repeated matrix $q_k$ supplies $c_j$ with bounded partial sums.
Put $s=\sup_m\sum_{j\le m}c_j$. For each $j$ there is $N_j$ such that
$d(A_n,A_m)\le c_j$ for $n,m\ge N_j$.
The recursion
$k_1\ge\max\{1,N_1\}$,
$k_{j+1}\ge\max\{k_j+1,j+1,N_{j+1}\}$ gives
$d(A_{k_j},A_{k_{j+1}})\le c_j$. Put $H_j=A_{k_j}$ and
\[
 e_N=\sup_{m\ge N}\tau\left(\bigcup_{j=N}^m
             (H_j\symdiff H_{j+1})\right).
\]
These are directed suprema bounded by $h$. The inequalities
\[
 0\le e_N\le h,\qquad e_N\le s-\sum_{j<N}c_j,\qquad
 d(H_N,H_m)\le e_N\quad(m>N)
\]
give $e_N\downarrow0$ and $\delta(\iota H_N,x)\le e_N$.
Here $[H_j]=[A_j]$, since $j,k_j\to\infty$ and $(A_n)$ is Cauchy.

For a $\delta$-Cauchy sequence $(x_n)$, apply this construction to obtain
$\delta(\iota H_{n,k},x_n)\le e_{n,k}\downarrow_k0$, $e_{n,k}\le h$.
Regularity supplies $\varphi$ for which
$s=\sup_m\sum_{n\le m}e_{n,\varphi(n)}$ exists.
Put $t_N=s-\sum_{n<N}e_{n,\varphi(n)}\downarrow0$. Then
\[
 d(H_{n,\varphi(n)},H_{m,\varphi(m)})
 \le2t_N+\delta(x_n,x_m)\quad(n,m\ge N).
\]
The diagonal sequence represents an element $x$ and converges to $x$:
a Cauchy bound $d(B_n,B_m)\le r_\gamma$ for $n,m\ge N_\gamma$
gives $\delta(\iota B_n,[B_m])\le r_\gamma$ after passage to the
limit in $m$, where $B_n=H_{n,\varphi(n)}$.
The triangle inequality
$\delta(x_n,x)\le e_{n,\varphi(n)}+\delta(\iota B_n,x)$
therefore gives $\delta(x_n,x)\to_o0$.
For a Cauchy net, countable extraction of its regulator provides
$q_j\downarrow0$ and increasing indices $i_j$ with
$\delta(x_i,x_{i'})\le q_j$ for $i,i'\ge i_j$.
Sequential completeness supplies $x$ with $\delta(x_{i_k},x)\to_o0$.
Let $r_\gamma\downarrow0$ regulate this last convergence, so that
for every $\gamma$ there is $K_\gamma$ with
$\delta(x_{i_k},x)\le r_\gamma$ whenever $k\ge K_\gamma$.
For $j\ge1$, $i\ge i_j$, and $k\ge\max\{j,K_\gamma\}$,
\[
 0\le\delta(x_i,x)\le\delta(x_i,x_{i_k})+
       \delta(x_{i_k},x)\le q_j+r_\gamma.
\]
Taking the infimum over $\gamma$ gives the exact tail estimate
\[
 \forall j\ge1\ \forall i\ge i_j:\qquad
 \delta(x_i,x)\le q_j,\qquad q_j\downarrow0.
\]
Thus every $\delta$-Cauchy net converges in $\mathbb D$.

For any nonempty $S\subseteq\mathbb D$, let
$z_F=\bigvee_{x\in F}x$ for finite nonempty $F\subseteq S$, and
$a=\sup_Fv(z_F)$. If $F,F'\supseteq F_0$, then
\[
 \delta(z_F,z_{F'})\le2(a-v(z_{F_0}))\downarrow_{F_0}0.
\]
Let $z$ be the limit. For $x\in S$, the identities
$z_F\wedge x=x$ hold whenever $x\in F$, hence $z\wedge x=x$.
If $y$ bounds $S$, then $z_F\wedge y=z_F$ for every $F$, hence
$z\wedge y=z$. Therefore $z=\bigvee S$ and $v(z)=a$.
Order separability supplies finite $F'_n$ with
$\sup_nv(z_{F'_n})=a$. Put $F_n=\bigcup_{j\le n}F'_j$ and
$w=\bigvee_nz_{F_n}$. Then
\[
 z_{F_n}\le w\le z,\qquad
 a=\sup_nv(z_{F_n})\le v(w)\le v(z)=a,
 \qquad v(z\setminus w)=0.
\]
Faithfulness gives $w=z$. This proves the countable sup
property. For $z_\alpha\uparrow z$ and
$a=\sup_\alpha v(z_\alpha)$,
$\delta(z_\beta,z_\gamma)\le2(a-v(z_\alpha))$ when
$\beta,\gamma\ge\alpha$. Let $y$ be its metric limit. At each fixed
$\alpha$, the identity $z_\beta\wedge z_\alpha=z_\alpha$ holds for
$\beta\ge\alpha$. The meet estimate gives
\[
 \delta(y\wedge z_\alpha,z_\alpha)
 =\delta(y\wedge z_\alpha,z_\beta\wedge z_\alpha)
 \le\delta(y,z_\beta)\to_o0,
\]
so $z_\alpha\le y$. For every upper bound $w$ of all $z_\alpha$,
\[
 \delta(y\wedge w,y)
 \le\delta(y\wedge w,z_\beta\wedge w)+\delta(z_\beta,y)
 \le2\delta(y,z_\beta)\to_o0,
\]
and therefore $y\le w$. Hence $y=z$. The bounds
$-\delta(z_\alpha,z)\le v(z_\alpha)-v(z)\le\delta(z_\alpha,z)$
identify the order limit of $(v(z_\alpha))$ with $v(z)$, while
$a-v(z_\alpha)\downarrow0$ identifies that limit with $a$.
Separation gives $v(z)=a$. For $z_\alpha\downarrow z$,
$v(z)=h-\sup_\alpha v(z_\alpha^c)=\inf_\alpha v(z_\alpha)$.

Finally $a_{n,k}=v(d_{n,k})\downarrow_k0$ is bounded by $h$.
For $q=\bigwedge_\psi\bigvee_nd_{n,\psi(n)}$ and
$\varphi\in\Phi(a)$, finite subadditivity and increasing continuity give
\[
 0\le v(q)\le v\left(\bigvee_nd_{n,\varphi(n)}\right)
 \le\sup_m\sum_{n\le m}a_{n,\varphi(n)}=s_\varphi\downarrow_\varphi0.
\]
Faithfulness yields $q=0$.
\end{proof}

\begin{lemma}[Additive transport]\label{volfull:step-properties}
\label{volfull:step-axioms}
The integer-step group $K=\mathbb Z[\mathbb D]$ is a Dedekind complete,
weakly sigma-distributive lattice-ordered group with the countable sup
property. The homomorphism
\[
 J\left(\sum_{i=1}^rn_i\mathbf1_{d_i}\right)
       =\sum_{i=1}^rn_iv(d_i)
\]
is positive, faithful on $K_+$ and order-continuous. Here $(d_i)$ is a
finite Boolean partition of $1$ and $n_i\in\mathbb Z$.
The map $\lambda(H)=\mathbf1_{\iota H}$ satisfies $J\lambda=\tau$
and every hypothesis of
Theorem~\ref{gen:extension} with sequential rows,
$M_K=\clo_o\lambda(\HH)$ and $D_K=[-2\mathbf1,2\mathbf1]$.
\end{lemma}
\begin{proof}
If $(d_i)$ and $(e_j)$ are finite Boolean partitions, then
\[
 \sum_i n_iv(d_i)=\sum_{i,j}n_iv(d_i\wedge e_j),\qquad
 \sum_j m_jv(e_j)=\sum_{i,j}m_jv(d_i\wedge e_j).
\]
On the common partition $(d_i\wedge e_j)_{i,j}$, addition adds integer
coefficients and lattice operations take their maxima and minima.
These identities give independence of representation,
$J(f+g)=Jf+Jg$, and $Jf\ge0$ for $f\ge0$.
For each $H\in\HH$, its definition gives
$J\lambda(H)=v(\iota H)=\tau H$.
If $f>0$, some nonzero $d$ has $\mathbf1_d\le f$, hence
$0<v(d)\le Jf$. For a nonempty upper-bounded family, fix one member
$f_0$ and replace each $f$ by $f\vee f_0$. The resulting family lies
between $a\mathbf1$ and $b\mathbf1$ for integers $a\le b$. Its supremum is
\[
 a\mathbf1+\sum_{k=a+1}^b\mathbf1_{e_k},\qquad
 e_k=\bigvee_f\{f\ge k\}.
\]
Indeed every member is bounded by this element, and every upper bound
$u$ satisfies $e_k\le\{u\ge k\}$. For each
$k\in\{a+1,\ldots,b\}$, the countable sup property of $\mathbb D$
supplies members $f_{k,j}$ of the original family such that
\[
 e_k=\bigvee_{j\ge1}\{f_{k,j}\vee f_0\ge k\}.
\]
The countable family consisting of $f_0$ and all $f_{k,j}$ has the
displayed supremum in $K$. If $a=b$, the single member $f_0$
has that supremum. This proves the countable sup property of $K$.
For $0\le f_{n,k}\le N\mathbf1$, $f_{n,k}\downarrow_k0$, their supports
$e_{n,k}$ decrease to zero: a common nonzero subprojection would give a
nonzero lower bound $\mathbf1_e\le f_{n,k}$. Thus
\[
 \sup_nf_{n,\varphi(n)}\le
 N\mathbf1_{\bigvee_ne_{n,\varphi(n)}},
 \qquad\inf_\varphi\sup_nf_{n,\varphi(n)}=0
\]
by \eqref{volfull:boolean-wsd}. For $r_\gamma\downarrow0$ in $K_+$,
one tail is bounded by $N\mathbf1$, so
$0\le Jr_\gamma\le N v(\operatorname{supp}r_\gamma)\downarrow0$.
Here the supports decrease to zero, since
$\mathbf1_{\bigwedge_\gamma\operatorname{supp}r_\gamma}
\le\inf_\gamma r_\gamma=0$.
For a general order-convergent net the inequalities
$-r_\gamma\le f_\alpha-f\le r_\gamma$ eventually imply
$-Jr_\gamma\le Jf_\alpha-Jf\le Jr_\gamma$ eventually. This proves
order continuity of $J$ on all net domains.

Corollary~\ref{ord:carriers} gives the reflection, summation and
uniform-closure principles on $D_K$.
The interval $[0,\mathbf1]$ is stable and contains $\lambda(\HH)$.
Thus $M_K\subseteq[0,\mathbf1]$, and Theorem~\ref{ord:complete}
makes this stable subset complete. If
$H=\dot\bigcup_nH_n\in\HH$, then
\[
 d=\bigvee_n\iota H_n\le\iota H,\qquad
 v(d)=\sup_N\sum_{n\le N}\tau H_n=\tau H,
 \qquad v(\iota H\setminus d)=0.
\]
Faithfulness gives $d=\iota H$, hence
$\sum_{n\le N}\lambda(H_n)\uparrow\lambda(H)$.
For arbitrary disjoint $H_n\in\HH$, put $d_N=\bigvee_{n\ge N}\iota H_n$
and $e=\bigwedge_Nd_N$. For every $j$,
\[
 e\wedge\iota H_j\le d_{j+1}\wedge\iota H_j=0,
 \qquad e=e\wedge d_1=\bigvee_j(e\wedge\iota H_j)=0.
\]
For $N\ge1$, $n\ge N$ and $H\in\HH$ with $H\subseteq H_n$,
$0\le\lambda(H)\le\mathbf1_{d_N}\downarrow_N0$.
This is the required hereditary exhaustivity.
Thus the data verified before extension are
\[
 M_K\subseteq[0,\mathbf1]\text{ complete},\qquad
 M_K-M_K\subseteq D_K,
\]
\[
 (\mathrm{HE})_\lambda,\qquad
 \mathsf B_\omega(D_K),\qquad
 \mathsf S_{\rm seq}(D_K),\qquad \mathsf U(M_K).
\]
The finite-response and unrestricted-row summation properties also hold
on $D_K$ by Corollary~\ref{ord:carriers}. The sequential properties
displayed above are the hypotheses used in
Theorem~\ref{gen:extension}.
\end{proof}

With this representation available, we can compare the extension in
$K$ with the double process in the original range. For
$A\subseteq\Omega$, put
\begin{equation}\label{volfull:projections}
 r_A=\bigvee_{H\in\HH,\ H\subseteq A}\iota H,
 \qquad p_A=\bigwedge_{U\in\Hs,\ A\subseteq U}r_U.
\end{equation}

\begin{theorem}[Partially ordered extension from the intrinsic theorem]
\label{vol:extension}
Let
\[
 \bar\lambda:\mathcal R_{\rm seq}(\lambda,D_K)\longrightarrow M_K
\]
be the extension supplied by Theorem~\ref{gen:extension} to the data
of Lemma~\ref{volfull:step-properties}. Then
$\bar\tau=J\bar\lambda$ is a positive order-countably additive
extension of $\tau$ with full upper and lower regularity.
The all-subset maps
$g(A)=v(r_A)$ and $\mu(A)=v(p_A)$ satisfy, in the original order
convergence of $G$,
\begin{align}
 g(A)&=\sup_{H\in\HH,\ H\subseteq A}\tau H
       =o\text{-}\lim_{H\in\HH^-(A)}\tau H,
       \label{volfull:inner}\\
 \mu(A)&=\inf_{U\in\Hs,\ A\subseteq U}g(U)
       =o\text{-}\lim_{U\in\HH^+(A)}g(U).
       \label{volfull:outer}
\end{align}
The two domains carry inclusion and reverse inclusion, respectively.
These are the unique double-process
values, lie in $\clo_o\tau(\HH)\subseteq[0,h]$, and satisfy
$g(A)\le\mu(A)$, $\mu|_{\Hs}=g|_{\Hs}$ and $\mu|_{\HH}=\tau$.
Moreover,
\[
 \mu(A\cup B)+\mu(A\cap B)\le\mu(A)+\mu(B),\qquad
 A_n\uparrow A\Longrightarrow\mu(A_n)\uparrow\mu(A).
\]
The following classes coincide:
\begin{equation}\label{volfull:class}
 \mathcal M_\mu
 =\{A:\mu(A)+\mu(A^c)=h\}
 =\{A:p_A\cap p_{A^c}=0\}
 =\mathcal R_{\mathrm{seq}}(\lambda,D_K).
\end{equation}
They form a sigma-algebra containing $\sigma(\HH)$, and $\mu$ restricted
to it is exactly the transported extension:
\begin{equation}\label{volfull:transported-extension}
 \bar\lambda(A)=\mathbf1_{p_A},\qquad
 \mu(A)=J\bar\lambda(A)=\bar\tau(A)
 \quad(A\in\mathcal M_\mu).
\end{equation}
This measure is complete. On $\sigma(\HH)$ it is unique even
among order-countably additive extensions without positivity.
\end{theorem}
\begin{proof}
Apply Theorem~\ref{gen:extension} with ambient Boolean algebra
$\mathcal P(\Omega)$, range $K$ with order convergence, source map
$\lambda$, stable range $M_K$, difference set $D_K$, and sequential
row domains. Lemma~\ref{volfull:step-properties} verifies each of its
hypotheses. Write $\mathcal R=\mathcal R_{\rm seq}(\lambda,D_K)$.
The theorem gives sigma-closure of $\mathcal R$, the inclusion
$\sigma(\HH)\subseteq\mathcal R$, countable additivity of $\bar\lambda$
and both full approximation limits. The first assertion of
Proposition~\ref{gen:transport}, with $T=J$ and $S=K$, gives these
additivity and regularity conclusions for $\bar\tau=J\bar\lambda$.
Since $J\lambda=\tau$ and $M_K\subseteq[0,\mathbf1]$, this is a
positive extension of $\tau$. We now identify its domain and values
with the original double process.

Let $H_n\uparrow U\in\Hs$ and $H\in\HH$, $H\subseteq U$.
By \eqref{volfull:original-countadd},
\[
 v\left(\bigvee_n\iota(H\cap H_n)\right)
 =\sup_n\tau(H\cap H_n)=\tau H=v(\iota H).
\]
The join on the left is at most $\iota H$. Faithfulness gives
$\iota H=\bigvee_n\iota(H\cap H_n)\le\bigvee_n\iota H_n$.
Taking the join over $H$ proves $r_U=\bigvee_n\iota H_n$.
For $H_n\uparrow U$, $K_n\uparrow V$, and
$H_{j,n}\uparrow U_j$ with all terms in $\HH$, Boolean distributivity gives
\[
 r_{U\cap V}=\bigvee_n\iota(H_n\cap K_n)=r_U\wedge r_V,
 \qquad
 r_{\bigcup_jU_j}=\bigvee_n\iota\left(\bigcup_{j\le n}H_{j,n}\right)
 =\bigvee_jr_{U_j}.
\]
Consequently
\[
 r_A\le p_A,\quad p_U=r_U\ (U\in\Hs),\quad
 p_{A\cup B}=p_A\vee p_B,\quad p_A\vee p_{A^c}=1.
\]
For $U\supseteq A$, $V\supseteq B$ in $\Hs$,
\[
 p_A\vee p_B\le p_{A\cup B}\le r_{U\cup V}=r_U\vee r_V,
 \qquad \bigwedge_{U,V}(r_U\vee r_V)=p_A\vee p_B.
\]
Here a fixed finite Boolean operation distributes over arbitrary extrema.
The core family is upward directed, and the cover family is downward
directed. Preservation of their extrema by $v$ gives
\eqref{volfull:inner}--\eqref{volfull:outer} and the regulators
$g(A)-\tau H\downarrow0$, $g(U)-\mu(A)\downarrow0$.
Since $\tau H\in\clo_o\tau(\HH)$, stability first puts every $g(A)$,
then every $\mu(A)$, in that hull. Separation gives uniqueness of both limits.

For $T,A\subseteq\Omega$, finite-union preservation gives the exact defect
\begin{equation}\label{volfull:splitting-defect}
 \mu(T\cap A)+\mu(T\setminus A)-\mu(T)
 =v(p_{T\cap A}\cap p_{T\setminus A}).
\end{equation}
If $p_A\cap p_{A^c}=0$, every defect vanishes. Conversely the test
$T=\Omega$ and faithfulness give the first three class equalities.
For such an $A$, the countable sup property applied to complements
supplies covers $U'_n\supseteq A$, $V'_n\supseteq A^c$ with
$\bigwedge_nr_{U'_n}=p_A$ and $\bigwedge_nr_{V'_n}=p_{A^c}$.
Put $U_n=\bigcap_{j\le n}U'_j$, $V_n=\bigcap_{j\le n}V'_j$,
$L_n=V_n^c$, and $e_n=r_{U_n}\cap r_{V_n}\downarrow0$. Then
\[
 L_n\subseteq A\subseteq U_n,\qquad
 \Gamma_\lambda(U_n\setminus L_n)\subseteq[0,\mathbf1_{e_n}].
\]
The rows $W_{1,n}=[0,\mathbf1_{e_n}]$ and $W_{i,n}=\{0\}$ for $i>1$
are decreasing and order null. Their selected finite-sum union and
its relative stable hull both equal $[0,\mathbf1_{e_{\varphi(1)}}]$.
In fact, if $H\in\HH$ and
$H\subseteq U_n\setminus L_n=U_n\cap V_n$, then
\[
 \iota H\le r_{U_n}\wedge r_{V_n}=e_n,\qquad
 0\le\lambda(H)\le\mathbf1_{e_n}.
\]
The order-stable interval $[0,\mathbf1_{e_n}]$ therefore contains
$\Gamma_\lambda(U_n\setminus L_n)$. For every selector $\varphi$,
the corresponding certificate inclusions are
\[
 L_{\varphi(1)}\subseteq A\subseteq U_{\varphi(1)},\qquad
 \Gamma_\lambda(U_{\varphi(1)}\setminus L_{\varphi(1)})
 \subseteq[0,\mathbf1_{e_{\varphi(1)}}]
 =D_K\cap\clo_o\!\left(\bigcup_{F\text{ finite}}
                \sum_{i\in F}W_{i,\varphi(i)}\right).
\]
They constitute a sequential regulated certificate.
Conversely, any regulated certificate $(L_\varphi,U_\varphi,Q_\varphi)$
satisfies
\begin{gather*}
 p_A\cap p_{A^c}\le r_{U_\varphi}\cap r_{L_\varphi^c}
 =r_{U_\varphi\setminus L_\varphi},\\
 \mathbf1_{r_{U_\varphi\setminus L_\varphi}}
 \in\Gamma_\lambda(U_\varphi\setminus L_\varphi)\subseteq Q_\varphi.
\end{gather*}
For an exhaustion $H_n\uparrow U_\varphi\setminus L_\varphi$,
$\lambda(H_n)\uparrow\mathbf1_{r_{U_\varphi\setminus L_\varphi}}$,
which proves the displayed membership. Put $e=p_A\cap p_{A^c}$.
Since the evaluation of $Q_\varphi$ is order null, its selected net
$\mathbf1_{r_{U_\varphi\setminus L_\varphi}}$ is order null.
The inequalities
$0\le\mathbf1_e\le\mathbf1_{r_{U_\varphi\setminus L_\varphi}}$
then give $\mathbf1_e=0$. This proves the final class equality.

The class just identified is therefore the domain $\mathcal R$ on which
the general theorem has already constructed $\bar\lambda$.
For $U\in\Hs$, source exhaustion gives
$\lambda_\sigma(U)=\mathbf1_{r_U}$. For every $A\in\mathcal R$,
the full upper net decreases to $\mathbf1_{p_A}$, whereas the
regularity conclusion of Theorem~\ref{gen:extension} gives its order
limit $\bar\lambda(A)$. Separation identifies these values and proves
\eqref{volfull:transported-extension}. Consequently
\[
 g(U)\to_o\mu(A),\qquad h-g(L^c)\to_o\mu(A)
\]
on the full upper and lower domains. These are the transported
approximation limits of the general theorem.
If $\mu(N)=0$ and $A\subseteq N$, faithfulness gives $p_N=0$, hence
$p_A=0$ and $p_A\cap p_{A^c}=0$. Hence
$A\in\mathcal M_\mu$ and $\mu(A)=0$, proving completeness.

Finite additivity of $v$ and $p_{A\cap B}\le p_A\cap p_B$ give
\[
 \mu(A\cup B)+\mu(A\cap B)
 \le v(p_A\vee p_B)+v(p_A\cap p_B)=\mu(A)+\mu(B).
\] For $A_n\uparrow A$, let
$s=\bigvee_np_{A_n}$. Decreasing covers $U_{n,k}\supseteq A_n$ can be
indexed so that $r_{U_{n,k}}\downarrow_kp_{A_n}$. Put
$e_{n,k}=r_{U_{n,k}}\setminus p_{A_n}$ and
$e_\varphi=\bigvee_ne_{n,\varphi(n)}\downarrow_\varphi0$.
The cover $\bigcup_nU_{n,\varphi(n)}$ gives
$p_A\le s\vee e_\varphi$ for every $\varphi$. Therefore
\[
 s\le p_A\le\bigwedge_\varphi(s\vee e_\varphi)
 =s\vee\bigwedge_\varphi e_\varphi=s,\qquad
 \mu(A_n)=v(p_{A_n})\uparrow v(s)=\mu(A).
\]
For a competing order-countably additive extension $\eta$ with values
in $G$, set $\mathcal D=\{A\in\sigma(\HH):\eta(A)=\mu(A)\}$.
The equality $\eta(\Omega)=\tau\Omega=\mu(\Omega)$ and finite
additivity give, for $A\in\mathcal D$,
\[
 \eta(A^c)=\eta(\Omega)-\eta(A)
 =\mu(\Omega)-\mu(A)=\mu(A^c).
\]
For pairwise disjoint $A_j\in\mathcal D$,
\[
 \eta\!\left(\bigcup_jA_j\right)
 =o\text{-}\lim_N\sum_{j\le N}\eta(A_j)
 =o\text{-}\lim_N\sum_{j\le N}\mu(A_j)
 =\mu\!\left(\bigcup_jA_j\right).
\]
Thus $\mathcal D$ is a Dynkin system containing the intersection-closed
algebra $\HH$. The Dynkin-system argument proved in
Proposition~\ref{gen:uniqueness}, applied to this $\mathcal D$ and
$\mathcal P=\HH$, gives $\sigma(\HH)\subseteq\mathcal D$.
This establishes uniqueness in the original range without a lift of
$\eta$ and without a positivity assumption on $\eta$.
\end{proof}

Having identified the transported extension and its full domain, we
turn to the estimates for the original process. We also consider
lattice ranges, for which the order criteria permit a direct
application to $\tau$.

\begin{corollary}[Process identification and direct lattice specializations]
\label{volfull:compatibility}\label{volfull:fragmenthulls}
\label{vol:lattice-classical}\label{class:positive}\label{class:volauf}
Under the partially ordered hypotheses of Theorem~\ref{vol:extension},
the restriction of $J\bar\lambda$ to $\sigma(\HH)$ is the unique
extension of $\tau$ in \cite[Theorem~3]{Volauf80}. On the full
splitting domain $\mathcal M_\mu$, Theorem~\ref{vol:extension}
gives $J\bar\lambda=\mu$. For $\KK=\alg(\Hs)$,
the Boolean approximation of the original double process is
$\nu=\mu|_{\KK}=(J\bar\lambda)|_{\KK}$. Its core and cover errors
and outer residuals satisfy
\begin{gather*}
 H,K\supseteq H_0\Longrightarrow
 -(g(A)-\tau H_0)\le\tau H-\tau K\le g(A)-\tau H_0,\\
 U,V\subseteq U_0\Longrightarrow
 -(g(U_0)-\mu(A))\le g(U)-g(V)\le g(U_0)-\mu(A),\\
 R_A(U)=\bigcup_{A\subseteq V\subseteq U,\ V\in\Hs}
       \{\nu F:F\in\KK,\ F\subseteq U\setminus V\},\qquad
 \clo_oR_A(U)\subseteq[0,g(U)-\mu(A)].
\end{gather*}
Here $H,K,H_0$ are algebra cores of $A$, and $U,V,U_0$ are countable
algebra covers. Both error bounds decrease to zero. For
$A\subseteq T\subseteq U$ one has
$g(U)-\mu(T)\in\clo_oR_A(U)$.
For the original fragment values put
$V_\tau(S)=\{\tau H:H\in\HH,\ H\subseteq S\}$,
$\Gamma(S)=\clo_oV_\tau(S)$ and
$\Delta(S)=\clo_o(V_\tau(S)-V_\tau(S))$. Then
\[
 g(A)\in\Gamma(A),\quad
 \mu(A)\in\bigcap_{S\supseteq A,\ S\in\Hs}\Gamma(S),\quad
 \mu(A)-\mu(B)\in\bigcap_{S\supseteq A\symdiff B,\ S\in\Hs}\Delta(S).
\]
For the lattice specializations, let $E$ be a Dedekind-complete abelian
lattice-ordered group which is WSD and has the countable sup property.
For a positive source premeasure $\tau:\HH\to E_+$ on an algebra,
put $h=\tau\Omega$, $M_E=\clo_o\tau(\HH)$ and $D_E=[-2h,2h]$.
Theorem~\ref{gen:extension} applies directly to $\tau$, and its
all-subset refinement, Theorem~\ref{thm:main}, gives
\[
 g(A)=\sup_{H\in\HH,\ H\subseteq A}\tau H,\qquad
 \mu(A)=\inf_{U\in\Hs,\ A\subseteq U}g(U),
\]
\[
 \mathcal M_\mu
 =\{A:\mu(A)+\mu(A^c)=h\}
 =\mathcal R_{\rm all}(\tau,D_E),\qquad
 \mu|_{\mathcal M_\mu}=\bar\tau.
\]
The map $\mu$ is monotone and submodular, satisfies
$\mu(A_n)\uparrow\mu(A)$ for $A_n\uparrow A$, and restricts to a
complete positive order-countably additive measure with full upper
and lower regularity. These assertions use the stated lattice
hypotheses on $E$ and require no regularity in the sense of
Definition~\ref{class:volauf-regular}.

More generally, let $\HH$ be a ring with $\Omega\in\Hs$ and
$\tau:\HH\to E$ an order-countably additive charge satisfying
$a\le\tau H\le b$ for all $H\in\HH$. Put $c=b-a$,
$M_E=\clo_o\tau(\HH)$ and $D_E=[-2c,2c]$.
Theorems~\ref{gen:extension} and~\ref{thm:main} again apply directly
to $\tau$. Its two original approximation nets define $g$ and $\mu$
on every subset, and
\[
 \mathcal M_\mu=\mathcal R_{\rm all}(\tau,D_E),\qquad
 \mu|_{\mathcal M_\mu}=\bar\tau,\qquad
 A_n\uparrow A\Longrightarrow\mu(A_n)\to_o\mu(A).
\]
The extension is order-countably additive and has both full regularity
nets, with lower values $g(\Omega)-g(L^c)$. In both lattice cases it
is unique among all order-countably additive extensions on
$\sigma(\HH)$ and among fully lower-regular extensions on the full
regulated domain.
\end{corollary}
\begin{proof}
For $E\in\KK$, express its indicator as a finite Boolean function of
$1_{B_i}$, $B_i\in\Hs$, and replace increasing exhaustions
$H_{i,n}\uparrow B_i$ in that function to obtain $P_n\in\HH$.
Then $E\symdiff P_n\subseteq C_n=\bigcup_i(B_i\setminus H_{i,n})$ and
\[
 0\le\mu(C_n)\le\sum_i(g(B_i)-\tau H_{i,n})\downarrow0,
 \qquad-\mu(C_n)\le\mu(E)-\tau P_n\le\mu(C_n).
\]
Hence $\tau P_n\to_o\mu(E)$. More generally, if $C_n\in\Hs$ and
$(V_\tau(C_n))_n\Rightarrow_o0$ with regulator $r_\gamma\downarrow0$, then
\[
 \forall\gamma\ \exists N_\gamma\ \forall n\ge N_\gamma:\quad
 0\le\mu(C_n)=g(C_n)=\sup_{H\in\HH,\,H\subseteq C_n}\tau H
 \le r_\gamma.
\]
The same bound on $\mu(E)-\tau P_n$ identifies every such approximation.
For the two Cauchy estimates, the bounding intervals are
\[
 \tau H_0\le\tau H,\tau K\le g(A),\qquad
 \mu(A)\le g(U),g(V)\le g(U_0).
\]
Subtracting either pair gives the displayed inequalities.
For $F\in\KK$, $F\subseteq U\setminus V$,
\[
 0\le\nu F\le\nu(U\setminus V)=g(U)-g(V)\le g(U)-\mu(A).
\]
Stability of the interval yields the bound on $\clo_oR_A(U)$.
The map $V\mapsto V\cap U$ makes the covers of $T$ contained in $U$
cofinal among all its covers. Hence
$\nu(U\setminus V)=g(U)-g(V)\to_o g(U)-\mu(T)$, proving membership
in the residual hull.
The core net lies in $V_\tau(A)$, so its limit $g(A)$ belongs to
$\Gamma(A)$. For $A\subseteq S\in\Hs$, covers contained in $S$ are
cofinal and satisfy $g(V)\in\Gamma(V)\subseteq\Gamma(S)$.
Thus $\mu(A)\in\Gamma(S)$.
If $A\symdiff B\subseteq S\in\Hs$, then $A\setminus S=B\setminus S$
and $S$ is splitting. Therefore
\[
 \mu(A)-\mu(B)=\mu(A\cap S)-\mu(B\cap S)
 \in\Gamma(S)-\Gamma(S)\subseteq\Delta(S),
\]
where the last inclusion is \eqref{eq:hull-addition}.

For the lattice-valued positive case and disjoint $H_n$, put
$s_m=\sum_{n\le m}\tau H_n=\tau(\bigcup_{n\le m}H_n)\le h$.
Dedekind completeness supplies $s=\sup_ms_m$. Then
\[
 t_N=s-\sum_{n<N}\tau H_n\downarrow0,\qquad
 0\le\tau H\le\tau H_n\le t_N
 \quad(H\in\HH,\ H\subseteq H_n,\ n\ge N).
\]
These bounds verify hereditary exhaustivity in $E$.
Order stability gives $M_E\subseteq[0,h]$ and
$M_E-M_E\subseteq[-h,h]\subseteq D_E$.
Theorem~\ref{ord:complete}, applied to the order-stable subset
$M_E\subseteq E$, makes $M_E$ complete. Corollary~\ref{ord:carriers}
with interval parameter $2h$ gives finite-response reflection and
unrestricted-row closed summation on $D_E=[-2h,2h]$.
For the uniform-closure principle, every one of its four point
families takes values in $M_E\subseteq[0,h]$.
Proposition~\ref{ord:uc}, with endpoints $a=0$ and $b=h$ and
with the original additional index set, therefore applies to those
families. The resulting principles are
\[
 \mathsf F(D_E),\qquad\mathsf S_{\rm all}(D_E),\qquad
 \mathsf U(M_E)
\]
Proposition~\ref{gen:reflection-block}, applied to the verified
$\mathsf F(D_E)$, gives $\mathsf B_\omega(D_E)$. Thus
Theorem~\ref{gen:extension}, with source algebra $\HH$, ambient
Boolean algebra $\mathcal P(\Omega)$, range $E$ with order
convergence, source map $\tau$, stable range $M_E$, difference set
$D_E$, and arbitrary directed row domains, constructs $\bar\tau$
on $\mathcal R_{\rm all}(\tau,D_E)$.
Theorem~\ref{thm:main} identifies this extension with the double
process and its full splitting class, and supplies its increasing-set
continuity and uniqueness assertions.

The positive inner and outer nets are monotone, so their limits equal
the displayed directed extrema. For $H_n\uparrow U$, $K_n\uparrow V$
with $H_n,K_n\in\HH$, finite additivity gives
\[
 \tau(H_n\cup K_n)+\tau(H_n\cap K_n)=\tau H_n+\tau K_n.
\]
Taking the four order limits gives
$g(U\cup V)+g(U\cap V)=g(U)+g(V)$.
If $A\subseteq U$ and $B\subseteq V$, then
\[
 \mu(A\cup B)+\mu(A\cap B)
 \le g(U\cup V)+g(U\cap V)=g(U)+g(V).
\]
Taking the two independent directed infima proves submodularity.
Suppose $\mu(A)+\mu(A^c)=h$. Submodularity and
$A\cup(T\cup A^c)=\Omega$ give
\[
 \mu(T\cap A)+h\le\mu(A)+\mu(T\cup A^c),\qquad
 \mu(T\cup A^c)+\mu(T\setminus A)\le\mu(T)+\mu(A^c).
\]
Adding and cancelling $h$ yields
$\mu(T\cap A)+\mu(T\setminus A)\le\mu(T)$.
Subadditivity gives $\mu(T)\le\mu(T\cap A)+\mu(T\setminus A)$. A splitting set satisfies the complement
equality by taking $T=\Omega$. This proves the positive class
description. If $\mu(N)=0$ and $A\subseteq N$, monotonicity gives
$\mu(T\cap A)=0$ for every $T$, and
\[
 \mu(T)\le\mu(T\cap A)+\mu(T\setminus A)
 =\mu(T\setminus A)\le\mu(T).
\]
Thus every subset of a null set is splitting, proving completeness.
Monotonicity together with the increasing-set convergence from
Theorem~\ref{thm:main} gives the asserted increasing supremum identity.

For a signed charge with $a\le\tau H\le b$, put
$u(H)=\sup_{B\subseteq H}\tau B$ and
$w(H)=\sup_{B\subseteq H}(-\tau B)$, with all fragments in $\HH$.
These suprema exist by Dedekind completeness of $E$ and the source
bounds. Since $\tau\varnothing=0$, one has $a\le0\le b$.
Complements inside $H$ give
$w(H)=u(H)-\tau H$. For disjoint $H,K$, the fragment identity
\[
 \{\tau B:B\in\HH,\ B\subseteq H\cup K\}
 =\{\tau C+\tau D:C,D\in\HH,\ C\subseteq H,\ D\subseteq K\}
\]
and preservation of suprema by translations give
$u(H\cup K)=u(H)+u(K)$. Subtracting
$\tau(H\cup K)=\tau H+\tau K$ gives
$w(H\cup K)=w(H)+w(K)$.
Moreover $0\le u(H)\le b$ and $0\le w(H)\le-a$.
For disjoint $H_n$,
\[
 \sum_{n\le m}u(H_n)=u\left(\bigcup_{n\le m}H_n\right)\le b,
 \qquad
 \sum_{n\le m}w(H_n)=w\left(\bigcup_{n\le m}H_n\right)\le-a.
\]
Their increasing suprema $s^+,s^-$ therefore exist. Hence
\[
 |\tau B|\le u(B)+w(B)
 \le s^++s^--\sum_{n<N}(u(H_n)+w(H_n))\downarrow_N0
 \quad(B\subseteq H_n, n\ge N).
\]
This verifies hereditary exhaustivity from the original source values.
Only positivity, finite additivity and the displayed bounds on $u,w$
enter this estimate. The source premeasure condition is the assumed
order-countable additivity of $\tau$.
The inequalities $a\le0\le b$ give
$M_E\subseteq[a,b]\subseteq[-c,c]$ and
$M_E-M_E\subseteq[-c,c]\subseteq D_E$.
Theorem~\ref{ord:complete} applies to the order-stable range $M_E$
and proves its completeness. Corollary~\ref{ord:carriers}, with
interval parameter $2c$, gives $\mathsf F(D_E)$ and
$\mathsf S_{\rm all}(D_E)$. Proposition~\ref{gen:reflection-block}
then gives $\mathsf B_\omega(D_E)$. For every quadruple of point
families in $M_E$ from the definition of $\mathsf U$, all values
lie in $[a,b]$, so Proposition~\ref{ord:uc} with these endpoints
gives $\mathsf U(M_E)$. Theorem~\ref{gen:extension}, with source
ring $\HH$, ambient algebra $\mathcal P(\Omega)$, source map
$\tau$, original order convergence on $E$, stable range $M_E$,
difference set $D_E$ and arbitrary directed row domains, constructs
the extension on $\mathcal R_{\rm all}(\tau,D_E)$.
Theorem~\ref{thm:main} uses these data together with the verified
$\mathsf F(D_E)$. Its inner and outer nets are exactly the two
nets displayed in this corollary, and its domain identity is
$\mathcal M_\mu=\mathcal R_{\rm all}(\tau,D_E)$.
Its increasing-set continuity, two full regularity limits, and
uniqueness clauses give the corresponding stated assertions.
\end{proof}

We can now compare the lattice hypotheses with those in Volauf's
earlier extension theorem.

In a monotone-complete lattice group, finite joins convert every
upper-bounded family into a directed family, so the group is Dedekind
complete. Order separability gives the countable sup property, and
regularity implies weak sigma-distributivity through
$\inf_\psi\sup_na_{n,\psi(n)}\le s_\varphi$ for
$\varphi\in\Phi(a)$. Thus the extension conclusions of
\cite[Theorems~2--4]{Volauf77} are included as well.
The stronger regularity condition is not implicit in the lattice
specialization: in $\ell^\infty(\N)$ the matrix
$a_{n,k}(j)=\mathbf1_{\{j\ge k\}}$ has decreasing null rows but no
bounded selector, since each selected $m$-term sum attains $m$.
For nonempty upper-bounded $S\subseteq\ell^\infty$, its coordinatewise
supremum $s$ lies between any fixed member and an upper bound. Elements
$f_{j,k}\in S$ with $f_{j,k}(j)>s(j)-1/k$ satisfy
$s=\sup_{j,k}f_{j,k}$. This proves Dedekind completeness and the countable
sup property. For a bounded regulator $(b_{n,k})$, the selectors
$\varphi_{j,\varepsilon}(n)=\min\{k:b_{n,k}(j)\le\varepsilon\}$ give
$0\le(\inf_\varphi\sup_nb_{n,\varphi(n)})(j)\le\varepsilon$ for every
$j$ and $\varepsilon>0$, proving weak sigma-distributivity. The next section
uses lattice-valued bounds to treat nonpositive ranges without replacing
semivariation by variation.

\section{Boolean vector variation and canonical semivariation}
All event domains in this section are Boolean. The partition and common-refinement identities proved here are used only in that setting. Corollary~\ref{app:lattice-normed} gives the separate pseudo-$D$-lattice result for a positive additive majorant already given on the source.

\label{sec:km}\label{sec:semiv}\label{sec:lattice-short}

Finite vector variation provides a positive additive majorant for the values of a premeasure. Kusraev and Malyugin use it to obtain an extension on a regulator-defined domain, preserving variation and full lower regularity \cite[Sections~1.4.2 and~1.4.11--1.4.12]{KM}. We recover their entire domain by representing the source premeasure in a space where the intrinsic extension theorem applies. This representation also permits vector norms that are not decomposable.

Finite semivariation requires a different representation. It bounds the norms of signed sums and can remain finite when vector variation is unbounded, as the example at the end of the section shows. Since semivariation is generally only subadditive, we cannot use it as an additive coordinate. We instead consider dominated linear maps $T$. Each coordinate charge $Tm$ has a positive variation, and we show that the supremum of these variations is the original semivariation. Adjoining all the variations to the represented source values gives an additive lift $\lambda$ whose vector norm is the semivariation.

We construct the representation with $P(Jy)=p(y)$, so its error bounds retain the original convergence. Completeness makes $JY$ bo-closed and keeps the extension of the lift in the domain of the continuous inverse of $J$. After applying Theorem~\ref{gen:extension}, the projection $L$ therefore returns both the extension and its full approximation nets to the original vector-normed space. The representation and the bounds needed for this application are constructed on the source ring.

\begin{definition}\label{km:def-lns}\label{sem:def-structure}
Let $F$ be a real Dedekind complete vector lattice and $Y$ a real vector space. A vector norm $p:Y\to F_+$ satisfies
\[
 p(y)=0\Longleftrightarrow y=0,\qquad p(ty)=|t|p(y),\qquad
 p(y+z)\le p(y)+p(z).
\]
Write $y_\alpha\to_{\rm bo}y$ when $p(y_\alpha-y)\to_o0$. Bo-completeness means that every net whose independent difference net is bo-null has a bo-limit. For $h\in F_+$ put $F_h=\{h\}^{\perp\perp}$ and $Y_h=\{y:p(y)\in F_h\}$.
\end{definition}
The two triangle inequalities
$p(y)\le p(z)+p(y-z)$ and $p(z)\le p(y)+p(y-z)$ give
$|p(y)-p(z)|\le p(y-z)$. If $p(y_\alpha-y)\le r_\gamma$ for
$\alpha\ge\alpha_\gamma$ and $r_\gamma\downarrow0$, then
\[
 \forall\gamma\ \forall\alpha\ge\alpha_\gamma:\qquad
 |p(y_\alpha)-p(y)|\le r_\gamma.
\]
For independent convergence regulators $r_\gamma$ and $t_\delta$,
\[
 p((y_\alpha+z_\beta)-(y+z))\le r_\gamma+t_\delta
 \quad(\alpha\ge\alpha_\gamma,\ \beta\ge\beta_\delta),\qquad
 \inf_{\gamma,\delta}(r_\gamma+t_\delta)=0.
\]
Negation preserves $p(y_\alpha-y)$. If $y$ and $z$ are limits of
one net, the estimate $p(y-z)\le r_\gamma+t_\delta$ for every
$(\gamma,\delta)$ gives $p(y-z)=0$, hence $y=z$. These estimates
prove the group-convergence and separation assertions. The band
projection $\pi_h$ gives
\[
 (I-\pi_h)p(y)\le (I-\pi_h)p(y-y_\alpha)\quad(y_\alpha\in Y_h).
\]
For a convergent net $y_\alpha\in Y_h$, every regulator index gives
$0\le(I-\pi_h)p(y)\le(I-\pi_h)r_\gamma\le r_\gamma$.
Its infimum is zero, so $p(y)\in F_h$ and $Y_h$ is bo-closed.
If $y_\alpha,y\in Y_h$, then
\[
 p(y_\alpha-y)=\pi_hp(y_\alpha-y)\le\pi_hr_\gamma
 \quad(\alpha\ge\alpha_\gamma),\qquad
 0\le\inf_\gamma\pi_hr_\gamma\le\inf_\gamma r_\gamma=0.
\]
Thus an ambient regulator produces an intrinsic regulator in $F_h$.
The reverse implication uses the inclusion $F_h\subseteq F$ and
the agreement of bounded suprema and infima in a band with those
in $F$ \cite{AB}. A bo-Cauchy net in $Y_h$ is therefore bo-Cauchy
in $Y$. When $Y$ is bo-complete, its limit lies in $Y_h$ by the
proved closedness, and the projected regulator proves convergence
in $Y_h$.

\begin{lemma}[Dominated coordinate representation]\label{km:embedding}\label{sem:representation}
For a vector-normed space over $F$, put
\[
 \mathscr T=\{T:Y\to F:T\text{ linear},\ |Ty|\le p(y)\ (y\in Y)\},
\]
\[
 E=\{z\in F^{\mathscr T}:\sup_T|z_T|\text{ exists in }F\},
 \qquad P(z)=\sup_T|z_T|.
\]
Then $E$ is bo-complete and $J:Y\to E$, $Jy=(Ty)_T$, satisfies
\begin{equation}\label{km:isometry}
 P(Jy)=p(y).
\end{equation}
The inverse $J^{-1}$ on $JY$ is bo-continuous. If $Y$ is bo-complete, its range $JY$ is bo-closed.
\end{lemma}
\begin{proof}
For $z,w\in E$ and $t\in\mathbb R$, the coordinate estimates
$|z_T+w_T|\le P(z)+P(w)$ and $|tz_T|=|t||z_T|$ give
\[
 z+w,tz\in E,\quad P(z+w)\le P(z)+P(w),\quad
 P(tz)=|t|P(z),\quad P(z)=0\Longleftrightarrow z=0.
\]
For $y\ne0$, the map $ty\mapsto tp(y)$ on $\mathbb Ry$ satisfies
$tp(y)\le |t|p(y)=p(ty)$. Hahn--Banach--Kantorovich supplies a linear extension $T_y$ with $T_yz\le p(z)$ \cite[Theorem~2.3]{KKHBK}. Applying this inequality to $z$ and $-z$ gives $|T_yz|\le p(z)$. Hence
\[
 p(y)=|T_yy|\le P(Jy)\le p(y).
\]
Let $z_\alpha$ be bo-Cauchy, with $P(z_\alpha-z_\beta)\le r_\gamma$ for $\alpha,\beta\ge\alpha_\gamma$, where $r_\gamma\downarrow0$. Fix $\gamma_0$ and put $\alpha_0=\alpha_{\gamma_0}$ and
$g=P(z_{\alpha_0})+r_{\gamma_0}$. Then $|z_{\alpha,T}|\le g$ for $\alpha\ge\alpha_0$ and every $T$. All subsequent tail infima and suprema are taken over this cofinal tail, so Dedekind completeness gives
\[
 l_{\alpha,T}=\inf_{\beta\ge\alpha}z_{\beta,T},\quad
 u_{\alpha,T}=\sup_{\beta\ge\alpha}z_{\beta,T},\quad
 0\le u_{\alpha,T}-l_{\alpha,T}\le r_\gamma
 \quad(\alpha\ge\alpha_\gamma).
\]
Both families lie in $[-g,g]$. Their gap is bounded by every $r_\gamma$, whence
$\sup_\alpha l_{\alpha,T}=\inf_\alpha u_{\alpha,T}=:z_T$.
The bound $|z_T|\le g$ gives $z\in E$. For $\alpha\ge\alpha_\gamma$, the inequalities
$z_{\alpha,T}-r_\gamma\le z_{\beta,T}\le z_{\alpha,T}+r_\gamma$
hold on a cofinal $\beta$-tail. Its infimum and supremum give
$|z_{\alpha,T}-z_T|\le r_\gamma$ for every $T$, hence $P(z_\alpha-z)\le r_\gamma$. If $Jy_\alpha\to_{\rm bo}z$, then
\[
 p(y_\alpha-y_\beta)\le P(Jy_\alpha-z)+P(Jy_\beta-z)\to_o0.
\]
When $Y$ is bo-complete, its Cauchy net $y_\alpha$ has a limit $y$, and separation gives $z=Jy$. Identity~\eqref{km:isometry} proves continuity of the inverse without any completeness assumption.
\end{proof}

\begin{lemma}[Bounded range principles]\label{sem:target}\label{km:uniform-bounds}
Suppose $F_h$ is WSD and $(W,N,F_h)$ is vector-normed. If $K\subseteq\{N\le h\}$ and $D=\{N\le2h\}$, then $D$ satisfies $\mathsf S_{\rm seq}$ and $\mathsf R_\omega$, hence $\mathsf B_\omega$, and $K$ satisfies $\mathsf U$.
\end{lemma}
\begin{proof}
For any uniformly bounded family $V_\alpha$, its evaluation is bo-null exactly when
\begin{equation}\label{sem:bounded-evaluation}
 \sup\{N(v):\beta\ge\alpha,\ v\in V_\beta\}\downarrow_\alpha0.
\end{equation}
If $r_\gamma\downarrow0$ regulates the evaluation, then for each $\gamma$ some $\alpha_\gamma$ satisfies
$N(v)\le r_\gamma$ for every $\beta\ge\alpha_\gamma$ and $v\in V_\beta$.
Thus the displayed supremum at $\alpha_\gamma$ is at most $r_\gamma$.
Conversely, those tail suprema bound every value on their corresponding evaluation tails and have infimum zero. For decreasing null rows $U_{i,k}\subseteq D$, put $a_{i,k}=\sup_{w\in U_{i,k}}N(w)$. Then $a_{i,k}\downarrow_k0$. WSD gives
\[
 q_\varphi=\sup_i(2h\wedge2^ia_{i,\varphi(i)})\downarrow_\varphi0.
\]
For finite $A\subseteq\mathbb N$ and $w_i\in U_{i,\varphi(i)}$,
\[
 N\Bigl(\sum_{i\in A}w_i\Bigr)
 \le\sum_{i\in A}a_{i,\varphi(i)}
 \le\bigvee_{i\in A}2^ia_{i,\varphi(i)},
 \qquad N\Bigl(\sum_{i\in A}w_i\Bigr)\wedge2h\le q_\varphi.
\]
To verify stability of the bounding set, let $z_\alpha$ belong to
$\{z:N(z)\wedge2h\le q_\varphi\}$ and suppose
$N(z_\alpha-z)\le r_\gamma$ for $\alpha\ge\alpha_\gamma$,
where $r_\gamma\downarrow0$. The lattice contraction inequality gives
\[
 \bigl|(N(z)\wedge2h)-(N(z_\alpha)\wedge2h)\bigr|
 \le|N(z)-N(z_\alpha)|\le r_\gamma,
\]
\[
 (N(z)\wedge2h-q_\varphi)^+\le r_\gamma\quad(\gamma),
 \qquad N(z)\wedge2h\le q_\varphi.
\]
Thus the set is bo-stable. It contains every finite selected sum,
so
\begin{equation}\label{sem:summation-bound}
 D\cap\clo_{\rm bo}\Bigl(\bigcup_{A\text{ finite}}\sum_{i\in A}U_{i,\varphi(i)}\Bigr)
 \subseteq\{N\le q_\varphi\}.
\end{equation}
This proves closed sequential summation.

For the matrix in $\mathsf R_\omega$, set
\[
 b_{n,k}=\sup_{w\in A_{n,k}}N(w),\quad b_n=\sup_kb_{n,k},
 \quad b=\inf_nb_n,\quad c_{n,k}=b_n-b_{n,k}\downarrow_k0.
\]
The matrix inclusions give $b_{n,k}\uparrow_k b_n$ and
$b_{n+1}\le b_n$, so the displayed row differences decrease to zero.
For each selector $\varphi$, selected evaluation nullity gives
a regulator $r^\varphi_\gamma\downarrow0$ and thresholds
$N^\varphi_\gamma$ such that
\[
 b_{n,\varphi(n)}\le r^\varphi_\gamma
 \quad(n\ge N^\varphi_\gamma).
\]
Put $c_\varphi=\sup_jc_{j,\varphi(j)}$. Then, for every $\gamma$
and $n\ge N^\varphi_\gamma$, one has
$b\le b_n=b_{n,\varphi(n)}+c_{n,\varphi(n)}
\le r^\varphi_\gamma+c_\varphi$. Consequently
\[
 0\le b\le b_{n,\varphi(n)}+\sup_jc_{j,\varphi(j)}
 \quad\Longrightarrow\quad b\le\sup_jc_{j,\varphi(j)}.
\]
The implication in this display is the infimum over $\gamma$ of
the preceding bound. All $c_{n,k}$ lie in $[0,2h]$. WSD yields
$b\le\inf_\varphi\sup_jc_{j,\varphi(j)}=0$.
For $w\in\bigcup_kA_{n,k}$ one has $N(w)\le b_n\downarrow0$, which proves row-union evaluation nullity. For $\mathsf U$, use its notation and put
\[
 A_N=\sup_{n\ge N,i}N(x_n(i)-y_n(i)),\quad
 A_N^\beta=\sup_{n\ge N,i}N(x_n^\beta(i)-y_n^\beta(i)),
\]
\[
 e_\beta=\sup_{n,i}\bigl(N(x_n(i)-x_n^\beta(i))+N(y_n(i)-y_n^\beta(i))\bigr).
\]
These suprema are bounded by $2h,2h,4h$, respectively. The hypotheses and~\eqref{sem:bounded-evaluation} give $A_N^\beta\downarrow_N0$ and $e_\beta\to_o0$. For every $N,\beta$, the triangle inequality gives $A_N\le A_N^\beta+e_\beta$. Consequently
\[
 0\le\inf_NA_N\le\inf_N(A_N^\beta+e_\beta)=e_\beta
 \quad(\beta),\qquad \inf_NA_N=0.
\]
\end{proof}

\begin{lemma}[Coordinate balls]\label{km:ball-decomposition}\label{sem:ball-decomposition}
In a space of order-bounded $F$-coordinates with supremum vector norm $N$, including finite products of $E$ and $F$, an order-bounded sequence $c_i\ge0$ satisfies
\begin{equation}\label{sem:ball-hull}
 \{N\le\sup_ic_i\}\subseteq
 \clo_{\rm bo}\Bigl(\bigcup_{A\text{ finite}}\sum_{i\in A}\{N\le c_i\}\Bigr).
\end{equation}
\end{lemma}
\begin{proof}
For $z$ in the left side, put $e=N(z)$, $t_n=e\wedge\bigvee_{i\le n}c_i$, $t_0=0$, and
$z_\gamma^{(n)}=(z_\gamma^+\wedge t_n)-(z_\gamma^-\wedge t_n)$ at every coordinate. Then
\[
 N(z^{(n)}-z^{(n-1)})\le t_n-t_{n-1}\le c_n,
 \qquad N(z-z^{(n)})\le e-t_n\downarrow0.
\]
The first inequality uses
$|(|f|\wedge t)-(|f|\wedge r)|\le t-r$ for $0\le r\le t$ and
$(e\wedge(v\vee c))-(e\wedge v)\le(v\vee c)-v\le c$.
The equality $\sup_nt_n=e\wedge\sup_ic_i=e$ justifies the null regulator above. Moreover
$z^{(n)}=\sum_{i=1}^n(z^{(i)}-z^{(i-1)})\in\sum_{i=1}^n\{N\le c_i\}$.
Thus $z$ is a bo-limit of elements of the displayed union and belongs to its stable hull.
\end{proof}

\begin{definition}[Source semivariation]\label{sem:def-semivariation}
Let $\mathcal H$ be a Boolean ring in a sigma-complete Boolean algebra $\mathbb B$, with $1\in\mathcal H_\sigma$, and let $m:\mathcal H\to Y$ be a bo-premeasure. Write $\Pi_{\mathcal H}(a)$ for its finite partitions and suppose the following family is order bounded:
\begin{equation}\label{sem:source-semivariation}
 s(a)=\sup_{\pi\in\Pi_{\mathcal H}(a),\ \varepsilon_c=\pm1}
 p\Bigl(\sum_{c\in\pi}\varepsilon_cm(c)\Bigr),
 \qquad h=\sup_{a\in\mathcal H}s(a).
\end{equation}
For the semivariation construction, $F_h$ is WSD and $Y_h$ is bo-complete. The coordinate representation is taken over $(Y_h,p,F_h)$. Its index family is denoted by $\mathscr T_h$.
\end{definition}

\begin{lemma}[Scalarized variations and the additive lift]\label{sem:variation}\label{sem:finite-coefficients}\label{sem:def-lift}
For $T\in\mathscr T_h$, define
\[
 \nu_T(a)=\sup_{\pi\in\Pi_{\mathcal H}(a)}\sum_{c\in\pi}|Tm(c)|,
 \quad q(a)=\sup_{c\in\mathcal H,\ c\le a}p(m(c)).
\]
Each $\nu_T$ is a positive order-premeasure, and
\begin{equation}\label{sem:coordinate-variation}
 |Tm(a)|\le\nu_T(a)\le s(a)=\sup_T\nu_T(a),
 \qquad q(a)\le s(a)\le2q(a).
\end{equation}
The semivariation is unchanged by allowing coefficients in $[-1,1]$. It is monotone and finitely subadditive. The maps
\[
 \kappa(a)=(\nu_T(a))_T,\quad
 \lambda(a)=(Jm(a),\kappa(a)),\quad
 Z=E\times E,\quad Q(z,w)=P(z)\vee P(w)
\]
are finitely additive and satisfy
\begin{equation}\label{sem:lift-norm}
 P\kappa(a)=Q\lambda(a)=s(a),\quad
 M:=\clo_{\rm bo}\lambda(\mathcal H)\subseteq\{Q\le h\}\cap(JY_h\times E).
\end{equation}
The set $M$ is bo-Cauchy complete and $M-M\subseteq D_h:=\{Q\le2h\}$.
\end{lemma}
\begin{proof}
For $t_j\in[-1,1]$, the numbers
$\alpha_\varepsilon=\prod_j\frac{1+\varepsilon_jt_j}{2}$ satisfy
$\alpha_\varepsilon\ge0$, $\sum_\varepsilon\alpha_\varepsilon=1$ and
$\sum_\varepsilon\alpha_\varepsilon\varepsilon_j=t_j$.
Therefore, for every partition $(c_j)$ of $a$,
\[
 p\Bigl(\sum_jt_jm(c_j)\Bigr)
 \le\sum_\varepsilon\alpha_\varepsilon
       p\Bigl(\sum_j\varepsilon_jm(c_j)\Bigr)\le s(a).
\]
The coefficients $1,0$ on $c,a\setminus c$ give $q(a)\le s(a)$.
For signs $\varepsilon_j$, put $a_\pm=\bigvee_{\varepsilon_j=\pm1}c_j$.
Then $p(m(a_+)-m(a_-))\le2q(a)$, proving $s(a)\le2q(a)$.
If $a\le b$, adjoining $b\setminus a$ with coefficient zero gives $s(a)\le s(b)$.
For a partition $(c_j)$ of $a\vee b$, splitting each $c_j$ into
$c_j\wedge a$ and $c_j\wedge(b\setminus a)$ gives the bound
$s(a)+s(b\setminus a)\le s(a)+s(b)$ for its coefficient sum.
Taking the partition supremum proves finite subadditivity.

Finite distributivity gives
\[
 \sum_{c\in\pi}|Tm(c)|
 =\bigvee_{\varepsilon_c=\pm1}T\Bigl(\sum_{c\in\pi}\varepsilon_cm(c)\Bigr)\le s(a).
\]
Thus $\nu_T$ exists. For disjoint $a,b$ and $\pi\in\Pi_{\mathcal H}(a\vee b)$,
\[
 \sum_{c\in\pi}|Tm(c)|
 \le\sum_{c\in\pi}|Tm(c\wedge a)|
   +\sum_{c\in\pi}|Tm(c\wedge b)|\le\nu_T(a)+\nu_T(b).
\]
The supremum over $\pi$ gives $\nu_T(a\vee b)\le\nu_T(a)+\nu_T(b)$. Conversely, the union of partitions $\pi_a$ of $a$ and $\pi_b$ of $b$ partitions $a\vee b$, so
\[
 \sum_{c\in\pi_a}|Tm(c)|+\sum_{d\in\pi_b}|Tm(d)|\le\nu_T(a\vee b).
\]
Taking the supremum first over $\pi_a$ and then over $\pi_b$ gives
$\nu_T(a)+\nu_T(b)\le\nu_T(a\vee b)$, since addition preserves bounded suprema in each variable. This proves positivity and finite additivity. If $a_n\uparrow a$ in the source, each source partition $\pi$ of $a$ satisfies
\[
 \sum_{c\in\pi}|Tm(c\wedge a_n)|\to_o\sum_{c\in\pi}|Tm(c)|,
 \qquad \sum_{c\in\pi}|Tm(c\wedge a_n)|\le\sup_n\nu_T(a_n).
\]
The supremum over $\pi$ proves $\nu_T(a_n)\uparrow\nu_T(a)$, hence its premeasure property. The norming identity gives the remaining equality:
\[
 p\Bigl(\sum_c\varepsilon_cm(c)\Bigr)
 =\sup_T\Bigl|\sum_c\varepsilon_cTm(c)\Bigr|\le\sup_T\nu_T(a)\le s(a).
\]
The identities $P\kappa(a)=s(a)$ and $P(Jm(a))=p(m(a))\le s(a)$ prove~\eqref{sem:lift-norm}. A bo-Cauchy net $(z_\alpha,w_\alpha)$ in $Z$ satisfies
$P(z_\alpha-z_\beta),P(w_\alpha-w_\beta)\le r_\gamma$ beyond each Cauchy threshold.
Lemma~\ref{sem:representation} gives coordinate limits $z,w$ with
$P(z_\alpha-z),P(w_\alpha-w)\le r_\gamma$ beyond that threshold.
Hence $Q((z_\alpha,w_\alpha)-(z,w))\le r_\gamma$. Thus $Z$ is bo-complete. If $Q(z_\alpha)\le h$ and $z_\alpha\to_{\rm bo}z$ with regulator $r_\gamma\downarrow0$, then $Q(z)\le h+r_\gamma$ for every $\gamma$, giving $Q(z)\le h$. Hence the norm ball is stable. The range $JY_h\times E$ is stable by Lemma~\ref{sem:representation}. Both sets contain the source lift and therefore contain $M$. A Cauchy net in $M$ converges in $Z$, and stability places its limit in $M$. Finally $Q(v-w)\le Q(v)+Q(w)\le2h$ for $v,w\in M$.
\end{proof}

The coordinates are additive, but countable additivity in the supremum vector norm requires one regulator valid for all of them. We obtain this estimate from hereditary exhaustivity of the original premeasure through Lemma~\ref{gen:ut}. It proves that the whole lift is a premeasure and allows us to construct its upper values from source exhaustions.

\begin{lemma}[Exhaustivity and upper values]\label{sem:source}\label{sem:upper-lift}\label{sem:upper-envelope-properties}
Suppose $m$ is uniformly hereditarily exhaustive. Then $\kappa$ and $\lambda$ are uniformly hereditarily exhaustive bo-premeasures. For $u\in\mathcal H_\sigma$, put
\[
 s_\sigma(u)=\sup_{a\in\mathcal H,\ a\le u}s(a),\quad
 \nu_{T,\sigma}(u)=\sup_{a\in\mathcal H,\ a\le u}\nu_T(a).
\]
Every increasing source exhaustion $a_n\uparrow u$ satisfies
\begin{equation}\label{sem:upper-envelope-identities}
 \nu_{T,\sigma}(u)=\sup_n\nu_T(a_n),\quad
 s_\sigma(u)=\sup_ns(a_n)=\sup_T\nu_{T,\sigma}(u),\quad
 Q\lambda_\sigma(u)=s_\sigma(u),
\end{equation}
where $\lambda_\sigma(u)=\mathrm{bo}\text{-}\lim_n\lambda(a_n)\in M$ is independent of the exhaustion. Moreover
\begin{equation}\label{sem:whole-shell}
 \lambda_\sigma(u)\in\Gamma_\lambda(u)
 :=\clo_{\rm bo}\{\lambda(a):a\in\mathcal H,\ a\le u\}
 \subseteq\{Q\le s_\sigma(u)\},\qquad
 s_\sigma(u\setminus a_n)\downarrow0.
\end{equation}
\end{lemma}
\begin{proof}
For disjoint source $a_n$, hereditary exhaustivity and boundedness give
\[
 u_N=\sup_{n\ge N,\ c\le a_n}p(m(c))\downarrow0,
 \qquad Q\lambda(c)=P\kappa(c)=s(c)\le2u_N
 \quad(c\le a_n,\ n\ge N).
\]
Thus both finitely additive lifts are hereditarily exhaustive.
Apply Lemma~\ref{sem:target} with $W=E$, $N=P$,
$K=\clo_{\rm bo}\kappa(\mathcal H)$, and
$D_\kappa=\{P\le2h\}$. The source bounds $P\kappa(a)=s(a)\le h$
put $\kappa(\mathcal H)$ in $D_\kappa$, and that lemma proves
$\mathsf B_\omega(D_\kappa)$. Thus Lemma~\ref{gen:ut} applies to the
finitely additive map $\kappa:\mathcal H\to E$ with this difference
set and with the hereditary estimate just proved. For a disjoint
source sequence $(b_n)$, its null tail families and
\eqref{sem:bounded-evaluation} give
\[
 r_N:=\sup\Bigl\{s(c):c\in\mathcal H,\ c\le\bigvee_{n\in A}b_n,
 \ A\subseteq\{N,N+1,\ldots\}\text{ finite}\Bigr\}\downarrow0.
\]
If $a_n\downarrow0$ in the source, use $b_n=a_n\setminus a_{n+1}$. Coordinate countable additivity, already proved, gives
\[
 s(a_N)=\sup_T\sup_{K\ge N}\sum_{n=N}^K\nu_T(b_n)
 =\sup_{K\ge N}s\Bigl(\bigvee_{n=N}^Kb_n\Bigr)\le r_N.
\]
Therefore $s(a_N)\downarrow0$. For a disjoint source union $a=\bigvee_na_n\in\mathcal H$, both lift errors have norm
$s(a\setminus\bigvee_{n\le N}a_n)\downarrow0$.
In particular, the partial sums converge to the prescribed values $\kappa(a)$ and $\lambda(a)$.

For an exhaustion $a_n\uparrow u$ and $c\in\mathcal H$ with
$c\le u$, the increasing sequence $c\wedge a_n$ has supremum
$c$ in the source. The proved positive premeasure property gives
\[
 \nu_T(c)=\sup_n\nu_T(c\wedge a_n)\le\sup_n\nu_T(a_n)
 \le\nu_{T,\sigma}(u).
\]
The supremum over $c$ gives
$\nu_{T,\sigma}(u)=\sup_n\nu_T(a_n)$. All of the following
suprema are bounded by $h$, and their index sets are products, so
\[
 \begin{split}
 s_\sigma(u)
 &=\sup_{c\in\mathcal H,\ c\le u}\sup_T\nu_T(c)
 =\sup_T\sup_{c\in\mathcal H,\ c\le u}\nu_T(c)\\
 &=\sup_T\sup_n\nu_T(a_n)
 =\sup_n\sup_T\nu_T(a_n)=\sup_ns(a_n).
 \end{split}
\]
Set $a_0=0$ and apply the finite-tail estimate above to
$b_n=a_n\setminus a_{n-1}$. For $c\le u\setminus a_N$ in the source
and $j>N$, one has
\[
 c\wedge a_j\le a_j\setminus a_N
 =\bigvee_{N<n\le j}b_n.
\]
For $j\le N$, the element $c\wedge a_j$ is zero. Hence
$Q\lambda(c\wedge a_j)\le r_{N+1}$.
Since $c\wedge a_j\uparrow c\in\mathcal H$, the premeasure property and norm continuity give $s(c)\le r_{N+1}$.
The defining supremum is bounded by $h$, and therefore
$s_\sigma(u\setminus a_N)\le r_{N+1}\downarrow0$.
For every $N\ge1$ and $p,q\ge N$, finite additivity gives
\[
 Q(\lambda(a_p)-\lambda(a_q))
 =s(a_{\max\{p,q\}}\setminus a_{\min\{p,q\}})
 \le s_\sigma(u\setminus a_N).
\]
Completeness of $M$ gives a limit $z_a\in M$. For another exhaustion
$d_k\uparrow u$, this estimate constructs a limit $z_d\in M$.
At each fixed $n$, the differences
$(a_n\wedge d_k)\setminus(a_n\wedge d_{k-1})$, with $d_0=0$,
are disjoint source elements with supremum $a_n$. Thus the proved
premeasure property gives
$\lambda(a_n\wedge d_k)\to_{\rm bo}\lambda(a_n)$. The identity
\[
 \lambda(d_k)-\lambda(a_n\wedge d_k)=\lambda(d_k\setminus a_n),
 \qquad Q\lambda(d_k\setminus a_n)\le s_\sigma(u\setminus a_n)
\]
gives $Q(z_d-\lambda(a_n))\le s_\sigma(u\setminus a_n)$ by
passing to the bo-limit in $k$ and using
$|Q(v)-Q(w)|\le Q(v-w)$. Consequently
$\lambda(a_n)\to_{\rm bo}z_d$, and separation gives $z_d=z_a$.
Write this common value as $\lambda_\sigma(u)$.
Coordinate and norm estimates identify its values explicitly:
\[
 (\lambda_\sigma(u))_{2,T}
 =o\text{-}\lim_n\nu_T(a_n)=\nu_{T,\sigma}(u),\qquad
 Q\lambda_\sigma(u)=o\text{-}\lim_ns(a_n)=s_\sigma(u).
\]
The source exhaustion lies in
$\{\lambda(c):c\in\mathcal H,\ c\le u\}$, so its limit belongs
to $\Gamma_\lambda(u)$. Every source value in this set has norm at
most $s_\sigma(u)$. For any bo-convergent net $z_\alpha$ with
$Qz_\alpha\le s_\sigma(u)$, a regulator $r_\gamma\downarrow0$
gives $Qz\le s_\sigma(u)+r_\gamma$ for every $\gamma$, hence
$Qz\le s_\sigma(u)$. This stable norm ball contains the whole hull,
proving \eqref{sem:whole-shell}.
\end{proof}

\begin{definition}[The full regulator domain]\label{sem:domain}\label{km:def-sigma}
For $\ell\in\mathcal H_\delta$, put
\[
 \lambda_\delta(\ell)=\lambda_\sigma(1)-\lambda_\sigma(\ell^c),\quad
 \nu_{T,\delta}(\ell)=\nu_{T,\sigma}(1)-\nu_{T,\sigma}(\ell^c),\quad
 s_\delta(\ell)=\sup_T\nu_{T,\delta}(\ell).
\]
Define $\Sigma_s(\mathcal H)$ by the existence of one bounded regulator $a_{i,k}\downarrow_k0$ and sandwiches
\begin{equation}\label{sem:domain-shell}
 \ell_\varphi\in\mathcal H_\delta,\quad u_\varphi\in\mathcal H_\sigma,
 \quad\ell_\varphi\le b\le u_\varphi,
 \quad s_\sigma(u_\varphi\setminus\ell_\varphi)\le\sup_i a_{i,\varphi(i)}
 \quad(\varphi\in\mathbb N^{\mathbb N}).
\end{equation}
Regulator entries may be taken in $F_h$. The projection $L(Jy,w)=y$ on $JY_h\times E$ defines $m_\sigma=L\lambda_\sigma$ and $m_\delta=L\lambda_\delta$.
\end{definition}
The lower lift belongs to $M$ by Lemma~\ref{gen:delta-construction}, so both projections are defined. Lower semivariation is recovered from the positive lower coordinate values.

The sandwich condition uses only the upper values already constructed from the source. We identify its domain with the intrinsic regulated class by turning regulator bounds into finite-sum certificates and recovering those bounds from the value of the whole shell. Theorem~\ref{gen:extension} then constructs $\bar\lambda$, and the continuous projection $L$ gives $\bar m$ on this domain. To prove preservation of variation and semivariation, we use all finite partitions in the enlarged domain.

\begin{theorem}[Extension with exact variation and semivariation]\label{sem:extension}\label{sem:domain-identification}\label{sem:exact-variation}\label{sem:exact-semivariation}
Under the hypotheses above, including hereditary exhaustivity,
\begin{equation}\label{sem:domain-equality}
 \Sigma_s(\mathcal H)=\mathcal R_{\rm seq}(\lambda,D_h).
\end{equation}
This domain is a Boolean sigma-subalgebra containing $\sigma(\mathcal H)$. The extension $\bar\lambda$ supplied by Theorem~\ref{gen:extension} defines a bo-countably additive $\bar m=L\bar\lambda$ and positive order-countably additive coordinates $\bar\nu_T=(\bar\lambda)_{2,T}$ with
\begin{equation}\label{sem:variation-preserved}
 |T\bar m|(b)=\bar\nu_T(b),\qquad
 \bar s(b)=\sup_T\bar\nu_T(b)=Q\bar\lambda(b)\le h.
\end{equation}
Here the variation and semivariation use all finite partitions in the full domain. They extend $\nu_T$ and $s$. On the complete upper and lower directed families,
\begin{equation}\label{sem:full-upper}
 m_\sigma(u)\to_{\rm bo}\bar m(b),\qquad
 m_\delta(\ell)\to_{\rm bo}\bar m(b)\quad
 (u\supseteq b,\ \ell\subseteq b),
\end{equation}
with upper covers ordered by reverse inclusion and lower approximants by inclusion. Moreover
\begin{equation}\label{sem:regularity}
 s_\sigma(u)\downarrow\bar s(b),\quad s_\delta(\ell)\uparrow\bar s(b),
 \qquad\bar s(u\setminus b)\downarrow0,\quad\bar s(b\setminus\ell)\downarrow0.
\end{equation}
Uniqueness holds among all bo-countably additive extensions on $\sigma(\mathcal H)$ and among those with full lower regularity on $\Sigma_s(\mathcal H)$.
\end{theorem}
\begin{proof}
For an intrinsic certificate, let $r_{i,k}=\sup_{z\in U_{i,k}}Q(z)\downarrow_k0$. Equation~\eqref{sem:summation-bound} and the whole-shell member in~\eqref{sem:whole-shell} give
\[
 s_\sigma(u_\varphi\setminus\ell_\varphi)
 \le\sup_i(2h\wedge2^ir_{i,\varphi(i)}).
\]
The entries $2h\wedge2^ir_{i,k}$ decrease to zero in $k$ and lie in $[0,2h]$, proving
$\mathcal R_{\rm seq}(\lambda,D_h)\subseteq\Sigma_s(\mathcal H)$.
Conversely, replace a domain regulator by $h\wedge a_{i,k}$, using
$h\wedge\sup_i a_{i,\varphi(i)}=\sup_i(h\wedge a_{i,\varphi(i)})$.
For the resulting rows $U_{i,k}=\{Q\le a_{i,k}\}\subseteq D_h$, put
$A_\varphi=\bigcup_{A\text{ finite}}\sum_{i\in A}U_{i,\varphi(i)}$ and
$v_\varphi=u_\varphi\setminus\ell_\varphi$. Lemma~\ref{sem:ball-decomposition} gives
\[
 \Gamma_\lambda(v_\varphi)
 \subseteq\{Q\le s_\sigma(v_\varphi)\}
 \subseteq\{Q\le\sup_i a_{i,\varphi(i)}\}
 \subseteq\clo_{\rm bo}A_\varphi.
\]
The first set lies in $\{Q\le h\}\subseteq D_h$, proving the reverse domain inclusion.

Lemma~\ref{sem:source} makes $\lambda$ a hereditarily exhaustive premeasure.
Lemma~\ref{sem:variation} gives the complete stable range $M$ and $M-M\subseteq D_h$.
Lemma~\ref{sem:target} gives $\mathsf S_{\rm seq}(D_h)$, $\mathsf B_\omega(D_h)$, and $\mathsf U(M)$.
Thus Theorem~\ref{gen:extension} supplies $\bar\lambda$ and its full approximation nets. The continuous additive projection $L$ and the coordinate projections preserve these limits and countable additivity. The cone
\[
 \{(z,w):0\le w_T,\ |z_T|\le w_T\text{ for every }T\}
\]
is bo-stable. Indeed, for a net $(z^\alpha,w^\alpha)$ in this cone
converging to $(z,w)$ with product-norm error at most $r_\gamma$
for $\alpha\ge\alpha_\gamma$, every coordinate satisfies
\[
 w_T\ge w_T^\alpha-r_\gamma\ge-r_\gamma,
 \qquad |z_T|\le|z_T^\alpha|+r_\gamma
 \le w_T^\alpha+r_\gamma\le w_T+2r_\gamma.
\]
The infimum over $\gamma$ gives $w_T\ge0$ and $|z_T|\le w_T$.
The source inequalities $0\le\nu_T(a)$ and
$|Tm(a)|\le\nu_T(a)$ put $\lambda(\mathcal H)$ in this cone,
so stability puts $M$ in it. Thus
$|T\bar m(b)|\le\bar\nu_T(b)$ and $Q\bar\lambda(b)=\sup_T\bar\nu_T(b)$.

Fix $T$ and let $w=|T\bar m|$. Coordinate domination bounds every partition sum by $\bar\nu_T(b)$, so $w$ exists and $w\le\bar\nu_T$. For disjoint $b,d$ and a partition $\pi$ of their join,
\[
 \sum_{c\in\pi}|T\bar m(c)|
 \le\sum_{c\in\pi}|T\bar m(c\wedge b)|+
       \sum_{c\in\pi}|T\bar m(c\wedge d)|\le w(b)+w(d).
\]
Conversely, partitions $\pi_b$ and $\pi_d$ give
\[
 \sum_{c\in\pi_b}|T\bar m(c)|+\sum_{c\in\pi_d}|T\bar m(c)|\le w(b\vee d).
\]
Their independent suprema prove finite additivity of $w$. For $b_n\uparrow b$, each partition $\pi$ of $b$ gives
\[
 \sum_{c\in\pi}|T\bar m(c\wedge b_n)|\to_o\sum_{c\in\pi}|T\bar m(c)|,
 \qquad\sum_{c\in\pi}|T\bar m(c\wedge b_n)|\le\sup_nw(b_n).
\]
Let $r=\sup_nw(b_n)\le w(b)$. The displayed order limits lie in
the order-stable interval $[0,r]$, so each partition sum for $b$
is at most $r$. Their supremum gives $w(b)\le r$. Thus
$w(b_n)\uparrow w(b)$. For disjoint $(b_n)$, apply this identity
to the finite joins $d_N=\bigvee_{n\le N}b_n$ and use
$w(d_N)=\sum_{n\le N}w(b_n)$ to obtain countable additivity.
For $a\in\mathcal H$, source partitions and domination give
\[
 \nu_T(a)
 =\sup_{\pi\in\Pi_{\mathcal H}(a)}\sum_{c\in\pi}|Tm(c)|
 \le w(a)\le\bar\nu_T(a)=\nu_T(a).
\]
Proposition~\ref{gen:uniqueness}, with source map $\nu_T$,
range $F_h$ with order convergence, and competing countably
additive maps $w|_{\sigma(\mathcal H)}$ and
$\bar\nu_T|_{\sigma(\mathcal H)}$, gives equality on that
generated sigma-algebra. In particular the two maps agree on
$\mathcal H_\delta$. The full lower net for the positive coordinate
is increasing and converges to $\bar\nu_T(b)$, so its supremum is
$\bar\nu_T(b)$. Therefore
\[
 w(b)\ge\sup_{\ell\in\mathcal H_\delta,\ \ell\le b}w(\ell)
 =\sup_{\ell\le b}\nu_{T,\delta}(\ell)=\bar\nu_T(b),
\]
proving the first identity in~\eqref{sem:variation-preserved}.

For disjoint $b_j\le b$ and $|t_j|\le1$,
\[
 p\Bigl(\sum_jt_j\bar m(b_j)\Bigr)
 =\sup_T\Bigl|\sum_jt_jT\bar m(b_j)\Bigr|
 \le\sup_T\bar\nu_T(b).
\]
Conversely, finite distributivity gives
$\sum_j|T\bar m(b_j)|=\bigvee_{\varepsilon_j=\pm1}T(\sum_j\varepsilon_j\bar m(b_j))\le\bar s(b)$.
Taking the supremum first over all finite partitions of $b$ and then
over $T$ yields
\[
 \sup_T\bar\nu_T(b)\le\bar s(b)
 \le\sup_T\bar\nu_T(b),
\]
which proves exact semivariation. For $u\in\mathcal H_\sigma$,
countable additivity along a source exhaustion identifies
$\bar\lambda(u)=\lambda_\sigma(u)$ and
$\bar\nu_T(u)=\nu_{T,\sigma}(u)$. For
$\ell\in\mathcal H_\delta$, finite additivity gives
\[
 \bar\lambda(\ell)=\bar\lambda(1)-\bar\lambda(\ell^c)
 =\lambda_\sigma(1)-\lambda_\sigma(\ell^c)=\lambda_\delta(\ell),
 \qquad\bar\nu_T(\ell)=\nu_{T,\delta}(\ell).
\]
Hence $\bar s(u)=s_\sigma(u)$ and
$\bar s(\ell)=s_\delta(\ell)$. On the full upper and lower nets,
the estimates $|Qv-Qw|\le Q(v-w)$ give
$Q\lambda_\sigma(u)\to_o Q\bar\lambda(b)$ and
$Q\lambda_\delta(\ell)\to_o Q\bar\lambda(b)$.
The positive coordinate inequalities for $u'\le u$ and $\ell\le\ell'$ give
$s_\sigma(u')\le s_\sigma(u)$ and $s_\delta(\ell)\le s_\delta(\ell')$, establishing the first two limits in~\eqref{sem:regularity}. Finite additivity yields
\[
 \bar s(u\setminus b)=Q(\lambda_\sigma(u)-\bar\lambda(b)),\quad
 \bar s(b\setminus\ell)=Q(\bar\lambda(b)-\lambda_\delta(\ell)),
\]
Both right-hand sides converge in order to zero by full regularity
of $\bar\lambda$. The left-hand sides decrease on their respective
directed domains by monotonicity of $\bar s$, proving the two shell
limits in \eqref{sem:regularity}.
For a competing bo-countably additive extension $n$, a source
exhaustion $a_j\uparrow u$ gives
\[
 n(u)=\mathrm{bo}\text{-}\lim_jm(a_j)=m_\sigma(u),\qquad
 n(\ell)=n(1)-n(\ell^c)=m_\delta(\ell).
\]
If $n$ has full lower regularity, then for every
$b\in\Sigma_s(\mathcal H)$,
\[
 n(b)=\mathrm{bo}\text{-}\lim_{\ell\in\mathcal H_\delta,\ \ell\le b}
          m_\delta(\ell)=\bar m(b).
\]
For two arbitrary bo-countably additive extensions on
$\sigma(\mathcal H)$, Proposition~\ref{gen:uniqueness} applies with
source $m$, range $Y$ with its original bo-convergence, and this generated
sigma-algebra. It gives the asserted uniqueness without regularity
assumptions on those competitors.
\end{proof}

When vector variation is finite, its positive variation supplies the additional coordinate in the lift $(Jm,v)$. We identify the regulated domain of this lift with $\Sigma_v(\mathcal A)$ and apply Theorem~\ref{gen:extension} to recover the Kusraev--Malyugin extension, including preservation of variation on the entire domain. The full lower approximation nets then show that this extension agrees there with the one obtained from semivariation.

\begin{corollary}[Finite vector variation on its full domain]\label{km:global-extension}\label{km:local-extension}\label{sem:finite-variation-inclusion}
Suppose the source is a unital Boolean algebra $\mathcal A$. Let
\begin{equation}\label{km:def-nu}
 v(a)=\sup_{\pi\in\Pi_{\mathcal A}(a)}\sum_{c\in\pi}p(m(c)),\qquad k=v(1)
\end{equation}
exist, assume $F_k$ is WSD, and assume $Y_k$ is bo-complete. Define $\Sigma_v(\mathcal A)$ by~\eqref{sem:domain-shell} with a bounded $F$-valued regulator and $v_\sigma(u)=\sup_{a\le u}v(a)$ in place of $s_\sigma$. There is a unique fully lower-regular bo-premeasure extension on this entire domain, its variation is the positive fully regular extension $\bar v$ of $v$, and both full approximation nets converge. Also
\[
 \Sigma_v(\mathcal A)\subseteq\Sigma_s(\mathcal A),
 \qquad \bar m_s|_{\Sigma_v(\mathcal A)}=\bar m_v.
\]
Neither decomposability nor the countable sup property is required.
\end{corollary}
\begin{proof}
Every partition of $a$ extends to one of $1$, so the defining supremum is bounded by $k$. The single-cell partition gives $p(m(a))\le v(a)$. For disjoint $a,b$ and any partition $\pi$ of $a\vee b$,
\[
 \sum_{c\in\pi}p(m(c))
 \le\sum_{c\in\pi}p(m(c\wedge a))+
       \sum_{c\in\pi}p(m(c\wedge b))\le v(a)+v(b).
\]
For independent partitions $\pi_a,\pi_b$, their union gives
\[
 \sum_{c\in\pi_a}p(m(c))+\sum_{c\in\pi_b}p(m(c))\le v(a\vee b).
\]
The respective suprema prove $v(a\vee b)=v(a)+v(b)$. If $w$ is positive additive and $p(m(c))\le w(c)$, then
$\sum_{c\in\pi}p(m(c))\le\sum_{c\in\pi}w(c)=w(a)$ for every partition $\pi$ of $a$.
Taking the supremum gives $v(a)\le w(a)$.

For $a_n\uparrow a$, put $r=\sup_nv(a_n)\le v(a)$. Each source partition $\pi$ of $a$ satisfies
\[
 \sum_{c\in\pi}p(m(c\wedge a_n))\to_o\sum_{c\in\pi}p(m(c)),
 \qquad\sum_{c\in\pi}p(m(c\wedge a_n))\le r.
\]
Order closedness gives $\sum_{c\in\pi}p(m(c))\le r$, and the partition supremum gives $v(a)\le r$. Thus $v(a_n)\uparrow v(a)$. Applied to finite partial unions of a disjoint source union, this proves the order-premeasure property. For an arbitrary disjoint source sequence, put
\[
 t_0=0,\quad t_N=\sum_{n\le N}v(a_n),\quad t=\sup_Nt_N\le k,\quad r_N=t-t_{N-1}\downarrow0.
\]
Every fragment $c$ of a finite union with indices at least $N$ satisfies
$v(c)\le\sum_{n\in A}v(a_n)\le r_N$.

Clipping regulator entries to $k$ preserves the domain, since every shell variation is at most $k$ and $k\wedge\sup_i a_i=\sup_i(k\wedge a_i)$. Thus its domain can be computed in $F_k$. Form
\[
 \widetilde m(a)=(Jm(a),v(a))\in E\times F_k,
 \qquad N(z,f)=P(z)\vee|f|,
\]
\[
 \qquad\widetilde M=\clo_{\rm bo}\widetilde m(\mathcal A),
 \quad\widetilde D=\{N\le2k\}.
\]
One has $N\widetilde m(a)=v(a)$ and
$N\widetilde m(c)\le r_N$ for $c\le a_n$, $n\ge N$.
For a disjoint source union $a=\bigvee_na_n$, finite additivity gives
$N(\widetilde m(a)-\sum_{n\le N}\widetilde m(a_n))
=v(a)-\sum_{n\le N}v(a_n)\downarrow0$.
The first estimate bounds every fragment in every tail of a disjoint
source sequence by $r_N\downarrow0$, proving hereditary exhaustivity.
The second estimate identifies the bo-limit of the partial sums
with the prescribed source value, proving the premeasure property.
For an $N$-Cauchy net $(z_\alpha,f_\alpha)$ with error regulator
$r_\gamma\downarrow0$, its coordinates are Cauchy in $E$ and in
$F_k$. Lemma~\ref{sem:representation} and Theorem~\ref{ord:complete}
give limits $z,f$ with
\[
 P(z_\alpha-z)\le r_\gamma,\qquad
 |f_\alpha-f|\le r_\gamma
 \quad(\alpha\ge\alpha_\gamma),
 \qquad N((z_\alpha,f_\alpha)-(z,f))\le r_\gamma.
\]
Thus $E\times F_k$ is bo-complete. The set $JY_k\times F_k$ is
bo-stable by Lemma~\ref{sem:representation}. The norm ball
$\{N\le k\}$ is stable because
$Nz\le Nz_\alpha+N(z-z_\alpha)\le k+r_\gamma$ implies
$Nz\le k$ after taking the infimum over $\gamma$.
Both sets contain the source lift, hence
\[
 \widetilde M\subseteq(JY_k\times F_k)\cap\{N\le k\},\qquad
 N(v-w)\le Nv+Nw\le2k\quad(v,w\in\widetilde M).
\]
A Cauchy net in $\widetilde M$ converges in $E\times F_k$, and
stability places its limit in $\widetilde M$. This proves
completeness and $\widetilde M-\widetilde M\subseteq\widetilde D$.
Lemma~\ref{sem:target}, with base $F_k$, bound $k$, space
$W=E\times F_k$, norm $N$, and $K=\widetilde M$, supplies
$\mathsf B_\omega(\widetilde D)$,
$\mathsf S_{\rm seq}(\widetilde D)$ and $\mathsf U(\widetilde M)$.
Lemma~\ref{sem:ball-decomposition} applies to the coordinate family
$\mathscr T_k\dot\cup\{*\}$, where $*$ indexes the last $F_k$
coordinate. If $a_n\uparrow u$ and $c\in\mathcal A$, $c\le u$, then
$v(c)=\sup_nv(c\wedge a_n)\le\sup_nv(a_n)$. Taking the supremum over $c$ proves $v_\sigma(u)=\sup_nv(a_n)$. Therefore
\[
 N(\widetilde m(a_q)-\widetilde m(a_p))
 =v(a_q)-v(a_p)\le v_\sigma(u)-v(a_N)\downarrow_N0
 \quad(q\ge p\ge N).
\]
Lemma~\ref{gen:sigma-construction}, with source $\widetilde m$,
range $E\times F_k$, complete stable set $\widetilde M$, and
difference set $\widetilde D$, therefore gives the
exhaustion-independent upper value $\widetilde m_\sigma(u)$.
The inequalities
\[
 |N\widetilde m(a_n)-N\widetilde m_\sigma(u)|
 \le N(\widetilde m(a_n)-\widetilde m_\sigma(u))\to_o0,
 \qquad N\widetilde m(a_n)=v(a_n)\uparrow v_\sigma(u)
\]
give $N\widetilde m_\sigma(u)=v_\sigma(u)$.
Every source fragment value in $u$ has norm at most $v_\sigma(u)$.
The stability estimate for this norm ball is
$Nz\le v_\sigma(u)+r_\gamma$ for every regulator index, hence
$Nz\le v_\sigma(u)$. The exhaustion and this stable bound give
\[
 \widetilde m_\sigma(u)\in\Gamma_{\widetilde m}(u)
 \subseteq\{N\le v_\sigma(u)\}.
\]
For an intrinsic certificate, \eqref{sem:summation-bound}, with $h$ replaced by $k$, gives
$v_\sigma(u_\varphi\setminus\ell_\varphi)\le\sup_i(2k\wedge2^ir_{i,\varphi(i)})$.
Conversely, a variation regulator clipped to $k$ defines rows $U_{i,j}=\{N\le a_{i,j}\}$, and Lemma~\ref{sem:ball-decomposition} gives
\[
 \Gamma_{\widetilde m}(u_\varphi\setminus\ell_\varphi)
 \subseteq\{N\le\sup_i a_{i,\varphi(i)}\}
 \subseteq\clo_{\rm bo}\bigcup_{A\text{ finite}}\sum_{i\in A}U_{i,\varphi(i)}.
\]
The left set also lies in $\{N\le k\}\subseteq\widetilde D$.
Both certificate implications identify the domain with $\Sigma_v(\mathcal A)$. Theorem~\ref{gen:extension} supplies the extension, sigma-closure and both full regularity nets.

For a convergent net $(Jy_\alpha,f_\alpha)$ in the cone
$\{(Jy,f):p(y)\le f\}$, the stable range $JY_k\times F_k$
places the limit in the form $(Jy,f)$. If the product-norm error
is at most $r_\gamma$ for $\alpha\ge\alpha_\gamma$, then
\[
 p(y)\le p(y-y_\alpha)+p(y_\alpha)
 \le r_\gamma+f_\alpha\le f+2r_\gamma.
\]
The infimum over $\gamma$ gives $p(y)\le f$, proving stability
of the cone. It contains the source lift and therefore its stable
hull and the extension values. Consequently
$p(\bar m_v(b))\le\bar v(b)$ for every $b$ in the extension domain.
Thus
\[
 w(b):=\sup_{\pi\in\Pi_{\Sigma_v}(b)}\sum_{c\in\pi}p(\bar m_v(c))\le\bar v(b)
\]
exists. For disjoint $b,d$, splitting each partition of $b\vee d$ gives
\[
 \sum_{c\in\pi}p(\bar m_v(c))\le
 \sum_{c\in\pi}p(\bar m_v(c\wedge b))+
 \sum_{c\in\pi}p(\bar m_v(c\wedge d))\le w(b)+w(d).
\]
Conversely, the union of any partitions of $b$ and $d$ has norm sum at most $w(b\vee d)$. Independent suprema prove additivity of $w$. For $b_n\uparrow b$, every partition $\pi$ of $b$ gives
\[
 \sum_{c\in\pi}p(\bar m_v(c\wedge b_n))\to_o
 \sum_{c\in\pi}p(\bar m_v(c)),\qquad
 \sum_{c\in\pi}p(\bar m_v(c\wedge b_n))\le\sup_nw(b_n).
\]
Writing $r=\sup_nw(b_n)$, order stability of $[0,r]$ bounds the
limit of each displayed partition sum by $r$. The supremum over
partitions gives $w(b)\le r$, while monotonicity gives $r\le w(b)$.
Hence $w(b)=\sup_nw(b_n)$. For a disjoint sequence, apply this
identity to its finite joins and then finite additivity of $w$ to
obtain countable additivity. On a source element $a$,
\[
 v(a)\le\sup_{\pi\in\Pi_{\Sigma_v}(a)}
                   \sum_{c\in\pi}p(\bar m_v(c))
 =w(a)\le\bar v(a)=v(a).
\]
Proposition~\ref{gen:uniqueness}, with source $v$, range $F_k$
with order convergence, and the extensions $w$ and $\bar v$,
gives equality on $\sigma(\mathcal A)$ and in particular on
$\mathcal A_\delta$. For arbitrary $b\in\Sigma_v(\mathcal A)$,
\[
 \bar v(b)\ge w(b)\ge
 \sup_{\ell\in\mathcal A_\delta,\ \ell\le b}w(\ell)
 =\sup_{\ell\le b}\bar v(\ell)=\bar v(b).
\]
Thus $|\bar m_v|=\bar v$ on the entire regulator domain.
For a competing fully lower-regular extension $n$, every source
exhaustion $a_j\uparrow u$ gives
$n(u)=\mathrm{bo}\text{-}\lim_jm(a_j)=m_\sigma(u)$.
For $\ell\in\mathcal A_\delta$, finite additivity gives
$n(\ell)=n(1)-n(\ell^c)=m_\delta(\ell)$.
Its full lower regularity then gives
\[
 n(b)=\mathrm{bo}\text{-}\lim_{\ell\in\mathcal A_\delta,\ \ell\le b}
                    m_\delta(\ell)=\bar m_v(b)
 \qquad(b\in\Sigma_v(\mathcal A)).
\]
This proves the stated uniqueness.

The inequality $s(a)\le v(a)$ proves finite semivariation and hereditary exhaustivity. Since $h\le k$, one has $F_h\subseteq F_k$.
Each $p(m(a))\le h$ belongs to $F_h$, so every partition sum defining $k$ belongs to that band.
Order closedness gives $k\in F_h$, hence $F_k\subseteq F_h$.
Thus $F_h=F_k$ is WSD and $Y_h=Y_k$ is bo-complete, verifying the
semivariation hypotheses. For a variation certificate,
\[
 s_\sigma(u_\varphi\setminus\ell_\varphi)
 \le h\wedge v_\sigma(u_\varphi\setminus\ell_\varphi)
 \le h\wedge\sup_i a_{i,\varphi(i)}
 =\sup_i(h\wedge a_{i,\varphi(i)}).
\]
The clipped rows decrease to zero and lie in $F_h$, proving
$\Sigma_v(\mathcal A)\subseteq\Sigma_s(\mathcal A)$.
The two extensions have the identical source exhaustion limit
$m_\sigma(u)$ for each $u\in\mathcal A_\sigma$ and therefore the
identical lower value $m_\delta(\ell)$ for each
$\ell\in\mathcal A_\delta$. For $b\in\Sigma_v(\mathcal A)$ their
full lower limits are
\[
 \bar m_s(b)
 =\mathrm{bo}\text{-}\lim_{\ell\in\mathcal A_\delta,\ \ell\le b}
                  m_\delta(\ell)=\bar m_v(b).
\]
If $k=0$, then $p(m(a))\le k=0$ gives $m(a)=0$.
For every $b\in\mathbb B$, the sandwich $0\le b\le1$ has zero shell variation.
The zero regulator gives both domains equal to $\mathbb B$, with both extensions zero.
\end{proof}

\begin{remark}[Quantifiers in the converse theorem]\label{cmp:km}
For a Dedekind complete $F$, Kusraev and Malyugin prove that WSD is equivalent to extension of every positive order-premeasure from every Boolean subalgebra $\mathcal A$ of every complete Boolean algebra $\mathbb B$ to the generated sigma-subalgebra $\sigma_{\mathbb B}(\mathcal A)$ \cite[Theorem~1.4.15]{KM}. It is also equivalent to this extension assertion for all finite-vector-variation bo-premeasures with values in all full bo-complete $F$-normed spaces. Here fullness means $p(Y)^{\perp\perp}=F$. For one fixed full bo-complete space whose vector norm is decomposable, the same universal source-extension assertion characterizes WSD \cite[Theorem~1.4.17]{KM}. Decomposability means
\[
 p(y)=f_1+f_2,\quad f_1,f_2\ge0
 \quad\Longrightarrow\quad
 \exists y_1,y_2:\ y=y_1+y_2,\quad p(y_i)=f_i.
\]
These are cited converse theorems, with their quantifiers over all source algebras and premeasures. Corollary~\ref{km:local-extension} concerns one fixed premeasure and only its variation band. Neither the cited converses nor this corollary asserts that a prescribed larger difference set must satisfy closed summation whenever that particular premeasure extends.
\end{remark}

We next express the source exhaustivity condition through a control on the generated sigma-algebra. Under the stated lattice hypotheses, every positive coordinate variation extends independently of hereditary exhaustivity of $m$. The supremum of these extensions therefore defines a canonical control before the vector extension has been constructed. We show that continuity from above of this control is equivalent to hereditary exhaustivity of the original premeasure.

\Needspace{10\baselineskip}
\begin{proposition}[Canonical control and intrinsic exhaustivity]\label{sem:canonical-control}\label{sem:he-equivalence}
Before hereditary exhaustivity is assumed, every $\nu_T$ has a unique positive extension $\widehat\nu_T$ to $\Sigma=\sigma(\mathcal H)$. Define
\[
 \rho(b)=\sup_T\widehat\nu_T(b)\in[0,h].
\]
Then $\rho|_{\mathcal H}=s$, $\rho(1)=h$, and the following are equivalent:
\begin{enumerate}
\item $m$ is uniformly hereditarily exhaustive.
\item $s(a_n)\to_o0$ for every disjoint source sequence.
\item $b_n\downarrow0$ in $\Sigma$ implies $\rho(b_n)\downarrow0$.
\end{enumerate}
Under these conditions $\rho=\bar s|_\Sigma$.
\end{proposition}
\begin{proof}
For each $T$ and each disjoint source sequence $(a_n)$, put
$t_T=\sup_N\sum_{n\le N}\nu_T(a_n)\le h$.
Positivity gives $0\le\nu_T(c)\le t_T-\sum_{n<N}\nu_T(a_n)\downarrow_N0$
whenever $c\le a_j$ and $j\ge N$.
Thus $\nu_T$ is hereditarily exhaustive. Set
\[
 M_T=\clo_o\nu_T(\mathcal H)\subseteq[0,h],\qquad
 D_T=[-2h,2h]\subseteq F_h.
\]
The interval is order-stable by \eqref{ord:carrier-bound}.
Theorem~\ref{ord:complete} makes $M_T$ complete, and
$M_T-M_T\subseteq[-h,h]\subseteq D_T$.
Lemma~\ref{sem:target}, with $W=F_h$, $N(f)=|f|$, bound $h$,
$K=M_T$ and $D=D_T$, proves
\[
 \mathsf B_\omega(D_T),\qquad
 \mathsf S_{\rm seq}(D_T),\qquad \mathsf U(M_T).
\]
The source map $\nu_T$ is a positive premeasure by
Lemma~\ref{sem:variation}. Theorem~\ref{gen:extension}, applied
to these data with ambient Boolean algebra $\mathbb B$ and source
ring $\mathcal H$, constructs an extension with values in $M_T$.
Restrict it to $\Sigma=\sigma_{\mathbb B}(\mathcal H)$ and denote
the restriction by $\widehat\nu_T$. The inclusion
$M_T\subseteq[0,h]$ proves positivity. The uniqueness clause of
Proposition~\ref{gen:uniqueness}, for this source map and range
$F_h$, proves uniqueness on $\Sigma$.
Consequently $\rho$ exists, $\rho|_{\mathcal H}=s$, and an exhaustion $e_n\uparrow1$ gives
\[
 \rho(1)=\sup_T\sup_n\nu_T(e_n)
 =\sup_ns(e_n)=\sup_{a\in\mathcal H}s(a)=h.
\]
For each source element $a$ and coordinate $T$, the increasing
sequence $a\wedge e_n$ has source supremum $a$, hence
\[
 \nu_T(a)=\sup_n\nu_T(a\wedge e_n)
 \le\sup_n\nu_T(e_n)\le\sup_ns(e_n).
\]
Taking the suprema over $T$ and then $a$ gives
$h=\sup_as(a)\le\sup_ns(e_n)$.
The inequalities $s(e_n)\le h$ give the reverse bound.
A dominated linear map $T_0:Y_h\to F_h$, viewed as $F$-valued,
satisfies $T_0y\le p(y)$. Hahn--Banach--Kantorovich, applied to
the sublinear map $p:Y\to F$, extends it to a linear $T:Y\to F$
with $Ty\le p(y)$. Applying this bound to $-y$ gives
$-Ty\le p(y)$, so $|Ty|\le p(y)$ on $Y$.
Conversely, a global dominated $T$ satisfies
$|Ty|\le p(y)\in F_h$ for $y\in Y_h$, and solidity gives
$Ty\in F_h$. Thus the restriction map from global dominated
coordinates onto $\mathscr T_h$ is surjective, and both coordinate
families produce the identical suprema in the definition of $\rho$.

For a disjoint $(a_n)$, boundedness and \eqref{sem:bounded-evaluation} give
\[
 \sup_{n\ge N}q(a_n)\downarrow0
 \ \Longleftrightarrow\
 \sup_{n\ge N}s(a_n)\downarrow0
 \ \Longleftrightarrow\ s(a_n)\to_o0,
\]
using $q\le s\le2q$. These are (1) and (2).
Under (3), put $b_N=\bigvee_{n\ge N}a_n\downarrow0$.
Then $p(m(c))\le s(c)=\rho(c)\le\rho(b_N)\downarrow0$ for $c\le a_n$ and $n\ge N$, proving (1).
Under (1), Theorem~\ref{sem:extension} applies to $m$ with the
stated base band $F_h$ and complete space $Y_h$. Both
$\widehat\nu_T$ and $\bar\nu_T|_\Sigma$ are countably additive
extensions of $\nu_T$, so the preceding uniqueness statement gives
$\widehat\nu_T=\bar\nu_T|_\Sigma$. Hence
\[
 \rho(b)=\sup_T\bar\nu_T(b)=Q\bar\lambda(b)=\bar s(b)
 \qquad(b\in\Sigma).
\]
For $b_n\downarrow0$, put $d_n=b_n\setminus b_{n+1}$.
The $d_n$ are disjoint and $\bigvee_nd_n=b_1$, whence
\[
 \bar\lambda(b_N)
 =\bar\lambda(b_1)-\sum_{n<N}\bar\lambda(d_n)
 \to_{\rm bo}0,\qquad \rho(b_N)=Q\bar\lambda(b_N)\to_o0.
\]
Since $\rho(b_N)$ decreases, order convergence to zero means its
infimum is zero. This proves (3).
\end{proof}

This continuity also allows approximation by the source ring. We combine the countably many approximation errors into one regulator for each element of the generated sigma-algebra. When the active band has the countable sup property, these bounds can be reduced to a sequence, giving sequential source approximation as well.

\begin{proposition}[Regulator density]\label{sem:regulator-density}\label{sem:canonical-density}
Let $r:\sigma(\mathcal H)\to[0,h]$ be monotone and finitely subadditive, with $r(0)=0$, and suppose $b_n\downarrow0$ implies $r(b_n)\downarrow0$. Every $b\in\sigma(\mathcal H)$ has source approximants $c_\varphi$ and a bounded regulator satisfying $r(b\mathbin\triangle c_\varphi)\le\sup_i a_{i,\varphi(i)}$. Existence of this certificate requires only Dedekind completeness. WSD makes its selected bounds decrease to zero.

Under Theorem~\ref{sem:extension}, every $b\in\Sigma$ has a source net $c_\varphi$, indexed by $\mathbb N^{\mathbb N}$, and a bounded regulator $a_{i,k}$ such that
\[
 \rho(b\mathbin\triangle c_\varphi)\le\sup_i a_{i,\varphi(i)}\downarrow_\varphi0,
 \qquad m(c_\varphi)\to_{\rm bo}\bar m(b).
\]
If $F_h$ has the countable sup property, a source sequence suffices. In that case sequential bo-completeness of $Y_h$ suffices for all preceding results. More generally, sequential bo-completeness implies bo-completeness over any Dedekind complete base with the countable sup property.
\end{proposition}
\begin{proof}
Let $r$ satisfy the stated hypotheses, and define $\mathcal C$ to be
exactly the set of elements admitting the stated regulator
approximation. A source element $c$ has the constant approximants
$c_\varphi=c$ and the zero regulator, so
$\mathcal H\subseteq\mathcal C$.
For $b,d\in\mathcal C$, write their approximants as
$c_\varphi,e_\psi\in\mathcal H$ and their bounds as
$q_\varphi=\sup_i a_{i,\varphi(i)}$ and
$t_\psi=\sup_i b_{i,\psi(i)}$. Define one regulator by
\[
 A_{2i-1,k}=h\wedge2a_{i,k},\qquad
 A_{2i,k}=h\wedge2b_{i,k}.
\]
For a selector $\Phi\in\mathbb N^{\mathbb N}$, put
$\varphi(i)=\Phi(2i-1)$ and $\psi(i)=\Phi(2i)$, and use the
source approximant $c_\varphi\setminus e_\psi$. The Boolean
inclusion
\[
 (b\setminus d)\mathbin\triangle(c_\varphi\setminus e_\psi)
 \le(b\mathbin\triangle c_\varphi)
       \vee(d\mathbin\triangle e_\psi)
\]
and finite subadditivity give
\[
 r\bigl((b\setminus d)\mathbin\triangle(c_\varphi\setminus e_\psi)\bigr)
 \le h\wedge(q_\varphi+t_\psi)
 \le h\wedge(2q_\varphi\vee2t_\psi)
 =\sup_j A_{j,\Phi(j)}.
\]
The middle inequality follows from
$q_\varphi,t_\psi\le q_\varphi\vee t_\psi$, and the final equality
uses preservation of bounded suprema by $f\mapsto h\wedge f$.
Every row $A_{j,k}$ decreases to zero, so $b\setminus d\in\mathcal C$.

For $b_n\in\mathcal C$, write $b=\bigvee_nb_n$ and
$B_N=b\setminus\bigvee_{n\le N}b_n$. Boolean distributivity gives
$B_N\downarrow0$, hence $u_N=r(B_N)\downarrow0$ by hypothesis.
Let $a_{n,i,k}$ be a regulator for $b_n$, with approximants
$c^n_\varphi$. On the row set $I=\{0\}\dot\cup\mathbb N^2$, use
\[
 d_{0,N}=h\wedge2u_N,\qquad d_{(n,i),k}=h\wedge2^{n+1}a_{n,i,k}.
\]
To index these rows by $\mathbb N$, put
\[
 \eta(0)=1,\qquad
 \eta(n,i)=1+\frac{(n+i-2)(n+i-1)}2+n.
\]
This is a bijection $I\to\mathbb N$. For a selector
$\Phi\in\mathbb N^{\mathbb N}$, set
$N=\Phi(1)$ and $\varphi_n(i)=\Phi(\eta(n,i))$.
Its source approximant is
$c=\bigvee_{n\le N}c^n_{\varphi_n}$.
The Boolean inclusion
\[
 b\mathbin\triangle c
 \le B_N\vee\bigvee_{n\le N}
                   (b_n\mathbin\triangle c^n_{\varphi_n})
\]
and the bounds $q_n=\sup_i a_{n,i,\varphi_n(i)}$ give
\[
 r(b\mathbin\triangle c)\le h\wedge\Bigl(u_N+\sum_{n\le N}q_n\Bigr)
 \le h\wedge\Bigl(2u_N\vee\bigvee_{n\le N}2^{n+1}q_n\Bigr)
 \le\sup_{j\in I}d_{j,\Phi(\eta(j))}.
\]
For the middle inequality, set
$w=2u_N\vee\bigvee_{n\le N}2^{n+1}q_n$. Then
\[
 u_N+\sum_{n\le N}q_n
 \le\left(\frac12+\sum_{n\le N}2^{-n-1}\right)w\le w.
\]
For the last inequality, distribute $h\wedge(\cdot)$ over the
finite join and then over each bounded supremum defining $q_n$.
Every row $d_{j,k}$ decreases to zero and is bounded by $h$.
Thus the reindexed matrix is a regulator for $b$, proving closure
under countable joins. We have proved
$\mathcal H\subseteq\mathcal C\subseteq\sigma(\mathcal H)$ and
closure of $\mathcal C$ under differences and countable joins.
An exhaustion of $1$ by source elements puts $1$ in $\mathcal C$, so $b^c=1\setminus b\in\mathcal C$ for $b\in\mathcal C$.
Therefore $\mathcal C=\sigma(\mathcal H)$. WSD makes every resulting selected bound decrease to zero. The canonical map $r=\rho$ has these hypotheses by Proposition~\ref{sem:he-equivalence}. For its approximants,
$p(\bar m(b)-m(c_\varphi))\le\bar s(b\mathbin\triangle c_\varphi)=\rho(b\mathbin\triangle c_\varphi)$.

The countable sup property, applied to the bounded family $\{-q_\varphi\}$ with supremum zero, gives selectors $\varphi_j$ with $\inf_jq_{\varphi_j}=0$.
Put $\psi_n(i)=\max_{j\le n}\varphi_j(i)$. Then $q_{\psi_n}\downarrow0$, giving the source sequence $c_{\psi_n}$.
For a bo-Cauchy net with regulator $r_\gamma\downarrow0$, the same property applied to $\{-r_\gamma\}$ gives indices $\eta_n$ with $\inf_nr_{\eta_n}=0$.
Set $\gamma_1=\eta_1$ and take $\alpha_1$ beyond its Cauchy threshold.
For $n\ge1$, take $\gamma_{n+1}\ge\gamma_n,\eta_{n+1}$ and then
$\alpha_{n+1}\ge\alpha_n$ beyond the $\gamma_{n+1}$-Cauchy threshold.
Denote the Cauchy threshold for $r_\gamma$ by $\lambda_\gamma$.
The recursion has $\alpha_n\ge\lambda_{\gamma_n}$ and
$\alpha_{n+1}\ge\alpha_n$. For $n,m\ge N$,
$p(y_{\alpha_n}-y_{\alpha_m})\le r_{\gamma_N}$.
Thus $r_{\gamma_n}\downarrow0$ and $(y_{\alpha_n})$ is bo-Cauchy.
Sequential completeness supplies its limit $y$. For every fixed
$N$ and every original index $\alpha\ge\lambda_{\gamma_N}$,
\[
 p(y_\alpha-y_{\alpha_n})\le r_{\gamma_N}\quad(n\ge N)
 \quad\Longrightarrow\quad p(y_\alpha-y)\le r_{\gamma_N}.
\]
The passage to the limit in this display uses
$|p(y_\alpha-y_{\alpha_n})-p(y_\alpha-y)|
\le p(y_{\alpha_n}-y)\to_o0$ and order stability of
$[0,r_{\gamma_N}]$. It proves
\[
 \forall N\ge1\ \forall\alpha\ge\lambda_{\gamma_N}:\qquad
 p(y_\alpha-y)\le r_{\gamma_N},\qquad r_{\gamma_N}\downarrow0.
\]
Thus the original net converges. No cofinality of $(\alpha_n)$ is used.
\end{proof}

We can now use the extended coordinate variations to determine the integral bounds for simple functions. Their supremum gives the least isotone submajorant and retains its original values on source simple functions. Countable additivity of the lifted measure then proves sequential order continuity of this submajorant and sequential order-to-bo continuity of the integration operator.

\begin{proposition}[The exact simple-function submajorant]\label{sem:exact-submajorant}\label{sem:submajorant-continuity}
For the Boolean simple-function lattice $S(\Sigma_s(\mathcal H))$, let
$T_{\bar m}(\sum_jt_j\mathbf1_{b_j})=\sum_jt_j\bar m(b_j)$. An isotone submajorant is an isotone map $H:S(\Sigma_s(\mathcal H))_+\to(F_h)_+$ satisfying $p(T_{\bar m}g)\le H(|g|)$ for every simple $g$. The exact submajorant is
\begin{equation}\label{sem:submajorant-scalarization}
 \overline{\mathcal Q}(f):=\sup_{|g|\le f}p(T_{\bar m}g)
 =\sup_T I_{\bar\nu_T}(f)\quad(f\ge0).
\end{equation}
This is the least isotone submajorant of $T_{\bar m}$. It is positively homogeneous and subadditive. It extends the corresponding source submajorant, satisfies $\overline{\mathcal Q}(\mathbf1_b)=\bar s(b)$, and is sequentially order continuous. The operator $T_{\bar m}$ is sequentially order-to-bo continuous.
\end{proposition}
\begin{proof}
All simple expressions are identified under common finite Boolean refinement, which makes their integrals well defined. For $f=\sum_jc_j\mathbf1_{b_j}\ge0$ and a finite-variation coordinate charge $v$, domination gives $|T_vg|\le I_{|v|}(f)$ when $|g|\le f$. Conversely, finite partitions $\pi_j$ of $b_j$ and signed functions with coefficients $c_j\varepsilon_a$ give
\[
 \bigvee_{\varepsilon_a=\pm1}|T_vg_\varepsilon|
 =\sum_jc_j\sum_{a\in\pi_j}|v(a)|.
\]
Independent partition suprema prove $\sup_{|g|\le f}|T_vg|=I_{|v|}(f)$. For $|g|\le f\le c\mathbf1_1$, semivariation domination gives
$p(T_{\bar m}g)\le ch$, so all the suprema below exist. Hence
\[
 \sup_{|g|\le f}p(T_{\bar m}g)
 =\sup_{|g|\le f}\sup_T|T(T_{\bar m}g)|
 =\sup_T I_{|T\bar m|}(f)=\sup_T I_{\bar\nu_T}(f).
\]
For a nonnegative source simple function
$f=\sum_{j=1}^rc_j\mathbf1_{a_j}$ on a finite disjoint source
family, the original-coordinate calculation gives
\[
 \mathcal Q(f)=\sup_T\sum_{j=1}^rc_j\nu_T(a_j),\qquad
 \overline{\mathcal Q}(f)
 =\sup_T\sum_{j=1}^rc_j\bar\nu_T(a_j)
 =\sup_T\sum_{j=1}^rc_j\nu_T(a_j).
\]
This proves exact restriction to the source simple-function lattice.
For $f,g\ge0$ and $t\ge0$, positivity and additivity of each
coordinate integral give
\[
 \begin{split}
 f\le g&\ \Longrightarrow\
 \sup_T I_{\bar\nu_T}(f)\le\sup_T I_{\bar\nu_T}(g),\\
 \sup_T I_{\bar\nu_T}(tf)&=t\sup_T I_{\bar\nu_T}(f),\\
 \sup_T I_{\bar\nu_T}(f+g)
 &\le\sup_T I_{\bar\nu_T}(f)+\sup_T I_{\bar\nu_T}(g).
 \end{split}
\]
For $t=0$ the middle equality is $0=0$. These are isotonicity,
positive homogeneity and subadditivity of $\overline{\mathcal Q}$.
For an isotone submajorant $H$ and $|g|\le f$,
$p(T_{\bar m}g)\le H(|g|)\le H(f)$.
The supremum over $g$ gives $\overline{\mathcal Q}(f)\le H(f)$,
proving minimality. The indicator identity is
\[
 \overline{\mathcal Q}(\mathbf1_b)
 =\sup_T I_{\bar\nu_T}(\mathbf1_b)
 =\sup_T\bar\nu_T(b)=\bar s(b).
\]

Write $\bar\rho=\bar s$. For $b_n\downarrow0$ in the full domain,
the disjoint elements $d_n=b_n\setminus b_{n+1}$ have join $b_1$.
Consequently
\[
 \bar\lambda(b_N)=\bar\lambda(b_1)-
              \sum_{n<N}\bar\lambda(d_n)\to_{\rm bo}0,\qquad
 \bar\rho(b_N)=Q\bar\lambda(b_N)\to_o0.
\]
Monotonicity gives $\bar\rho(b_N)\downarrow0$.
If $f_n\downarrow0$ in the simple-function lattice, with
$f_1\le c\mathbf1_1$, put $e_{n,k}=\{f_n>1/k\}$ and
$e_k=\bigwedge_ne_{n,k}$. For every $n$,
$0\le k^{-1}\mathbf1_{e_k}\le f_n$.
The defining infimum of $(f_n)$ is zero, hence $e_k=0$.
Thus $e_{n,k}\downarrow_n0$. The pointwise Boolean-simple bound
$0\le f_n\le k^{-1}\mathbf1_1+c\mathbf1_{e_{n,k}}$ gives
\[
 0\le\overline{\mathcal Q}(f_n)\le k^{-1}h+c\bar\rho(e_{n,k}),
 \qquad 0\le\inf_n\overline{\mathcal Q}(f_n)\le k^{-1}h\quad(k\ge1),
\]
The infimum over $k$ is zero since a Dedekind complete real vector
lattice is Archimedean. Thus
$\overline{\mathcal Q}(f_n)\downarrow0$.
For an order-null sequence $g_n$, let $u_\gamma\downarrow0$ be a
regulator in the simple-function lattice, with thresholds
$N_\gamma$ such that $|g_n|\le u_\gamma$ for $n\ge N_\gamma$.
Fix one index $\gamma_0$. Its simple function $u_{\gamma_0}$ has a
finite upper bound $c\mathbf1_1$, and set $N_0=N_{\gamma_0}$.
For $N\ge N_0$, put
$b_{N,k}=\bigvee_{n\ge N}\{|g_n|>1/k\}$ and
$b_k=\bigwedge_{N\ge N_0}b_{N,k}$.
For $N\ge\max\{N_0,N_\gamma\}$, every element in this join is
bounded by $\{u_\gamma>1/k\}$, so
\[
 k^{-1}\mathbf1_{b_k}
 \le k^{-1}\mathbf1_{b_{N,k}}
 \le u_\gamma\qquad(\gamma).
\]
Taking the infimum over $\gamma$ gives $b_k=0$. For
$n\ge N\ge N_0$ the simple-function bound
$|g_n|\le k^{-1}\mathbf1_1+c\mathbf1_{b_{N,k}}$ gives
\[
 \sup_{n\ge N}\overline{\mathcal Q}(|g_n|)
 \le k^{-1}h+c\bar\rho(b_{N,k})
\]
Define $v_N=\sup_{n\ge N}\overline{\mathcal Q}(|g_n|)$.
These suprema exist because $0\le v_N\le ch$. For each $k$,
\[
 0\le\inf_{N\ge N_0}v_N
 \le\inf_{N\ge N_0}(k^{-1}h+c\bar\rho(b_{N,k}))
 =k^{-1}h.
\]
The infimum over $k$ gives $v_N\downarrow0$. In particular,
\[
 \forall N\ge N_0\ \forall n\ge N:\qquad
 p(T_{\bar m}g_n)\le\overline{\mathcal Q}(|g_n|)\le v_N,
\]
which proves sequential order-to-bo continuity of $T_{\bar m}$.
Finally, for nonnegative $f,g$ the inequalities
$f\le g+|f-g|$ and $g\le f+|f-g|$ imply
\[
 |\overline{\mathcal Q}(f)-\overline{\mathcal Q}(g)|
 \le\overline{\mathcal Q}(|f-g|).
\]
If $f_n\to_o f$ in the nonnegative simple-function cone, apply
the proved null-sequence estimate to $g_n=f_n-f$. It bounds the
last right-hand side by a decreasing null tail regulator, and
therefore proves $\overline{\mathcal Q}(f_n)\to_o
\overline{\mathcal Q}(f)$. This proves the full sequential order
continuity assertion.
\end{proof}

\begin{example}[Finite semivariation and infinite variation]\label{sem:harmonic-example}
On the finite--cofinite algebra $\mathcal H$ of $\mathbb N$, the map
\[
 m(a)=(\mathbf1_a(j)/j)_{j\ge1}\in c_0
\]
satisfies the semivariation theorem with $F=\mathbb R$. Its regulator domain is $\mathcal P(\mathbb N)$ and its extension has the same formula. Its norm variation is infinite.
\end{example}
\begin{proof}
For disjoint source cells and $|t_i|\le1$,
\[
 \Bigl\|\sum_it_im(a_i)\Bigr\|_\infty\le
 \sup_{j\in\bigcup_ia_i}j^{-1}.
\]
The single-cell coefficient $1$ gives the reverse bound, so
$s(a)=\sup_{j\in a}j^{-1}$, with $s(0)=0$ and $h=1$. Finite additivity holds coordinatewise. If $a=\bigcup_na_n\in\mathcal H$ is disjoint, then, for each $K$, its finitely many members in $\{1,\ldots,K\}$ lie in some finite partial union $\bigcup_{n\le N_K}a_n$. Thus
\[
 N\ge N_K\quad\Longrightarrow\quad
 \Bigl\|m(a)-\sum_{n\le N}m(a_n)\Bigr\|_\infty
 \le (K+1)^{-1},
\]
proving the premeasure property. For an arbitrary disjoint source sequence, each $j\le K$ belongs to at most one $a_n$. Hence some $N_K$ satisfies $a_n\cap\{1,\ldots,K\}=\varnothing$ for all $n\ge N_K$, giving
\[
 \sup_{n\ge N_K,\ c\le a_n}\|m(c)\|_\infty\le(K+1)^{-1}.
\]
This proves uniform hereditary exhaustivity directly. The Banach space $c_0$ is complete and the base $\mathbb R$ is Dedekind complete and WSD.

Every subset of $\mathbb N$ is a countable union of source singletons, so $\mathcal H_\sigma=\mathcal H_\delta=\mathcal P(\mathbb N)$. For any $b$, the identical lower and upper approximants $b$ have zero shell and certify $b\in\Sigma_s(\mathcal H)$. Its finite truncations satisfy
\[
 \left\|m(b\cap\{1,\ldots,N\})-(\mathbf1_b(j)/j)_j\right\|_\infty
 \le(N+1)^{-1},
\]
which identifies the upper value and hence the extension. Finally, the partition of $1$ into its first $N$ singletons and their complement has variation sum at least $\sum_{j\le N}j^{-1}$. For $N=2^r$,
\[
 \sum_{j=1}^{2^r}j^{-1}\ge
 \sum_{k=0}^{r-1}\sum_{j=2^k+1}^{2^{k+1}}2^{-k-1}=r/2.
\]
The variation sums are unbounded.
\end{proof}

The example exhibits an extension under finite semivariation for which vector variation is infinite. For premeasures with finite vector variation, the construction recovers the extension on the full Kusraev--Malyugin regulator domain. Both applications use weak sigma-distributivity only in the band generated by the source bound, and their approximation limits remain limits in the original vector norm.

\section{Scope of the extension results}\label{sec:scope}

The uncontrolled theorem proves extension of modular measures from a $D$-sublattice to the ambiently generated $\sigma$-$D$-sublattice using only convergence conditions in the range. Its defining limits, representation independence, normality and countable additivity are conclusions of the construction. The relative theorem proves a common lifting result for this modular mechanism and the nonmodular positive mechanism. Its cone assumptions are verified separately in the positive, lattice-normed, product and subgroup realizations.

The two Sion processes exist on the ambient pseudo-$D$-lattice and agree with the extension on its generated domain. The Boolean theorem additionally identifies the full splitting class with a regulated $\sigma$-algebra. The finite-variation and finite-semivariation preservation results use that Boolean domain. These are distinct domain conclusions, and no identification of them is needed for the common lifting theorem or its corollaries.

The range principles are sufficient conditions. Apart from the stated equivalences for ordered ranges, the paper does not characterize all convergences that admit extension. In particular, removing the cone from the nonmodular lifting theorem would require a further argument beyond the uncontrolled modular construction proved here.


\begingroup
\small
\begin{thebibliography}{99}
\setlength{\itemsep}{2pt}
\setlength{\parskip}{0pt}
\bibitem{AB}
C.~D. Aliprantis and O. Burkinshaw,
\emph{Positive Operators}, Springer, Dordrecht, 2006.
\href{https://doi.org/10.1007/978-1-4020-5008-4}{doi:10.1007/978-1-4020-5008-4}.

\bibitem{Birkhoff}
G. Birkhoff, \emph{Lattice Theory}, third edition,
American Mathematical Society Colloquium Publications, vol.~25,
American Mathematical Society, Providence, 1967.
\href{https://doi.org/10.1090/coll/025}{doi:10.1090/coll/025}.

\bibitem{Fremlin75}
D.~H. Fremlin, A direct proof of the Matthes--Wright integral extension theorem,
\emph{J. London Math. Soc.} (2) \textbf{11} (1975), no.~3, 276--284.
\href{https://doi.org/10.1112/jlms/s2-11.3.276}{doi:10.1112/jlms/s2-11.3.276}.

\bibitem{Sion}
M. Sion, Outer measures with values in a topological group,
\emph{Proc. London Math. Soc.} (3) \textbf{19} (1969), no.~1, 89--106.
\href{https://doi.org/10.1112/plms/s3-19.1.89}{doi:10.1112/plms/s3-19.1.89}.

\bibitem{Volauf77}
P. Volauf, Extension and regularity of $\ell$-group valued measures,
\emph{Math. Slovaca} \textbf{27} (1977), no.~1, 47--53.
\url{https://dml.cz/handle/10338.dmlcz/129443}.

\bibitem{Volauf80}
P. Volauf, On extension of maps with values in ordered spaces,
\emph{Math. Slovaca} \textbf{30} (1980), no.~4, 351--361.
\url{https://dml.cz/handle/10338.dmlcz/129034}.

\bibitem{Wright}
J.~D. Maitland Wright, The measure extension problem for vector lattices,
\emph{Ann. Inst. Fourier (Grenoble)} \textbf{21} (1971), no.~4, 65--85.
\href{https://doi.org/10.5802/aif.393}{doi:10.5802/aif.393}.

\bibitem{OTW}
M. O'Brien, V.~G. Troitsky, and J.~H. van der Walt,
Net convergence structures with applications to vector lattices,
\emph{Quaestiones Mathematicae} \textbf{46} (2023), no.~2, 243--280.
\href{https://doi.org/10.2989/16073606.2021.2012721}{doi:10.2989/16073606.2021.2012721}.

\bibitem{BCTW}
E. Bilokopytov, J. Conradie, V.~G. Troitsky, and J.~H. van der Walt,
Locally solid convergence structures,
\emph{Dissertationes Mathematicae}, to appear.
\href{https://arxiv.org/abs/2404.15641}{arXiv:2404.15641} (2024).

\bibitem{KM}
A.~G. Kusraev and S.~A. Malyugin,
\emph{Some Questions of the Theory of Vector Measures},
Institute of Mathematics, Siberian Branch of the USSR Academy of Sciences,
Novosibirsk, 1988, 182 pp. (Russian).


\bibitem{KKHBK}
A.~G. Kusraev and S.~S. Kutateladze,
Kantorovich spaces and optimization,
\emph{J. Math. Sci. (N.Y.)} \textbf{133} (2006), no.~4, 1449--1455.
\href{https://doi.org/10.1007/s10958-006-0060-7}{doi:10.1007/s10958-006-0060-7}.

\bibitem{Kupka78}
J. Kupka, The Carath\'eodory extension theorem for vector valued measures,
\emph{Proc. Amer. Math. Soc.} \textbf{72} (1978), no.~1, 57--61.
\href{https://doi.org/10.1090/S0002-9939-1978-0493327-7}{doi:10.1090/S0002-9939-1978-0493327-7}.

\bibitem{Weber76}
H. Weber, Fortsetzung von Ma{\ss}en mit Werten in uniformen Halbgruppen,
\emph{Arch. Math. (Basel)} \textbf{27} (1976), no.~4, 412--423.
\href{https://doi.org/10.1007/BF01224694}{doi:10.1007/BF01224694}.

\bibitem{FoxMorales83}
G. Fox and P. Morales, On uniform semigroup-valued additive set functions,
\emph{Canad. Math. Bull.} \textbf{26} (1983), no.~1, 26--34.
\href{https://doi.org/10.4153/CMB-1983-005-4}{doi:10.4153/CMB-1983-005-4}.

\bibitem{Trombetta05}
A. Trombetta, On the construction of outer measures with values in a uniform semigroup,
\emph{Math. Slovaca} \textbf{55} (2005), no.~3, 263--269.
\url{https://dml.cz/handle/10338.dmlcz/131587}.

\bibitem{Smirnov26}
A.~G. Smirnov and M.~S. Smirnov, Measure theory without infinities,
preprint, arXiv:2609.03875, 2026.
\url{https://arxiv.org/abs/2609.03875}.

\bibitem{Sipos90}
J. \v Sipo\v s, On extension of group valued measures,
\emph{Math. Slovaca} \textbf{40} (1990), no.~3, 279--286.
\url{https://dml.cz/handle/10338.dmlcz/131920}.

\bibitem{Kratochvil77}
P. Kratochv\'il, Multisequences and measures,
in J. Nov\'ak (ed.), \emph{General Topology and Its Relations to Modern Analysis and Algebra~IV},
Part~B: Contributed Papers (Prague, 1976),
Society of Czechoslovak Mathematicians and Physicists, Prague, 1977, 237--244.
\url{https://dml.cz/handle/10338.dmlcz/700679}.

\bibitem{Fric05}
R. Fri\v c, Extension of measures: a categorical approach,
\emph{Math. Bohem.} \textbf{130} (2005), no.~4, 397--407.
\href{https://doi.org/10.21136/MB.2005.134212}{doi:10.21136/MB.2005.134212}.

\bibitem{BoccutoSambucini98}
A. Boccuto and A.~R. Sambucini, Abstract integration in convergence groups,
\emph{Atti Sem. Mat. Fis. Univ. Modena} \textbf{46} (1998), 315--333.
\url{https://sambucini.sites.dmi.unipg.it/artic/BCCSMB.pdf}.

\bibitem{Halmos}
P.~R. Halmos, \emph{Measure Theory}, Graduate Texts in Mathematics, vol.~18,
Springer-Verlag, New York, 1974, reprint of the 1950 edition.
\href{https://doi.org/10.1007/978-1-4684-9440-2}{doi:10.1007/978-1-4684-9440-2}.

\bibitem{Gaina63}
S. G\u{a}in\u{a}, Extension of vector measures,
\emph{Rev. Math. Pures Appl.} \textbf{8} (1963), no.~1, 151--154.

\bibitem{Kluvanek61}
I. Kluv\'anek, Contribution to the theory of vector measures (Russian),
\emph{Mat.-Fyz. \v Casopis} \textbf{11} (1961), no.~3, 173--191.
\url{https://dml.cz/handle/10338.dmlcz/126689}.

\bibitem{Kluvanek73}
I. Kluv\'anek, The extension and closure of vector measure,
in D.~H. Tucker and H.~B. Maynard (eds.),
\emph{Vector and Operator Valued Measures and Applications},
Academic Press, New York--London, 1973, 175--190.
\href{https://doi.org/10.1016/B978-0-12-702450-9.50024-5}{doi:10.1016/B978-0-12-702450-9.50024-5}.

\bibitem{Dinculeanu67}
N. Dinculeanu, \emph{Vector Measures},
International Series of Monographs in Pure and Applied Mathematics, vol.~95,
Pergamon Press, Oxford, 1967.
\href{https://doi.org/10.1016/C2013-0-07847-4}{doi:10.1016/C2013-0-07847-4}.

\bibitem{DiestelUhl}
J. Diestel and J.~J. Uhl, Jr., \emph{Vector Measures},
Mathematical Surveys, vol.~15, American Mathematical Society, Providence, RI, 1977.
\href{https://doi.org/10.1090/surv/015}{doi:10.1090/surv/015}.

\bibitem{DAndreaDeLuciaWright94}
A.~B. d'Andrea, P. de Lucia, and J.~D. Maitland Wright,
On Kalmbach measurability,
\emph{Appl. Math.} \textbf{39} (1994), no.~6, 445--447.
\href{https://doi.org/10.21136/AM.1994.134270}{doi:10.21136/AM.1994.134270}.

\bibitem{Avallone06}
A. Avallone, A. De Simone, and P. Vitolo,
Effect algebras and extensions of measures,
\emph{Boll. Unione Mat. Ital. Sez. B} (8) \textbf{9-B} (2006), no.~2, 423--444.
\url{https://www.bdim.eu/item?id=BUMI_2006_8_9B_2_423_0}.

\bibitem{Barbieri09}
G. Barbieri, An extension theorem for modular measures on effect algebras,
\emph{Czechoslovak Math. J.} \textbf{59} (2009), no.~3, 707--719.
\href{https://doi.org/10.1007/s10587-009-0042-2}{doi:10.1007/s10587-009-0042-2}.

\bibitem{Barbieri14}
G. Barbieri, The extension of measures on $D$-lattices,
\emph{Fuzzy Sets and Systems} \textbf{244} (2014), 123--129.
\href{https://doi.org/10.1016/j.fss.2013.07.002}{doi:10.1016/j.fss.2013.07.002}.

\bibitem{Avallone16}
A. Avallone, A. De Simone, and P. Vitolo,
Extension of measures on pseudo-$D$-lattices,
\emph{Math. Slovaca} \textbf{66} (2016), no.~2, 421--438.
\href{https://doi.org/10.1515/ms-2015-0147}{doi:10.1515/ms-2015-0147}.
\bibitem{DvurecenskijVetterlein01}
A. Dvure\v{c}enskij and T. Vetterlein,
Pseudoeffect algebras. I. Basic properties,
\emph{Internat. J. Theoret. Phys.} \textbf{40} (2001), 685--701.
\href{https://doi.org/10.1023/A:1004192715509}{doi:10.1023/A:1004192715509}.

\bibitem{AvalloneVitolo12}
A. Avallone and P. Vitolo,
Pseudo-$D$-lattices and topologies generated by measures,
\emph{Ital. J. Pure Appl. Math.} \textbf{29} (2012), 25--42.
\url{https://journals.uniurb.it/index.php/ijpam/article/view/6063}.

\bibitem{AvalloneVitolo12Uniformities}
A. Avallone and P. Vitolo,
Lattice uniformities on pseudo-$D$-lattices,
\emph{Math. Slovaca} \textbf{62} (2012), no.~6, 1019--1044.
\href{https://doi.org/10.2478/s12175-012-0062-5}{doi:10.2478/s12175-012-0062-5}.

\bibitem{Weber18}
H. Weber, Extension of modular functions and measures,
\emph{Ann. Mat. Pura Appl.} (4) \textbf{197} (2018), no.~3, 761--774.
\href{https://doi.org/10.1007/s10231-017-0703-y}{doi:10.1007/s10231-017-0703-y}.
\bibitem{Birkhoff48}
G. Birkhoff, \emph{Lattice Theory}, revised edition,
American Mathematical Society Colloquium Publications, vol.~25,
American Mathematical Society, New York, 1948.


\end{thebibliography}
\endgroup
\end{document}